\pdfoutput=1
\RequirePackage{silence}
\documentclass[a4paper,11pt]{amsart}
\usepackage[hmarginratio={1:1},vmarginratio={1:1},lmargin=70.0pt,tmargin=70.0pt]{geometry}

\usepackage{lmodern}
\usepackage{ifdraft}
\ifdraft{\usepackage[draft]{showkeys}}{\usepackage[final]{showkeys}}
\usepackage{calligra}
\usepackage[T1]{fontenc}
\usepackage[ugly]{nicefrac}
\usepackage{mathtools}
\mathtoolsset{mathic}
\usepackage{multicol}
\usepackage{tabularx,tabu}
\usepackage{latexsym,exscale,amsfonts,amssymb,mathtools}
\usepackage{amsmath,amsthm,amsfonts,amssymb,amscd,textcomp,bbm}
\usepackage{stmaryrd}
\SetSymbolFont{stmry}{bold}{U}{stmry}{m}{n}
\usepackage[normalem]{ulem}
\usepackage{thmtools}
\usepackage{enumitem}
\setlist[enumerate]{itemsep=0.15cm,label=\emph{\upshape(\alph*)}}
\usepackage[table]{xcolor}
\usepackage{mathrsfs}
\usepackage{array}
\newcolumntype{C}{>{$}c<{$}}
\usepackage{colortbl}
\usepackage{caption} 
\usepackage[yyyymmdd,hhmmss]{datetime}

\allowdisplaybreaks

\let\emph\relax
\DeclareTextFontCommand{\emph}{\bfseries\em}

\DeclareMathAlphabet{\mathscrbf}{OMS}{mdugm}{b}{n}
\DeclareFontEncoding{LS1}{}{}
\DeclareFontSubstitution{LS1}{stix}{m}{n}
\DeclareMathAlphabet{\mathpzc}{LS1}{stixscr}{m}{n}

\newcommand{\ie}{i.e.}
\newcommand{\eg}{e.g.}
\newcommand{\cf}{cf.}
\newcommand{\etc}{etc.}

\newcommand{\ver}{verbatim}

\newcommand{\muta}{mutatis mutandis}

\definecolor{mygray}{gray}{0.6}
\definecolor{mygraydark}{gray}{0.4}
\definecolor{mygraylight}{gray}{0.8}
\definecolor{specialgray}{RGB}{135,135,200}

\definecolor{cherry}{RGB}{222,49,99}
\definecolor{cream}{RGB}{255,253,208}
\definecolor{corn}{RGB}{251,236,93}
\definecolor{citron}{RGB}{190,180,90}

\definecolor{spinach}{RGB}{46,139,87}
\definecolor{tomato}{RGB}{255,99,71}
\definecolor{pumpkin}{RGB}{224,180,80}

\definecolor{orchid}{RGB}{143,40,194}
\definecolor{lava}{RGB}{207,16,32}
\definecolor{mydarkblue}{RGB}{10,10,150}

\definecolor{myorange}{RGB}{225,127,0}
\definecolor{mygreen}{RGB}{0,225,0}
\definecolor{mypurple}{RGB}{128,0,128}
\definecolor{myred}{RGB}{255,0,0}
\definecolor{myblue}{RGB}{0,0,195}
\definecolor{myyellow}{RGB}{210,210,0}

\usepackage[all]{xy}
\SelectTips{cm}{}

\usepackage{tikz}
\usetikzlibrary{cd}
\usetikzlibrary{decorations}
\usetikzlibrary{decorations.markings}
\usetikzlibrary{decorations.pathreplacing}
\usetikzlibrary{decorations.pathmorphing}
\usetikzlibrary{arrows.meta,shapes,positioning,matrix,calc}
\usetikzlibrary{shapes.callouts}
\tikzstyle{densely dotted}=[dash pattern=on \pgflinewidth off 0.5pt]
\tikzset{anchorbase/.style={baseline={([yshift=-0.5ex]current bounding box.center)}},
tinynodes/.style={font=\tiny,text height=0.25ex,text depth=0.05ex},
smallnodes/.style={font=\scriptsize,text height=0.75ex,text depth=0.15ex},
usual/.style={line width=0.9,color=black},
dusual/.style={line width=0.9,color=spinach,densely dashed},
pole/.style={line width=3.0,color=specialgray},
crossline/.style={preaction={draw=white,line width=5.0pt,-},preaction={draw=black,line width=0.9pt,-}},
crosspole/.style={preaction={draw=white,line width=6.0pt,-},preaction={draw=specialgray,line width=3.0pt,-}},
mor/.style={line width=0.75,color=black,fill=cream},
blob/.style={circle,fill,minimum size=5.0pt,inner sep=0pt,outer sep=0pt},
blobbed/.style n args={3}{decoration={markings,post length=0.5mm,pre length=0.5mm,
mark=at position #1 with {\node[blob,#3,label=left:$#2\!$]at (0,0){};}
},postaction={decorate}},
rblobbed/.style n args={3}{decoration={markings,post length=0.5mm,pre length=0.5mm,
mark=at position #1 with {\node[blob,#3,label=right:$\!#2$]at (0,0){};}
},postaction={decorate}},
}

\renewcommand{\dots}{\text{...}}
\renewcommand{\vdots}{\raisebox{-0.05cm}{\rotatebox{90}{\text{...}}}}
\renewcommand{\ddots}{\raisebox{0.175cm}{\rotatebox{-45}{\text{...}}}}
\newcommand{\placeholder}{{}_{-}}

\newcommand{\leftsquigarrow}{\,\reflectbox{$\rightsquigarrow$}\,}

\newcommand{\actsleft}{\mathop{\,\;\raisebox{1.7ex}{\rotatebox{-90}{$\circlearrowright$}}\;\,}}
\newcommand{\actsright}{\mathop{\,\;\raisebox{0ex}{\rotatebox{90}{$\circlearrowleft$}}\;\,}}

\newcommand{\setstuff}[1]{\mathrm{#1}}

\newcommand{\obstuff}[1]{\mathtt{#1}}
\newcommand{\morstuff}[1]{\mathrm{#1}}

\newcommand{\idmor}{\morstuff{id}}

\newcommand{\Hom}{\setstuff{Hom}}

\newcommand{\hcirc}{\otimes}

\newcommand{\Z}{\mathbb{Z}}

\newcommand{\KK}{\mathbbm{K}}
\newcommand{\KKL}{\mathbbm{L}}
\newcommand{\N}{\mathbb{N}}

\newcommand{\sand}[1][\lambda]{\mathbb{S}_{#1}}
\newcommand{\sandbasis}[1][\lambda]{\setstuff{B}_{#1}}

\newcommand{\bsym}[1]{\boldsymbol{#1}}
\newcommand{\varsym}[1]{\mathtt{#1}}

\newcommand{\zvar}{\varsym{z}}
\newcommand{\Zf}{\Z_{\zvar,Y}}
\newcommand{\jm}{L}
\newcommand{\jmc}{l}

\newcommand{\qvar}{\varsym{q}}

\newcommand{\cpar}{\bsym{c}}
\newcommand{\cvar}{\varsym{c}}

\newcommand{\bpar}{\bsym{b}}
\newcommand{\bvar}{\varsym{b}}
\newcommand{\dpar}{\bsym{d}}
\newcommand{\dvar}{\varsym{d}}
\newcommand{\bbvar}{\varsym{BN}}

\newcommand{\avar}{\varsym{a}}

\newtheorem{theoremm}{Theorem}[section]

\declaretheoremstyle[
headfont=\bfseries, 
notebraces={[}{]},
bodyfont=\normalfont\itshape,
headpunct={},
postheadspace=1em,
spacebelow=10pt,
spaceabove=10pt, 
]{ourtheo}

\declaretheoremstyle[
headfont=\normalfont\bfseries,
notefont=\mdseries,
notebraces={(}{)},
bodyfont=\normalfont\slshape,
headpunct={},
postheadspace=1em,
spacebelow=10pt,
spaceabove=10pt, 
]{ourdef}

\declaretheorem[style=ourtheo,name=Theorem,numberlike=theoremm]{theorem}
\declaretheorem[style=ourtheo,name=Lemma,numberlike=theoremm]{lemma}
\declaretheorem[style=ourtheo,name=Proposition,numberlike=theoremm]{proposition}

\declaretheorem[style=ourtheo,name=Lemma,qed=$\blacksquare$,numberlike=theoremm]{lemmaqed}

\declaretheorem[style=ourtheo,name=Corollary,qed=$\blacksquare$,numberlike=theoremm]{corollary}

\declaretheorem[style=ourdef,name=Definition,numberlike=theorem]{definition}
\declaretheorem[style=ourdef,name=Example,numberlike=theorem]{example}
\declaretheorem[style=ourdef,name=Remark,numberlike=theorem]{remark}
\declaretheorem[style=ourdef,name=Convention,numberlike=theorem]{convention}

\allowdisplaybreaks

\numberwithin{equation}{section}
\usepackage[hypertexnames=false]{hyperref}
\usepackage{cleveref,bookmark}
\hypersetup{
pdftoolbar=true,
pdfmenubar=true,
pdffitwindow=false,
pdfstartview={FitH},
pdftitle={Handlebody diagram algebras},
pdfauthor={Daniel Tubbenhauer and Pedro Vaz},
pdfsubject={},
pdfcreator={Daniel Tubbenhauer and Pedro Vaz},
pdfproducer={Daniel Tubbenhauer and Pedro Vaz},
pdfkeywords={},
pdfnewwindow=true,
colorlinks=true,
linkcolor=mydarkblue,
citecolor=teal,
filecolor=magenta,
urlcolor=orchid,
linkbordercolor=lava,
citebordercolor=teal,
urlbordercolor=orchid,
linktocpage=true
}

\renewcommand{\theequation}{\thesection-\arabic{equation}}
\let\fullref\autoref
\def\makeautorefname#1#2{\expandafter\def\csname#1autorefname\endcsname{#2}}
\makeautorefname{equation}{Equation}%
\makeautorefname{footnote}{footnote}%
\makeautorefname{item}{item}%
\makeautorefname{figure}{Figure}%
\makeautorefname{table}{Table}%
\makeautorefname{part}{Part}%
\makeautorefname{appendix}{Appendix}%
\makeautorefname{chapter}{Chapter}%
\makeautorefname{section}{Section}%
\makeautorefname{subsection}{Section}%
\makeautorefname{subsubsection}{Section}%
\makeautorefname{paragraph}{Paragraph}%
\makeautorefname{subparagraph}{Paragraph}%
\makeautorefname{theorem}{Theorem}%
\makeautorefname{addendum}{Addendum}%
\makeautorefname{maintheorem}{Main theorem}%
\makeautorefname{corollary}{Corollary}%
\makeautorefname{lemma}{Lemma}%
\makeautorefname{sublemma}{Sublemma}%
\makeautorefname{proposition}{Proposition}%
\makeautorefname{property}{Property}
\makeautorefname{conjecture}{Conjecture}%
\makeautorefname{question}{Question}
\makeautorefname{definition}{Definition}%
\makeautorefname{notation}{Notation}
\makeautorefname{remark}{Remark}%
\makeautorefname{remarks}{Remarks}%
\makeautorefname{example}{Example}%
\makeautorefname{algorithm}{Algorithm}%
\makeautorefname{axiom}{Axiom}%
\makeautorefname{claim}{Claim}%
\makeautorefname{assumption}{Assumption}%
\makeautorefname{conclusion}{Conclusion}%
\makeautorefname{condition}{Condition}%
\makeautorefname{construction}{Construction}%
\makeautorefname{criterion}{Criterion}%
\makeautorefname{exercise}{Exercise}%
\makeautorefname{problem}{Problem}%
\makeautorefname{convention}{Convention}%

\newcommand{\nnfootnote}[1]{%
\begin{NoHyper}
\renewcommand\thefootnote{}\footnote{#1}%
\addtocounter{footnote}{-1}%
\end{NoHyper}
}

\title[Handlebody diagram algebras]
{Handlebody diagram algebras}
\author[Daniel Tubbenhauer and Pedro Vaz]{Daniel Tubbenhauer and Pedro Vaz}

\address{D.T.: The University of Sydney,
School of Mathematics and Statistics,
F07 - Carslaw Building,
Office Carslaw 827,
NSW 2006, Australia, \href{www.dtubbenhauer.com}{www.dtubbenhauer.com}}
\email{daniel.tubbenhauer@sydney.edu.au}

\address{P.V.: Institut de recherche en math\'{e}matique et physique (IRMP), Universit\'{e} 
catholique de Louvain, Chemin du Cyclotron 2, building L7.01.02, 1348 Louvain-la-Neuve, Belgium,\newline \href{https://perso.uclouvain.be/pedro.vaz/}{https://perso.uclouvain.be/pedro.vaz/}}
\email{pedro.vaz@uclouvain.be}

\begin{document}

\begin{abstract}
In this paper we study handlebody versions 
of some classical diagram algebras, most prominently, 
handlebody versions of Temperley--Lieb, blob, 
Brauer, BMW, Hecke and Ariki--Koike algebras.
Moreover, motivated by Green--Kazhdan--Lusztig's 
theory of cells, we reformulate the notion of 
(sandwich, inflated or affine) cellular algebras.
We explain this reformulation
and how all of the above algebras 
are part of this theory.
\end{abstract}

\nnfootnote{\textit{Mathematics Subject Classification 2020.} Primary: 16G99; Secondary: 18M30, 20C08, 20F36, 57K14.}
\nnfootnote{\textit{Keywords.} Diagrammatic algebra, cellular algebras, braid groups, handlebodies.}

\renewcommand{\theequation}{\thesection-\arabic{equation}}

\addtocontents{toc}{\protect\setcounter{tocdepth}{1}}

\maketitle

\tableofcontents

\section{Introduction}\label{section:intro}

A large collection of diagram algebras, such as Temperley--Lieb 
or (type A) Hecke algebras, are interesting from at least two perspectives:
they are of fundamental importance in low-dimensional topology 
and they also have a rich representation theory.

Having an eye on the study of links in 
3-manifolds brings the topology of the ambient space into play.
In all the classical examples, like Temperley--Lieb or 
Hecke algebras, the ambient space is the 3-ball, and these algebras 
are related to spherical Coxeter groups and Artin--Tits braid groups.
In the simplest case beyond the 3-ball, 
when one passes to links in a solid torus, 
these diagram algebras get replaced by their (extended) affine versions, 
and the related objects are now affine Coxeter groups and Artin--Tits braid groups.
A natural question is what kind of diagrammatics and Coxeter combinatorics 
one could expect for more general 3-manifolds.

In this paper we consider (three-dimensional) handlebodies of genus $g$. 
The 3-ball and the solid torus correspond to 
$g=0$ (called classical in this paper) and $g=1$, respectively.
As we will see, the Temperley--Lieb and Hecke algebras, their 
affine versions as well 
as algebras along the same lines, 
can be seen as a low genus class
of a more general, higher genus, class of diagram algebras.

In case $g=0$ these diagram algebras and their associated 
braid groups are around for donkey's years. 
For $g=1$ there is a long history of work on this topic which 
goes back to at least Brieskorn \cite{Brieskorn}.
For example (with more references to come later on), see
\cite{allcock} for braid pictures, \cite{Geck-Lambropoulou} or \cite{OrRa-affine-braids} for connections 
to knot theory, \cite{Gr-gen-tl-algebra} 
or \cite{HaOl-cyclotomic-bmw} for affine diagram algebras.
To the best of our knowledge, the first attempts to give a
description of braids in handlebodies are due to Vershinin \cite{Ve-handlebodies}
and H{\"a}ring-Oldenburg--Lambropoulou \cite{La-handlebodies}, \cite{HaOlLa-handlebodies}, 
while Lambropoulou studied links in 3-manifolds even before that \cite{La-PhDThesis}.

From the representation theoretical point of view all these 
algebras share the common feature of being cellular in a certain way
that we will make precise. In a nutshell, in the $g=0$ case 
these algebras are often cellular in the sense of Graham--Lehrer \cite{GrLe-cellular}, 
and in the $g=1$ case they are often affine cellular in the sense of K{\"o}nig--Xi \cite{KoXi-affine-cellular}.

We see this paper as a continuation of these works, focusing on the diagrammatic, 
algebraic and representation theoretical aspects. That is, we generalize 
some diagram algebras to higher genus, and 
we show that they are sandwich cellular.
(Sandwich cellular is a notion that 
generalizes cellularity. Roughly speaking it means that the original 
algebra can be obtained by sandwiching smaller algebras.
One of the upshots of being sandwich cellular is that the 
classification of simple modules is reduced 
from the original algebra to the sandwiched algebras.)

\subsection{What this paper does}

Our starting point is a diagrammatic description of 
handlebody braid groups of genus $g$, {\ie} a diagrammatic 
description of the configuration space of a disk with 
$g$ punctures. The pictures hereby are {\eg}
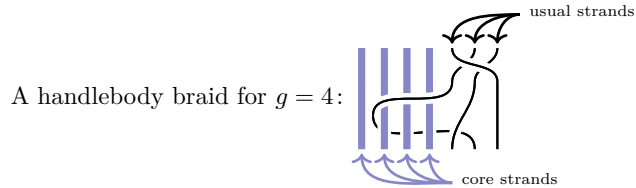
\begin{gather*}
\scalebox{0.85}{$\text{A handlebody braid for $g=4$}\colon
\begin{tikzpicture}[anchorbase,scale=0.7,tinynodes]
\draw[pole,crosspole] (-0.5,0) to[out=90,in=270] (-0.5,1.5);
\draw[usual,crossline] (2,0) to[out=90,in=270] (-0.25,0.75);
\draw[pole,crosspole] (1,0) to[out=90,in=270] (1,1.5);
\draw[pole,crosspole] (0.5,0) to[out=90,in=270] (0.5,1.5);
\draw[usual,crossline] (1.5,0) to[out=90,in=270] (2,1.5);
\draw[usual,crossline] (2.5,0) to[out=90,in=270] (2.5,1.5);
\draw[pole,crosspole] (0,0) to[out=90,in=270] (0,1.5);
\draw[usual,crossline] (-0.25,0.75) to[out=90,in=270] (1.5,1.5);
\draw[pole,crosspole] (-0.5,1.5) to[out=90,in=270] (-0.5,2.25);
\draw[pole,crosspole] (0,1.5) to[out=90,in=270] (0,2.25);
\draw[pole,crosspole] (0.5,1.5) to[out=90,in=270] (0.5,2.25);
\draw[pole,crosspole] (1,1.5) to[out=90,in=270] (1,2.25);
\draw[usual,crossline] (1.5,1.5) to[out=90,in=270] (2,2.25);
\draw[usual,crossline] (2,1.5) to[out=90,in=270] (2.5,2.25);
\draw[usual,crossline] (2.5,1.5) to[out=90,in=270] (1.5,2.25);
\draw[very thick,specialgray,->] (1.5,-0.75) to[out=180,in=270] (-0.5,-0.1);
\draw[very thick,specialgray,->] (1.5,-0.75) to[out=180,in=270] (0,-0.1);
\draw[very thick,specialgray,->] (1.5,-0.75) to[out=180,in=270] (0.5,-0.1);
\draw[very thick,specialgray,->] (1.5,-0.75)node[right,black]{core strands} to[out=180,in=270] (1,-0.1);
\draw[very thick,black,->] (3,3) to[out=180,in=90] (1.5,2.35);
\draw[very thick,black,->] (3,3) to[out=180,in=90] (2,2.35);
\draw[very thick,black,->] (3,3) node[right,black]{usual strands} to[out=180,in=90] (2.5,2.35);
\end{tikzpicture}$}
.
\end{gather*}
This illustrates a handlebody braid of genus $4$: 
The three strands on the right are usual strands.
The four thick and blue/grayish strands on the left are core
strands and they correspond to the punctures 
of the disk respectively the cores of the handlebody. 
The point is that by an appropriate closure, {\ie} merging 
the core strands at infinity, illustrated by
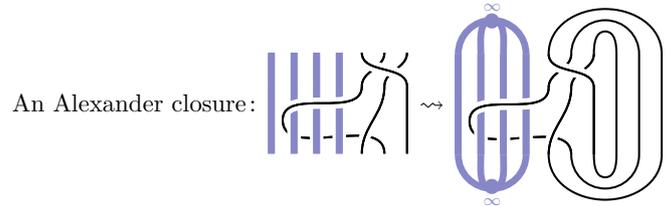
\begin{gather*}
\scalebox{0.85}{$\text{An Alexander closure}\colon
\begin{tikzpicture}[anchorbase,scale=0.7,tinynodes]
\draw[pole,crosspole] (-0.5,0) to[out=90,in=270] (-0.5,1.5);
\draw[usual,crossline] (2,0) to[out=90,in=270] (-0.25,0.75);
\draw[pole,crosspole] (1,0) to[out=90,in=270] (1,1.5);
\draw[pole,crosspole] (0.5,0) to[out=90,in=270] (0.5,1.5);
\draw[usual,crossline] (1.5,0) to[out=90,in=270] (2,1.5);
\draw[usual,crossline] (2.5,0) to[out=90,in=270] (2.5,1.5);
\draw[pole,crosspole] (0,0) to[out=90,in=270] (0,1.5);
\draw[usual,crossline] (-0.25,0.75) to[out=90,in=270] (1.5,1.5);
\draw[pole,crosspole] (-0.5,1.5) to[out=90,in=270] (-0.5,2.25);
\draw[pole,crosspole] (0,1.5) to[out=90,in=270] (0,2.25);
\draw[pole,crosspole] (0.5,1.5) to[out=90,in=270] (0.5,2.25);
\draw[pole,crosspole] (1,1.5) to[out=90,in=270] (1,2.25);
\draw[usual,crossline] (1.5,1.5) to[out=90,in=270] (2,2.25);
\draw[usual,crossline] (2,1.5) to[out=90,in=270] (2.5,2.25);
\draw[usual,crossline] (2.5,1.5) to[out=90,in=270] (1.5,2.25);
\end{tikzpicture}
\rightsquigarrow
\begin{tikzpicture}[anchorbase,scale=0.7,tinynodes]
\draw[pole,crosspole] (-0.5,0) to[out=90,in=270] (-0.5,1.5);
\draw[usual,crossline] (2,0) to[out=90,in=270] (-0.25,0.75);
\draw[pole,crosspole] (1,0) to[out=90,in=270] (1,1.5);
\draw[pole,crosspole] (0.5,0) to[out=90,in=270] (0.5,1.5);
\draw[usual,crossline] (1.5,0) to[out=90,in=270] (2,1.5);
\draw[usual,crossline] (2.5,0) to[out=90,in=270] (2.5,1.5);
\draw[pole,crosspole] (0,0) to[out=90,in=270] (0,1.5);
\draw[usual,crossline] (-0.25,0.75) to[out=90,in=270] (1.5,1.5);
\draw[pole,crosspole] (-0.5,1.5) to[out=90,in=270] (-0.5,2.25);
\draw[pole,crosspole] (0,1.5) to[out=90,in=270] (0,2.25);
\draw[pole,crosspole] (0.5,1.5) to[out=90,in=270] (0.5,2.25);
\draw[pole,crosspole] (1,1.5) to[out=90,in=270] (1,2.25);
\draw[usual,crossline] (1.5,1.5) to[out=90,in=270] (2,2.25);
\draw[usual,crossline] (2,1.5) to[out=90,in=270] (2.5,2.25);
\draw[usual,crossline] (2.5,1.5) to[out=90,in=270] (1.5,2.25);
\draw[pole] (-0.5,0) to[out=270,in=180] (0.25,-0.75);
\draw[pole] (0,0) to[out=270,in=180] (0.25,-0.75);
\draw[pole] (0.5,0) to[out=270,in=0] (0.25,-0.75);
\draw[pole] (1,0) to[out=270,in=0] (0.25,-0.75);
\draw[pole] (-0.5,2.25) to[out=90,in=180] (0.25,3);
\draw[pole] (0,2.25) to[out=90,in=180] (0.25,3);
\draw[pole] (0.5,2.25) to[out=90,in=0] (0.25,3);
\draw[pole] (1,2.25) to[out=90,in=0] (0.25,3);
\draw node[pole] at (0.25,2.8) {{\color{specialgray}\LARGE$\bullet$}};
\draw node[pole,above] at (0.25,2.95) {\color{specialgray}$\infty$};
\draw node[pole] at (0.25,-0.9) {{\color{specialgray}\LARGE$\bullet$}};
\draw node[pole,below] at (0.25,-0.8) {\color{specialgray}$\infty$};
\draw[usual] (1.5,0) to[out=270,in=180] (2.75,-1) 
to[out=0,in=270] (4,0) to (4,2.25) to[out=90,in=0] (2.75,3.25) to[out=180,in=90] (1.5,2.25);
\draw[usual] (2,0) to[out=270,in=180] (2.75,-0.75) 
to[out=0,in=270] (3.5,0) to (3.5,2.25) to[out=90,in=0] (2.75,3) to[out=180,in=90] (2,2.25);
\draw[usual] (2.5,0) to[out=270,in=180] (2.75,-0.5) 
to[out=0,in=270] (3,0) to (3,2.25) to[out=90,in=0] (2.75,2.75) to[out=180,in=90] (2.5,2.25);
\end{tikzpicture}$}
\,,
\end{gather*}
the core strands correspond to cores of a 
handlebody as explained in {\eg} 
\cite[Section 2]{RoTu-homflypt-typea}, hence the name. All links in such handlebodies can 
be obtained by this closing procedure, and there is also an associated 
Markov theorem. In other words, the handlebody braid group gives an 
algebraic way to study links in handlebodies. The main references 
here are \cite{Ve-handlebodies} and \cite{HaOlLa-handlebodies}.
After explaining this setup more carefully, building upon the aforementioned 
works, in \fullref{section:braids} 
we also study an associated handlebody Coxeter group for which we find a basis 
using versions of Jucys--Murphy elements.

These handlebody braid pictures are also our starting point to define and study 
various diagram algebras associated to handlebodies:

\begin{enumerate}

\item In \fullref{section:tlblob} we study handlebody Temperley--Lieb and blob algebras. The pictures to keep in mind are crossingless matchings 
and core strands
(left, Temper\-ley--Lieb)
respectively crossingless matchings decorated with colored blobs
(ri\-ght, blob):
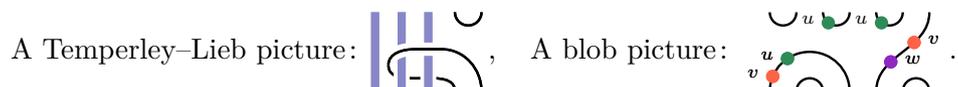
\begin{gather*}
\text{A Temperley--Lieb picture}\colon
\begin{tikzpicture}[anchorbase,scale=0.7,tinynodes]
\draw[pole,crosspole] (-0.5,0) to[out=90,in=270] (-0.5,1.5);
\draw[usual,crossline] (1,0) to[out=90,in=270] (-0.25,0.5);
\draw[pole,crosspole] (0,0) to[out=90,in=270] (0,1.5);
\draw[pole,crosspole] (0.5,0) to[out=90,in=270] (0.5,1.5);
\draw[usual,crossline] (-0.25,0.5) to[out=90,in=180] (0.25,0.8)
to[out=0,in=180] (0.75,0.8) to[out=0,in=90] (1.5,0);
\draw[usual,crossline] (1,1.5) to[out=270,in=180] (1.25,1.25) 
to[out=0,in=270] (1.5,1.5);
\end{tikzpicture}
,\quad
\text{A blob picture}\colon
\begin{tikzpicture}[anchorbase,scale=0.7,tinynodes]
\draw[usual] (1,1.5) to[out=270,in=180] (1.25,1.25) to[out=0,in=270] (1.5,1.5);
\draw[usual,blobbed={0.25}{u}{spinach}] (2,1.5) 
to[out=270,in=180] (2.25,1.25) to[out=0,in=270] (2.5,1.5);
\draw[usual,blobbed={0.25}{u}{spinach}] (3,1.5) 
to[out=270,in=180] (3.25,1.25) to[out=0,in=270] (3.5,1.5);
\draw[usual,rblobbed={0.33}{w}{orchid},rblobbed={0.66}{v}{tomato}] 
(3,0) to[out=90,in=270] (4,1.5);
\draw[usual,blobbed={0.1}{v}{tomato},blobbed={0.3}{u}{spinach}] (1,0) 
to[out=90,in=180] (1.75,0.75) to[out=0,in=90] (2.5,0);
\draw[usual] (1.5,0) to[out=90,in=180] (1.75,0.25) to[out=0,in=90] (2,0);
\draw[usual] (3.5,0) to[out=90,in=180] (3.75,0.25) to[out=0,in=90] (4,0);
\end{tikzpicture}
.
\end{gather*}
Note that the blobs illustrated in the right picture come in colors, 
corresponding to the various cores strands. These algebras 
generalize Temperley--Lieb algebras and blob algebras: 
If $g=0$, then these two algebras are the same as the classical 
Temperley--Lieb algebra.
For $g\geq 1$ they are not the same anymore (at least 
in our formulation), and have a long history of study
starting with {\eg} \cite{MaSa-blob}.

\item In \fullref{section:brauer} we move on to handlebody versions 
of Brauer and BMW algebras. These are tangle algebras with core strands
and the picture is
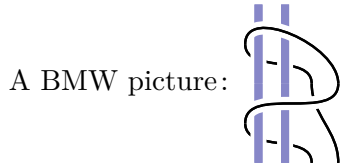
\begin{gather*}
\text{A BMW picture}\colon
\begin{tikzpicture}[anchorbase,scale=0.7,tinynodes]
\draw[usual,crossline] (1.5,0) to[out=90,in=270] (1,1.5);
\draw[usual,crossline] (1,0) to[out=90,in=270] (-0.25,0.75);
\draw[pole,crosspole] (0,0) to[out=90,in=270] (0,1.5);
\draw[pole,crosspole] (0.5,0) to[out=90,in=270] (0.5,1.5);
\draw[usual,crossline] (-0.25,0.75) to[out=90,in=270] (1.5,1.5);
\draw[usual,crossline] (1,1.5) to[out=90,in=270] (-0.25,2.25);
\draw[pole,crosspole] (0,1.5) to[out=90,in=270] (0,3);
\draw[pole,crosspole] (0.5,1.5) to[out=90,in=270] (0.5,3);
\draw[usual,crossline] (-0.25,2.25) to[out=90,in=90] (1.5,1.5);
\end{tikzpicture}
.
\end{gather*}
As before, the case $g=0$ is classical and goes back to Brauer, while 
$g=1$ appears in \cite{HaOl-actions-tensor-categories}.
In the same section we also study their cyclotomic quotients.

\item Finally, \fullref{section:hecke} studies handlebody versions of Hecke and Ariki--Koike algebras.
The picture for handlebody Hecke algebras is the same as for handlebody braid groups, 
while we choose to illustrate handlebody versions of Ariki--Koike algebras using blobs, {\eg}
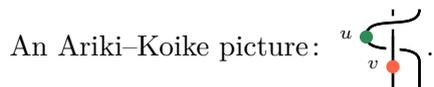
\begin{gather*}
\text{An Ariki--Koike picture}\colon
\begin{tikzpicture}[anchorbase,scale=0.7,tinynodes]
\draw[usual] (2,1) to[out=90,in=270] (2,1.5);
\draw[usual,blobbed={0.58}{u}{spinach},crossline] (2.5,0) 
to (2.5,0.5) to[out=90,in=270] (1.5,1) to[out=90,in=270] (2.5,1.5);
\draw[usual,blobbed={0.4}{v}{tomato},crossline] (2,0) to[out=90,in=270] (2,1);
\end{tikzpicture}
.
\end{gather*}
The cases $g=0$ and $g=1$ are, of course, well-studied and they correspond to 
Hecke respectively extended affine Hecke algebras or cyclotomic quotients.
We learned about the $g>1$ case from \cite{La-handlebodies}
and \cite{Ba-braid-handlebodies}.

\end{enumerate}

We also study a generalization of cellularity in 
\fullref{section:cells}, giving us a toolkit to parameterize the simples 
modules of the aforementioned algebras.
Note hereby that this generalization heavily builds on and borrows from \cite{Gr-semigroups}, \cite{KoXi-cellular-inflation-morita}, \cite{GuWi-almost-cellular} 
or \cite{EhTu-relcell}. 
Although it might be known to experts,
our exposition is new. 

\subsection{Speculations}

Let us mention a number of possible future directions.

\begin{enumerate}[label=$\bullet$]

\item \textbf{Quantum topology.} 
A manifest direction which we do not explore in 
this work would be to study these algebras 
in connections to quantum topology and its ramifications.

For example, for $g=1$ \cite{Geck-Lambropoulou} and \cite{OrRa-affine-braids}
construct link invariants 
from Markov traces, and
these link invariants admit categorifications \cite{WeWi}. 
For $g>1$ \cite{RoTu-homflypt-typea} takes a few first steps towards 
categorical handlebody link invariants, 
but this direction appears to be widely open otherwise.

Moreover, for $g=0$ most of these diagram algebras are related to 
representation theory by some form of Schur--Weyl duality. 
(In fact, this was the reason to define {\eg} the Temperley--Lieb 
algebras to begin with, see \cite{RuTeWe}.) 
Some work for higher genus on this representation theoretical 
aspect is done, 
{\eg} in relation to Verma modules \cite{IoLeZh-verma-schur-weyl}, 
\cite{LV}, \cite{DaRa-two-boundary-hecke} or complex reflection groups \cite{MaSt}, \cite{SaSh}. Following this track for higher genus
seems to be a worthwhile goal.

\item \textbf{Diagram algebras.} There are plenty 
of diagram algebras that we do not considered in this 
paper, but for which (some version of) our discussion goes through.

Examples of such algebras that come to mind that appear in classical 
literature are partition algebras \cite{Martin},
rook monoid algebras \cite{Solomon}, walled 
Brauer algebras \cite{Koike} and alike. 
Other examples are related to knot theory and categorification 
such as (type A) webs appearing in a version 
of Schur--Weyl duality \cite{CKM}, \cite{RoTu-symmetric-howe}, 
\cite{QS}, 
\cite{TVW}.

Diagrammatic algebras are also important in categorification. 
After the introduction of the diagrammatic version of the 
KLR algebra \cite{KL1} 
(see also \cite{Rouquier}) 
they have become quite popular, and might admit handlebody extensions.
For example, alongside with KLR algebras 
Webster's tensor product algebras \cite{webster}, 
algebras related to Verma categorifications 
\cite{naissevaz2}, \cite{naissevaz3}, \cite{maksimau-vaz}, \cite{lacabanne-naisse-vaz}, Soergel diagrammatics \cite{EW}, 
potentially admit handlebody versions, just to name a few.

We also expect these handlebody diagram 
algebras to have ``nice'' sandwich cellular bases.

\item \textbf{Categories instead of algebras.}
All of our algebras and concepts under study also have appropriate 
categorical versions.

For example, it should be fairly straightforward 
to generalize our discussion of \fullref{section:cells} 
to cellular categories \cite{We-tensors-cellular-categories},
\cite{ElLa-trace-hecke}.
However, let us mention that a reason why we have not 
touched upon categorical versions of our handlebody diagram 
algebras is that these do not form monoidal categories 
for $g>0$ (at least not in any reasonable sense as far as we are aware), but 
rather module categories. We think this deserves a thorough 
treatment, following for example \cite{HaOl-actions-tensor-categories} 
or \cite{SaTu}.

\end{enumerate}

\subsection{How to read this paper}

\fullref{section:cells} explains our generalization 
of cellularity and is independent of the rest. It can be easily skipped 
during a first reading.
\fullref{section:braids}
treats handlebody braid and Coxeter groups and
is fundamental for 
all sections following it. Moreover, to avoid 
too much repetition, we decided to construct the remaining sections 
assuming the reader knows \fullref{section:tlblob}, in which we
define handlebody Temperley--Lieb and blob algebras. So 
this section is mandatory if one wants to read either \fullref{section:brauer}, the BMW part, 
or \fullref{section:hecke}, the Hecke part.

\medskip

\noindent\textbf{Acknowledgments.}
We thank a referee, Abel Lacabanne, David Rose, Catharina Stroppel and Arik Wilbert for comments on this paper and discussions related 
to the diagrammatics of handlebodies.

D.T. does not deserve to be supported, but was still supported
by the Hausdorff Research Institute for Mathematics (HIM) 
during the Junior Trimester Program New Trends in Representation Theory.
Part of this paper were written during that program, which 
is gratefully acknowledged.
P.V. was supported by the Fonds de la Recherche
Scientifique - FNRS under Grant no. MIS-F.4536.19.

\section{A generalization of cellularity}\label{section:cells}

Let $\KK$ be a unital, commutative, Noetherian domain, {\eg} 
the integers or a field, or polynomial rings over these. Everything 
in this paper is linear over $\KK$. 
In particular, algebras are $\KK$-algebras.

\subsection{Inflation by algebras}\label{subsection:groups-inflation}

We start by defining sandwich cellular algebras.

\begin{remark}\label{remark:green}
Our discussion below is motivated by Green's 
theory of cells \cite{Gr-semigroups},
often called Green's relations, 
and the Clifford--Munn--Ponizovski\u{\i} theorem 
(see {\eg} \cite{GaMaSt-irreps-semigroups} 
for a modern formulation). 
In fact, we have borrowed 
part of the terminology from the literature on semigroups.
\end{remark}

The following generalization of the notion 
of cellularity from \cite{GrLe-cellular} is well-known 
to experts, see {\eg} \cite{KoXi-cellular-inflation-morita}
or \cite{GuWi-almost-cellular} for
basis-free formulations.
Nevertheless, we will state this generalization and some consequences 
of it.

\begin{definition}\label{definition:cellular}
A \emph{sandwich cellular algebra} over 
$\KK$ is an associative, unital algebra 
$\setstuff{A}$ together with a \emph{sandwich cellular datum}, that is:
\begin{enumerate}[label=$\bullet$]

\item A partial ordered set $\Lambda=\text{(}\Lambda,\leq_{\Lambda}\text{)}$ (we also write $<_{\Lambda}$ {\etc} having the usual meaning);

\item finite sets $M_{\lambda}$ (bottom) and $N_{\lambda}$ (top) for all $\lambda\in\Lambda$;

\item an algebra $\sand$ and a fixed basis $\sandbasis$ of it for all $\lambda\in\Lambda$;

\item a $\KK$-basis $\{c_{D,b,U}^{\lambda}\mid\lambda\in\Lambda,D\in M_{\lambda},U\in N_{\lambda},
b\in\sandbasis\}$ of $\setstuff{A}$;

\end{enumerate}
such that we have

\begin{enumerate}

\item For all $x\in\setstuff{A}$ there exist scalars 
$r(S,D)\in R$ that do not depend
on $U$ or on $b$, such that
\begin{gather}\label{eq:cell-mult}
xc_{D,b,U}^{\lambda}\equiv
\sum_{S\in M_{\lambda},a\in\sandbasis}r(S,D)\cdot c_{S,a,U}^{\lambda}\pmod{\setstuff{A}^{<_{\Lambda}\lambda}},
\end{gather}
where $\setstuff{A}^{<_{\Lambda}\lambda}$ is the $\KK$-submodule of $\setstuff{A}$ spanned by the set
$\{c_{D,b,U}^{\mu}|\mu\in\Lambda,\mu<_{\Lambda}\lambda,D\in M_{\mu},U\in N_{\mu},b\in\sandbasis[\mu]\}$. We also have a similar condition for right 
multiplication.

\item Let $\setstuff{A}(\lambda)=\setstuff{A}^{\leq_{\Lambda}\lambda}/\setstuff{A}^{<_{\Lambda}\lambda}$, where 
$\setstuff{A}^{\leq_{\Lambda}\lambda}$ is the $\KK$-submodule of $\setstuff{A}$ spanned 
by the set $\{c_{D,b,U}^{\mu} |\mu\in\Lambda,\mu\leq_{\Lambda}\lambda,D\in M_{\mu},U\in N_{\mu},b\in\sandbasis[\mu]\}$. 
Then $\setstuff{A}(\lambda)$ is isomorphic to 
$\Delta(\lambda)\hcirc_{\sand[\lambda]}(\lambda)\Delta$ for free graded right and left 
$\sand[\lambda]$-modules $\Delta(\lambda)$ and $(\lambda)\Delta$, respectively.

\end{enumerate}

The set $\{c_{D,b,U}^{\lambda}\mid\lambda\in\Lambda,D\in M_{\lambda},U\in
N_{\lambda},b\in\sandbasis\}$
is called a \emph{(sandwich) cellular basis}.
\end{definition}

We very often have $M_{\lambda}=N_{\lambda}$, and we will then 
omit $N_{\lambda}$ from the notation. In particular, from \autoref{subsection:brauer} 
onward we always have $M_{\lambda}=N_{\lambda}$.

\begin{remark}
One of the advantages of the basis-focused formulation 
above is that
\fullref{definition:cellular} works, {\muta}, for relative 
cellular algebras as in \cite{EhTu-relcell} 
or (strictly object-adapted) cellular 
categories \cite{We-tensors-cellular-categories}, \cite{ElLa-trace-hecke}.
\end{remark}

We also define:

\begin{definition}\label{definition:cellular-involution}
A sandwich cellular algebra $\setstuff{A}$ is called \emph{involutive} 
if $M_{\lambda}=N_{\lambda}$ for all $\lambda\in\Lambda$ 
and $\setstuff{A}$ admits an antiinvolution $(\placeholder)^{\star}\colon\setstuff{A}\to\setstuff{A}$ 
compatible with the cell structure. That is, 
$(\placeholder)^{\star}$ restricts to
an antiinvolution $(\placeholder)^{\star}\colon\sand\to\sand$ that is a bijection on $\sandbasis$
for all $\lambda\in\Lambda$, and we have
\begin{gather}\label{eq:antiinvolution}
(c_{D,b,U}^{\lambda})^{\star}=c_{U,b^{\star},D}^{\lambda}
.
\end{gather}
\end{definition} 

\begin{convention}\label{convention:diagram-conventions}
We will use diagrammatics from now on. 
Our reading conventions for diagrams are
summarized by 
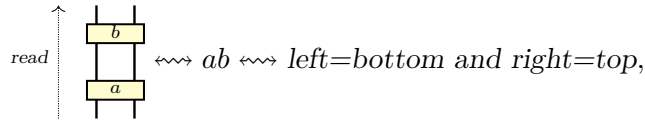
\begin{gather*}
\begin{tikzpicture}[anchorbase,scale=0.5,tinynodes]
\draw[usual] (-1,0) to[out=90,in=270] (-1,0.5);
\draw[usual] (0,0) to[out=90,in=270] (0,0.5);
\draw[usual] (-1,1) to[out=90,in=270] (-1,2);
\draw[usual] (0,1) to[out=90,in=270] (0,2);
\draw[usual] (-1,2.5) to[out=90,in=270] (-1,3);
\draw[usual] (0,2.5) to[out=90,in=270] (0,3);
\draw[mor] (-1.25,0.5) rectangle (0.25,1);
\draw[mor] (-1.25,2) rectangle (0.25,2.5);
\node at (-0.5,0.7) {$a$};
\node at (-0.5,2.17) {$b$};
\draw[densely dotted,->] (-2,0)to (-2,1.5)node[left]{read}to (-2,3);
\end{tikzpicture}
\leftrightsquigarrow
ab
\leftrightsquigarrow
\text{left=bottom and right=top}
,
\end{gather*}
which is bottom to top.
We omit data, such as a label, if it is not of importance for the 
situation at hand.
Moreover, we use colors in this paper, but they 
are non-essential and for illustration purpose only. 
We however still recommend to read the paper in color.
\end{convention}

The pictures for \eqref{eq:cell-mult} (with $x=c_{D^{\prime},b^{\prime},U^{\prime}}^{\lambda^{\prime}}$)
and \eqref{eq:antiinvolution} are
\begin{gather}\label{eq:cell-action}
\scalebox{0.85}{$\begin{tikzpicture}[anchorbase,scale=1]
\draw[mor] (0,-0.5) to (0.25,0) to (0.75,0) to (1,-0.5) to (0,-0.5);
\node at (0.5,-0.25){$D^{\prime}$};
\draw[mor] (0,1) to (0.25,0.5) to (0.75,0.5) to (1,1) to (0,1);
\node at (0.5,0.75){$U^{\prime}$};
\draw[mor] (0.25,0) to (0.25,0.5) to (0.75,0.5) to (0.75,0) to (0.25,0);
\node at (0.5,0.25){$b^{\prime}$};
\draw[mor] (0,1) to (0.25,1.5) to (0.75,1.5) to (1,1) to (0,1);
\node at (0.5,1.25){$D$};
\draw[mor] (0,2.5) to (0.25,2) to (0.75,2) to (1,2.5) to (0,2.5);
\node at (0.5,2.25){$U$};
\draw[mor] (0.25,1.5) to (0.25,2) to (0.75,2) to (0.75,1.5) to (0.25,1.5);
\node at (0.5,1.75){$b$};
\end{tikzpicture}
\equiv
r(U^{\prime},D)\cdot
\begin{tikzpicture}[anchorbase,scale=1]
\draw[mor] (0,-0.5) to (0.25,0) to (0.75,0) to (1,-0.5) to (0,-0.5);
\node at (0.5,-0.25){$D^{\prime}$};
\draw[mor] (0,1) to (0.25,0.5) to (0.75,0.5) to (1,1) to (0,1);
\node at (0.5,0.75){$U$};
\draw[mor] (0.25,0) to (0.25,0.5) to (0.75,0.5) to (0.75,0) to (0.25,0);
\node at (0.5,0.25){$b^{\prime}b$};
\end{tikzpicture}
\pmod{\setstuff{A}^{<_{\Lambda}\lambda}}
,\quad
\left(
\begin{tikzpicture}[anchorbase,scale=1]
\draw[mor] (0,-0.5) to (0.25,0) to (0.75,0) to (1,-0.5) to (0,-0.5);
\node at (0.5,-0.25){$D$};
\draw[mor] (0,1) to (0.25,0.5) to (0.75,0.5) to (1,1) to (0,1);
\node at (0.5,0.75){$U$};
\draw[mor] (0.25,0) to (0.25,0.5) to (0.75,0.5) to (0.75,0) to (0.25,0);
\node at (0.5,0.25){$b$};
\end{tikzpicture}
\right)^{\star}
=
\begin{tikzpicture}[anchorbase,scale=1]
\draw[mor] (0,-0.5) to (0.25,0) to (0.75,0) to (1,-0.5) to (0,-0.5);
\node at (0.5,-0.25){$U$};
\draw[mor] (0,1) to (0.25,0.5) to (0.75,0.5) to (1,1) to (0,1);
\node at (0.5,0.75){$D$};
\draw[mor] (0.25,0) to (0.25,0.5) to (0.75,0.5) to (0.75,0) to (0.25,0);
\node at (0.5,0.25){$b^{\star}$};
\end{tikzpicture}$}
.
\end{gather}
The picture in \eqref{eq:cell-action} is an accurate description for all diagram algebras that we use in this paper. However, these pictures should be taken with care as the definition of a sandwich cellular datum is more general.

Throughout the rest of this section 
we write $\setstuff{A}$ for a sandwich cellular algebra 
with a fixed sandwich cell datum, using the notation from 
\fullref{definition:cellular}. 
We will use the terminology of being a sandwich cellular algebra 
in the sense that we have fixed a sandwich cell datum.
As we will see in {\eg} \fullref{theorem:classification}, our focus 
is indeed not on whether an algebra is sandwich cellular but rather 
whether one can find a useful sandwich cell datum. Here is an explicit example 
of a not very useful sandwich cell datum:

\begin{example}\label{example:groups}
For any group $\setstuff{G}$ the group element basis 
is a sandwich cellular basis in the sense of \fullref{definition:cellular}. 
To see this we let $\Lambda=\{\bullet\}=M_{\bullet}=N_{\bullet}$ be trivial,
and set $c_{\bullet,b,\bullet}=b$ for $b\in\setstuff{G}$
seen as an element of $\sand[\bullet]=\KK\setstuff{G}$.
For this choice \fullref{theorem:classification} 
does not reduce the classification 
problem of finding the simple $\KK\setstuff{G}$-modules.
\end{example}

\begin{example}
For (important) special cases, the above has appeared in the literature. 
If $\setstuff{A}$ is involutive, then:
\begin{enumerate}

\item If $\sand=\KK$ for all $\lambda\in\Lambda$, 
then the notion of a sandwich cell datum above is the same as 
the classical one from \cite{GrLe-cellular}. Conversely, any cell datum in 
the sense of \cite{GrLe-cellular} is a sandwich cell datum in the above sense 
by letting $\sand=\KK$ for all $\lambda\in\Lambda$.

\item If $\sand=\KK[X,X^{-1}]$ 
for all $\lambda\in\Lambda$, 
then a sandwich cellular algebras is affine cellular 
as in \cite{KoXi-affine-cellular}. Allowing 
any quotient of a finite polynomial ring as $\sand$, 
the converse is also true as one can check.

\end{enumerate}
Having \fullref{section:braids} in mind,
we note that $\KK\cong\KK\big[\pi_{1}(D_{0})\big]$ 
and $\KK[X,X^{-1}]\cong\KK\big[\pi_{1}(D_{1})\big]$,
where $\pi_{1}(D_{0})$ respectively $\pi_{1}(D_{1})$
are the fundamental groups of a disc 
$D_{0}$ or a punctured disc $D_{1}$.
\end{example}

\begin{example}\label{example:sandwich}
The notion of being (involutive) sandwich cellular 
is a strict generalization of 
being cellular. An easy, albeit silly, example is to take 
$\Lambda=\{\bullet\}=M_{\bullet}=N_{\bullet}$ and $\sand[\bullet]$
to be a non-cellular algebra. As an explicit example 
consider the set of upper triangular 
$2$x$2$ matrices over $\KK$, and view it as a semigroup 
so that $\{c_{\bullet,b_{i},\bullet}^{\bullet}=b_{i}\}_{i}$ 
is the semigroup basis of the semigroup ring 
$\sand[\bullet]$.
See \cite{EhTu-relcell} for several examples of non-cellular algebras
which one could take as $\sand[\bullet]$.
\end{example}

The comparison of sandwich cellular to cellular algebras is:

\begin{proposition}\label{proposition:cell-compare}
An involutive sandwich cellular algebra $\setstuff{A}$ 
such that all $\sand$ are cellular 
(with the same antiinvolution $(\placeholder)^{\star}$) is cellular 
with a refined sandwich cell datum. 
Conversely, if at least one $\sand$ is non-cellular, then 
$\setstuff{A}$ is non-cellular.
\end{proposition}

Note that an algebra can be sandwich cellular without the 
$\sand$ being cellular, {\cf} \fullref{example:sandwich}.

\begin{proof}
The trick is to use the sandwich cell datum of the 
$\sand$ 
(let us fix any such datum compatible with 
$(\placeholder)^{\star}$) to make $\Lambda$ finer. The picture is
\begin{gather}\label{eq:finer-cell}
\begin{tikzpicture}[anchorbase,scale=1]
\draw[mor] (0,-0.5) to (0.25,0) to (0.75,0) to (1,-0.5) to (0,-0.5);
\node at (0.5,-0.25){$D$};
\draw[mor] (0,1) to (0.25,0.5) to (0.75,0.5) to (1,1) to (0,1);
\node at (0.5,0.75){$U$};
\draw[mor] (0.25,0) to (0.25,0.5) to (0.75,0.5) to (0.75,0) to (0.25,0);
\node at (0.5,0.25){$b$};
\end{tikzpicture}
\rightsquigarrow
\begin{tikzpicture}[anchorbase,scale=1]
\draw[mor] (0,-0.5) to (0.25,0) to (0.75,0) to (1,-0.5) to (0,-0.5);
\node at (0.5,-0.25){$D$};
\draw[mor] (0,1.5) to (0.25,1) to (0.75,1) to (1,1.5) to (0,1.5);
\node at (0.5,1.25){$U$};
\draw[mor] (0.25,0) to (0.35,0.5) to (0.65,0.5) to (0.75,0) to (0.25,0);
\node at (0.5,0.25){$\scalebox{0.75}{$D_{\mu}$}$};
\draw[mor] (0.25,1) to (0.35,0.5) to (0.65,0.5) to (0.75,1) to (0.25,1);
\node at (0.5,0.75){$\scalebox{0.75}{$U_{\nu}$}$};
\end{tikzpicture}
,\quad
\left(
\begin{tikzpicture}[anchorbase,scale=1]
\draw[mor] (0,-0.5) to (0.25,0) to (0.75,0) to (1,-0.5) to (0,-0.5);
\node at (0.5,-0.25){$D$};
\draw[mor] (0,1.5) to (0.25,1) to (0.75,1) to (1,1.5) to (0,1.5);
\node at (0.5,1.25){$U$};
\draw[mor] (0.25,0) to (0.35,0.5) to (0.65,0.5) to (0.75,0) to (0.25,0);
\node at (0.5,0.25){$\scalebox{0.75}{$D_{\mu}$}$};
\draw[mor] (0.25,1) to (0.35,0.5) to (0.65,0.5) to (0.75,1) to (0.25,1);
\node at (0.5,0.75){$\scalebox{0.75}{$U_{\nu}$}$};
\end{tikzpicture}
\right)^{\star}
=
\begin{tikzpicture}[anchorbase,scale=1]
\draw[mor] (0,-0.5) to (0.25,0) to (0.75,0) to (1,-0.5) to (0,-0.5);
\node at (0.5,-0.25){$U$};
\draw[mor] (0,1.5) to (0.25,1) to (0.75,1) to (1,1.5) to (0,1.5);
\node at (0.5,1.25){$D$};
\draw[mor] (0.25,0) to (0.35,0.5) to (0.65,0.5) to (0.75,0) to (0.25,0);
\node at (0.5,0.25){$\scalebox{0.75}{$U_{\mu}$}$};
\draw[mor] (0.25,1) to (0.35,0.5) to (0.65,0.5) to (0.75,1) to (0.25,1);
\node at (0.5,0.75){$\scalebox{0.75}{$D_{\nu}$}$};
\end{tikzpicture}
.
\end{gather}
Precisely, we define 
$\Lambda^{\prime}=\{(\lambda,\mu)
\mid\lambda\in\Lambda,\mu\in\Lambda_{\lambda}\}$ with $\Lambda_{\lambda}$ 
being the poset associated to $\sand$. The order on 
$\Lambda^{\prime}$ is $(\lambda,\mu)\leq_{\Lambda^{\prime}}(\lambda^{\prime},\mu^{\prime})$ 
if ($\lambda<_{\Lambda}\lambda^{\prime}$) or 
($\lambda=\lambda^{\prime}$ and $\mu\leq_{\Lambda_{\lambda}}\mu^{\prime}$). 
The sets $M_{(\lambda,\mu)}$ are now tuples $(D,D_{\mu})$ 
with $D\in M_{\lambda}$ and $D_{\mu}\in M_{\mu}$, while the basis 
elements are $c_{(D,D_{\mu}),(U,U_{\nu})}^{(\lambda,\mu)}$, defined 
as in \eqref{eq:finer-cell}, where we also 
indicated the antiinvolution. By construction, this 
is a cell datum for $\setstuff{A}$ in the sense of \cite{GrLe-cellular}.

For the converse one can apply (or rather copy)
\cite[Sections 3 and 4]{KoXi-cellular-inflation-morita}: if one 
inflates along a non-cellular algebra, then the result can not 
be cellular.
\end{proof}

\subsection{Cell modules}\label{subsection:cell-modules}

The theory of cellular algebras is particularly nice 
for finite-dimensional algebras. This is however not always the 
case in the situation we have in mind, 
see {\eg} \fullref{section:tlblob}. Nevertheless, 
parts of the theory still goes through for 
infinite-dimensional cellular algebras, see \cite{GrLe-cellular}, \cite{KoXi-affine-cellular} or \cite{EhTu-relcell}. In particular, the 
existence of cell modules, cells 
and some of their properties, as we will discuss now.

Recall from \autoref{convention:diagram-conventions} that, in pictures, left actions and left multiplications are stacking from the bottom.
For each $\lambda\in\Lambda$ and $U\in N_{\lambda}$ we have a \emph{left
cell} $\mathcal{L}(\lambda,U)$ given by
\begin{gather*}
\mathcal{L}(\lambda,U)=\KK\{
c_{D,b,U}^{\lambda}\mid D\in M_{\lambda},b\in\sandbasis\}
,\quad
x\actsleft
c_{D,b,U}^{\lambda}
=
\eqref{eq:cell-mult},
\end{gather*}
which we endow with the left $\setstuff{A}$-module 
structure given above. 
(In this paper actions are distinguished from multiplications 
by using the symbols $\actsleft$ respectively 
$\actsright$.) In pictures, this means acting on the bottom.
There is also a \emph{right 
cell} $\mathcal{R}(\lambda,D)$ 
defined {\ver}, where the action is from the top.

\begin{lemma}\label{lemma:delta-welldefined}
We have the following.
\begin{enumerate}
\item The left and right cells are $\setstuff{A}$-modules. 

\item As $\setstuff{A}$-modules, $\mathcal{L}(\lambda,U)
\cong\mathcal{L}(\lambda,U^{\prime})$ and 
$\mathcal{R}(\lambda,D)\cong
\mathcal{R}(\lambda,D^{\prime})$ for all $U$, $U^{\prime}$ and $D$, $D^{\prime}$.
\end{enumerate}
\end{lemma}

\begin{proof}
This is immediate from the definitions.
\end{proof}

Using \fullref{lemma:delta-welldefined},
we will write $\Delta(\lambda)$ and $(\lambda)\Delta$ (see 
also (b) of \autoref{definition:cellular}) for 
any choice of $D\in M_{\lambda},U\in N_{\lambda}$ (for the basis elements 
of these modules we omit the fixed index). 
In the theory of cellular algebras, 
these are then also called left, respectively right, 
cell modules, so we call them \emph{sandwich cell modules}.

The space $\mathcal{H}_{\lambda,D,U}=\mathcal{R}(\lambda,D)\hcirc_{\setstuff{A}}
\mathcal{L}(\lambda,U)$ 
is called an 
\emph{$\mathcal{H}$-cell}.
Moreover, the space $\mathcal{J}_{\lambda}=\mathcal{L}(\lambda,U)
\hcirc_{\sand}\mathcal{R}(\lambda,D)$ 
is called a two-sided cell or \emph{$\mathcal{J}$-cell}. 
As free $\KK$-modules we clearly have ({\cf} part (b) of \autoref{definition:cellular})
\begin{gather}\label{eq:hj-cells}
\begin{gathered}
\scalebox{0.8}{$\mathcal{L}(\lambda,U)\cong M_{\lambda}\hcirc_{\KK}\sand
\leftrightsquigarrow
\begin{tikzpicture}[anchorbase,scale=1]
\draw[mor] (0,-0.5) to (0.25,0) to (0.75,0) to (1,-0.5) to (0,-0.5);
\node at (0.5,-0.25){$D$};
\draw[mor,orchid] (0,1) to (0.25,0.5) to (0.75,0.5) to (1,1) to (0,1);
\node at (0.5,0.75){$U$};
\draw[mor] (0.25,0) to (0.25,0.5) to (0.75,0.5) to (0.75,0) to (0.25,0);
\node at (0.5,0.25){$b$};
\end{tikzpicture}
,\quad
\mathcal{R}(\lambda,D)\cong 
\sand\hcirc_{\KK}N_{\lambda}
\leftrightsquigarrow
\begin{tikzpicture}[anchorbase,scale=1]
\draw[mor,orchid] (0,-0.5) to (0.25,0) to (0.75,0) to (1,-0.5) to (0,-0.5);
\node at (0.5,-0.25){$D$};
\draw[mor] (0,1) to (0.25,0.5) to (0.75,0.5) to (1,1) to (0,1);
\node at (0.5,0.75){$U$};
\draw[mor] (0.25,0) to (0.25,0.5) to (0.75,0.5) to (0.75,0) to (0.25,0);
\node at (0.5,0.25){$b$};
\end{tikzpicture}$}
,
\\
\scalebox{0.8}{$\mathcal{J}_{\lambda}\cong M_{\lambda}\hcirc_{\KK}\sand\hcirc_{\KK}N_{\lambda}
\leftrightsquigarrow
\begin{tikzpicture}[anchorbase,scale=1]
\draw[mor] (0,-0.5) to (0.25,0) to (0.75,0) to (1,-0.5) to (0,-0.5);
\node at (0.5,-0.25){$D$};
\draw[mor] (0,1) to (0.25,0.5) to (0.75,0.5) to (1,1) to (0,1);
\node at (0.5,0.75){$U$};
\draw[mor] (0.25,0) to (0.25,0.5) to (0.75,0.5) to (0.75,0) to (0.25,0);
\node at (0.5,0.25){$b$};
\end{tikzpicture}
,\quad
\mathcal{H}_{\lambda,D,U}\cong
\sand
\leftrightsquigarrow
\begin{tikzpicture}[anchorbase,scale=1]
\draw[mor,orchid] (0,-0.5) to (0.25,0) to (0.75,0) to (1,-0.5) to (0,-0.5);
\node at (0.5,-0.25){$D$};
\draw[mor,orchid] (0,1) to (0.25,0.5) to (0.75,0.5) to (1,1) to (0,1);
\node at (0.5,0.75){$U$};
\draw[mor] (0.25,0) to (0.25,0.5) to (0.75,0.5) to (0.75,0) to (0.25,0);
\node at (0.5,0.25){$b$};
\end{tikzpicture}$}
.
\end{gathered}
\end{gather}
In \eqref{eq:hj-cells} we highlighted 
the parts which are fixed and do not vary.

Note that 
$\mathcal{J}$-cells are $\setstuff{A}$-$\setstuff{A}$-bimodules 
isomorphic to $\Delta(\lambda)\hcirc_{\sand}(\lambda)\Delta$,
by definition, and thus, in general non-unital, algebras.
In contrast, the $\mathcal{H}$-cells are only $\KK$-modules, but 
are multiplicatively closed, as follows from 
\eqref{eq:cell-action} and \eqref{eq:hj-cells} below, 
so they form, in general non-unital, subalgebras of the $\mathcal{J}$-cells.

The above can be illustrated by
\begin{gather}\label{eq:cell-diagram}
\scalebox{0.8}{$\begin{tikzpicture}[baseline=(A.center),every node/.style=
{anchor=base,minimum width=1.4cm,minimum height=1cm}]
\matrix (A) [matrix of math nodes,ampersand replacement=\&] 
{
c_{D_{1}*U_{1}} \& c_{D_{1}*U_{2}} 
\& c_{D_{1}*U_{3}} \& c_{D_{1}*U_{4}} \& \dots \\
c_{D_{2}*U_{1}} \& c_{D_{2}*U_{2}} 
\& c_{D_{2}*U_{3}} \& c_{D_{2}*U_{4}} \& \ddots \\
c_{D_{3}*U_{1}} \& c_{D_{3}*U_{2}} 
\& c_{D_{3}*U_{3}} \& c_{D_{3}*U_{4}} \& \ddots \\
c_{D_{4}*U_{1}} \& c_{D_{4}*U_{2}} 
\& c_{D_{4}*U_{3}} \& c_{D_{4}*U_{4}} \& \ddots \\
\vdots \& \ddots \& \ddots \& \ddots \& \ddots \\
};
\draw[fill=blue,opacity=0.25] (A-3-1.north west) node[blue,left,yshift=-0.5cm,opacity=1] 
{$\mathcal{R}(\lambda,D_{3})$} to (A-3-5.north east) 
to (A-3-5.south east) to (A-3-1.south west) to (A-3-1.north west);
\draw[fill=red,opacity=0.25] (A-1-3.north west) node[red,above,xshift=0.7cm,opacity=1] 
{$\mathcal{L}(\lambda,U_{3})$} to (A-5-3.south west) 
to (A-5-3.south east) to (A-1-3.north east) to (A-1-3.north west);
\draw[black] (A-1-1.north west) node[black,above] {$\mathcal{J}_{\lambda}$} to 
(A-1-5.north east) to (A-5-5.south east) 
to (A-5-1.south west) to (A-1-1.north west);
\draw[black,->] ($(A-3-5.south east)+(0.5,0.15)$) 
node[right]{$\mathcal{H}_{\lambda,D_{3},U_{3}}$} 
to[out=180,in=0] ($(A-3-3.south east)+(0,0.2)$);
\draw[black,->] ($(A-2-5.south east)+(0.5,0.15)$) 
node[right]{$\mathcal{H}_{\lambda,D_{2},U_{3}}$} 
to[out=180,in=0] ($(A-2-3.south east)+(0,0.2)$);
\end{tikzpicture}$
},
\end{gather}
where $*$ runs over all $b\in\sandbasis$.

\begin{lemma}\label{lemma:h-subalgebra}
Assume that $c_{D^{\prime},b^{\prime},U^{\prime}}^{\lambda^{\prime}}c_{D,b,U}^{\lambda}=
r(U^{\prime},D)\equiv c_{D,b^{\prime}Fb,U^{\prime}}^{\lambda}
\pmod{\setstuff{A}^{<_{\Lambda}\lambda}}$ for some $F\in\sand$.
If $r(U^{\prime},D)$ is invertible in $\KK$, then 
$\mathcal{H}_{\lambda,D,U}\cong\sand$ as algebras.
\end{lemma}

\begin{proof}
Note that $\mathcal{H}_{\lambda,D,U}$ 
has a $\KK$-basis given by 
$\{c_{D,b,U}^{\lambda}\mid b\in\sandbasis\}$. 
For $F(U,D)=1$ we have
$c_{D,b,U}^{\lambda}c_{D,b^{\prime},U}^{\lambda}
=r(U,D)\cdot c_{D,bb^{\prime},U}^{\lambda}$, and hence 
we can scale everything by $1/r(U,D)$ to get the result.
For $F(U,D)$ being invertible we calculate 
\begin{gather*}
c_{D,bF(U,D)^{-1},U}^{\lambda}c_{D,b^{\prime}F(U,D)^{-1},U}^{\lambda}
=r(U,D)\cdot c_{D,bb^{\prime}F(U,D)^{-1},U}^{\lambda}, 
\end{gather*}
so the map $b\mapsto bF(U,D)^{-1}$ induces an isomorphism 
upon division by $r(U,D)$.
\end{proof}

\begin{lemma}\label{lemma:bilinearform}
The concatenation multiplication on the $\setstuff{A}$-$\setstuff{A}$-bimodule
$\Delta(\lambda)\hcirc_{\sand}(\lambda)\Delta$ is determined 
by a bilinear map
\begin{gather*}
\phi^{\lambda}\colon
(\lambda)\Delta\hcirc_{\setstuff{A}}
\Delta(\lambda)\to\sand
\end{gather*}
which satisfies, for all $a\in\setstuff{A}$, that
\begin{gather*}
\phi^{\lambda}
(x,y\actsright a)=
\phi^{\lambda}
(a\actsleft x,y).
\end{gather*}
It gives rise to a symmetric bilinear form $\phi^{\lambda}$ on $\Delta(\lambda)$. 
This symmetric bilinear form extends to 
a symmetric bilinear form $\phi^{\lambda}_{K}$
on $\Delta(\lambda,K)=\Delta(\lambda)\hcirc_{\sand}K$ 
for any simple $\sand$-module $K$.
\end{lemma}

\begin{proof}
In general this follows from a standard lemma in 
ring theory, see {\eg} \cite[Lemma 2]{GuWi-almost-cellular}. 
In the restricted setting of \eqref{eq:cell-action} one has a nice diagrammatic description:
The map $\phi^{\lambda}$ is clearly bilinear, 
as it is just the multiplication in $\mathcal{J}_{\lambda}$.
The remaining claim can be illustrated by
\begin{gather*}
\phi^{\lambda}
\left(
\begin{tikzpicture}[anchorbase,scale=1]
\draw[mor] (0,0.5) to (0.25,0) to (0.75,0) to (1,0.5) to (0,0.5);
\node at (0.5,0.25){$U$};
\draw[mor] (0.25,-0.5) to (0.25,0) to (0.75,0) to (0.75,-0.5) to (0.25,-0.5);
\node at (0.5,-0.25){$b$};
\draw[mor] (0,0.5) to (0,1) to (1,1) to (1,0.5) to (0,0.5);
\node at (0.5,0.75){$a$};
\end{tikzpicture}
,
\begin{tikzpicture}[anchorbase,scale=1]
\draw[mor] (0,1) to (0.25,1.5) to (0.75,1.5) to (1,1) to (0,1);
\node at (0.5,1.25){$D^{\prime}$};
\draw[mor] (0.25,1.5) to (0.25,2) to (0.75,2) to (0.75,1.5) to (0.25,1.5);
\node at (0.5,1.75){$b^{\prime}$};
\end{tikzpicture}
\right)
\quad\rightsquigarrow\quad
\begin{tikzpicture}[anchorbase,scale=1]
\draw[mor] (0,0.5) to (0.25,0) to (0.75,0) to (1,0.5) to (0,0.5);
\node at (0.5,0.25){$U$};
\draw[mor] (0.25,-0.5) to (0.25,0) to (0.75,0) to (0.75,-0.5) to (0.25,-0.5);
\node at (0.5,-0.25){$b$};
\draw[mor] (0,0.5) to (0,1) to (1,1) to (1,0.5) to (0,0.5);
\node at (0.5,0.75){$a$};
\draw[mor] (0,1) to (0.25,1.5) to (0.75,1.5) to (1,1) to (0,1);
\node at (0.5,1.25){$D^{\prime}$};
\draw[mor] (0.25,1.5) to (0.25,2) to (0.75,2) to (0.75,1.5) to (0.25,1.5);
\node at (0.5,1.75){$b^{\prime}$};
\end{tikzpicture}
\quad\leftsquigarrow\quad
\phi^{\lambda}
\left(
\begin{tikzpicture}[anchorbase,scale=1]
\draw[mor] (0,0.5) to (0.25,0) to (0.75,0) to (1,0.5) to (0,0.5);
\node at (0.5,0.25){$U$};
\draw[mor] (0.25,-0.5) to (0.25,0) to (0.75,0) to (0.75,-0.5) to (0.25,-0.5);
\node at (0.5,-0.25){$b$};
\end{tikzpicture}
,
\begin{tikzpicture}[anchorbase,scale=1]
\draw[mor] (0,0.5) to (0,1) to (1,1) to (1,0.5) to (0,0.5);
\node at (0.5,0.75){$a$};
\draw[mor] (0,1) to (0.25,1.5) to (0.75,1.5) to (1,1) to (0,1);
\node at (0.5,1.25){$D^{\prime}$};
\draw[mor] (0.25,1.5) to (0.25,2) to (0.75,2) to (0.75,1.5) to (0.25,1.5);
\node at (0.5,1.75){$b^{\prime}$};
\end{tikzpicture}
\right)
.
\end{gather*}
The extension to $K$ is then just 
$\phi^{\lambda}_{K}=\phi^{\lambda}\hcirc_{\sand}\idmor_{K}$, 
having the same properties.
\end{proof}

Denote by $\mathrm{rad}(\lambda,K)$ the radical of the 
bilinear form $\phi^{\lambda}_{K}$ defined as follows. 
Let $\bar{\phi}^{\lambda}_{K}$ be the associated map from $\Delta(\lambda)\hcirc_{\sand}K\to\Hom_{\KK}\big((\lambda)\Delta,\KK\big)\hcirc_{\sand}K$. Then the radical $\mathrm{rad}(\lambda,K)$ is defined as the kernel of this map.

\begin{lemma}\label{lemma:simples}
Let $\KK$ be a field and $K$ be a simple 
$\sand$-module.
If $\phi^{\lambda}_{K}$ is not constant zero, then
$L(\lambda,K)=\Delta(\lambda,K)/\mathrm{rad}(\lambda,K)$ is 
the simple head of $\Delta(\lambda,K)$.
\end{lemma}

\begin{proof}
The same arguments as in \cite[Section 3A]{EhTu-relcell} work: 
Any $z\in\Delta(\lambda,K)\setminus\mathrm{rad}(\lambda,K)$ is a generator 
of $\Delta(\lambda,K)$, while $\mathrm{rad}(\lambda,K)$ is a submodule 
of $\Delta(\lambda,K)$ by \fullref{lemma:bilinearform}. 
Thus, $L(\lambda,K)$ is a well-defined and 
simple submodule. It is the head of $\Delta(\lambda,K)$ since $\mathrm{rad}(\lambda,K)$ equals 
the representation-theoretical radical because all 
$z\in\Delta(\lambda,K)\setminus\mathrm{rad}(\lambda,K)$ generate.
\end{proof}

\subsection{Classification of simple modules}\label{subsection:simples}

For an $\setstuff{A}$-module $M$ let 
$\mathrm{Ann}_{\setstuff{A}}(M)=
\{a\in A\mid a\actsleft M=0\}$ be the annihilator.

\begin{definition}\label{definition:apex}
An \emph{apex} of an $\setstuff{A}$-module $M$ 
is a $\lambda\in\Lambda$ such that 
$\mathrm{Ann}_{\setstuff{A}}(M)=
\mathcal{J}_{{>_{\Lambda}}\lambda}=
\bigcup_{\mu>_{\Lambda}\lambda}\mathcal{J}_{\mu}$ 
and $\phi^{\lambda}$ is not constant zero.
\end{definition}

Recall that
in the setting of non-unital algebras 
simple modules are defined using the usual assumption of 
having no non-trivial submodules but one also additionally 
assumes that at least one element acts not as zero.

\begin{lemma}\label{lemma:apex}
We have the following.

\begin{enumerate}

\item Every simple $\setstuff{A}$-module $L$ has a unique apex 
$\lambda\in\Lambda$.

\item If $\KK$ is a field, then the simple modules $L(\lambda,K)$ 
from \fullref{lemma:simples}
have apex $\lambda$.

\item A simple $\setstuff{A}$-module $L$ of apex $\lambda$ is a 
simple $\mathcal{J}_{\lambda}$-module. Conversely, 
every simple $\mathcal{J}_{\lambda}$-module $L$ is a simple 
$\setstuff{A}$-module $L$ with apex $\lambda$, by inflation.

\end{enumerate}
\end{lemma}

\begin{proof}
One can reformulate {\eg} \cite[Corollary 3.2]{KoXi-affine-cellular} 
to get the claimed results. (Here we use that $\KK$ is 
Noetherian and a domain.) Precisely:	

\textit{(a).} Since $\setstuff{A}\actsleft L\neq 0$ there 
is a $\leq_{\Lambda}$-minimal $\lambda$ such that 
$\mathcal{J}_{{\geq_{\Lambda}}\lambda}=
\bigcup_{\mu\geq_{\Lambda}\lambda}\mathcal{J}_{\mu}$ is not contained
in $\mathrm{Ann}_{\setstuff{A}}(M)$.
It is not hard to see 
that $\mathrm{Ann}_{\setstuff{A}}(M)$ is a maximal 
left ideal of $\mathcal{J}_{{\geq_{\Lambda}}\lambda}=
\bigcup_{\mu\geq_{\Lambda}\lambda}\mathcal{J}_{\mu}$, so it has to 
be $\mathcal{J}_{{>_{\Lambda}}\lambda}$. In particular, 
$\mathcal{J}_{\lambda}\actsleft L\neq 0$ and actually 
$\mathcal{J}_{\lambda}\actsleft L=L$. This can only happen 
if some linear combination of the $c_{D,b,U}^{\lambda}$ is an idempotent.

\textit{(b).} By construction.

\textit{(c).} By (a), all elements bigger than $\lambda$ 
annihilate $L$ which together with the partial ordering 
implies that $\mathcal{J}_{\lambda}\actsleft L$, by the same formulas, 
and this action is not zero. Since $\mathrm{Ann}_{\setstuff{A}}(L)$ 
is maximal, it follows that 
$L\cong\mathcal{J}_{{\geq_{\Lambda}}\lambda}/\mathrm{Ann}_{\setstuff{A}}(L)$
stays simple as a $\mathcal{J}_{\lambda}$-module.
Conversely, take a simple $\mathcal{J}_{\lambda}$-module $L$ 
and inflate it to a $\mathcal{J}_{{\geq_{\Lambda}}\lambda}$-module 
such that its annihilator is $\mathcal{J}_{{>_{\Lambda}}\lambda}$. 
Note that $\mathcal{J}_{{\geq_{\Lambda}}\lambda}$ is an ideal in 
$\setstuff{A}$, so $\setstuff{A}\actsleft L$.
Then the same arguments as in \cite[Lemma 3.1]{KoXi-affine-cellular} 
imply that $L$ stays simple.
\end{proof}

We get an analog of the Clifford--Munn--Ponizovski\u{\i} theorem ({\cf} 
\fullref{remark:green}):

\begin{theorem}\label{theorem:classification}
Let $\KK$ be a field.
\begin{enumerate}

\item A $\lambda\in\Lambda$ is an apex if and 
only if the form $\phi^{\lambda}$ is not constant zero 
if and only if the form $\phi^{\lambda}_{K}$ is not constant zero 
for any simple $\sand$-module $K$. If the assumptions of \autoref{lemma:h-subalgebra} hold, then $\lambda\in\Lambda$ is an apex if and 
only if $r(U,D)\neq 0$ for some $D\in M_{\lambda},U\in N_{\lambda}$.

\item Assume that $\sand$ is unital and Artinian.
For a fixed apex $\lambda\in\Lambda$
the simple $\setstuff{A}$-modules of that apex are 
parameterized by simple 
modules of $\sand$. In other words, we have
\begin{gather*}
\left\{
\text{simple $\setstuff{A}$-modules with apex $\lambda$}
\right\}
\xleftrightarrow{1:1}
\left\{
\text{simple $\sand$-modules}
\right\}
.
\end{gather*}
Under this bijection the 
simple $\setstuff{A}$-module $L(\lambda,K)$ 
associated to the simple $\sand$-module $K$
is the head of
$\Delta(\lambda,K)$.

\item Assume that the assumptions of \autoref{lemma:h-subalgebra} hold. For a fixed apex $\lambda\in\Lambda$ there exists 
$\mathcal{H}_{\lambda,D,U}$ such that there is a 1:1-correspondence
\begin{gather*}
\left\{
\text{simple $\setstuff{A}$-modules with apex $\lambda$}
\right\}
\xleftrightarrow{1:1}
\left\{
\text{simple $\mathcal{H}_{\lambda,D,U}$-modules}
\right\}
.
\end{gather*}	
Under this bijection the 
simple $\setstuff{A}$-module $L(\lambda,K)$ 
for the simple $\mathcal{H}_{\lambda,D,U}$-module $K$
is the head of the induced module 
$\mathrm{Ind}_{\mathcal{H}_{\lambda,D,U}}^{\setstuff{A}}(K)$.

\end{enumerate}
\end{theorem}

Note that (c) is not a special case of (b) as the sandwiched algebras in (c) need not to be unital and Artinian. This is only relevant in the infinite-dimensional world.

\begin{proof}
\textit{(a).} Clearly, $\phi^{\lambda}$ is not constant zero 
if and only if $\phi^{\lambda}_{K}$ is not constant zero. Moreover, 
$\phi^{\lambda}$ being not constant zero implies that 
$L(\lambda,K)$ exists, see \fullref{lemma:simples},
and has apex $\lambda$ by \fullref{lemma:apex}. The converse 
follows by the definition of an apex.

\textit{(c).} First, we choose $\mathcal{H}_{\lambda,D,U}$ such that 
it contains an idempotent $e=1/r(U,D)\cdot c_{D,1,U}^{\lambda}$, 
which we can do by (a) and the calculation
$1/r(U,D)\cdot c_{D,1,U}^{\lambda}1/r(U,D)\cdot c_{D,1,U}^{\lambda}
=1/r(U,D)\cdot c_{D,1,U}^{\lambda}$.
By part (c) of \fullref{lemma:apex} 
we can reduce the question to matching simple $\mathcal{J}_{\lambda}$-modules 
and simple $\mathcal{H}_{\lambda,D,U}$-modules. The picture
\begin{gather*}
\begin{tikzpicture}[anchorbase,scale=1]
\draw[mor,orchid] (0,1) to (0.25,0.5) to (0.75,0.5) to (1,1) to (0,1);
\node at (0.5,0.75){$U$};
\draw[mor] (0,1) to (0.25,1.5) to (0.75,1.5) to (1,1) to (0,1);
\node at (0.5,1.25){$D^{\prime}$};
\draw[mor] (0,2.5) to (0.25,2) to (0.75,2) to (1,2.5) to (0,2.5);
\node at (0.5,2.25){$U^{\prime}$};
\draw[mor] (0.25,1.5) to (0.25,2) to (0.75,2) to (0.75,1.5) to (0.25,1.5);
\node at (0.5,1.75){$b$};
\draw[mor,orchid] (0,2.5) to (0.25,3) to (0.75,3) to (1,2.5) to (0,2.5);
\node at (0.5,2.75){$D$};
\end{tikzpicture}
=
r(U,D^{\prime})r(U^{\prime},D)\cdot
\begin{tikzpicture}[anchorbase,scale=1]
\draw[mor] (0.25,0) to (0.25,0.5) to (0.75,0.5) to (0.75,0) to (0.25,0);
\node at (0.5,0.25){$b$};
\end{tikzpicture}
\end{gather*}
shows that $\mathcal{H}_{\lambda,D,U}\cong e\mathcal{J}_{\lambda}e$ 
(note that at least $r(U,D)$ 
is invertible, so $e\mathcal{J}_{\lambda}e\neq 0$).
By a classical theorem of Green, see {\eg} 
\cite[Lemma 6]{GaMaSt-irreps-semigroups}, 
it remains to show that simple $\mathcal{J}_{\lambda}$-modules 
are not annihilated by $e$. To see this we observe that 
any two pseudo-idempotents of the form $c_{D,1,U}^{\lambda}$ 
are related by appropriate conjugation. That is, 
for $e=1/r(U,D)\cdot c_{D,1,U}^{\lambda}$ and 
$e^{\prime}=1/r(U^{\prime},D^{\prime})
\cdot c_{D^{\prime},1,U^{\prime}}^{\lambda}$ we have
\begin{gather*}
r(U,D)r(U^{\prime},D^{\prime})\cdot e^{\prime}=
c_{D^{\prime},1,U}^{\lambda}ec_{D,1,U^{\prime}}^{\lambda},
\end{gather*}
and all appearing scalars are non-zero, and thus invertible. Hence, $e$ 
annihilates a simple $\mathcal{J}_{\lambda}$-module 
if and only if $e^{\prime}$ does. All other $\mathcal{H}$-cells 
do not contain idempotents, so a simple 
$\mathcal{J}_{\lambda}$-module can not be annihilated by any $e$ in 
$\mathcal{J}_{\lambda}$.

\textit{(b).} The proof is not much different from the one in (c), and follows by well-established arguments in the theory of cellular algebras, see {\eg} \cite{GrLe-cellular}, \cite{KoXi-affine-cellular}, \cite{AnStTu-cellular-tilting}, \cite{EhTu-relcell} or \cite{GuWi-almost-cellular}. However, the general proof requires 
the sandwiched algebras to be unital and Artinian.
\end{proof}

\begin{remark}
Note a crucial difference to the case $\sand=\KK$, 
which is (up to having an involution) the cellular case: 
one apex can have any number of simples associated to it.
\end{remark}

\subsection{The Brauer algebra as an example}\label{subsection:brauer}

Let $\setstuff{Br}_{n}(\cvar)$ be the \emph{Brauer algebra} in $n$-strands 
with circle evaluation parameter $\cvar\in\KK$.
The reader unfamiliar with the Brauer algebra is referred 
to {\eg} \cite[Section 4]{GrLe-cellular}. Alternatively,
take $g=0$ in \fullref{section:brauer} below. Let also $\setstuff{S}_{\lambda}$ 
denote the symmetric group in $\lambda$ strands (or on $\{1,\dots,\lambda\}$).

\begin{remark}
The construction we present below 
is not new, see {\eg} \cite{FG} 
or \cite{KoXi-brauer}. However, 
our exposition is new and might be of some use.
\end{remark}

The Brauer algebra has a well-known diagrammatic description 
given by \emph{perfect mat\-chings of $2n$ points}, and typical 
\emph{Brauer diagrams} for $n=4$ are
\begin{gather*}
\begin{tikzpicture}[anchorbase,scale=0.7,tinynodes]
\draw[usual] (0.5,0) to[out=90,in=180] (1,0.25) to[out=0,in=90] (1.5,0);
\draw[usual] (0.5,0.75) to[out=270,in=180] (1,0.5) to[out=0,in=270] (1.5,0.75);
\draw[usual] (0,0) to (1,0.75);
\draw[usual] (1,0) to (0,0.75);
\end{tikzpicture}
,\mspace{60mu}
\left(
\begin{tikzpicture}[anchorbase,scale=0.7,tinynodes]
\draw[usual] (0.5,0) to[out=270,in=180] (1,-0.25) to[out=0,in=270] (1.5,0);
\draw[usual] (0,-0.75) to[out=90,in=180] (0.25,-0.5) to[out=0,in=90] (0.5,-0.75);
\draw[usual] (0,0) to (1.5,-0.75);
\draw[usual] (1,0) to (1,-0.75);
\end{tikzpicture}
\right)^{\star}
=
\begin{tikzpicture}[anchorbase,scale=0.7,tinynodes]
\draw[usual] (0.5,0) to[out=90,in=180] (1,0.25) to[out=0,in=90] (1.5,0);
\draw[usual] (0,0.75) to[out=270,in=180] (0.25,0.5) to[out=0,in=270] (0.5,0.75);
\draw[usual] (0,0) to (1.5,0.75);
\draw[usual] (1,0) to (1,0.75);
\end{tikzpicture}
.
\end{gather*}
We have also illustrated the antiinvolution $(\placeholder)^{\star}$ on 
$\setstuff{Br}_{n}(\cvar)$ given by vertical mirroring.
Note that Brauer diagrams also make sense in a categorical setting, meaning 
with a different number of bottom and top points.

It is known that the Brauer algebra is cellular, see 
\cite[Section 4]{GrLe-cellular} or \cite[Section 5]{AnStTu-cellular-tilting}. 
To make $\setstuff{Br}_{n}(\cvar)$ cellular one needs 
$\mathcal{H}$-cells of size one, which is achieved in 
\cite[Section 4]{GrLe-cellular} by using 
that $\KK\setstuff{S}_{n}$ is 
a subalgebra of $\setstuff{Br}_{n}(\cvar)$. Then
they work with
the Kazhdan--Lusztig 
basis of $\KK\setstuff{S}_{n}$.
As we will describe now this is not necessary if one wants 
to parameterize simple modules.

We let $\Lambda=(\{n,n-2,n-4,\dots\},\leq_{\N})\subset\N$ 
with the usual partial order $\leq_{\N}$.
The $\lambda\in\Lambda$ are the \emph{through strands} of the Brauer diagrams,
that is, we let the set $M_{\lambda}$ consists of all Brauer diagrams 
from $n$ bottom points to $\lambda$ top points. These are the diagrams of the form $D,U$ below.
Moreover, we let $\sand=\KK\setstuff{S}_{\lambda}$ with the group element basis $\sandbasis$.
As our $\KK$-basis we choose 
\begin{gather*}
\{c_{D,b,U}^{\lambda}\mid\lambda\in\Lambda,D,U\in M_{\lambda},
b\in\sandbasis\}.
\end{gather*}
The picture for $n=4$ and $\lambda=2$ is: 
\begin{gather*}
\begin{aligned}
\begin{tikzpicture}[anchorbase,scale=1]
\draw[mor] (0,1) to (0.25,0.5) to (0.75,0.5) to (1,1) to (0,1);
\node at (0.5,0.75){$U$};
\end{tikzpicture}
&=
\begin{tikzpicture}[anchorbase,scale=0.7,tinynodes]
\draw[usual] (-0.5,0) to[out=270,in=90] (-0.5,-0.75);
\draw[usual] (0.5,0) to[out=270,in=90] (0,-0.75);
\draw[usual] (0,0) to[out=270,in=90] (0.5,-0.75);
\draw[usual] (1,0) to[out=270,in=90] (1,-0.75);
\draw[usual] (-0.5,-0.75) to[out=270,in=90] (-0.5,-1.5);
\draw[usual] (0,-0.75) to[out=270,in=90] (0,-1.5);
\draw[usual] (0.5,-0.75) to[out=270,in=180] (0.75,-1) to[out=0,in=270] (1,-0.75);
\end{tikzpicture}
,
\\
\begin{tikzpicture}[anchorbase,scale=1]
\draw[mor] (0.25,0) to (0.25,0.5) to (0.75,0.5) to (0.75,0) to (0.25,0);
\node at (0.5,0.25){$b$};
\node at (0.8,0.25){\phantom{$b$}};
\end{tikzpicture}
&=
\begin{tikzpicture}[anchorbase,scale=0.7,tinynodes]
\draw[usual] (0.5,0) to[out=90,in=270] (0,0.75);
\draw[usual] (0,0) to[out=90,in=270] (0.5,0.75);
\end{tikzpicture}
,
\\
\begin{tikzpicture}[anchorbase,scale=1]
\draw[mor] (0,-0.5) to (0.25,0) to (0.75,0) to (1,-0.5) to (0,-0.5);
\node at (0.5,-0.25){$D$};
\end{tikzpicture}
&=
\begin{tikzpicture}[anchorbase,scale=0.7,tinynodes]
\draw[usual] (-0.5,0) to[out=90,in=270] (-0.5,0.75);
\draw[usual] (0.5,0) to[out=90,in=270] (0,0.75);
\draw[usual] (0,0) to[out=90,in=270] (0.5,0.75);
\draw[usual] (1,0) to[out=90,in=270] (1,0.75);
\draw[usual] (-0.5,0.75) to[out=90,in=270] (-0.5,1.5);
\draw[usual] (0,0.75) to[out=90,in=270] (0,1.5);
\draw[usual] (0.5,0.75) to[out=90,in=180] (0.75,1) to[out=0,in=90] (1,0.75);
\end{tikzpicture}
.
\end{aligned}
\end{gather*}
That is, we divide a Brauer diagram into 
a diagram only containing caps and crossings, a diagram 
only containing cups and crossings, and a part only containing 
crossings.

\begin{proposition}
The above defines an involutive sandwich cell datum for $\setstuff{Br}_{n}(\cvar)$.
\end{proposition}

\begin{proof}
Identifying Brauer diagrams with immersed one-dimensional cobordisms 
(which is a well-known identification), all axioms are easily verified.
\end{proof}

\begin{example}\label{example:brauer}
The following illustrates basis elements and the cell structure 
of the cell $\mathcal{J}_{2}$ with two through strands for $n=4$ 
(using the same conventions as in \eqref{eq:cell-diagram}). 
In particular, the columns are $\mathcal{L}$-cells, the rows are $\mathcal{R}$-cells 
and the small boxes are $\mathcal{H}$-cells.
\begin{gather*}
\noalign{\global\arrayrulewidth=1mm}
\scalebox{0.85}{$\begin{tabular}{C!{\color{tomato}\vrule width 1mm}C!{\color{tomato}\vrule width 1mm}C!{\color{tomato}\vrule width 1mm}C!{\color{tomato}\vrule width 1mm}C!{\color{tomato}\vrule width 1mm}C}
\arrayrulecolor{tomato}
\begin{gathered}
\begin{tikzpicture}[anchorbase,scale=0.7,tinynodes,yscale=1]
\draw[usual] (0,0) to[out=90,in=180] (0.25,0.25) to[out=0,in=90] (0.5,0);
\draw[usual] (0,0.75) to[out=270,in=180] (0.25,0.5) to[out=0,in=270] (0.5,0.75);
\draw[usual] (1,0) to (1,0.75);
\draw[usual] (1.5,0) to (1.5,0.75);
\end{tikzpicture}
\\[3pt]
\begin{tikzpicture}[anchorbase,scale=0.7,tinynodes,yscale=1]
\draw[usual] (0,0) to[out=90,in=180] (0.25,0.25) to[out=0,in=90] (0.5,0);
\draw[usual] (0,0.75) to[out=270,in=180] (0.25,0.5) to[out=0,in=270] (0.5,0.75);
\draw[usual] (1,0) to (1.5,0.75);
\draw[usual] (1.5,0) to (1,0.75);
\end{tikzpicture}
\end{gathered} &
\begin{gathered}
\begin{tikzpicture}[anchorbase,scale=0.7,tinynodes,yscale=1]
\draw[usual] (0,0) to[out=90,in=180] (0.25,0.25) to[out=0,in=90] (0.5,0);
\draw[usual] (0.5,0.75) to[out=270,in=180] (0.75,0.5) to[out=0,in=270] (1,0.75);
\draw[usual] (1,0) to (0,0.75);
\draw[usual] (1.5,0) to (1.5,0.75);
\end{tikzpicture}
\\[3pt]
\begin{tikzpicture}[anchorbase,scale=0.7,tinynodes,yscale=1]
\draw[usual] (0,0) to[out=90,in=180] (0.25,0.25) to[out=0,in=90] (0.5,0);
\draw[usual] (0.5,0.75) to[out=270,in=180] (0.75,0.5) to[out=0,in=270] (1,0.75);
\draw[usual] (1.5,0) to (0,0.75);
\draw[usual] (1,0) to (1.5,0.75);
\end{tikzpicture}
\end{gathered} &
\begin{gathered}
\begin{tikzpicture}[anchorbase,scale=0.7,tinynodes,yscale=1]
\draw[usual] (0,0) to[out=90,in=180] (0.25,0.25) to[out=0,in=90] (0.5,0);
\draw[usual] (1,0.75) to[out=270,in=180] (1.25,0.5) to[out=0,in=270] (1.5,0.75);
\draw[usual] (1,0) to (0,0.75);
\draw[usual] (1.5,0) to (0.5,0.75);
\end{tikzpicture}
\\[3pt]
\begin{tikzpicture}[anchorbase,scale=0.7,tinynodes,yscale=1]
\draw[usual] (0,0) to[out=90,in=180] (0.25,0.25) to[out=0,in=90] (0.5,0);
\draw[usual] (1,0.75) to[out=270,in=180] (1.25,0.5) to[out=0,in=270] (1.5,0.75);
\draw[usual] (1,0) to (0.5,0.75);
\draw[usual] (1.5,0) to (0,0.75);
\end{tikzpicture}
\end{gathered} &
\begin{gathered}
\begin{tikzpicture}[anchorbase,scale=0.7,tinynodes,yscale=1]
\draw[usual] (0,0) to[out=90,in=180] (0.25,0.25) to[out=0,in=90] (0.5,0);
\draw[usual] (0,0.75) to[out=270,in=180] (0.5,0.5) to[out=0,in=270] (1,0.75);
\draw[usual] (1,0) to (0.5,0.75);
\draw[usual] (1.5,0) to (1.5,0.75);
\end{tikzpicture}
\\[3pt]
\begin{tikzpicture}[anchorbase,scale=0.7,tinynodes,yscale=1]
\draw[usual] (0,0) to[out=90,in=180] (0.25,0.25) to[out=0,in=90] (0.5,0);
\draw[usual] (0,0.75) to[out=270,in=180] (0.5,0.5) to[out=0,in=270] (1,0.75);
\draw[usual] (1,0) to (1.5,0.75);
\draw[usual] (1.5,0) to (0.5,0.75);
\end{tikzpicture}
\end{gathered} &
\begin{gathered}
\begin{tikzpicture}[anchorbase,scale=0.7,tinynodes,yscale=1]
\draw[usual] (0,0) to[out=90,in=180] (0.25,0.25) to[out=0,in=90] (0.5,0);
\draw[usual] (0.5,0.75) to[out=270,in=180] (1,0.5) to[out=0,in=270] (1.5,0.75);
\draw[usual] (1,0) to (0,0.75);
\draw[usual] (1.5,0) to (1,0.75);
\end{tikzpicture}
\\[3pt]
\begin{tikzpicture}[anchorbase,scale=0.7,tinynodes,yscale=-1]
\draw[usual] (0.5,0) to[out=90,in=180] (1,0.25) to[out=0,in=90] (1.5,0);
\draw[usual] (0,0.75) to[out=270,in=180] (0.25,0.5) to[out=0,in=270] (0.5,0.75);
\draw[usual] (1,0) to (1,0.75);
\draw[usual] (0,0) to (1.5,0.75);
\end{tikzpicture}
\end{gathered} &
\begin{gathered}
\begin{tikzpicture}[anchorbase,scale=0.7,tinynodes,yscale=-1]
\draw[usual] (0,0) to[out=90,in=180] (0.75,0.3) to[out=0,in=90] (1.5,0);
\draw[usual] (0,0.75) to[out=270,in=180] (0.25,0.5) to[out=0,in=270] (0.5,0.75);
\draw[usual] (0.5,0) to (1,0.75);
\draw[usual] (1,0) to (1.5,0.75);
\end{tikzpicture}
\\[3pt]
\begin{tikzpicture}[anchorbase,scale=0.7,tinynodes,yscale=-1]
\draw[usual] (0,0) to[out=90,in=180] (0.75,0.3) to[out=0,in=90] (1.5,0);
\draw[usual] (0,0.75) to[out=270,in=180] (0.25,0.5) to[out=0,in=270] (0.5,0.75);
\draw[usual] (0.5,0) to (1.5,0.75);
\draw[usual] (1,0) to (1,0.75);
\end{tikzpicture}
\end{gathered}
\\
\arrayrulecolor{tomato}\hline
\begin{gathered}
\begin{tikzpicture}[anchorbase,scale=0.7,tinynodes,yscale=1]
\draw[usual] (0.5,0) to[out=90,in=180] (0.75,0.25) to[out=0,in=90] (1,0);
\draw[usual] (0,0.75) to[out=270,in=180] (0.25,0.5) to[out=0,in=270] (0.5,0.75);
\draw[usual] (0,0) to (1,0.75);
\draw[usual] (1.5,0) to (1.5,0.75);
\end{tikzpicture}
\\[3pt]
\begin{tikzpicture}[anchorbase,scale=0.7,tinynodes,yscale=1]
\draw[usual] (0.5,0) to[out=90,in=180] (0.75,0.25) to[out=0,in=90] (1,0);
\draw[usual] (0,0.75) to[out=270,in=180] (0.25,0.5) to[out=0,in=270] (0.5,0.75);
\draw[usual] (0,0) to (1.5,0.75);
\draw[usual] (1.5,0) to (1,0.75);
\end{tikzpicture}
\end{gathered} &
\begin{gathered}
\begin{tikzpicture}[anchorbase,scale=0.7,tinynodes,yscale=1]
\draw[usual] (0.5,0) to[out=90,in=180] (0.75,0.25) to[out=0,in=90] (1,0);
\draw[usual] (0.5,0.75) to[out=270,in=180] (0.75,0.5) to[out=0,in=270] (1,0.75);
\draw[usual] (0,0) to (0,0.75);
\draw[usual] (1.5,0) to (1.5,0.75);
\end{tikzpicture}
\\[3pt]
\begin{tikzpicture}[anchorbase,scale=0.7,tinynodes,yscale=1]
\draw[usual] (0.5,0) to[out=90,in=180] (0.75,0.25) to[out=0,in=90] (1,0);
\draw[usual] (0.5,0.75) to[out=270,in=180] (0.75,0.5) to[out=0,in=270] (1,0.75);
\draw[usual] (0,0) to (1.5,0.75);
\draw[usual] (1.5,0) to (0,0.75);
\end{tikzpicture}
\end{gathered} &
\begin{gathered}
\begin{tikzpicture}[anchorbase,scale=0.7,tinynodes,yscale=1]
\draw[usual] (0.5,0) to[out=90,in=180] (0.75,0.25) to[out=0,in=90] (1,0);
\draw[usual] (1,0.75) to[out=270,in=180] (1.25,0.5) to[out=0,in=270] (1.5,0.75);
\draw[usual] (0,0) to (0,0.75);
\draw[usual] (1.5,0) to (0.5,0.75);
\end{tikzpicture}
\\[3pt]
\begin{tikzpicture}[anchorbase,scale=0.7,tinynodes,yscale=1]
\draw[usual] (0.5,0) to[out=90,in=180] (0.75,0.25) to[out=0,in=90] (1,0);
\draw[usual] (1,0.75) to[out=270,in=180] (1.25,0.5) to[out=0,in=270] (1.5,0.75);
\draw[usual] (1.5,0) to (0,0.75);
\draw[usual] (0,0) to (0.5,0.75);
\end{tikzpicture}
\end{gathered} & 
\begin{gathered}
\begin{tikzpicture}[anchorbase,scale=0.7,tinynodes,yscale=1]
\draw[usual] (0.5,0) to[out=90,in=180] (0.75,0.25) to[out=0,in=90] (1,0);
\draw[usual] (0,0.75) to[out=270,in=180] (0.5,0.5) to[out=0,in=270] (1,0.75);
\draw[usual] (0,0) to (0.5,0.75);
\draw[usual] (1.5,0) to (1.5,0.75);
\end{tikzpicture}
\\[3pt]
\begin{tikzpicture}[anchorbase,scale=0.7,tinynodes,yscale=1]
\draw[usual] (0.5,0) to[out=90,in=180] (0.75,0.25) to[out=0,in=90] (1,0);
\draw[usual] (0,0.75) to[out=270,in=180] (0.5,0.5) to[out=0,in=270] (1,0.75);
\draw[usual] (0,0) to (1.5,0.75);
\draw[usual] (1.5,0) to (0.5,0.75);
\end{tikzpicture}
\end{gathered} &
\begin{gathered}
\begin{tikzpicture}[anchorbase,scale=0.7,tinynodes,yscale=1]
\draw[usual] (0.5,0) to[out=90,in=180] (0.75,0.25) to[out=0,in=90] (1,0);
\draw[usual] (0.5,0.75) to[out=270,in=180] (1,0.5) to[out=0,in=270] (1.5,0.75);
\draw[usual] (0,0) to (0,0.75);
\draw[usual] (1.5,0) to (1,0.75);
\end{tikzpicture}
\\[3pt]
\begin{tikzpicture}[anchorbase,scale=0.7,tinynodes,yscale=1]
\draw[usual] (0.5,0) to[out=90,in=180] (0.75,0.25) to[out=0,in=90] (1,0);
\draw[usual] (0.5,0.75) to[out=270,in=180] (1,0.5) to[out=0,in=270] (1.5,0.75);
\draw[usual] (1.5,0) to (0,0.75);
\draw[usual] (0,0) to (1,0.75);
\end{tikzpicture}
\end{gathered} &
\begin{gathered}
\begin{tikzpicture}[anchorbase,scale=0.7,tinynodes,yscale=-1]
\draw[usual] (0,0) to[out=90,in=180] (0.75,0.3) to[out=0,in=90] (1.5,0);
\draw[usual] (0.5,0.75) to[out=270,in=180] (0.75,0.5) to[out=0,in=270] (1,0.75);
\draw[usual] (0.5,0) to (0,0.75);
\draw[usual] (1,0) to (1.5,0.75);
\end{tikzpicture}
\\[3pt]
\begin{tikzpicture}[anchorbase,scale=0.7,tinynodes,yscale=-1]
\draw[usual] (0,0) to[out=90,in=180] (0.75,0.3) to[out=0,in=90] (1.5,0);
\draw[usual] (0.5,0.75) to[out=270,in=180] (0.75,0.5) to[out=0,in=270] (1,0.75);
\draw[usual] (0.5,0) to (1.5,0.75);
\draw[usual] (1,0) to (0,0.75);
\end{tikzpicture}
\end{gathered}
\\
\arrayrulecolor{tomato}\hline
\begin{gathered}
\begin{tikzpicture}[anchorbase,scale=0.7,tinynodes,yscale=1]
\draw[usual] (1,0) to[out=90,in=180] (1.25,0.25) to[out=0,in=90] (1.5,0);
\draw[usual] (0,0.75) to[out=270,in=180] (0.25,0.5) to[out=0,in=270] (0.5,0.75);
\draw[usual] (0,0) to (1,0.75);
\draw[usual] (0.5,0) to (1.5,0.75);
\end{tikzpicture}
\\[3pt]
\begin{tikzpicture}[anchorbase,scale=0.7,tinynodes,yscale=1]
\draw[usual] (1,0) to[out=90,in=180] (1.25,0.25) to[out=0,in=90] (1.5,0);
\draw[usual] (0,0.75) to[out=270,in=180] (0.25,0.5) to[out=0,in=270] (0.5,0.75);
\draw[usual] (0,0) to (1.5,0.75);
\draw[usual] (0.5,0) to (1,0.75);
\end{tikzpicture}
\end{gathered} &
\begin{gathered}
\begin{tikzpicture}[anchorbase,scale=0.7,tinynodes,yscale=1]
\draw[usual] (1,0) to[out=90,in=180] (1.25,0.25) to[out=0,in=90] (1.5,0);
\draw[usual] (0.5,0.75) to[out=270,in=180] (0.75,0.5) to[out=0,in=270] (1,0.75);
\draw[usual] (0,0) to (0,0.75);
\draw[usual] (0.5,0) to (1.5,0.75);
\end{tikzpicture}
\\[3pt]
\begin{tikzpicture}[anchorbase,scale=0.7,tinynodes,yscale=1]
\draw[usual] (1,0) to[out=90,in=180] (1.25,0.25) to[out=0,in=90] (1.5,0);
\draw[usual] (0.5,0.75) to[out=270,in=180] (0.75,0.5) to[out=0,in=270] (1,0.75);
\draw[usual] (0,0) to (1.5,0.75);
\draw[usual] (0.5,0) to (0,0.75);
\end{tikzpicture}
\end{gathered} &
\begin{gathered}
\begin{tikzpicture}[anchorbase,scale=0.7,tinynodes,yscale=1]
\draw[usual] (1,0) to[out=90,in=180] (1.25,0.25) to[out=0,in=90] (1.5,0);
\draw[usual] (1,0.75) to[out=270,in=180] (1.25,0.5) to[out=0,in=270] (1.5,0.75);
\draw[usual] (0,0) to (0,0.75);
\draw[usual] (0.5,0) to (0.5,0.75);
\end{tikzpicture}
\\[3pt]
\begin{tikzpicture}[anchorbase,scale=0.7,tinynodes,yscale=1]
\draw[usual] (1,0) to[out=90,in=180] (1.25,0.25) to[out=0,in=90] (1.5,0);
\draw[usual] (1,0.75) to[out=270,in=180] (1.25,0.5) to[out=0,in=270] (1.5,0.75);
\draw[usual] (0,0) to (0.5,0.75);
\draw[usual] (0.5,0) to (0,0.75);
\end{tikzpicture}
\end{gathered} &
\begin{gathered}
\begin{tikzpicture}[anchorbase,scale=0.7,tinynodes,yscale=1]
\draw[usual] (1,0) to[out=90,in=180] (1.25,0.25) to[out=0,in=90] (1.5,0);
\draw[usual] (0,0.75) to[out=270,in=180] (0.5,0.5) to[out=0,in=270] (1,0.75);
\draw[usual] (0,0) to (0.5,0.75);
\draw[usual] (0.5,0) to (1.5,0.75);
\end{tikzpicture}
\\[3pt]
\begin{tikzpicture}[anchorbase,scale=0.7,tinynodes,yscale=1]
\draw[usual] (1,0) to[out=90,in=180] (1.25,0.25) to[out=0,in=90] (1.5,0);
\draw[usual] (0,0.75) to[out=270,in=180] (0.5,0.5) to[out=0,in=270] (1,0.75);
\draw[usual] (0.5,0) to (0.5,0.75);
\draw[usual] (0,0) to (1.5,0.75);
\end{tikzpicture}
\end{gathered} &
\begin{gathered}
\begin{tikzpicture}[anchorbase,scale=0.7,tinynodes,yscale=1]
\draw[usual] (1,0) to[out=90,in=180] (1.25,0.25) to[out=0,in=90] (1.5,0);
\draw[usual] (0.5,0.75) to[out=270,in=180] (1,0.5) to[out=0,in=270] (1.5,0.75);
\draw[usual] (0,0) to (0,0.75);
\draw[usual] (0.5,0) to (1,0.75);
\end{tikzpicture}
\\[3pt]
\begin{tikzpicture}[anchorbase,scale=0.7,tinynodes,yscale=1]
\draw[usual] (1,0) to[out=90,in=180] (1.25,0.25) to[out=0,in=90] (1.5,0);
\draw[usual] (0.5,0.75) to[out=270,in=180] (1,0.5) to[out=0,in=270] (1.5,0.75);
\draw[usual] (0.5,0) to (0,0.75);
\draw[usual] (0,0) to (1,0.75);
\end{tikzpicture}
\end{gathered} &
\begin{gathered}
\begin{tikzpicture}[anchorbase,scale=0.7,tinynodes,yscale=-1]
\draw[usual] (0,0) to[out=90,in=180] (0.75,0.3) to[out=0,in=90] (1.5,0);
\draw[usual] (1,0.75) to[out=270,in=180] (1.25,0.5) to[out=0,in=270] (1.5,0.75);
\draw[usual] (0.5,0) to (0,0.75);
\draw[usual] (1,0) to (0.5,0.75);
\end{tikzpicture}
\\[3pt]
\begin{tikzpicture}[anchorbase,scale=0.7,tinynodes,yscale=-1]
\draw[usual] (0,0) to[out=90,in=180] (0.75,0.3) to[out=0,in=90] (1.5,0);
\draw[usual] (1,0.75) to[out=270,in=180] (1.25,0.5) to[out=0,in=270] (1.5,0.75);
\draw[usual] (0.5,0) to (0.5,0.75);
\draw[usual] (1,0) to (0,0.75);
\end{tikzpicture}
\end{gathered}
\\
\arrayrulecolor{tomato}\hline
\begin{gathered}
\begin{tikzpicture}[anchorbase,scale=0.7,tinynodes,yscale=1]
\draw[usual] (0,0) to[out=90,in=180] (0.5,0.25) to[out=0,in=90] (1,0);
\draw[usual] (0,0.75) to[out=270,in=180] (0.25,0.5) to[out=0,in=270] (0.5,0.75);
\draw[usual] (0.5,0) to (1,0.75);
\draw[usual] (1.5,0) to (1.5,0.75);
\end{tikzpicture}
\\[3pt]
\begin{tikzpicture}[anchorbase,scale=0.7,tinynodes,yscale=1]
\draw[usual] (0,0) to[out=90,in=180] (0.5,0.25) to[out=0,in=90] (1,0);
\draw[usual] (0,0.75) to[out=270,in=180] (0.25,0.5) to[out=0,in=270] (0.5,0.75);
\draw[usual] (1.5,0) to (1,0.75);
\draw[usual] (0.5,0) to (1.5,0.75);
\end{tikzpicture}
\end{gathered} &
\begin{gathered}
\begin{tikzpicture}[anchorbase,scale=0.7,tinynodes,yscale=1]
\draw[usual] (0,0) to[out=90,in=180] (0.5,0.25) to[out=0,in=90] (1,0);
\draw[usual] (0.5,0.75) to[out=270,in=180] (0.75,0.5) to[out=0,in=270] (1,0.75);
\draw[usual] (0.5,0) to (0,0.75);
\draw[usual] (1.5,0) to (1.5,0.75);
\end{tikzpicture}
\\[3pt]
\begin{tikzpicture}[anchorbase,scale=0.7,tinynodes,yscale=1]
\draw[usual] (0,0) to[out=90,in=180] (0.5,0.25) to[out=0,in=90] (1,0);
\draw[usual] (0.5,0.75) to[out=270,in=180] (0.75,0.5) to[out=0,in=270] (1,0.75);
\draw[usual] (0.5,0) to (1.5,0.75);
\draw[usual] (1.5,0) to (0,0.75);
\end{tikzpicture}
\end{gathered} &
\begin{gathered}
\begin{tikzpicture}[anchorbase,scale=0.7,tinynodes,yscale=1]
\draw[usual] (0,0) to[out=90,in=180] (0.5,0.25) to[out=0,in=90] (1,0);
\draw[usual] (1,0.75) to[out=270,in=180] (1.25,0.5) to[out=0,in=270] (1.5,0.75);
\draw[usual] (0.5,0) to (0,0.75);
\draw[usual] (1.5,0) to (0.5,0.75);
\end{tikzpicture}
\\[3pt]
\begin{tikzpicture}[anchorbase,scale=0.7,tinynodes,yscale=1]
\draw[usual] (0,0) to[out=90,in=180] (0.5,0.25) to[out=0,in=90] (1,0);
\draw[usual] (1,0.75) to[out=270,in=180] (1.25,0.5) to[out=0,in=270] (1.5,0.75);
\draw[usual] (0.5,0) to (0.5,0.75);
\draw[usual] (1.5,0) to (0,0.75);
\end{tikzpicture}
\end{gathered} &
\begin{gathered}
\begin{tikzpicture}[anchorbase,scale=0.7,tinynodes,yscale=1]
\draw[usual] (0,0) to[out=90,in=180] (0.5,0.25) to[out=0,in=90] (1,0);
\draw[usual] (0,0.75) to[out=270,in=180] (0.5,0.5) to[out=0,in=270] (1,0.75);
\draw[usual] (0.5,0) to (0.5,0.75);
\draw[usual] (1.5,0) to (1.5,0.75);
\end{tikzpicture}
\\[3pt]
\begin{tikzpicture}[anchorbase,scale=0.7,tinynodes,yscale=1]
\draw[usual] (0,0) to[out=90,in=180] (0.5,0.25) to[out=0,in=90] (1,0);
\draw[usual] (0,0.75) to[out=270,in=180] (0.5,0.5) to[out=0,in=270] (1,0.75);
\draw[usual] (0.5,0) to (1.5,0.75);
\draw[usual] (1.5,0) to (0.5,0.75);
\end{tikzpicture}
\end{gathered} & 
\begin{gathered}
\begin{tikzpicture}[anchorbase,scale=0.7,tinynodes,yscale=1]
\draw[usual] (0,0) to[out=90,in=180] (0.5,0.25) to[out=0,in=90] (1,0);
\draw[usual] (0.5,0.75) to[out=270,in=180] (1,0.5) to[out=0,in=270] (1.5,0.75);
\draw[usual] (0.5,0) to (0,0.75);
\draw[usual] (1.5,0) to (1,0.75);
\end{tikzpicture}
\\[3pt]
\begin{tikzpicture}[anchorbase,scale=0.7,tinynodes,yscale=1]
\draw[usual] (0,0) to[out=90,in=180] (0.5,0.25) to[out=0,in=90] (1,0);
\draw[usual] (0.5,0.75) to[out=270,in=180] (1,0.5) to[out=0,in=270] (1.5,0.75);
\draw[usual] (1.5,0) to (0,0.75);
\draw[usual] (0.5,0) to (1,0.75);
\end{tikzpicture}
\end{gathered} &
\begin{gathered}
\begin{tikzpicture}[anchorbase,scale=0.7,tinynodes,yscale=-1]
\draw[usual] (0,0) to[out=90,in=180] (0.75,0.3) to[out=0,in=90] (1.5,0);
\draw[usual] (0,0.75) to[out=270,in=180] (0.5,0.5) to[out=0,in=270] (1,0.75);
\draw[usual] (0.5,0) to (0.5,0.75);
\draw[usual] (1,0) to (1.5,0.75);
\end{tikzpicture}
\\[3pt]
\begin{tikzpicture}[anchorbase,scale=0.7,tinynodes,yscale=-1]
\draw[usual] (0,0) to[out=90,in=180] (0.75,0.3) to[out=0,in=90] (1.5,0);
\draw[usual] (0,0.75) to[out=270,in=180] (0.5,0.5) to[out=0,in=270] (1,0.75);
\draw[usual] (0.5,0) to (1.5,0.75);
\draw[usual] (1,0) to (0.5,0.75);
\end{tikzpicture}
\end{gathered}
\\
\arrayrulecolor{tomato}\hline
\cellcolor{mydarkblue!25}
\begin{gathered}
\begin{tikzpicture}[anchorbase,scale=0.7,tinynodes,yscale=1]
\draw[usual] (0.5,0) to[out=90,in=180] (1,0.25) to[out=0,in=90] (1.5,0);
\draw[usual] (0,0.75) to[out=270,in=180] (0.25,0.5) to[out=0,in=270] (0.5,0.75);
\draw[usual] (0,0) to (1,0.75);
\draw[usual] (1,0) to (1.5,0.75);
\end{tikzpicture}
\\[3pt]
\begin{tikzpicture}[anchorbase,scale=0.7,tinynodes,yscale=1]
\draw[usual] (0.5,0) to[out=90,in=180] (1,0.25) to[out=0,in=90] (1.5,0);
\draw[usual] (0,0.75) to[out=270,in=180] (0.25,0.5) to[out=0,in=270] (0.5,0.75);
\draw[usual] (0,0) to (1.5,0.75);
\draw[usual] (1,0) to (1,0.75);
\end{tikzpicture}
\end{gathered} &
\begin{gathered}
\begin{tikzpicture}[anchorbase,scale=0.7,tinynodes,yscale=1]
\draw[usual] (0.5,0) to[out=90,in=180] (1,0.25) to[out=0,in=90] (1.5,0);
\draw[usual] (0.5,0.75) to[out=270,in=180] (0.75,0.5) to[out=0,in=270] (1,0.75);
\draw[usual] (0,0) to (0,0.75);
\draw[usual] (1,0) to (1.5,0.75);
\end{tikzpicture}
\\[3pt]
\begin{tikzpicture}[anchorbase,scale=0.7,tinynodes,yscale=1]
\draw[usual] (0.5,0) to[out=90,in=180] (1,0.25) to[out=0,in=90] (1.5,0);
\draw[usual] (0.5,0.75) to[out=270,in=180] (0.75,0.5) to[out=0,in=270] (1,0.75);
\draw[usual] (0,0) to (1.5,0.75);
\draw[usual] (1,0) to (0,0.75);
\end{tikzpicture}
\end{gathered} &
\begin{gathered}
\begin{tikzpicture}[anchorbase,scale=0.7,tinynodes,yscale=1]
\draw[usual] (0.5,0) to[out=90,in=180] (1,0.25) to[out=0,in=90] (1.5,0);
\draw[usual] (1,0.75) to[out=270,in=180] (1.25,0.5) to[out=0,in=270] (1.5,0.75);
\draw[usual] (0,0) to (0,0.75);
\draw[usual] (1,0) to (0.5,0.75);
\end{tikzpicture}
\\[3pt]
\begin{tikzpicture}[anchorbase,scale=0.7,tinynodes,yscale=1]
\draw[usual] (0.5,0) to[out=90,in=180] (1,0.25) to[out=0,in=90] (1.5,0);
\draw[usual] (1,0.75) to[out=270,in=180] (1.25,0.5) to[out=0,in=270] (1.5,0.75);
\draw[usual] (0,0) to (0.5,0.75);
\draw[usual] (1,0) to (0,0.75);
\end{tikzpicture}
\end{gathered} &
\begin{gathered}
\begin{tikzpicture}[anchorbase,scale=0.7,tinynodes,yscale=1]
\draw[usual] (0.5,0) to[out=90,in=180] (1,0.25) to[out=0,in=90] (1.5,0);
\draw[usual] (0,0.75) to[out=270,in=180] (0.5,0.5) to[out=0,in=270] (1,0.75);
\draw[usual] (0,0) to (0.5,0.75);
\draw[usual] (1,0) to (1.5,0.75);
\end{tikzpicture}
\\[3pt]
\begin{tikzpicture}[anchorbase,scale=0.7,tinynodes,yscale=1]
\draw[usual] (0.5,0) to[out=90,in=180] (1,0.25) to[out=0,in=90] (1.5,0);
\draw[usual] (0,0.75) to[out=270,in=180] (0.5,0.5) to[out=0,in=270] (1,0.75);
\draw[usual] (0,0) to (1.5,0.75);
\draw[usual] (1,0) to (0.5,0.75);
\end{tikzpicture}
\end{gathered} &
\begin{gathered}
\begin{tikzpicture}[anchorbase,scale=0.7,tinynodes,yscale=1]
\draw[usual] (0.5,0) to[out=90,in=180] (1,0.25) to[out=0,in=90] (1.5,0);
\draw[usual] (0.5,0.75) to[out=270,in=180] (1,0.5) to[out=0,in=270] (1.5,0.75);
\draw[usual] (0,0) to (0,0.75);
\draw[usual] (1,0) to (1,0.75);
\end{tikzpicture}
\\[3pt]
\begin{tikzpicture}[anchorbase,scale=0.7,tinynodes,yscale=1]
\draw[usual] (0.5,0) to[out=90,in=180] (1,0.25) to[out=0,in=90] (1.5,0);
\draw[usual] (0.5,0.75) to[out=270,in=180] (1,0.5) to[out=0,in=270] (1.5,0.75);
\draw[usual] (0,0) to (1,0.75);
\draw[usual] (1,0) to (0,0.75);
\end{tikzpicture}
\end{gathered} &
\begin{gathered}
\begin{tikzpicture}[anchorbase,scale=0.7,tinynodes,yscale=-1]
\draw[usual] (0,0) to[out=90,in=180] (0.75,0.3) to[out=0,in=90] (1.5,0);
\draw[usual] (0.5,0.75) to[out=270,in=180] (1,0.5) to[out=0,in=270] (1.5,0.75);
\draw[usual] (0.5,0) to (0,0.75);
\draw[usual] (1,0) to (1,0.75);
\end{tikzpicture}
\\[3pt]
\begin{tikzpicture}[anchorbase,scale=0.7,tinynodes,yscale=-1]
\draw[usual] (0,0) to[out=90,in=180] (0.75,0.3) to[out=0,in=90] (1.5,0);
\draw[usual] (0.5,0.75) to[out=270,in=180] (1,0.5) to[out=0,in=270] (1.5,0.75);
\draw[usual] (0.5,0) to (1,0.75);
\draw[usual] (1,0) to (0,0.75);
\end{tikzpicture}
\end{gathered}
\\
\arrayrulecolor{tomato}\hline
\begin{gathered}
\begin{tikzpicture}[anchorbase,scale=0.7,tinynodes,yscale=1]
\draw[usual] (0,0) to[out=90,in=180] (0.75,0.3) to[out=0,in=90] (1.5,0);
\draw[usual] (0,0.75) to[out=270,in=180] (0.25,0.5) to[out=0,in=270] (0.5,0.75);
\draw[usual] (0.5,0) to (1,0.75);
\draw[usual] (1,0) to (1.5,0.75);
\end{tikzpicture}
\\[3pt]
\begin{tikzpicture}[anchorbase,scale=0.7,tinynodes,yscale=1]
\draw[usual] (0,0) to[out=90,in=180] (0.75,0.3) to[out=0,in=90] (1.5,0);
\draw[usual] (0,0.75) to[out=270,in=180] (0.25,0.5) to[out=0,in=270] (0.5,0.75);
\draw[usual] (0.5,0) to (1.5,0.75);
\draw[usual] (1,0) to (1,0.75);
\end{tikzpicture}
\end{gathered} &
\begin{gathered}
\begin{tikzpicture}[anchorbase,scale=0.7,tinynodes,yscale=1]
\draw[usual] (0,0) to[out=90,in=180] (0.75,0.3) to[out=0,in=90] (1.5,0);
\draw[usual] (0.5,0.75) to[out=270,in=180] (0.75,0.5) to[out=0,in=270] (1,0.75);
\draw[usual] (0.5,0) to (0,0.75);
\draw[usual] (1,0) to (1.5,0.75);
\end{tikzpicture}
\\[3pt]
\begin{tikzpicture}[anchorbase,scale=0.7,tinynodes,yscale=1]
\draw[usual] (0,0) to[out=90,in=180] (0.75,0.3) to[out=0,in=90] (1.5,0);
\draw[usual] (0.5,0.75) to[out=270,in=180] (0.75,0.5) to[out=0,in=270] (1,0.75);
\draw[usual] (0.5,0) to (1.5,0.75);
\draw[usual] (1,0) to (0,0.75);
\end{tikzpicture}
\end{gathered} &
\begin{gathered}
\begin{tikzpicture}[anchorbase,scale=0.7,tinynodes,yscale=1]
\draw[usual] (0,0) to[out=90,in=180] (0.75,0.3) to[out=0,in=90] (1.5,0);
\draw[usual] (1,0.75) to[out=270,in=180] (1.25,0.5) to[out=0,in=270] (1.5,0.75);
\draw[usual] (0.5,0) to (0,0.75);
\draw[usual] (1,0) to (0.5,0.75);
\end{tikzpicture}
\\[3pt]
\begin{tikzpicture}[anchorbase,scale=0.7,tinynodes,yscale=1]
\draw[usual] (0,0) to[out=90,in=180] (0.75,0.3) to[out=0,in=90] (1.5,0);
\draw[usual] (1,0.75) to[out=270,in=180] (1.25,0.5) to[out=0,in=270] (1.5,0.75);
\draw[usual] (0.5,0) to (0.5,0.75);
\draw[usual] (1,0) to (0,0.75);
\end{tikzpicture}
\end{gathered} &
\begin{gathered}
\begin{tikzpicture}[anchorbase,scale=0.7,tinynodes,yscale=1]
\draw[usual] (0,0) to[out=90,in=180] (0.75,0.3) to[out=0,in=90] (1.5,0);
\draw[usual] (0,0.75) to[out=270,in=180] (0.5,0.5) to[out=0,in=270] (1,0.75);
\draw[usual] (0.5,0) to (0.5,0.75);
\draw[usual] (1,0) to (1.5,0.75);
\end{tikzpicture}
\\[3pt]
\begin{tikzpicture}[anchorbase,scale=0.7,tinynodes,yscale=1]
\draw[usual] (0,0) to[out=90,in=180] (0.75,0.3) to[out=0,in=90] (1.5,0);
\draw[usual] (0,0.75) to[out=270,in=180] (0.5,0.5) to[out=0,in=270] (1,0.75);
\draw[usual] (0.5,0) to (1.5,0.75);
\draw[usual] (1,0) to (0.5,0.75);
\end{tikzpicture}
\end{gathered} &
\begin{gathered}
\begin{tikzpicture}[anchorbase,scale=0.7,tinynodes,yscale=1]
\draw[usual] (0,0) to[out=90,in=180] (0.75,0.3) to[out=0,in=90] (1.5,0);
\draw[usual] (0.5,0.75) to[out=270,in=180] (1,0.5) to[out=0,in=270] (1.5,0.75);
\draw[usual] (0.5,0) to (0,0.75);
\draw[usual] (1,0) to (1,0.75);
\end{tikzpicture}
\\[3pt]
\begin{tikzpicture}[anchorbase,scale=0.7,tinynodes,yscale=1]
\draw[usual] (0,0) to[out=90,in=180] (0.75,0.3) to[out=0,in=90] (1.5,0);
\draw[usual] (0.5,0.75) to[out=270,in=180] (1,0.5) to[out=0,in=270] (1.5,0.75);
\draw[usual] (0.5,0) to (1,0.75);
\draw[usual] (1,0) to (0,0.75);
\end{tikzpicture}
\end{gathered} &
\begin{gathered}
\begin{tikzpicture}[anchorbase,scale=0.7,tinynodes,yscale=1]
\draw[usual] (0,0) to[out=90,in=180] (0.75,0.25) to[out=0,in=90] (1.5,0);
\draw[usual] (0,0.75) to[out=270,in=180] (0.75,0.5) to[out=0,in=270] (1.5,0.75);
\draw[usual] (0.5,0) to (0.5,0.75);
\draw[usual] (1,0) to (1,0.75);
\end{tikzpicture}
\\[3pt]
\begin{tikzpicture}[anchorbase,scale=0.7,tinynodes,yscale=1]
\draw[usual] (0,0) to[out=90,in=180] (0.75,0.25) to[out=0,in=90] (1.5,0);
\draw[usual] (0,0.75) to[out=270,in=180] (0.75,0.5) to[out=0,in=270] (1.5,0.75);
\draw[usual] (0.5,0) to (1,0.75);
\draw[usual] (1,0) to (0.5,0.75);
\end{tikzpicture}
\end{gathered}
\end{tabular}$}
\quad,
\\
\begin{gathered}
\text{multiplication}
\\[-5pt]
\text{table of the}
\\[-5pt]
\text{colored box}
\end{gathered}
\colon
\begin{gathered}
\begin{tikzpicture}[anchorbase,scale=0.7,tinynodes]
\draw[usual] (0.5,0) to[out=90,in=180] (1,0.25) to[out=0,in=90] (1.5,0);
\draw[usual] (0,0.75) to[out=270,in=180] (0.25,0.5) to[out=0,in=270] (0.5,0.75);
\draw[usual] (0,0) to (1.5,0.75);
\draw[usual] (1,0) to (1,0.75);
\draw[usual] (0.5,0.75) to[out=90,in=180] (1,1) to[out=0,in=90] (1.5,0.75);
\draw[usual] (0,1.5) to[out=270,in=180] (0.25,1.25) to[out=0,in=270] (0.5,1.5);
\draw[usual] (0,0.75) to (1.5,1.5);
\draw[usual] (1,0.75) to (1,1.5);
\end{tikzpicture}
=
\begin{tikzpicture}[anchorbase,scale=0.7,tinynodes]
\draw[usual] (0.5,0) to[out=90,in=180] (1,0.25) to[out=0,in=90] (1.5,0);
\draw[usual] (0,0.75) to[out=270,in=180] (0.25,0.5) to[out=0,in=270] (0.5,0.75);
\draw[usual] (0,0) to (1.5,0.75);
\draw[usual] (1,0) to (1,0.75);
\end{tikzpicture}
,\quad
\begin{tikzpicture}[anchorbase,scale=0.7,tinynodes]
\draw[usual] (0.5,0) to[out=90,in=180] (1,0.25) to[out=0,in=90] (1.5,0);
\draw[usual] (0,0.75) to[out=270,in=180] (0.25,0.5) to[out=0,in=270] (0.5,0.75);
\draw[usual] (0,0) to (1,0.75);
\draw[usual] (1,0) to (1.5,0.75);
\draw[usual] (0.5,0.75) to[out=90,in=180] (1,1) to[out=0,in=90] (1.5,0.75);
\draw[usual] (0,1.5) to[out=270,in=180] (0.25,1.25) to[out=0,in=270] (0.5,1.5);
\draw[usual] (0,0.75) to (1.5,1.5);
\draw[usual] (1,0.75) to (1,1.5);
\end{tikzpicture}
=
\begin{tikzpicture}[anchorbase,scale=0.7,tinynodes]
\draw[usual] (0.5,0) to[out=90,in=180] (1,0.25) to[out=0,in=90] (1.5,0);
\draw[usual] (0,0.75) to[out=270,in=180] (0.25,0.5) to[out=0,in=270] (0.5,0.75);
\draw[usual] (0,0) to (1,0.75);
\draw[usual] (1,0) to (1.5,0.75);
\end{tikzpicture}
,
\\
\begin{tikzpicture}[anchorbase,scale=0.7,tinynodes]
\draw[usual] (0.5,0) to[out=90,in=180] (1,0.25) to[out=0,in=90] (1.5,0);
\draw[usual] (0,0.75) to[out=270,in=180] (0.25,0.5) to[out=0,in=270] (0.5,0.75);
\draw[usual] (0,0) to (1.5,0.75);
\draw[usual] (1,0) to (1,0.75);
\draw[usual] (0.5,0.75) to[out=90,in=180] (1,1) to[out=0,in=90] (1.5,0.75);
\draw[usual] (0,1.5) to[out=270,in=180] (0.25,1.25) to[out=0,in=270] (0.5,1.5);
\draw[usual] (0,0.75) to (1,1.5);
\draw[usual] (1,0.75) to (1.5,1.5);
\end{tikzpicture}
=
\begin{tikzpicture}[anchorbase,scale=0.7,tinynodes]
\draw[usual] (0.5,0) to[out=90,in=180] (1,0.25) to[out=0,in=90] (1.5,0);
\draw[usual] (0,0.75) to[out=270,in=180] (0.25,0.5) to[out=0,in=270] (0.5,0.75);
\draw[usual] (0,0) to (1,0.75);
\draw[usual] (1,0) to (1.5,0.75);
\end{tikzpicture}
,\quad
\begin{tikzpicture}[anchorbase,scale=0.7,tinynodes]
\draw[usual] (0.5,0) to[out=90,in=180] (1,0.25) to[out=0,in=90] (1.5,0);
\draw[usual] (0,0.75) to[out=270,in=180] (0.25,0.5) to[out=0,in=270] (0.5,0.75);
\draw[usual] (0,0) to (1,0.75);
\draw[usual] (1,0) to (1.5,0.75);
\draw[usual] (0.5,0.75) to[out=90,in=180] (1,1) to[out=0,in=90] (1.5,0.75);
\draw[usual] (0,1.5) to[out=270,in=180] (0.25,1.25) to[out=0,in=270] (0.5,1.5);
\draw[usual] (0,0.75) to (1,1.5);
\draw[usual] (1,0.75) to (1.5,1.5);
\end{tikzpicture}
=
\begin{tikzpicture}[anchorbase,scale=0.7,tinynodes]
\draw[usual] (0.5,0) to[out=90,in=180] (1,0.25) to[out=0,in=90] (1.5,0);
\draw[usual] (0,0.75) to[out=270,in=180] (0.25,0.5) to[out=0,in=270] (0.5,0.75);
\draw[usual] (0,0) to (1.5,0.75);
\draw[usual] (1,0) to (1,0.75);
\end{tikzpicture}
.
\end{gathered}
\end{gather*}
Each $\mathcal{H}$-cell is of size two, and
the colored box is an $\mathcal{H}$-cell
isomorphic to $\KK\setstuff{S}_{2}$ regardless
of $\cvar$. This is illustrated on the right.
\end{example}

We obtain the well-known classification of simple 
$\setstuff{Br}_{n}(\cvar)$-modules, {\cf} \cite[Theorem 4.17]{GrLe-cellular}:

\begin{theorem}\label{theorem:brauer}
Let $\KK$ be a field.
\begin{enumerate}

\item If $\cvar\neq 0$, or $\cvar=0$ and $\lambda\neq 0$ is odd, 
then all $\lambda\in\Lambda$ are apexes. In the remaining case, 
$\cvar=0$ and $\lambda=0$ (this only happens if $n$ is even), 
all $\lambda\in\Lambda-\{0\}$ are apexes, but $\lambda=0$ is not an apex.

\item The simple $\setstuff{Br}_{n}(\cvar)$-modules of 
apex $\lambda\in\Lambda$ 
are parameterized by simple $\KK\setstuff{S}_{\lambda}$-modules.

\item The simple $\setstuff{Br}_{n}(\cvar)$-modules of 
apex $\lambda\in\Lambda$ can be constructed as 
the simple heads of
$\mathrm{Ind}_{\KK\setstuff{S}_{\lambda}}^{\setstuff{Br}_{n}(\cvar)}(K)$, 
where $K$ runs over (equivalence classes of) 
simple $\KK\setstuff{S}_{\lambda}$-modules.

\end{enumerate}
\end{theorem}

\begin{proof}
We apply \fullref{theorem:classification}.(c) together with the 
following observations.	

Firstly, we are clearly in the situation of \autoref{lemma:h-subalgebra} 
for the Brauer algebras. 
Second, it is easy to see that the $\mathcal{H}$-cells are of a 
similar form 
as in \fullref{example:brauer}, and,
if $\cvar\neq 0$, then all $\mathcal{H}$-cells are 
isomorphic to $\KK\setstuff{S}_{\lambda}$.
For $\cvar=0$ and $\lambda\neq 0$, and any 
two half-diagrams 
we can find an element such that their pairing is one, by 
straightening cups-caps. 
Here is an example that easily generalizes:
\begin{gather*}
\phi^{\lambda}
\left(
\begin{tikzpicture}[anchorbase,scale=0.7,tinynodes]
\draw[usual] (-1,0) to[out=270,in=0] (-1.25,-0.25) 
to[out=180,in=270] (-1.5,0);
\draw[usual] (0,0) to (0,-0.75);
\draw[usual] (-0.5,0) to (-0.5,-0.75);
\end{tikzpicture}
,
\begin{tikzpicture}[anchorbase,scale=0.7,tinynodes]
\draw[usual] (0,-0.75) to[out=90,in=0] (-0.5,-0.5) 
to[out=180,in=90] (-1,-0.75);
\draw[usual] (0,0) to (-0.5,-0.75);
\draw[usual] (-1,0) to (-1.5,-0.75);
\end{tikzpicture}
\right)
=
\begin{tikzpicture}[anchorbase,scale=0.7,tinynodes]
\draw[usual] (0,-0.75) to[out=90,in=0] (-0.5,-0.5) 
to[out=180,in=90] (-1,-0.75);
\draw[usual] (0,0) to (-0.5,-0.75);
\draw[usual] (-1,0) to (-1.5,-0.75);
\draw[usual] (-1,-0.75) to[out=270,in=0] (-1.25,-1) 
to[out=180,in=270] (-1.5,-0.75);
\draw[usual] (0,-0.75) to (0,-1.5);
\draw[usual] (-0.5,-0.75) to (-0.5,-1.5);
\end{tikzpicture}
=
1\cdot
\begin{tikzpicture}[anchorbase,scale=0.7,tinynodes]
\draw[usual] (0,0) to (0.5,0.75);
\draw[usual] (0.5,0) to (0,0.75);
\end{tikzpicture}
.
\end{gather*}
This trick works unless $\lambda=0$, 
which is clearly degenerate if $\cvar=0$.
\end{proof}

\begin{remark}
The same strategy works, {\muta}, for the 
oriented (or walled) Brauer algebra, other diagram algebras 
in the same spirit, {\eg} partition algebras, and 
the quantum versions of these diagram algebras such as 
the Birman--Murakami--Wenzl 
algebra (we will treat this case for higher genus in 
\fullref{section:brauer} below).
We leave the details to the interested reader. 
\end{remark}

\section{Handlebody braid and Coxeter groups}\label{section:braids}

Throughout, we fix the \emph{genus} $g\in\N$ as well as the 
\emph{number of strands} $n\in\N_{>0}$. As 
a general conventions, all notions involving $g$ are vacuous if $g=0$, 
and similarly all notions involving $n$ are vacuous for $n=1$.

\subsection{Handlebody braid diagrams}\label{subsection:braids}

In this section we consider 
\emph{handlebody braid diagrams} (in $n$ strands and of genus $g$).
These diagrams are similar to classical braid diagrams with $g+n$ strands 
in the following sense. We let
\begin{gather}\label{eq:lefttoright-labels}
\tau_{u}
=
\begin{tikzpicture}[anchorbase,scale=0.7,tinynodes]
\draw[usual,crossline] (1.5,0)node[below,black,yshift=-1pt]{$1$} 
to[out=90,in=270] (-0.25,0.75);
\draw[pole,crosspole] (-0.5,0)node[below,black,yshift=-1pt]{$u{-}1$} 
to[out=90,in=270] (-0.5,1.5)node[above,black,yshift=-3pt]{$u{-}1$};
\draw[pole,crosspole] (-1,0)node[below,black,yshift=-1pt]{$1$} 
to[out=90,in=270] (-1,1.5)node[above,black,yshift=-3pt]{$1$};
\draw[pole,crosspole] (0.5,0)node[below,black,yshift=-1pt]{$u{+}1$} 
to[out=90,in=270] (0.5,1.5)node[above,black,yshift=-3pt]{$u{+}1$};
\draw[pole,crosspole] (1,0)node[below,black,yshift=-1pt]{$g$} 
to[out=90,in=270] (1,1.5)node[above,black,yshift=-3pt]{$g$};
\draw[pole,crosspole] (0,0)node[below,black,yshift=-1pt]{$u$} 
to[out=90,in=270] (0,1.5)node[above,black,yshift=-3pt]{$u$};
\draw[usual,crossline] (-0.25,0.75) to[out=90,in=270] (1.5,1.5)node[above,black,yshift=-3pt]{$1$};
\node at (0.75,0.75) {$\dots$};
\node at (-0.75,0.75) {$\dots$};
\end{tikzpicture}
,\quad
\beta_{i}
=
\begin{tikzpicture}[anchorbase,scale=0.7,tinynodes]
\draw[usual,crossline] (0.5,0)node[below,black,yshift=-1pt]{$i{+}1$} 
to[out=90,in=270] (0,0.75)node[above,black,yshift=-3pt]{$i$};
\draw[usual,crossline] (0,0)node[below,black,yshift=-1pt]{$i$} 
to[out=90,in=270] (0.5,0.75)node[above,black,yshift=-3pt]{$i{+}1$};
\end{tikzpicture}
.
\end{gather}
Here the numbers indicate the corresponding positions, reading left to right.
We have \emph{usual strands}, illustrated in black, and \emph{core strands}, 
illustrated thick and blue-grayish. We note that all of our diagrams have $g$ core strands 
and $n$ usual strands, but we tend to illustrate local pictures, 
as we already did above. The elements of the form $\tau_{u}$ 
and their inverses are called \emph{positive coils} 
and \emph{negative coils}, respectively. We also 
say \emph{coils} for short.

\begin{definition}\label{definition:handlebody-braidgroup}
We let the \emph{handlebody braid group} (in $n$ strands and of genus $g$)
$\setstuff{B}_{g,n}$ be the group generated by 
$\{\tau_{u},\beta_{i}\mid 1\leq u\leq g,1\leq i\leq n-1\}$ modulo
\begin{gather}\label{eq:handlebody-summary1}
\beta_{i}\beta_{j}=\beta_{j}\beta_{i}\;\text{if }|i-j|>1,\quad
\beta_{i}\beta_{j}\beta_{i}=\beta_{j}\beta_{i}\beta_{j}\;\text{if }|i-j|=1,
\\ 
\label{eq:handlebody-summary2}
\tau_{u}\beta_{i}
=\beta_{i}\tau_{u}\;\text{if }i>1,
\quad
\tau_{v}(\beta_{1}\tau_{u}\beta_{1})
=(\beta_{1}\tau_{u}\beta_{1})\tau_{v}
\;\text{if }u\leq v.
\end{gather}
\end{definition}

We think of $\setstuff{B}_{g,n}$ as a 
handlebody generalization of the extended affine braid group of type A.
Note that is $\setstuff{B}_{g,n}$ not attached to a Coxeter group in any 
straightforward way, see {\eg} \cite[Remark 4]{La-handlebodies} 
and \cite[(1-7)]{RoTu-homflypt-typea}. 

\begin{remark}\label{remark:braidcox}
Special cases of \fullref{definition:handlebody-braidgroup} are:	
\begin{enumerate}

\setlength\itemsep{0.15cm}

\item The case $g=0$ is the classical braid group $\setstuff{B}_{n}=\setstuff{B}_{0,n}$.

\item For $g=1$ the handlebody braid group is the 
braid group of extended affine type A, which 
is also the braid group of Coxeter type C$=$B, 
see \cite{Brieskorn} or \cite{allcock}.

\item A perhaps surprising fact is that the handlebody braid group 
for $g=2$ is 
isomorphic to the braid group of affine Coxeter type C, 
see \cite{allcock}.

\item To the best of our knowledge the case $g>2$ was first 
studied in \cite{Ve-handlebodies}, 
and then further in \cite{HaOlLa-handlebodies}.

\end{enumerate}
\end{remark}

\begin{remark}\label{remark:handlebody-closing}
The handlebody braid group describes 
the configuration space of a disk with $g$ punctures \cite{Ve-handlebodies}, 
\cite{La-handlebodies},
\cite{HaOlLa-handlebodies}. 
Moreover, after taking an appropriate 
version of an Alexander closure, as explained {\eg} in 
\cite[Theorem 2]{HaOlLa-handlebodies} or 
\cite[Section 2]{RoTu-homflypt-typea} 
and illustrated in \eqref{eq:handlebody-closing}, these braid 
groups give an algebraic way to study links in handlebodies. 
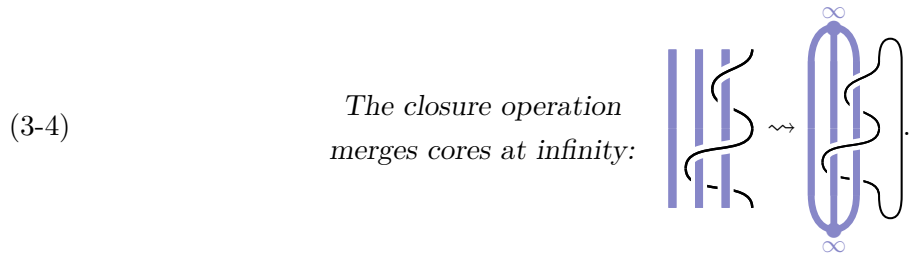
\begin{gather}\label{eq:handlebody-closing}
\begin{gathered}
\text{The closure operation}
\\[-10pt]
\text{merges cores at infinity:}
\end{gathered}
\quad
\begin{tikzpicture}[anchorbase,scale=0.7,tinynodes]
\draw[pole,crosspole] (-0.5,0) to[out=90,in=270] (-0.5,1.5);
\draw[usual,crossline] (1,0) to[out=90,in=270] (-0.25,0.75);
\draw[pole,crosspole] (0,0) to[out=90,in=270] (0,1.5);
\draw[pole,crosspole] (0.5,0) to[out=90,in=270] (0.5,1.5);
\draw[usual,crossline] (-0.25,0.75) to[out=90,in=270] (1,1.5);
\draw[pole,crosspole] (-0.5,1.5) to[out=90,in=270] (-0.5,3);
\draw[usual,crossline] (1,1.5) to[out=90,in=270] (0.25,2.25);
\draw[pole,crosspole] (0,1.5) to[out=90,in=270] (0,3);
\draw[pole,crosspole] (0.5,1.5) to[out=90,in=270] (0.5,3);
\draw[usual,crossline] (0.25,2.25) to[out=90,in=270] (1,3);
\end{tikzpicture}
\rightsquigarrow
\scalebox{0.85}{$\begin{tikzpicture}[anchorbase,scale=0.7]
\draw[pole,crosspole] (-0.5,0) to[out=90,in=270] (-0.5,1.5);
\draw[usual,crossline] (1,0) to[out=90,in=270] (-0.25,0.75);
\draw[pole,crosspole] (0,0) to[out=90,in=270] (0,1.5);
\draw[pole,crosspole] (0.5,0) to[out=90,in=270] (0.5,1.5);
\draw[usual,crossline] (-0.25,0.75) to[out=90,in=270] (1,1.5);
\draw[pole,crosspole] (-0.5,1.5) to[out=90,in=270] (-0.5,3);
\draw[usual,crossline] (1,1.5) to[out=90,in=270] (0.25,2.25);
\draw[pole,crosspole] (0,1.5) to[out=90,in=270] (0,3);
\draw[pole,crosspole] (0.5,1.5) to[out=90,in=270] (0.5,3);
\draw[usual,crossline] (0.25,2.25) to[out=90,in=270] (1,3);
\draw[pole] (-0.5,0) to[out=270,in=180] (0,-0.75);
\draw[pole] (0,0) to[out=270,in=90] (0,-0.75);
\draw[pole] (0.5,0) to[out=270,in=0] (0,-0.75);
\draw[pole] (-0.5,3) to[out=90,in=180] (0,3.75);
\draw[pole] (0,3) to[out=90,in=270] (0,3.75);
\draw[pole] (0.5,3) to[out=90,in=0] (0,3.75);
\draw node[pole] at (0,3.7) {{\color{specialgray}\LARGE$\bullet$}};
\draw node[pole,above] at (0,3.7) {\color{specialgray}$\infty$};
\draw node[pole] at (0,-0.8) {{\color{specialgray}\LARGE$\bullet$}};
\draw node[pole,below] at (0,-0.75) {\color{specialgray}$\infty$};
\draw[usual] (1,0) to[out=270,in=180] (1.25,-0.5) 
to[out=0,in=270] (1.5,0) to (1.5,3) to[out=90,in=0] (1.25,3.5) to[out=180,in=90] (1,3);
\end{tikzpicture}$}
.
\end{gather}
In the topological interpretation, as explained 
{\eg} in \cite[Section 2]{RoTu-homflypt-typea}, 
the core strands correspond to the cores of the handles 
of a handlebody.
This motivates our nomenclature.
\end{remark}

The diagrammatic interpretations of the relations
\eqref{eq:handlebody-summary1} and \eqref{eq:handlebody-summary2} are
\begin{gather*}
\begin{tikzpicture}[anchorbase,scale=0.7,tinynodes]
\draw[usual,crossline] (0.5,0) to[out=90,in=270] (0,0.75) to (0,1.5);
\draw[usual,crossline] (0,0) to[out=90,in=270] (0.5,0.75) to (0.5,1.5);
\draw[usual,crossline] (1.5,0) to (1.5,0.75) to[out=90,in=270] (1,1.5);
\draw[usual,crossline] (1,0) to (1,0.75) to[out=90,in=270] (1.5,1.5);
\end{tikzpicture}
=
\begin{tikzpicture}[anchorbase,scale=0.7,tinynodes]
\draw[usual,crossline] (0.5,0) to (0.5,0.75) to[out=90,in=270] (0,1.5);
\draw[usual,crossline] (0,0) to (0,0.75) to[out=90,in=270] (0.5,1.5);
\draw[usual,crossline] (1.5,0) to[out=90,in=270] (1,0.75) to (1,1.5);
\draw[usual,crossline] (1,0) to[out=90,in=270] (1.5,0.75) to (1.5,1.5);
\end{tikzpicture}
,\quad
\begin{tikzpicture}[anchorbase,scale=0.7,tinynodes]
\draw[usual,crossline] (0.5,0) to[out=90,in=270] (0,0.75);
\draw[usual,crossline] (0,0) to[out=90,in=270] (0.5,0.75);
\draw[usual,crossline] (1,0) to[out=90,in=270] (1,0.75);
\draw[usual,crossline] (1,0.75) to[out=90,in=270] (0.5,1.5);
\draw[usual,crossline] (0.5,0.75) to[out=90,in=270] (1,1.5);
\draw[usual,crossline] (0,0.75) to[out=90,in=270] (0,1.5);
\draw[usual,crossline] (0.5,1.5) to[out=90,in=270] (0,2.25);
\draw[usual,crossline] (0,1.5) to[out=90,in=270] (0.5,2.25);
\draw[usual,crossline] (1,1.5) to[out=90,in=270] (1,2.25);
\end{tikzpicture}
=
\begin{tikzpicture}[anchorbase,scale=0.7,tinynodes]
\draw[usual,crossline] (1,0) to[out=90,in=270] (0.5,0.75);
\draw[usual,crossline] (0.5,0) to[out=90,in=270] (1,0.75);
\draw[usual,crossline] (0,0) to[out=90,in=270] (0,0.75);
\draw[usual,crossline] (0.5,0.75) to[out=90,in=270] (0,1.5);
\draw[usual,crossline] (0,0.75) to[out=90,in=270] (0.5,1.5);
\draw[usual,crossline] (1,0.75) to[out=90,in=270] (1,1.5);
\draw[usual,crossline] (1,1.5) to[out=90,in=270] (0.5,2.25);
\draw[usual,crossline] (0.5,1.5) to[out=90,in=270] (1,2.25);
\draw[usual,crossline] (0,1.5) to[out=90,in=270] (0,2.25);
\end{tikzpicture}
,
\\
\begin{tikzpicture}[anchorbase,scale=0.7,tinynodes]
\draw[pole,crosspole] (-0.5,0) to[out=90,in=270] (-0.5,2.25);
\draw[usual,crossline] (1.5,0) to[out=90,in=270] (-0.25,0.75);
\draw[pole,crosspole] (0.5,0) to[out=90,in=270] (0.5,1.5) to (0.5,2.25);
\draw[pole,crosspole] (1,0) to[out=90,in=270] (1,1.5) to (1,2.25);
\draw[pole,crosspole] (0,0) to[out=90,in=270] (0,1.5) to (0,2.25);
\draw[usual,crossline] (-0.25,0.75) to[out=90,in=270] (1.5,1.5) to (1.5,2.25);
\draw[usual,crossline] (2.5,0) to (2.5,1.5) to[out=90,in=270] (2,2.25);
\draw[usual,crossline] (2,0) to (2,1.5) to[out=90,in=270] (2.5,2.25);
\end{tikzpicture}
=
\begin{tikzpicture}[anchorbase,scale=0.7,tinynodes]
\draw[pole,crosspole] (-0.5,0) to[out=90,in=270] (-0.5,2.25);
\draw[usual,crossline] (1.5,0) to (1.5,0.75) to[out=90,in=270] (-0.25,1.5);
\draw[pole,crosspole] (0.5,0) to (0.5,0.75) to[out=90,in=270] (0.5,2.25);
\draw[pole,crosspole] (1,0) to (1,0.75) to[out=90,in=270] (1,2.25);
\draw[pole,crosspole] (0,0) to (0,0.75) to[out=90,in=270] (0,2.25);
\draw[usual,crossline] (-0.25,1.5) to[out=90,in=270] (1.5,2.25);
\draw[usual,crossline] (2.5,0) to[out=90,in=270] (2,0.75) to (2,2.25);
\draw[usual,crossline] (2,0) to[out=90,in=270] (2.5,0.75) to (2.5,2.25);
\end{tikzpicture},
\quad
\scalebox{0.85}{$\begin{tikzpicture}[anchorbase,scale=0.7,tinynodes]
\draw[pole,crosspole] (-0.5,0) to[out=90,in=270] (-0.5,4.5);
\draw[usual,crossline] (1.5,0) to[out=90,in=270] (0.25,0.75);
\draw[pole,crosspole] (0.5,0)node[below,black,yshift=-1pt]{$v$} 
to[out=90,in=270] (0.5,1.5) to (0.5,2.25);
\draw[pole,crosspole] (1,0) to[out=90,in=270] (1,1.5) to (1,2.25);
\draw[pole,crosspole] (0,0)node[below,black,yshift=-1pt]{$u$} 
to[out=90,in=270] (0,1.5) to (0,2.25);
\draw[usual,crossline] (2,0) to (2,1.5) to[out=90,in=270] (1.5,2.25);
\draw[usual,crossline] (0.25,0.75) to[out=90,in=270] (1.5,1.5);
\draw[usual,crossline] (1.5,1.5) to[out=90,in=270] (2,2.25);
\draw[usual,crossline] (1.5,2.25) to[out=90,in=270] (-0.25,3);
\draw[pole,crosspole] (1,2.25) to[out=90,in=270] (1,3.75) to (1,4.5);
\draw[pole,crosspole] (0,2.25) to[out=90,in=270] (0,3.75) 
to (0,4.5)node[above,black,yshift=-3pt]{$u$};
\draw[pole,crosspole] (0.5,2.25) to[out=90,in=270] (0.5,3.75) 
to (0.5,4.5)node[above,black,yshift=-3pt]{$v$};
\draw[usual,crossline] (-0.25,3) to[out=90,in=270] (1.5,3.75);
\draw[usual,crossline] (2,2.25) to (2,3.75) to[out=90,in=270] (1.5,4.5);
\draw[usual,crossline] (1.5,3.75) to[out=90,in=270] (2,4.5);
\end{tikzpicture}
=
\begin{tikzpicture}[anchorbase,scale=0.7,tinynodes]
\draw[pole,crosspole] (-0.5,0) to[out=90,in=270] (-0.5,4.5);
\draw[usual,crossline] (1.5,0.75) to[out=90,in=270] (-0.25,1.5);
\draw[pole,crosspole] (1,0) to (1,0.75) to[out=90,in=270] (1,2.25);
\draw[pole,crosspole] (0,0)node[below,black,yshift=-1pt]{$u$} 
to (0,0.75) to[out=90,in=270] (0,2.25);
\draw[pole,crosspole] (0.5,0)node[below,black,yshift=-1pt]{$v$} 
to (0.5,0.75) to[out=90,in=270] (0.5,2.25);
\draw[usual,crossline] (-0.25,1.5) to[out=90,in=270] (1.5,2.25);
\draw[usual,crossline] (2,0) to[out=90,in=270] (1.5,0.75);
\draw[usual,crossline] (1.5,0) to[out=90,in=270] (2,0.75) to (2,2.25);
\draw[usual,crossline] (1.5,3) to[out=90,in=270] (0.25,3.75);
\draw[pole,crosspole] (0.5,2.25) to (0.5,3) to[out=90,in=270] 
(0.5,4.5)node[above,black,yshift=-3pt]{$v$};
\draw[pole,crosspole] (1,2.25) to (1,3) to[out=90,in=270] (1,4.5);
\draw[pole,crosspole] (0,2.25) to (0,3) to[out=90,in=270] 
(0,4.5)node[above,black,yshift=-3pt]{$u$};
\draw[usual,crossline] (0.25,3.75) to[out=90,in=270] (1.5,4.5);
\draw[usual,crossline] (2,2.25) to[out=90,in=270] (1.5,3);
\draw[usual,crossline] (1.5,2.25) to[out=90,in=270] (2,3) to (2,4.5);
\end{tikzpicture}$}
\;\text{if }u\leq v
.
\end{gather*}
We will use these diagrammatics whenever appropriate. 
For completeness, and to
make connection to the presentation from 
\cite[Theorem 2]{HaOlLa-handlebodies} or \cite[Section 2]{RoTu-homflypt-typea},
for $u=1,\dots,g$ we defined recursively $\tilde{\tau}_{g}$ by
$\tilde{\tau}_{g}=
\tau_{g}$, and for $1\leq u<g$ we let
\begin{gather*}
\tilde{\tau}_{u}=
\tilde{\tau}_{u-1}^{-1}
\dots\tilde{\tau}_{g}^{-1}\tau_{u}
=
\begin{tikzpicture}[anchorbase,scale=0.7,tinynodes]
\draw[pole,crosspole] (-0.5,0)node[below,black,yshift=-1pt]{$u{-}1$} 
to[out=90,in=270] (-0.5,1.5)node[above,black,yshift=-3pt]{$u{-}1$};
\draw[pole,crosspole] (-1,0)node[below,black,yshift=-1pt]{$1$} 
to[out=90,in=270] (-1,1.5)node[above,black,yshift=-3pt]{$1$};
\draw[pole,crosspole] (0.5,0)node[below,black,yshift=-1pt]{$u{+}1$} 
to[out=90,in=270] (0.5,1.5)node[above,black,yshift=-3pt]{$u{+}1$};
\draw[pole,crosspole] (1,0)node[below,black,yshift=-1pt]{$g$} 
to[out=90,in=270] (1,1.5)node[above,black,yshift=-3pt]{$g$};
\draw[usual,crossline] (1.5,0)node[below,black,yshift=-1pt]{$1$} 
to[out=90,in=270] (-0.25,0.75);
\draw[pole,crosspole] (0,0)node[below,black,yshift=-1pt]{$u$} 
to[out=90,in=270] (0,1.5)node[above,black,yshift=-3pt]{$u$};
\draw[usual,crossline] (-0.25,0.75) to[out=90,in=270] 
(1.5,1.5)node[above,black,yshift=-3pt]{$1$};
\node at (0.75,0.75) {$\dots$};
\node at (-0.75,0.75) {$\dots$};
\end{tikzpicture}
.
\end{gather*}

\begin{proposition}\label{proposition:alternative-presentation}
The handlebody braid group admits the following alternative presentation.
\begin{gather*}
\setstuff{B}_{g,n}\cong
\Bigg\langle
\begin{gathered}
\tilde{\tau}_{u},\,u=1,\dots,g,
\\
\beta_{i},\,i=1,\dots,n-1
\end{gathered}
\,\Bigg\vert\,
\scalebox{0.85}{$
\begin{gathered}
\beta_{i}\beta_{j}=\beta_{j}\beta_{i}\;\text{if }|i-j|>1,\quad
\beta_{i}\beta_{j}\beta_{i}=\beta_{j}\beta_{i}\beta_{j}\;\text{if }|i-j|=1,
\\
\tilde{\tau}_{u}\beta_{i}
=\beta_{i}\tilde{\tau}_{u}\;\text{if }i>1,
\quad
\tilde{\tau}_{u}\beta_{1}\tilde{\tau}_{u}\beta_{1}
=
\beta_{1}\tilde{\tau}_{u}\beta_{1}\tilde{\tau}_{u},
\quad
\\
\tilde{\tau}_{v}(\beta_{1}\tilde{\tau}_{u}\beta_{1}^{-1})
=(\beta_{1}\tilde{\tau}_{u}\beta_{1}^{-1})\tilde{\tau}_{v}
\;\text{if }u<v
\end{gathered}
$}
\Bigg\rangle.
\end{gather*}
\end{proposition}

\begin{proof}
A pleasant exercise (see also {\eg} \cite[Section 5]{HaOlLa-handlebodies}).
\end{proof}

The following allows us to use topological 
arguments and is used several times.

\begin{proposition}\label{proposition:embedding}
The rule
\begin{gather*}
\begin{tikzpicture}[anchorbase,scale=0.7,tinynodes]
\draw[usual,crossline] (1.5,0)node[below,black,yshift=-1pt]{$1$} 
to[out=90,in=270] (-0.25,0.75);
\draw[pole,crosspole] (-0.5,0) to[out=90,in=270] (-0.5,1.5);
\draw[pole,crosspole] (0.5,0) to[out=90,in=270] (0.5,1.5);
\draw[pole,crosspole] (1,0) to[out=90,in=270] (1,1.5);
\draw[pole,crosspole] (0,0)node[below,black,yshift=-1pt]{$u$} to[out=90,in=270] 
(0,1.5)node[above,black,yshift=-3pt]{$u$};
\draw[usual,crossline] (-0.25,0.75) to[out=90,in=270] 
(1.5,1.5)node[above,black,yshift=-3pt]{$1$};
\end{tikzpicture}
\mapsto
\begin{tikzpicture}[anchorbase,scale=0.7,tinynodes]
\draw[usual,crossline] (1.5,0)node[below,black,yshift=-1pt]{$g{+}1$} 
to[out=90,in=270] (-0.25,0.75);
\draw[usual,crossline] (-0.5,0) to[out=90,in=270] (-0.5,1.5);
\draw[usual,crossline] (0.5,0) to[out=90,in=270] (0.5,1.5);
\draw[usual,crossline] (1,0) to[out=90,in=270] (1,1.5);
\draw[usual,crossline] (0,0)node[below,black,yshift=-1pt]{$u$} 
to[out=90,in=270] (0,1.5)node[above,black,yshift=-3pt]{$u$};
\draw[usual,crossline] (-0.25,0.75) to[out=90,in=270] 
(1.5,1.5)node[above,black,yshift=-3pt]{$g{+}1$};
\end{tikzpicture}
,\quad
\begin{tikzpicture}[anchorbase,scale=0.7,tinynodes]
\draw[usual,crossline] (0.5,0)node[below,black,yshift=-1pt]{$i{+}1$} 
to[out=90,in=270] (0,0.75)node[above,black,yshift=-3pt]{$i$};
\draw[usual,crossline] (0,0)node[below,black,yshift=-1pt]{$i$} 
to[out=90,in=270] (0.5,0.75)node[above,black,yshift=-3pt]{$i{+}1$};
\end{tikzpicture}
\mapsto
\begin{tikzpicture}[anchorbase,scale=0.7,tinynodes]
\draw[usual,crossline] (0.5,0) to[out=90,in=270] (0,0.75)node[above,black,yshift=-3pt,xshift=-2pt]{$g{+}i$};
\draw[usual,crossline] (0,0)node[below,black,xshift=-2pt]{$g{+}i$} 
to[out=90,in=270] (0.5,0.75);
\draw[usual,white] (0.75,0)node[below,black,xshift=2pt]{$g{+}i{+}1$} 
to[out=90,in=270] (0.75,0.75)node[above,black,yshift=-3pt,xshift=2pt]{$g{+}i{+}1$};
\end{tikzpicture}
,
\end{gather*}
defines an injective group homomorphism
$\iota_{g,n}\colon\setstuff{B}_{g,n}\hookrightarrow\setstuff{B}_{g+n}$.
\end{proposition}

\begin{proof}
Using \fullref{proposition:alternative-presentation}, this is 
\cite[Theorem 1]{Ve-handlebodies}
and \cite[Section 5]{La-handlebodies}.
\end{proof}

We use the presentation without the tildes in this paper, but 
the other presentation can also be chosen, if preferred.

\subsection{Handlebody Coxeter groups}\label{subsection:coxdia}

The appropriate Coxeter groups in this setup are the following.

\begin{definition}\label{definition:handlebody-coxetergroup}
We let the \emph{handlebody Coxeter group} (in $n$ strands and of genus $g$)
$\setstuff{W}_{g,n}$ be the quotient group of 
$\setstuff{B}_{g,n}$ by the relations
\begin{gather*}
\beta_{i}^{2}=1
\;\text{for }i=1,\dots,n-1.
\end{gather*}
We write $t_{u}$ and $s_{i}$ for the images of 
$\tau_{u}$ respectively $\beta_{i}$ in the quotient.
\end{definition}

Similarly as for the handlebody braid group $\setstuff{B}_{g,n}$, we think of \fullref{definition:handlebody-coxetergroup} as a 
handlebody generalization of the extended affine Coxeter group of type A.

The asymmetry in \eqref{eq:handlebody-summary2} 
vanishes and we have the defining relations:
\begin{gather}\label{eq:coxhandlebody-summary1}
s_{i}^{2}=1
,\quad
s_{i}s_{j}=s_{j}s_{i}\;\text{if }|i-j|>1,
\quad
s_{i}s_{j}s_{i}=s_{j}s_{i}s_{j}\;\text{if }|i-j|=1,
\\
\label{eq:coxhandlebody-summary2}
t_{u}s_{i}
=s_{i}t_{u}\;\text{if }i>1,
\quad
t_{v}(s_{1}t_{u}s_{1})
=(s_{1}t_{u}s_{1})t_{v}
\;\forall
u,v
.
\end{gather}
By \eqref{eq:coxhandlebody-summary1}
we see that 
we have an embedding of groups $\setstuff{S}_{n}\cong\setstuff{W}_{0,n}\hookrightarrow\setstuff{W}_{g,n}$
by identifying the simple transpositions with the $s_{i}$.
Moreover, the $t_{u}$ span a free group $\setstuff{F}_{g}\cong\setstuff{W}_{g,1}$, 
and $\setstuff{F}_{g}$ is thus also a subgroup of $\setstuff{W}_{g,n}$. 
We use both below. 
Note also that $\setstuff{F}_{g}$, and thus $\setstuff{W}_{g,n}$, 
is of infinite order unless $g=0$. 
Hence, the analog of \fullref{proposition:embedding} 
does not hold for $\setstuff{W}_{g,n}$. 

\begin{remark}
Special cases of \fullref{definition:handlebody-coxetergroup} are:
\begin{enumerate}

\setlength\itemsep{0.15cm}

\item The case $g=0$ is the symmetric group $\setstuff{S}_{n}$.

\item For $g=1$ the handlebody Coxeter group is not the 
Coxeter group of type C$=$B, but rather the extended affine 
Coxeter (Weyl) group of type A.

\end{enumerate}
\end{remark}

For the further study of $\setstuff{W}_{g,n}$ we use
$\Zf=\Z[\zvar^{\pm 1},Y_{1},\dots,Y_{g}]$ as our ground ring.

\begin{definition}\label{definition:action}
We define a (right) action $
\Zf[X_{1},\dots,X_{n}]\actsright\setstuff{W}_{g,n}$
by:
\begin{enumerate}

\setlength\itemsep{0.15cm}

\item The generators $s_{i}$ act by the permutation action of
$\setstuff{S}_{n}$.

\item The generators $t_{u}$ act by
\begin{gather}\label{eq:action}
\begin{aligned}
X_{i}\actsright t_{u}
&=
\begin{cases}
\zvar^{u}X_{1}+Y_{u}&\text{if }i=1,
\\
X_{i}&\text{otherwise},
\end{cases}
\\
X_{i}\actsright t_{u}^{-1}
&=
\begin{cases}
\zvar^{-u}X_{1}-\zvar^{-u}Y_{u}&\text{if }i=1,
\\
X_{i}&\text{otherwise}.
\end{cases}
\end{aligned}
\end{gather}

\end{enumerate}
\end{definition}

The pictorial version of the action \eqref{eq:action} is
\begin{gather*}
\begin{tikzpicture}[anchorbase,scale=0.7,tinynodes]
\draw[usual,crossline] (1.5,0)node[below,black,yshift=-2pt]{$X_{1}$} 
to[out=90,in=270] (-0.25,0.75);
\draw[pole,crosspole] (-0.5,0) to[out=90,in=270] (-0.5,1.5);
\draw[pole,crosspole] (0.5,0) to[out=90,in=270] (0.5,1.5);
\draw[pole,crosspole] (1,0) to[out=90,in=270] (1,1.5);
\draw[pole,crosspole] (0,0)node[below,black,yshift=-1pt]{$u$} 
to[out=90,in=270] (0,1.5)node[above,black,yshift=-3pt]{$u$};
\draw[usual,crossline] (-0.25,0.75) to[out=90,in=270] 
(1.5,1.5)node[above,black,yshift=-2pt]{$\zvar^{u}X_{1}{+}Y_{u}$};
\end{tikzpicture}
.
\end{gather*}

\begin{remark}\label{remark:action}
Special cases of \fullref{definition:action} are:	
\begin{enumerate}

\item The case $g=0$ and $\zvar=1$ is the permutation 
representation of $\setstuff{S}_{n}$ on the polynomial ring 
$\Z[X_{1},\dots,X_{n}]$.

\item The case $g=1$ and $\zvar=1$ recovers the usual 
polynomial representation of the extended affine Weyl group of type A.

\item The case $g=1$, $\zvar=-1$ and $Y_{1}=0$ recovers 
the root-theoretic version of Tits' reflection representation 
of type C$=$B, {\ie} the representation where the type 
$A$ subdiagram acts by permutation and the 
additional generator acts as 
$-1$ on $X_{1}$.
\end{enumerate}
\end{remark}

\begin{lemma}\label{lemma:action}
The action in \fullref{definition:action} is well-defined.
\end{lemma}

\begin{proof}
Note that 
$t_{u}$ and $t_{u}^{-1}$ act slightly asymmetrically, 
but this ensures that their actions invert each other since {\eg}
\begin{gather*}
X_{1}\actsright t_{u}t_{u}^{-1}
=
(\zvar^{u}X_{1}+Y_{u})\actsright t_{u}^{-1}
=
\zvar^{u}t_{u}^{-1}(X_{1})+Y_{u}
=
\zvar^{u}(\zvar^{-u}X_{1}-\zvar^{-u}Y_{u})+Y_{u}
=
X_{1}.
\end{gather*}
Moreover, by construction of the action, the only 
non-trivial check is that 
\eqref{eq:coxhandlebody-summary2} holds.
The left equality in \eqref{eq:coxhandlebody-summary2} 
is immediate, and for the right we compute
\begin{gather*}
\begin{aligned}
X_{1}\actsright t_{u}s_{1}t_{v}s_{1}
=\zvar^{u}X_{2}\actsright t_{v}s_{1}+Y_{u}
=\zvar^{u}X_{2}\actsright s_{1}+Y_{u}
&=\zvar^{u}X_{1}+Y_{u},
\\
X_{1}\actsright s_{1}t_{v}s_{1}t_{u}=
X_{2}\actsright t_{v}s_{1}t_{u}
=X_{2}\actsright s_{1}t_{u}
=X_{1}\actsright t_{u}
&=\zvar^{u}X_{1}+Y_{u}.
\end{aligned}
\end{gather*}
Thus, on the value $X_{1}$ the actions agree. Furthermore, 
a very similar calculation 
shows that they also agree on the value $X_{2}$.
For all other values the relation \eqref{eq:coxhandlebody-summary2} 
holds evidently.
\end{proof}

The action in \fullref{definition:action} is actually 
faithful, and this is what we are going to show next.

\begin{definition}\label{definition:jm-elements}
Define \emph{Jucys--Murphy elements} 
$\jm_{u,i}^{\pm 1}\in\setstuff{B}_{g,n}$ as follows. 
We let $\jm_{u,1}^{\pm 1}=\tau_{u}^{\pm 1}$.
For $i>1$ we define $\jm_{u,i}^{\pm 1}
=\beta_{i-1}^{\pm 1}\dots\beta_{1}^{\pm 1}
\tau_{u}^{\pm 1}\beta_{1}^{\pm 1}\dots\beta_{i-1}^{\pm 1}$, or diagrammatically
\begin{gather}\label{eq:jm-elements}
\jm_{u,i}
=
\begin{tikzpicture}[anchorbase,scale=0.7,tinynodes]
\draw[pole,crosspole] (-0.5,0) to[out=90,in=270] (-0.5,1.5);
\draw[usual,crossline] (2.5,0)node[below,black,yshift=-1pt]{$i$} 
to[out=90,in=270] (-0.25,0.75);
\draw[usual,crossline] (1.5,0) to[out=90,in=270] (1.5,1.5);
\draw[usual,crossline] (2,0) to[out=90,in=270] (2,1.5);
\draw[pole,crosspole] (0.5,0) to[out=90,in=270] (0.5,1.5);
\draw[pole,crosspole] (1,0) to[out=90,in=270] (1,1.5);
\draw[pole,crosspole] (0,0)node[below,black,yshift=-1pt]{$u$} 
to[out=90,in=270] (0,1.5)node[above,black,yshift=-3pt]{$u$};
\draw[usual,crossline] (-0.25,0.75) to[out=90,in=270] 
(2.5,1.5)node[above,black,yshift=-3pt]{$i$};
\end{tikzpicture}
,\quad
\jm_{u,i}^{-1}
=
\begin{tikzpicture}[anchorbase,scale=0.7,tinynodes]
\draw[pole,crosspole] (-0.5,0) to[out=90,in=270] (-0.5,1.5);
\draw[usual,crossline] (-0.25,0.75) to[out=90,in=270] 
(2.5,1.5)node[above,black,yshift=-3pt]{$i$};
\draw[pole,crosspole] (0.5,0) to[out=90,in=270] (0.5,1.5);
\draw[pole,crosspole] (1,0) to[out=90,in=270] (1,1.5);
\draw[usual,crossline] (1.5,0) to[out=90,in=270] (1.5,1.5);
\draw[usual,crossline] (2,0) to[out=90,in=270] (2,1.5);
\draw[pole,crosspole] (0,0)node[below,black,yshift=-1pt]{$u$} 
to[out=90,in=270] (0,1.5)node[above,black,yshift=-3pt]{$u$};
\draw[usual,crossline] (2.5,0)node[below,black,yshift=-1pt]{$i$} 
to[out=90,in=270] (-0.25,0.75);
\end{tikzpicture}
.
\end{gather}
\end{definition}

\begin{lemma}\label{lemma:jm-elements}
We have
\begin{gather}\label{eq:jm-relations}
\begin{gathered}
\jm_{u,i}^{\pm 1}\jm_{u,i}^{\mp 1}=1
,\quad
\beta_{j}\jm_{u,i}^{\pm 1}
=
\jm_{u,i}^{\pm 1}\beta_{j}
\;\text{if }i-1,i\neq j,
\\
\beta_{i-1}^{-1}\jm_{u,i}
=
\jm_{u,i-1}\beta_{i-1}
,\quad
\beta_{i}\jm_{u,i}
=
\jm_{u,i+1}\beta_{i}^{-1}
,\\
\beta_{i-1}\jm_{u,i}^{-1}
=
\jm_{u,i-1}^{-1}\beta_{i-1}^{-1}
,\quad
\beta_{i}^{-1}\jm_{u,i}^{-1}
=
\jm_{u,i+1}^{-1}\beta_{i}
,
\\
\jm_{u,i}\jm_{v,j}^{\pm 1}
=\jm_{v,j}^{\pm 1}\jm_{u,i}
\text{ and }
\jm_{u,i}^{-1}\jm_{v,j}^{\pm 1}
=\jm_{v,j}^{\pm 1}\jm_{u,i}^{-1}
\;\text{if }[v,j]\subset[u,i[,
\end{gathered}
\end{gather}
where $[v,j]\subset[u,i[$ means $u\leq v$ and $j<i$. 
\end{lemma}

\begin{proof}
Using topological arguments, all of these are easy 
to verify. For example, the middle relations in 
\eqref{eq:jm-relations} are of the form
\begin{gather*}
\scalebox{0.85}{$\begin{tikzpicture}[anchorbase,scale=0.7,tinynodes]
\draw[pole,crosspole] (-0.5,0) to[out=90,in=270] (-0.5,1.5);
\draw[usual,crossline] (2,0) to[out=90,in=270] (-0.25,0.75);
\draw[pole,crosspole] (1,0) to[out=90,in=270] (1,1.5);
\draw[pole,crosspole] (0.5,0) to[out=90,in=270] (0.5,1.5);
\draw[usual,crossline] (1.5,0) to[out=90,in=270] (1.5,1.5);
\draw[usual,crossline] (2.5,0) to[out=90,in=270] (2.5,1.5);
\draw[pole,crosspole] (0,0) to[out=90,in=270] (0,1.5);
\draw[usual,crossline] (-0.25,0.75) to[out=90,in=270] (2,1.5);
\draw[pole,crosspole] (-0.5,1.5) to[out=90,in=270] (-0.5,2.25);
\draw[pole,crosspole] (0,1.5) to[out=90,in=270] (0,2.25);
\draw[pole,crosspole] (0.5,1.5) to[out=90,in=270] (0.5,2.25);
\draw[pole,crosspole] (1,1.5) to[out=90,in=270] (1,2.25);
\draw[usual,crossline] (1.5,1.5) to[out=90,in=270] (2,2.25);
\draw[usual,crossline] (2,1.5) to[out=90,in=270] (1.5,2.25);
\draw[usual,crossline] (2.5,1.5) to[out=90,in=270] (2.5,2.25);
\end{tikzpicture}
=
\begin{tikzpicture}[anchorbase,scale=0.7,tinynodes]
\draw[pole,crosspole] (-0.5,0) to[out=90,in=270] (-0.5,1.5);
\draw[usual,crossline] (2,0) to[out=90,in=270] (-0.25,0.75);
\draw[pole,crosspole] (1,0) to[out=90,in=270] (1,1.5);
\draw[pole,crosspole] (0.5,0) to[out=90,in=270] (0.5,1.5);
\draw[usual,crossline] (1.5,0) to[out=90,in=270] (2,1.5);
\draw[usual,crossline] (2.5,0) to[out=90,in=270] (2.5,1.5);
\draw[pole,crosspole] (0,0) to[out=90,in=270] (0,1.5);
\draw[usual,crossline] (-0.25,0.75) to[out=90,in=270] (1.5,1.5);
\draw[pole,crosspole] (-0.5,1.5) to[out=90,in=270] (-0.5,2.25);
\draw[pole,crosspole] (0,1.5) to[out=90,in=270] (0,2.25);
\draw[pole,crosspole] (0.5,1.5) to[out=90,in=270] (0.5,2.25);
\draw[pole,crosspole] (1,1.5) to[out=90,in=270] (1,2.25);
\draw[usual,crossline] (1.5,1.5) to[out=90,in=270] (1.5,2.25);
\draw[usual,crossline] (2,1.5) to[out=90,in=270] (2,2.25);
\draw[usual,crossline] (2.5,1.5) to[out=90,in=270] (2.5,2.25);
\end{tikzpicture}$}
,
\end{gather*}
The bottom relations in \eqref{eq:jm-relations} take the form
\begin{gather*}
\scalebox{0.85}{$\begin{tikzpicture}[anchorbase,scale=0.7,tinynodes]
\draw[pole,crosspole] (-0.5,0) to[out=90,in=270] (-0.5,3);
\draw[usual,crossline] (2.5,0) to[out=90,in=270] (-0.25,0.75);
\draw[pole,crosspole] (0.5,0) to[out=90,in=270] (0.5,1.5);
\draw[pole,crosspole] (1,0) to[out=90,in=270] (1,1.5);
\draw[usual,crossline] (1.5,0) to[out=90,in=270] (1.5,1.5);
\draw[usual,crossline] (2,0) to[out=90,in=270] (2,1.5);
\draw[pole,crosspole] (0,0) to[out=90,in=270] (0,1.5);
\draw[usual,crossline] (-0.25,0.75) to[out=90,in=270] (2.5,1.5);
\draw[usual,crossline] (2,1.5) to[out=90,in=270] (0.25,2.25);
\draw[pole,crosspole] (1,1.5) to[out=90,in=270] (1,3);
\draw[usual,crossline] (1.5,1.5) to[out=90,in=270] (1.5,3);
\draw[usual,crossline] (2.5,1.5) to[out=90,in=270] (2.5,3);
\draw[pole,crosspole] (0.5,1.5) to[out=90,in=270] (0.5,3);
\draw[pole,crosspole] (0,1.5) to[out=90,in=270] (0,3);
\draw[usual,crossline] (0.25,2.25) to[out=90,in=270] (2,3);
\node at (0,-0.5) {\phantom{a}};
\node at (0,3.5) {\phantom{a}};
\end{tikzpicture}
=
\begin{tikzpicture}[anchorbase,scale=0.7,tinynodes]
\draw[pole,crosspole] (-0.5,0) to[out=90,in=270] (-0.5,3);
\draw[usual,crossline] (2.5,0) to[out=90,in=270] (-0.25,0.75);
\draw[pole,crosspole] (0.5,0) to[out=90,in=270] (0.5,0.75);
\draw[pole,crosspole] (1,0) to[out=90,in=270] (1,0.75);
\draw[usual,crossline] (1.5,0) to[out=90,in=270] (1.5,0.75);
\draw[usual,crossline] (2,0) to[out=90,in=270] (2,0.75);
\draw[pole,crosspole] (0,0) to[out=90,in=270] (0,0.75);
\draw[usual,crossline] (2,0.75) to[out=90,in=270] (0.25,1.5);
\draw[pole,crosspole] (1,0.75) to[out=90,in=270] (1,2.25);
\draw[usual,crossline] (1.5,0.75) to[out=90,in=270] (1.5,2.25);
\draw[pole,crosspole] (0.5,0.75) to[out=90,in=270] (0.5,2.25);
\draw[pole,crosspole] (0,0.75) to[out=90,in=270] (0,2.25);
\draw[usual,crossline] (0.25,1.5) to[out=90,in=270] (2,2.25);
\draw[pole,crosspole] (0.5,2.25) to[out=90,in=270] (0.5,3);
\draw[pole,crosspole] (1,2.25) to[out=90,in=270] (1,3);
\draw[usual,crossline] (1.5,2.25) to[out=90,in=270] (1.5,3);
\draw[usual,crossline] (2,2.25) to[out=90,in=270] (2,3);
\draw[pole,crosspole] (0,2.25) to[out=90,in=270] (0,3);
\draw[usual,crossline] (-0.25,2.25) to[out=90,in=270] (2.5,3);
\draw[usual,crossline] (-0.25,0.75) to (-0.25,2.25);
\node at (0,-0.5) {\phantom{a}};
\node at (0,3.5) {\phantom{a}};
\end{tikzpicture}
=
\begin{tikzpicture}[anchorbase,scale=0.7,tinynodes]
\draw[pole,crosspole] (-0.5,0) to[out=90,in=270] (-0.5,3);
\draw[usual,crossline] (2,0) to[out=90,in=270] (0.25,0.75);
\draw[pole,crosspole] (1,0) to[out=90,in=270] (1,1.5);
\draw[usual,crossline] (1.5,0) to[out=90,in=270] (1.5,1.5);
\draw[usual,crossline] (2.5,0) to[out=90,in=270] (2.5,1.5);
\draw[pole,crosspole] (0,0) to[out=90,in=270] (0,1.5);
\draw[pole,crosspole] (0.5,0) to[out=90,in=270] (0.5,1.5);
\draw[usual,crossline] (0.25,0.75) to[out=90,in=270] (2,1.5);
\draw[usual,crossline] (2.5,1.5) to[out=90,in=270] (-0.25,2.25);
\draw[pole,crosspole] (0.5,1.5) to[out=90,in=270] (0.5,3);
\draw[pole,crosspole] (1,1.5) to[out=90,in=270] (1,3);
\draw[usual,crossline] (1.5,1.5) to[out=90,in=270] (1.5,3);
\draw[usual,crossline] (2,1.5) to[out=90,in=270] (2,3);
\draw[pole,crosspole] (0,1.5) to[out=90,in=270] (0,3);
\draw[usual,crossline] (-0.25,2.25) to[out=90,in=270] (2.5,3);
\node at (0,-0.5) {\phantom{a}};
\node at (0,3.5) {\phantom{a}};
\end{tikzpicture}
,\quad
\begin{tikzpicture}[anchorbase,scale=0.7,tinynodes]
\draw[pole,crosspole] (-0.5,0) to[out=90,in=270] (-0.5,3);
\draw[usual,crossline] (2,0) to[out=90,in=270] (-0.25,0.75);
\draw[pole,crosspole] (1,0) to[out=90,in=270] (1,1.5);
\draw[pole,crosspole] (0.5,0) to[out=90,in=270] (0.5,1.5);
\draw[usual,crossline] (1.5,0) to[out=90,in=270] (1.5,1.5);
\draw[usual,crossline] (2.5,0) to[out=90,in=270] (2.5,1.5);
\draw[pole,crosspole] (0,0) to[out=90,in=270] (0,1.5);
\draw[usual,crossline] (-0.25,0.75) to[out=90,in=270] (2,1.5);
\draw[usual,crossline] (2.5,1.5) to[out=90,in=270] (-0.25,2.25);
\draw[pole,crosspole] (0.5,1.5) to[out=90,in=270] (0.5,3);
\draw[pole,crosspole] (1,1.5) to[out=90,in=270] (1,3);
\draw[usual,crossline] (1.5,1.5) to[out=90,in=270] (1.5,3);
\draw[usual,crossline] (2,1.5) to[out=90,in=270] (2,3);
\draw[pole,crosspole] (0,1.5) to[out=90,in=270] (0,3);
\draw[usual,crossline] (-0.25,2.25) to[out=90,in=270] (2.5,3);
\node at (0,-0.5) {\phantom{a}};
\node at (0,3.5) {\phantom{a}};
\end{tikzpicture}
=
\begin{tikzpicture}[anchorbase,scale=0.7,tinynodes]
\draw[pole,crosspole] (-0.5,0) to[out=90,in=270] (-0.5,3);
\draw[usual,crossline] (2.5,0) to[out=90,in=270] (-0.25,0.75);
\draw[pole,crosspole] (0.5,0) to[out=90,in=270] (0.5,1.5);
\draw[pole,crosspole] (1,0) to[out=90,in=270] (1,1.5);
\draw[usual,crossline] (1.5,0) to[out=90,in=270] (1.5,1.5);
\draw[usual,crossline] (2,0) to[out=90,in=270] (2,1.5);
\draw[pole,crosspole] (0,0) to[out=90,in=270] (0,1.5);
\draw[usual,crossline] (-0.25,0.75) to[out=90,in=270] (2.5,1.5);
\draw[usual,crossline] (2,1.5) to[out=90,in=270] (-0.25,2.25);
\draw[pole,crosspole] (1,1.5) to[out=90,in=270] (1,3);
\draw[pole,crosspole] (0.5,1.5) to[out=90,in=270] (0.5,3);
\draw[usual,crossline] (1.5,1.5) to[out=90,in=270] (1.5,3);
\draw[usual,crossline] (2.5,1.5) to[out=90,in=270] (2.5,3);
\draw[pole,crosspole] (0,1.5) to[out=90,in=270] (0,3);
\draw[usual,crossline] (-0.25,2.25) to[out=90,in=270] (2,3);
\node at (0,-0.5) {\phantom{a}};
\node at (0,3.5) {\phantom{a}};
\end{tikzpicture}$}
,
\end{gather*}
The other relations can be verified {\ver}.
\end{proof}

Note that if $[v,j]\not\subset[u,i[$ 
in the final relation in \eqref{eq:jm-relations}, 
then the displayed elements do not commute (unless $\jm_{u,i}^{\pm 1}\jm_{u,i}^{\mp 1}=1$ applies). For example,
\begin{gather*}
\scalebox{0.85}{$\begin{tikzpicture}[anchorbase,scale=0.7,tinynodes]
\draw[pole,crosspole] (-0.5,0) to[out=90,in=270] (-0.5,3);
\draw[usual,crossline] (2.5,0) to[out=90,in=270] (0.25,0.75);
\draw[pole,crosspole] (0.5,0) to[out=90,in=270] (0.5,1.5);
\draw[pole,crosspole] (1,0) to[out=90,in=270] (1,1.5);
\draw[usual,crossline] (1.5,0) to[out=90,in=270] (1.5,1.5);
\draw[usual,crossline] (2,0) to[out=90,in=270] (2,1.5);
\draw[pole,crosspole] (0,0) to[out=90,in=270] (0,1.5);
\draw[usual,crossline] (0.25,0.75) to[out=90,in=270] (2.5,1.5);
\draw[usual,crossline] (2,1.5) to[out=90,in=270] (-0.25,2.25);
\draw[pole,crosspole] (1,1.5) to[out=90,in=270] (1,3);
\draw[usual,crossline] (1.5,1.5) to[out=90,in=270] (1.5,3);
\draw[usual,crossline] (2.5,1.5) to[out=90,in=270] (2.5,3);
\draw[pole,crosspole] (0.5,1.5) to[out=90,in=270] (0.5,3);
\draw[pole,crosspole] (0,1.5) to[out=90,in=270] (0,3);
\draw[usual,crossline] (-0.25,2.25) to[out=90,in=270] (2,3);
\node at (0,-0.5) {\phantom{a}};
\node at (0,3.5) {\phantom{a}};
\end{tikzpicture}
\neq
\begin{tikzpicture}[anchorbase,scale=0.7,tinynodes]
\draw[pole,crosspole] (-0.5,0) to[out=90,in=270] (-0.5,3);
\draw[usual,crossline] (2,0) to[out=90,in=270] (-0.25,0.75);
\draw[pole,crosspole] (1,0) to[out=90,in=270] (1,1.5);
\draw[usual,crossline] (1.5,0) to[out=90,in=270] (1.5,1.5);
\draw[usual,crossline] (2.5,0) to[out=90,in=270] (2.5,1.5);
\draw[pole,crosspole] (0,0) to[out=90,in=270] (0,1.5);
\draw[pole,crosspole] (0.5,0) to[out=90,in=270] (0.5,1.5);
\draw[usual,crossline] (-0.25,0.75) to[out=90,in=270] (2,1.5);
\draw[usual,crossline] (2.5,1.5) to[out=90,in=270] (0.25,2.25);
\draw[pole,crosspole] (0.5,1.5) to[out=90,in=270] (0.5,3);
\draw[pole,crosspole] (1,1.5) to[out=90,in=270] (1,3);
\draw[usual,crossline] (1.5,1.5) to[out=90,in=270] (1.5,3);
\draw[usual,crossline] (2,1.5) to[out=90,in=270] (2,3);
\draw[pole,crosspole] (0,1.5) to[out=90,in=270] (0,3);
\draw[usual,crossline] (0.25,2.25) to[out=90,in=270] (2.5,3);
\node at (0,-0.5) {\phantom{a}};
\node at (0,3.5) {\phantom{a}};
\end{tikzpicture}$}
.
\end{gather*}

\begin{lemma}\label{lemma:jm-order-lemma}
Any word in the Jucys--Murphy elements can be ordered 
such that $\jm_{u,i}^{\pm}$ appears right=above of
$\jm_{v,j}^{\pm}$ only if $j<i$.
\end{lemma}

\begin{proof}
We use the final relation in \eqref{eq:jm-relations}
inductively: First, start with $i=n$ and pull 
all $\jm_{u,n}^{\pm}$, for all $u$, to the right=top, 
without changing their order (they form a free group).
We freeze these elements in their positions and
we can thus use induction 
on the remaining Jucys--Murphy elements 
as their maximal second index is $n-1$.
\end{proof}

Let us denote the images of the Jucys--Murphy elements
in $\setstuff{W}_{g,n}$ by $\jmc_{u,i}^{\pm 1}$.

\begin{lemma}\label{lemma:jm-elements-span}
The set
\begin{gather}\label{eq:jm-basis}
\left\{ 
\jmc_{u_{1},i_{1}}^{a_{1}}\dots 
\jmc_{u_{m},i_{m}}^{a_{m}}w 
\,\middle\vert\,
\begin{gathered}
w\in\setstuff{S}_{n},
m\in\N,
\bsym{a}\in\Z^{m},
\\
(\bsym{u},\bsym{i})\in(\{1,\dots,g\}\times\{1,\dots,n\})^{m},
i_{1}\leq\dots\leq i_{m}
\end{gathered}
\right\}
\end{gather}
spans $\Z\setstuff{W}_{g,n}$.
\end{lemma}

\begin{proof}
Since $\tau_{u}=\jm_{u,1}$, words in 
$\jm_{u,i}$ and $\beta_{i}$ can clearly 
generate arbitrary words in $\setstuff{B}_{g,n}$, and thus in 
$\setstuff{W}_{g,n}$.
Further, using the relations \eqref{eq:jm-relations}, 
we see that we can always pull all
$s_{i}\in\setstuff{W}_{g,n}$ to the right=top, since 
$s_{i}=s_{i}^{-1}$. Finally, with the last 
relation in \eqref{eq:jm-relations} we can order the $\jm_{u,i}$ 
as claimed, {\cf} \fullref{lemma:jm-order-lemma}, which of course 
also works for the handlebody Coxeter group.
\end{proof}

\begin{proposition}\label{proposition:jm-elements-basis}
The set in \eqref{eq:jm-basis}
is a $\Z$-basis of $\Z\setstuff{W}_{g,n}$.
\end{proposition}

\begin{proof}
By \fullref{lemma:jm-elements-span}, it only remains to verify that 
the elements of the set \eqref{eq:jm-basis}
are linearly independent. To this end, we 
first observe that we can let $w$ be trivial 
since the action of $\setstuff{S}_{n}$ 
used in \fullref{lemma:jm-elements-span} is 
the faithful permutation action. Now, the 
only non-trivial action of $\jmc_{u,i}$ is on $X_{i}$:
\begin{gather*}
X_{j}\actsright\jmc_{u,i}
=
\begin{cases}
\zvar^{u}X_{i}+Y_{u}&\text{if }i=j,
\\
X_{j}&\text{otherwise},
\end{cases}
\leftrightsquigarrow
\begin{tikzpicture}[anchorbase,scale=0.7,tinynodes]
\draw[pole,crosspole] (-0.5,0) to[out=90,in=270] (-0.5,1.5);
\draw[usual,crossline] (2.5,0)node[below,black,yshift=-1pt]{$X_{i}$} 
to[out=90,in=270] (-0.25,0.75);
\draw[pole,crosspole] (0.5,0) to[out=90,in=270] (0.5,1.5);
\draw[pole,crosspole] (1,0) to[out=90,in=270] (1,1.5);
\draw[pole,crosspole] (0,0)node[below,black,yshift=-1pt]{$u$} 
to[out=90,in=270] (0,1.5)node[above,black,yshift=-3pt]{$u$};
\draw[usual,crossline] (-0.25,0.75) to[out=90,in=270] 
(2.5,1.5)node[above,black,yshift=-2pt]{$\zvar^{u}X_{i}{+}Y_{u}$};
\draw[usual] (1.5,0) to[out=90,in=270] (1.5,1.5);
\draw[usual] (2,0) to[out=90,in=270] (2,1.5);
\node at (0.75,0.75) {$\dots$};
\node at (1.75,0.75) {$\dots$};
\end{tikzpicture}
.
\end{gather*}
Note that the involved strand $i$ is only relevant for 
$\jmc_{u,i}$, for all $u$, and its powers. Moreover, we 
can distinguish $\jmc_{u,i}$ and $\jmc_{v,i}$ by the appearing 
variables $Y_{u}$ respectively $Y_{v}$, while their 
order does not matter due to \eqref{eq:jm-relations}.
The same argument works {\muta} 
for the inverses $\jmc_{u,i}^{-1}$, of course.
Taking all this together shows that the elements of 
\eqref{eq:jm-basis} are linearly independent.
\end{proof}

\begin{theorem}\label{theorem:action}
The action of $\setstuff{W}_{g,n}$ in \fullref{definition:action} is faithful.
\end{theorem}

\begin{proof}
Directly from the proof of \fullref{proposition:jm-elements-basis}.
\end{proof}

\section{Handlebody Temperley--Lieb and blob algebras}\label{section:tlblob}

Recall that $\KK$ denotes a unital, commutative, Noetherian domain. 
The example the reader should keep in mind throughout this section is
$\KK=\Z[\cpar]$ where $\cpar=(\cvar_{\gamma})_{\gamma}$ is a 
collection of variables $\cvar_{\gamma}$ 
which are circle evaluations.
Recall further that we have fixed $g\in\N$ and $n\in\N_{>0}$.

\subsection{Handlebody Temperley--Lieb algebras}\label{subsection:handlebody-tl}

In this section we
consider non-topological
\emph{crossingless matchings of $2n$ points of genus $g$}, which 
are all pairings $(i,k)$ of integers from $\{1,\dots,2n\}$ 
such that $i<j<k<l$ for all $(i,k)$ and $(j,l)$ together 
with a choice of a reduced word in the free group $\setstuff{F}_{g}$
for each appearing pair.

In slightly misleading pictures, {\cf} \fullref{remark:non-topological},
these are crossingless matchings of $2n$ points where usual strands 
can wind around the cores, but not among themselves.
Here we use the same conventions as in \fullref{section:braids} in the sense 
that all cores are to the left.
For example, if $g=3$ and $n=2$, then
\begin{gather}\label{eq:twist}
\begin{tikzpicture}[anchorbase,scale=0.7,tinynodes]
\draw[pole,crosspole] (-0.5,0) to[out=270,in=90] (-0.5,-1.5);
\draw[usual,crossline] (-0.25,-0.5) to[out=270,in=180] (0.25,-0.8) 
to[out=0,in=180] (0.75,-0.8) 
to[out=0,in=270] (1.5,0);
\draw[pole,crosspole] (0,0)node[above,white,yshift=-3pt]{$u$} to[out=270,in=90] (0,-1.5)node[below,white,yshift=-1pt]{$u$};
\draw[pole,crosspole] (0.5,0) to[out=270,in=90] (0.5,-1.5);
\draw[usual,crossline] (1,0) to[out=270,in=90] (-0.25,-0.5);
\draw[usual,crossline] (1,-1.5) to[out=90,in=180] (1.25,-1.25) 
to[out=0,in=90] (1.5,-1.5);
\end{tikzpicture}
,\quad
\tau_{u}=
\begin{tikzpicture}[anchorbase,scale=0.7,tinynodes]
\draw[pole,crosspole] (-0.5,0) to[out=90,in=270] (-0.5,1.5);
\draw[usual,crossline] (1,0) to[out=90,in=270] (-0.25,0.75);
\draw[usual,crossline] (1.5,0) to[out=90,in=270] (1.5,1.5);
\draw[pole,crosspole] (0,0)node[below,black,yshift=-1pt]{$u$} 
to[out=90,in=270] (0,1.5)node[above,black,yshift=-3pt]{$u$};
\draw[pole,crosspole] (0.5,0) to[out=90,in=270] (0.5,1.5);
\draw[usual,crossline] (-0.25,0.75) to[out=90,in=270] (1,1.5);
\end{tikzpicture}
,
\end{gather}
are examples of such crossingless matchings.

\begin{remark}
We will use a similar notation as in 
\fullref{section:braids}, {\eg} $\tau_{u}$ for elements 
as illustrated on the right in \eqref{eq:twist}.
\end{remark}

These crossingless matchings are not allowed to have 
any circles (circles, by definition, are a connected components of usual strings not touching the bottom or top).
But such circles, as usual for these types of algebras, could appear after concatenation. To address this we need to associate circles in such diagrams to conjugacy classes in $\setstuff{F}_{g}$.

To this end, for each circle in the diagrams we will
associate a word in $\setstuff{F}_{g}$ by starting somewhere 
generic on the circle, say the rightmost point, and read clockwise. 
This gives an element of $\setstuff{F}_{g}$, keeping the identification of 
coils and generators of $\setstuff{F}_{g}$ in mind, associated to each circle 
which is well-defined up to conjugacy. We then say the circles are \emph{$\setstuff{F}_{g}$-colored}, meaning 
tuples of a circle and a representative of a conjugacy class in $\setstuff{F}_{g}$. These circles will index our parameters momentarily.

For the 
following definition we choose a set of parameters 
$\cpar=(\cvar_{\gamma})_{\gamma}\in\KK$, one for 
each $\setstuff{F}_{g}$-colored circle $\gamma$. Thus, the parameters 
are constant on conjugacy classes in $\setstuff{F}_{g}$.

\begin{definition}
The evaluation of 
a $\setstuff{F}_{g}$-colored circle $\gamma$ is defined to be the removal 
of a closed component, contributing a factor 
$\cvar_{\gamma}$. We call this \emph{circle evaluation}.
\end{definition}

We could choose all $\cvar_{\gamma}$ to be different 
or equal, there is no restriction on the choice of these parameters 
(see however \fullref{remark:non-topological} below).
Note that all non-essential circles are associated to 
the trivial coloring $\gamma=1$ and are evaluated to $\cvar_{1}$.

\begin{example}\label{example:essential-loops}
Here are a few examples:
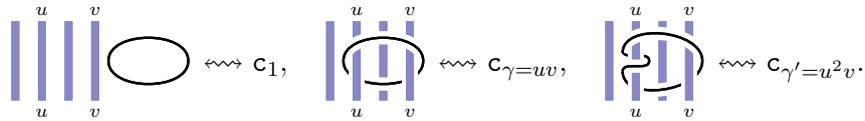
\begin{gather*}
\begin{tikzpicture}[anchorbase,scale=0.7,tinynodes]
\draw[pole,crosspole] (-2,0) to[out=90,in=270] (-2,1.5);
\draw[pole,crosspole] (-1.5,0)node[below,black]{$u$} 
to[out=90,in=270] (-1.5,1.5)node[above,black,yshift=-3pt]{$u$};
\draw[pole,crosspole] (-1,0) to[out=90,in=270] (-1,1.5);
\draw[pole,crosspole] (-0.5,0)node[below,black]{$v$} 
to[out=90,in=270] (-0.5,1.5)node[above,black,yshift=-3pt]{$v$};
\draw[usual,crossline] (1.25,0.75) to[out=270,in=270] (-0.25,0.75);
\draw[usual,crossline] (-0.25,0.75) to[out=90,in=90] (1.25,0.75);
\end{tikzpicture}
\leftrightsquigarrow
\cvar_{1}
,\quad
\begin{tikzpicture}[anchorbase,scale=0.7,tinynodes]
\draw[pole,crosspole] (0.5,0) to[out=90,in=270] (0.5,1.5);
\draw[pole,crosspole] (-0.5,0) to[out=90,in=270] (-0.5,1.5);
\draw[usual,crossline] (1.25,0.75) to[out=270,in=270] (-0.25,0.75);
\draw[pole,crosspole] (0,0)node[below,black]{$u$} 
to[out=90,in=270] (0,1.5)node[above,black,yshift=-3pt]{$u$};
\draw[pole,crosspole] (1,0)node[below,black]{$v$} 
to[out=90,in=270] (1,1.5)node[above,black,yshift=-3pt]{$v$};
\draw[usual,crossline] (-0.25,0.75) to[out=90,in=90] (1.25,0.75);
\end{tikzpicture}
\leftrightsquigarrow
\cvar_{\gamma=uv}
,\quad
\begin{tikzpicture}[anchorbase,scale=0.7,tinynodes]
\draw[pole,crosspole] (0.5,0) to[out=90,in=270] (0.5,1.5);
\draw[pole,crosspole] (-0.5,0) to[out=90,in=270] (-0.5,1.5);
\draw[usual,crossline] (1.25,0.75) to[out=270,in=270] (-0.25,0.55);
\draw[usual,crossline] (0.25,0.75) to[out=90,in=270] (-0.25,0.95);
\draw[pole,crosspole] (0,0)node[below,black]{$u$} 
to[out=90,in=270] (0,1.5)node[above,black,yshift=-3pt]{$u$};
\draw[pole,crosspole] (1,0)node[below,black]{$v$} 
to[out=90,in=270] (1,1.5)node[above,black,yshift=-3pt]{$v$};
\draw[usual,crossline] (-0.25,0.9) to[out=90,in=90] (1.25,0.75);
\draw[usual,crossline] (-0.25,0.55) to[out=90,in=270] (0.25,0.75);
\end{tikzpicture}
\leftrightsquigarrow
\cvar_{\gamma^{\prime}=u^{2}v}.
\end{gather*}
The first word in $\setstuff{F}_{g}$ is the trivial word, 
the second is $uv$, and the third is $u^{2}v$.
\end{example}

To each usual strand in a crossingless matching we associate 
a word in $\setstuff{F}_{g}$ by starting at the rightmost boundary point 
of the string and read to the other one, again using the 
identification of coils and generators of $\setstuff{F}_{g}$.
A \emph{$\setstuff{F}_{g}$-colored concatenation} of 
crossingless matchings $x$ and $y$ of $2n$ points of genus $g$ 
is, up to colors, a concatenation $xy$ in the usual sense 
and the colors are concatenated reading bottom to top along 
the strings respectively starting anywhere and moving 
clockwise along closed components.

\begin{definition}\label{definition:handlebody-tl}
We let $\setstuff{TL}_{g,n}(\cpar)$, 
the \emph{handlebody Temperley--Lieb algebra} 
(in $n$ strands and of genus $g$), be 
the algebra whose underlying free $\KK$-module is the
$\KK$-linear span of all 
crossingless matchings of $2n$ points of genus $g$, and
with multiplication given by $\setstuff{F}_{g}$-colored 
concatenation of diagrams modulo 
circle evaluation.
\end{definition}

\begin{remark}\label{remark:non-topological}
There are various ways to define 
Temperley--Lieb algebras beyond the classical case 
and some of them are not topological in nature, and 
the construction in \fullref{definition:handlebody-tl} 
is one of those that are not topological. 
In particular,
the \emph{Kauffman skein relation} 
\begin{gather}\label{eq:kauffman}
\begin{tikzpicture}[anchorbase,scale=0.7,tinynodes]
\draw[usual,crossline] (0.5,0) to[out=90,in=270] (0,0.75);
\draw[usual,crossline] (0,0) to[out=90,in=270] (0.5,0.75);
\end{tikzpicture}
=
\qvar^{1/2}\cdot
\begin{tikzpicture}[anchorbase,scale=0.7,tinynodes]
\draw[usual,crossline] (0,0) to[out=90,in=270] (0,0.75);
\draw[usual,crossline] (0.5,0) to[out=90,in=270] (0.5,0.75);
\end{tikzpicture}
+
\qvar^{-1/2}\cdot
\begin{tikzpicture}[anchorbase,scale=0.7,tinynodes]
\draw[usual,crossline] (1,0) to[out=90,in=180] (1.25,0.25) to[out=0,in=90] (1.5,0);
\draw[usual,crossline] (1,0.75) to[out=270,in=180] (1.25,0.5) to[out=0,in=270] (1.5,0.75);
\end{tikzpicture}
\end{gather}
does not behave topologically in $\setstuff{TL}_{g,n}(\cpar)$, 
even for appropriate choices of parameters. 
We will come back to a topological model of the 
handlebody Temperley--Lieb algebra in 
\fullref{subsection:handlebody-tl-topology} and for now we just note that:
\begin{enumerate}

\item One reason why we do not want \eqref{eq:kauffman} for the time 
being is that this relation implies that coils satisfy
an order two relation. This follows from the calculation
\begin{gather}\label{eq:order-two}
\begin{aligned}
\begin{tikzpicture}[anchorbase,scale=0.7,tinynodes]
\draw[pole,crosspole] (-0.5,0) to[out=90,in=270] (-0.5,1.5);
\draw[usual,crossline] (1,0) to[out=90,in=270] (-0.25,0.75);
\draw[pole,crosspole] (0,0) to[out=90,in=270] (0,1.5);
\draw[pole,crosspole] (0.5,0) to[out=90,in=270] (0.5,1.5);
\draw[usual,crossline] (-0.25,0.75) to[out=90,in=270] 
(1,1.5);
\end{tikzpicture}
&=
\begin{tikzpicture}[anchorbase,scale=0.7,tinynodes]
\draw[usual,crossline] (0.5,0.5) to[out=90,in=270] (0,1);
\draw[usual,crossline] (0,0.5) to[out=90,in=270] (0.5,1);
\draw[usual,crossline] (0.5,0.5) to[out=270,in=90] (0,0);
\draw[usual,crossline] (0,1.5) to[out=270,in=90] (0.5,1);
\draw[usual,crossline] (0,1) to[out=90,in=90] (-1.25,0.75);
\draw[pole,crosspole] (-1.5,0) to[out=90,in=270] (-1.5,1.5);
\draw[pole,crosspole] (-1,0) to[out=90,in=270] (-1,1.5);
\draw[pole,crosspole] (-0.5,0) to[out=90,in=270] (-0.5,1.5);
\draw[usual,crossline] (-1.25,0.75) to[out=270,in=270] 
(0,0.5);
\end{tikzpicture}
=
\qvar^{1/2}\cdot
\begin{tikzpicture}[anchorbase,scale=0.7,tinynodes]
\draw[usual,crossline] (0.5,0.5) to[out=90,in=270] (0.5,1);
\draw[usual,crossline] (0,0.5) to[out=90,in=270] (0,1);
\draw[usual,crossline] (0.5,0.5) to[out=270,in=90] (0,0);
\draw[usual,crossline] (0,1.5) to[out=270,in=90] (0.5,1);
\draw[usual,crossline] (0,1) to[out=90,in=90] (-1.25,0.75);
\draw[pole,crosspole] (-1.5,0) to[out=90,in=270] (-1.5,1.5);
\draw[pole,crosspole] (-1,0) to[out=90,in=270] (-1,1.5);
\draw[pole,crosspole] (-0.5,0) to[out=90,in=270] (-0.5,1.5);
\draw[usual,crossline] (-1.25,0.75) to[out=270,in=270] 
(0,0.5);
\end{tikzpicture}
+
\qvar^{-1/2}\cdot
\begin{tikzpicture}[anchorbase,scale=0.7,tinynodes]
\draw[usual,crossline] (0,0.5) to[out=90,in=90] (0.5,0.5);
\draw[usual,crossline] (0,1) to[out=270,in=270] (0.5,1);
\draw[usual,crossline] (0.5,0.5) to[out=270,in=90] (0,0);
\draw[usual,crossline] (0,1.5) to[out=270,in=90] (0.5,1);
\draw[usual,crossline] (0,1) to[out=90,in=90] (-1.25,0.75);
\draw[pole,crosspole] (-1.5,0) to[out=90,in=270] (-1.5,1.5);
\draw[pole,crosspole] (-1,0) to[out=90,in=270] (-1,1.5);
\draw[pole,crosspole] (-0.5,0) to[out=90,in=270] (-0.5,1.5);
\draw[usual,crossline] (-1.25,0.75) to[out=270,in=270] 
(0,0.5);
\end{tikzpicture}
\\
&=
\qvar^{1/2}\cvar_{\gamma}\cdot
\begin{tikzpicture}[anchorbase,scale=0.7,tinynodes]
\draw[pole,crosspole] (-0.5,0) to[out=90,in=270] (-0.5,1.5);
\draw[pole,crosspole] (0,0) to[out=90,in=270] (0,1.5);
\draw[pole,crosspole] (0.5,0) to[out=90,in=270] (0.5,1.5);
\draw[usual,crossline] (1,0) to[out=90,in=270] 
(1,1.5);
\end{tikzpicture}
+
\qvar^{-1/2}\cdot
\begin{tikzpicture}[anchorbase,scale=0.7,tinynodes]
\draw[usual,crossline] (-0.25,0.75) to[out=90,in=270] 
(1,1.5);
\draw[pole,crosspole] (-0.5,0) to[out=90,in=270] (-0.5,1.5);
\draw[pole,crosspole] (0,0) to[out=90,in=270] (0,1.5);
\draw[pole,crosspole] (0.5,0) to[out=90,in=270] (0.5,1.5);
\draw[usual,crossline] (1,0) to[out=90,in=270] (-0.25,0.75);
\end{tikzpicture}
.
\end{aligned}
\end{gather}

\item In a topological model \eqref{eq:kauffman} also implies relations among 
the circle parameters $\cvar_{\gamma}$:
\begin{gather*}
\begin{tikzpicture}[anchorbase,scale=0.7,tinynodes]
\draw[pole,crosspole] (0.5,0) to[out=90,in=270] (0.5,1.5);
\draw[usual,crossline] (1.25,0.75) to[out=270,in=270] (-0.25,0.75);
\draw[pole,crosspole] (0,0)node[below,black]{$u$} 
to[out=90,in=270] (0,1.5)node[above,black,yshift=-3pt]{$u$};
\draw[pole,crosspole] (1,0)node[below,black]{$v$} 
to[out=90,in=270] (1,1.5)node[above,black,yshift=-3pt]{$v$};
\draw[usual,crossline] (-0.25,0.75) to[out=90,in=90] (1.25,0.75);
\end{tikzpicture}
=
\begin{tikzpicture}[anchorbase,scale=0.7,tinynodes]
\draw[pole,crosspole] (0.5,0) to[out=90,in=270] (0.5,1.5);
\draw[usual,crossline] (1.25,0.75) to[out=90,in=270] (-0.25,0.75);
\draw[pole,crosspole] (0,0)node[below,black]{$u$} 
to[out=90,in=270] (0,1.5)node[above,black,yshift=-3pt]{$u$};
\draw[pole,crosspole] (1,0)node[below,black]{$v$} 
to[out=90,in=270] (1,1.5)node[above,black,yshift=-3pt]{$v$};
\draw[usual,crossline] (-0.25,0.75) to[out=90,in=270] (1.25,0.75);
\end{tikzpicture}
=
\qvar^{1/2}
\begin{tikzpicture}[anchorbase,scale=0.7,tinynodes]
\draw[pole,crosspole] (1,0)node[below,black]{$v$} 
to[out=90,in=270] (1,0.75);
\draw[pole,crosspole] (0.5,0) to[out=90,in=270] (0.5,1.5);
\draw[usual,crossline] (1.25,0.75) to[out=270,in=270] (-0.25,0.75);
\draw[pole,crosspole] (0,0)node[below,black]{$u$} 
to[out=90,in=270] (0,1.5)node[above,black,yshift=-3pt]{$u$};
\draw[usual,crossline] (-0.25,0.75) to[out=90,in=90] (1.25,0.75);
\draw[pole,crosspole] (1,0.75) 
to[out=90,in=270] (1,1.5)node[above,black,yshift=-3pt]{$v$};
\end{tikzpicture}
+
\qvar^{-1/2}\cdot
\begin{tikzpicture}[anchorbase,scale=0.7,tinynodes]
\draw[pole,crosspole] (0.5,0) to[out=90,in=270] (0.5,1.5);
\draw[usual,crossline] (-0.25,0.75) to[out=270,in=270] (0.25,0.75);
\draw[usual,crossline] (1.25,0.75) to[out=270,in=270] (0.75,0.75);
\draw[pole,crosspole] (0,0)node[below,black]{$u$} 
to[out=90,in=270] (0,1.5)node[above,black,yshift=-3pt]{$u$};
\draw[pole,crosspole] (1,0)node[below,black]{$v$} 
to[out=90,in=270] (1,1.5)node[above,black,yshift=-3pt]{$v$};
\draw[usual,crossline] (-0.25,0.75) to[out=90,in=90] (0.25,0.75);
\draw[usual,crossline] (0.75,0.75) to[out=90,in=90] (1.25,0.75);
\end{tikzpicture}
.
\end{gather*}

\end{enumerate}

\end{remark}

For completeness we note an easy fact:

\begin{lemma}\label{lemma:tl-basis}
The algebra $\setstuff{TL}_{g,n}(\cpar)$ is an associative, unital 
algebra with a $\KK$-basis given by all crossingless 
matchings of $2n$ points of genus $g$.
\end{lemma}

\begin{proof}
Note that the $\setstuff{F}_{g}$-colored 
concatenation of closed components ensures that the 
circle evaluation only depends on the associated word in $\setstuff{F}_{g}$ 
modulo conjugation.
Thus, the only claim which is not immediate is 
that the $\setstuff{F}_{g}$-colored 
concatenation is associative. This however follows 
by identifying the colors by dots on the strings.
\end{proof}

Note that the algebra $\setstuff{TL}_{g,n}(\cpar)$ is 
infinite-dimensional unless $g=0$, since the group-like elements $\tau_{u}$ 
are of infinite order and span $\setstuff{F}_{g}\cong
\setstuff{TL}_{g,1}(\cpar)\hookrightarrow\setstuff{TL}_{g,n}(\cpar)$.

\begin{remark}\label{remark:tl}
For low genus $\setstuff{TL}_{g,n}(\cpar)$ is well-studied 
(although for $g>0$ the precise definitions and the 
nomenclature vary 
throughout the literature):
\begin{enumerate}

\setlength\itemsep{0.15cm}

\item For $g=0$ the algebra 
$\setstuff{TL}_{0,n}(\cpar)$ is the 
Temperley--Lieb algebra in its crossingless matching definition.

\item In case $g=1$ there are the so-called 
affine Temperley--Lieb algebra or the
type C$=$B Temperley--Lieb algebra 
(the name comes from the relation of this algebra 
to the braid group of type C$=$B
as recalled in \fullref{remark:braidcox}), 
which were studied in many works such as {\eg} \cite{GrLe-affine-tl}.
The algebra $\setstuff{TL}_{1,n}(\cpar)$ is a version of these.

\item Similarly, for 
$g=2$ there is a so-called two-boundary 
Temperley--Lieb algebra 
and an affine type C Temperley--Lieb 
algebra, again independently 
introduced in many works.
The algebra $\setstuff{TL}_{2,n}(\cpar)$ 
is a version of these.

\end{enumerate}
Note that Temperley--Lieb algebras for $g<2$
are sometimes assumed to satisfy a quadratic relation for the coil 
elements, but we do not assume that.
These algebras are studied in 
the sections below.
\end{remark}

Similarly as in \fullref{subsection:brauer}, 
we let $\Lambda=\big(\{n,n-2,n-4,\dots\},\leq_{\N}\big)\subset\N$ 
with the usual partial order.
The $\lambda\in\Lambda$ are again the through strands of our diagrams, 
and we let $\sand=\KK\setstuff{F}_{g}$, 
the group ring of the free group in $g$ 
generators $\setstuff{F}_{g}\cong\setstuff{Br}_{g,1}\hookrightarrow\setstuff{Br}_{g,\lambda}$. We take the group element basis of $\setstuff{F}_{g}$ as the sandwich basis $\sandbasis$.

\begin{remark}\label{remark:fgtrivial}
We assume that $\setstuff{Br}_{g,\lambda}$ 
is trivial if $\lambda=0$.
In particular, $\KK\setstuff{F}_{g}\cong\KK$ for $\lambda=0$, since 
we identify $\setstuff{F}_{g}$ with a subgroup of $\setstuff{Br}_{g,\lambda}$, 
and the latter is trivial for $\lambda=0$.
\end{remark}

The construction of the basis 
\begin{gather}\label{eq:tl-basis}
\{c_{D,b,U}^{\lambda}\mid\lambda\in\Lambda,D,U\in M_{\lambda},
b\in\sand\}
\end{gather} 
works {\muta} as for the 
classical Temperley--Lieb algebra, having a concatenation 
of a cup diagram $D$, a cup diagram $U$ and through strands $b$, but now 
the caps and cups are allowed to wind around the cores 
(but not the through strands), 
and we have coils of through strands in the middle.
The following picture clarifies the construction:
\begin{gather}\label{eq:tl-basis-illustration}
\begin{tikzpicture}[anchorbase,scale=1]
\draw[usual,crossline] (0.25,1.2) to[out=90,in=180] (0.75,1.45) 
to[out=0,in=270] (4,2.5) to[out=90,in=270] (4,3);
\draw[pole,crosspole] (-0.5,3) to[out=270,in=90] (-0.5,1.5);
\draw[usual,crossline] (2,3) to[out=270,in=0] (0,2.8) 
to[out=180,in=90] (-0.25,2.6);
\draw[usual,crossline] (-0.25,1.8) to[out=270,in=180] (0,1.6) 
to[out=0,in=270] (3.5,3);
\draw[pole,crosspole] (0,3) to[out=270,in=90] (0,1.5);
\draw[pole,crosspole] (0.5,3) to[out=270,in=90] (0.5,1.5);
\draw[usual,crossline] (-0.25,2.6) to[out=270,in=180] (0,2.4) 
to[out=0,in=270] (2.5,3);
\draw[usual,crossline] (3,3) to[out=270,in=0] (0,2) 
to[out=180,in=90] (-0.25,1.8);
\draw[usual,crossline] (1,3) to[out=270,in=180] (1.25,2.75) 
to[out=0,in=270] (1.5,3);
\draw[pole,crosspole] (-0.5,0) to[out=90,in=270] (-0.5,1.5);
\draw[usual,crossline] (-0.25,0.5) to[out=90,in=180] (0,0.75) 
to[out=0,in=90] (2.5,0);
\draw[pole,crosspole] (0,0) to[out=90,in=270] (0,1.5);
\draw[pole,crosspole] (0.5,0) to[out=90,in=270] (0.5,1.5);
\draw[usual,crossline] (1,0) to[out=90,in=270] (-0.25,0.5);
\draw[usual,crossline] (1.5,0) to[out=90,in=180] (1.75,0.25) 
to[out=0,in=90] (2,0);
\draw[usual,crossline] (3,0) to[out=90,in=270] (3,0.5) 
to[out=90,in=0] (0.75,0.95) to[out=180,in=270] (0.25,1.2);
\draw[usual,crossline] (3.5,0) to[out=90,in=180] (3.75,0.25) 
to[out=0,in=90] (4,0);
\end{tikzpicture}
,\quad
\begin{aligned}
\begin{tikzpicture}[anchorbase,scale=1]
\draw[mor] (0,1) to (0.25,0.5) to (0.75,0.5) to (1,1) to (0,1);
\node at (0.5,0.75){$U$};
\end{tikzpicture}
&=
\begin{tikzpicture}[anchorbase,scale=0.7,tinynodes]
\draw[pole,crosspole] (-0.5,0) to[out=270,in=90] (-0.5,-1.5);
\draw[usual,crossline] (2,0) to[out=270,in=0] (0,-0.2) 
to[out=180,in=90] (-0.25,-0.4);
\draw[usual,crossline] (-0.25,-1.2) to[out=270,in=180] (0,-1.4) 
to[out=0,in=270] (3.5,0);
\draw[pole,crosspole] (0,0) to[out=270,in=90] (0,-1.5);
\draw[pole,crosspole] (0.5,0) to[out=270,in=90] (0.5,-1.5);
\draw[usual,crossline] (-0.25,-0.4) to[out=270,in=180] (0,-0.6) 
to[out=0,in=270] (2.5,0);
\draw[usual,crossline] (3,0) to[out=270,in=0] (0,-1) 
to[out=180,in=90] (-0.25,-1.2);
\draw[usual,crossline] (1,0) to[out=270,in=180] (1.25,-0.25) 
to[out=0,in=270] (1.5,0);
\draw[usual,crossline] (4,0) to[out=270,in=90] (4,-1.5);
\end{tikzpicture}
,\\
\begin{tikzpicture}[anchorbase,scale=1]
\draw[mor] (0.25,0) to (0.25,0.5) to (0.75,0.5) to (0.75,0) to (0.25,0);
\node at (0.5,0.25){$b$};
\node at (0.8,0.25){\phantom{$b$}};
\end{tikzpicture}
&=
\begin{tikzpicture}[anchorbase,scale=0.7,tinynodes]
\draw[pole,crosspole] (-0.5,0) to[out=90,in=270] (-0.5,1.5);
\draw[pole,crosspole] (0,0) to[out=90,in=270] (0,1.5);
\draw[usual,crossline] (0.25,0.75) to[out=90,in=270] (1,1.5);
\draw[pole,crosspole] (0.5,0) to[out=90,in=270] (0.5,1.5);
\draw[usual,crossline] (1,0) to[out=90,in=270] (0.25,0.75);
\end{tikzpicture}
,
\\
\begin{tikzpicture}[anchorbase,scale=1]
\draw[mor] (0,-0.5) to (0.25,0) to (0.75,0) to (1,-0.5) to (0,-0.5);
\node at (0.5,-0.25){$D$};
\end{tikzpicture}
&=
\begin{tikzpicture}[anchorbase,scale=0.7,tinynodes]
\draw[pole,crosspole] (-0.5,0) to[out=90,in=270] (-0.5,1.5);
\draw[usual,crossline] (-0.25,0.5) to[out=90,in=180] (0,0.75) 
to[out=0,in=90] (2.5,0);
\draw[pole,crosspole] (0,0) to[out=90,in=270] (0,1.5);
\draw[pole,crosspole] (0.5,0) to[out=90,in=270] (0.5,1.5);
\draw[usual,crossline] (1,0) to[out=90,in=270] (-0.25,0.5);
\draw[usual,crossline] (1.5,0) to[out=90,in=180] (1.75,0.25) 
to[out=0,in=90] (2,0);
\draw[usual,crossline] (3,0) to[out=90,in=270] (3,1.5);
\draw[usual,crossline] (3.5,0) to[out=90,in=180] (3.75,0.25) 
to[out=0,in=90] (4,0);
\end{tikzpicture}
.
\end{aligned}
\end{gather}
We also have the antiinvolution 
$(\placeholder)^{\star}\colon\setstuff{TL}_{g,n}(\cpar)
\to\setstuff{TL}_{g,n}(\cpar)$ given by 
flipping pictures upside down but keeping positive coils positive and negative coils negative (that is, 
one only reverses the words in $\setstuff{F}_{g}$ but does not 
invert them), {\eg} 
\begin{gather*}
\left(
\begin{tikzpicture}[anchorbase,scale=0.7,tinynodes]
\draw[pole,crosspole] (-0.5,0) to[out=90,in=270] (-0.5,1.5);
\draw[usual,crossline] (1,0) to[out=90,in=270] (-0.25,0.5);
\draw[pole,crosspole] (0,0) to[out=90,in=270] (0,1.5);
\draw[pole,crosspole] (0.5,0) to[out=90,in=270] (0.5,1.5);
\draw[usual,crossline] (-0.25,0.5) to[out=90,in=180] (0.25,0.8)
to[out=0,in=180] (0.75,0.8) to[out=0,in=90] (1.5,0);
\draw[usual,crossline] (1,1.5) to[out=270,in=180] (1.25,1.25) 
to[out=0,in=270] (1.5,1.5);
\end{tikzpicture}
\right)^{\star}
=
\begin{tikzpicture}[anchorbase,scale=0.7,tinynodes]
\draw[usual,crossline] (-0.25,-0.5) to[out=270,in=180] (0.25,-0.8) 
to[out=0,in=180] (0.75,-0.8) to[out=0,in=270] (1.5,0);
\draw[pole,crosspole] (0,0) to[out=270,in=90] (0,-1.5);
\draw[pole,crosspole] (0.5,0) to[out=270,in=90] (0.5,-1.5);
\draw[pole,crosspole] (-0.5,0) to[out=270,in=90] (-0.5,-1.5);
\draw[usual,crossline] (1,0) to[out=270,in=90] (-0.25,-0.5);
\draw[usual,crossline] (1,-1.5) to[out=90,in=180] (1.25,-1.25) to[out=0,in=90] (1.5,-1.5);
\end{tikzpicture}
.
\end{gather*}

\begin{proposition}\label{proposition:tl}
The above defines an 
involutive sandwich cell datum for $\setstuff{TL}_{g,n}(\cpar)$.
\end{proposition}

\begin{proof}
It is clear that one can cut any crossingless matching 
of $2n$ points of genus $g$ uniquely into pieces as 
illustrated in \eqref{eq:tl-basis-illustration}, hence that 
\eqref{eq:tl-basis} is a $\KK$-basis follows.
So it remains to verify \eqref{eq:cell-mult}. 
That the multiplication is ordered with respect to through 
strands follows (as for the classical Temperley--Lieb algebra) 
from the fact that one can not remove cups and caps, they 
can only be created. All the other requirements are clear by
construction; one basically calculates
\begin{gather*}
\scalebox{0.85}{$\begin{tikzpicture}[anchorbase,scale=1]
\draw[mor] (0,-0.5) to (0.25,0) to (0.75,0) to (1,-0.5) to (0,-0.5);
\node at (0.5,-0.25){$D$};
\draw[mor] (0,-0.5) to (0.25,-1) to (0.75,-1) to (1,-0.5) to (0,-0.5);
\node at (0.5,-0.75){$U$};
\end{tikzpicture}
=
\begin{tikzpicture}[anchorbase,scale=0.7,tinynodes]
\draw[pole,crosspole] (-0.5,0) to[out=90,in=270] (-0.5,1.5);
\draw[usual,crossline] (-0.25,0.5) to[out=90,in=180] (0,0.75) to[out=0,in=90] (2.5,0);
\draw[pole,crosspole] (0,0) to[out=90,in=270] (0,1.5);
\draw[pole,crosspole] (0.5,0) to[out=90,in=270] (0.5,1.5);
\draw[usual,crossline] (1,0) to[out=90,in=270] (-0.25,0.5);
\draw[usual,crossline] (1.5,0) to[out=90,in=180] (1.75,0.25)node[above]{$\gamma$} 
to[out=0,in=90] (2,0);
\draw[usual,crossline] (3,0) to[out=90,in=270] (3,1.5);
\draw[usual,crossline] (3.5,0) to[out=90,in=180] (3.75,0.25) to[out=0,in=90] (4,0);
\draw[pole,crosspole] (-0.5,0) to[out=270,in=90] (-0.5,-1.5);
\draw[usual,crossline] (2,0) to[out=270,in=0] (0,-0.2) to[out=180,in=90] (-0.25,-0.4);
\draw[usual,crossline] (-0.25,-1.2) to[out=270,in=180] (0,-1.4) to[out=0,in=270] (3.5,0);
\draw[pole,crosspole] (0,0) to[out=270,in=90] (0,-1.5);
\draw[pole,crosspole] (0.5,0) to[out=270,in=90] (0.5,-1.5);
\draw[usual,crossline] (-0.25,-0.4) to[out=270,in=180] (0,-0.6) to[out=0,in=270] (2.5,0);
\draw[usual,crossline] (3,0) to[out=270,in=0] (0,-1) to[out=180,in=90] (-0.25,-1.2);
\draw[usual,crossline] (1,0) to[out=270,in=180] (1.25,-0.25) to[out=0,in=270] (1.5,0);
\draw[usual,crossline] (4,0) to[out=270,in=90] (4,-1.5);
\end{tikzpicture}
=
\cvar_{\gamma}\cdot
\begin{tikzpicture}[anchorbase,scale=0.7,tinynodes]
\draw[pole,crosspole] (-0.5,0) to[out=90,in=270] (-0.5,1.5);
\draw[usual,crossline] (1,0) to[out=90,in=270] (-0.25,0.75);
\draw[pole,crosspole] (0,0) to[out=90,in=270] (0,1.5);
\draw[pole,crosspole] (0.5,0) to[out=90,in=270] (0.5,1.5);
\draw[usual,crossline] (-0.25,0.75) to[out=90,in=270] (1,1.5);
\end{tikzpicture}$}
,
\end{gather*}
where we used the elements in \eqref{eq:tl-basis-illustration}.
The property of being involutive also holds since any potential winding is allowed 
for the diagrams $D$ and $U$, so $(\placeholder)^{\star}$ is a $1$:$1$-correspondence between down $D$ and up $U$ diagrams.
\end{proof}

\begin{example}\label{example:cells-tl}
The cells of the algebra $\setstuff{TL}_{n,g}(\cpar)$ are 
of infinite size, so we can not display them here. 
However, let us illustrate strands of $\setstuff{TL}_{4,g}(\cpar)$ 
dashed (and colored) if they can wind around cores.
Using the same conventions 
as in \fullref{example:brauer}: then the cells $\mathcal{J}_{4}$ 
(bottom) to $\mathcal{J}_{0}$ (top), in order 
from smallest (bottom) to biggest (top) are given by
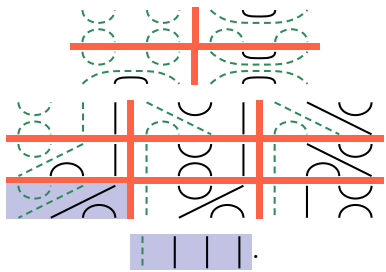
\begin{gather*}
\scalebox{0.85}{$\begin{tabular}{C!{\color{tomato}\vrule width 1mm}C}
\arrayrulecolor{tomato}
\begin{tikzpicture}[anchorbase,yscale=-1]
\draw[dusual] (0,0) to[out=90,in=180] (0.25,0.2) to[out=0,in=90] (0.5,0);
\draw[dusual] (0,0.5) to[out=270,in=180] (0.25,0.3) to[out=0,in=270] (0.5,0.5);
\draw[dusual] (1,0) to[out=90,in=180] (1.25,0.2) to[out=0,in=90] (1.5,0);
\draw[dusual] (1,0.5) to[out=270,in=180] (1.25,0.3) to[out=0,in=270] (1.5,0.5);
\end{tikzpicture} &
\begin{tikzpicture}[anchorbase,yscale=-1]
\draw[dusual] (0,0) to[out=45,in=180] (0.75,0.20) to[out=0,in=135] (1.5,0);
\draw[usual] (0.5,0) to[out=90,in=180] (0.75,0.1) to[out=0,in=90] (1,0);
\draw[dusual] (0,0.5) to[out=270,in=180] (0.25,0.3) to[out=0,in=270] (0.5,0.5);
\draw[dusual] (1,0.5) to[out=270,in=180] (1.25,0.3) to[out=0,in=270] (1.5,0.5);
\end{tikzpicture}
\\
\arrayrulecolor{tomato}\hline
\begin{tikzpicture}[anchorbase,yscale=-1]
\draw[dusual] (0,0) to[out=90,in=180] (0.25,0.2) to[out=0,in=90] (0.5,0);
\draw[dusual] (1,0) to[out=90,in=180] (1.25,0.2) to[out=0,in=90] (1.5,0);
\draw[dusual] (0,0.5) to[out=315,in=180] (0.75,0.3) to[out=0,in=225] (1.5,0.5);
\draw[usual] (0.5,0.5) to[out=270,in=180] (0.75,0.4) to[out=0,in=270] (1,0.5);
\end{tikzpicture} &
\begin{tikzpicture}[anchorbase,yscale=-1]
\draw[dusual] (0,0) to[out=45,in=180] (0.75,0.20) to[out=0,in=135] (1.5,0);
\draw[dusual] (0,0.5) to[out=315,in=180] (0.75,0.3) to[out=0,in=225] (1.5,0.5);
\draw[usual] (0.5,0) to[out=90,in=180] (0.75,0.1) to[out=0,in=90] (1,0);
\draw[usual] (0.5,0.5) to[out=270,in=180] (0.75,0.4) to[out=0,in=270] (1,0.5);
\end{tikzpicture}
\end{tabular}$}
\\[1pt]
\scalebox{0.85}{$\begin{tabular}{C!{\color{tomato}\vrule width 1mm}C!{\color{tomato}\vrule width 1mm}C}
\arrayrulecolor{tomato}
\begin{tikzpicture}[anchorbase,yscale=-1]
\draw[dusual] (0,0) to[out=90,in=180] (0.25,0.2) to[out=0,in=90] (0.5,0);
\draw[dusual] (0,0.5) to[out=270,in=180] (0.25,0.3) to[out=0,in=270] (0.5,0.5);
\draw[dusual] (1,0) to (1,0.5);
\draw[usual] (1.5,0) to (1.5,0.5);
\end{tikzpicture} & 
\begin{tikzpicture}[anchorbase,yscale=-1]
\draw[usual] (0.5,0) to[out=90,in=180] (0.75,0.2) to[out=0,in=90] (1,0);
\draw[dusual] (0,0.5) to[out=270,in=180] (0.25,0.3) to[out=0,in=270] (0.5,0.5);
\draw[dusual] (0,0) to (1,0.5);
\draw[usual] (1.5,0) to (1.5,0.5);
\end{tikzpicture} &
\begin{tikzpicture}[anchorbase,yscale=-1]
\draw[usual] (1,0) to[out=90,in=180] (1.25,0.2) to[out=0,in=90] (1.5,0);
\draw[dusual] (0,0.5) to[out=270,in=180] (0.25,0.3) to[out=0,in=270] (0.5,0.5);
\draw[dusual] (0,0) to (1,0.5);
\draw[usual] (0.5,0) to (1.5,0.5);
\end{tikzpicture}
\\
\arrayrulecolor{tomato}\hline
\begin{tikzpicture}[anchorbase,yscale=-1]
\draw[dusual] (0,0) to[out=90,in=180] (0.25,0.2) to[out=0,in=90] (0.5,0);
\draw[usual] (0.5,0.5) to[out=270,in=180] (0.75,0.3) to[out=0,in=270] (1,0.5);
\draw[dusual] (1,0) to (0,0.5);
\draw[usual] (1.5,0) to (1.5,0.5);
\end{tikzpicture} & 
\begin{tikzpicture}[anchorbase,yscale=-1]
\draw[dusual] (0,0) to (0,0.5);
\draw[usual] (0.5,0) to[out=90,in=180] (0.75,0.2) to[out=0,in=90] (1,0);
\draw[usual] (0.5,0.5) to[out=270,in=180] (0.75,0.3) to[out=0,in=270] (1,0.5);
\draw[usual] (1.5,0) to (1.5,0.5);
\end{tikzpicture} &
\begin{tikzpicture}[anchorbase,yscale=-1]
\draw[dusual] (0,0) to (0,0.5);
\draw[usual] (0.5,0) to (1.5,0.5);
\draw[usual] (1,0) to[out=90,in=180] (1.25,0.2) to[out=0,in=90] (1.5,0);
\draw[usual] (0.5,0.5) to[out=270,in=180] (0.75,0.3) to[out=0,in=270] (1,0.5);
\end{tikzpicture}
\\
\arrayrulecolor{tomato}\hline
\cellcolor{mydarkblue!25}
\begin{tikzpicture}[anchorbase,yscale=-1]
\draw[dusual] (0,0) to[out=90,in=180] (0.25,0.2) to[out=0,in=90] (0.5,0);
\draw[usual] (1,0.5) to[out=270,in=180] (1.25,0.3) to[out=0,in=270] (1.5,0.5);
\draw[dusual] (1,0) to (0,0.5);
\draw[usual] (1.5,0) to (0.5,0.5);
\end{tikzpicture} & 
\begin{tikzpicture}[anchorbase,yscale=-1]
\draw[dusual] (0,0) to (0,0.5);
\draw[usual] (1.5,0) to (0.5,0.5);
\draw[usual] (0.5,0) to[out=90,in=180] (0.75,0.2) to[out=0,in=90] (1,0);
\draw[usual] (1,0.5) to[out=270,in=180] (1.25,0.3) to[out=0,in=270] (1.5,0.5);
\end{tikzpicture} &
\begin{tikzpicture}[anchorbase,yscale=-1]
\draw[dusual] (0,0) to (0,0.5);
\draw[usual] (0.5,0) to (0.5,0.5);
\draw[usual] (1,0) to[out=90,in=180] (1.25,0.2) to[out=0,in=90] (1.5,0);
\draw[usual] (1,0.5) to[out=270,in=180] (1.25,0.3) to[out=0,in=270] (1.5,0.5);
\end{tikzpicture}
\end{tabular}$}
\\[1pt]
\scalebox{0.85}{$\begin{tabular}{C}
\cellcolor{mydarkblue!25}
\begin{tikzpicture}[anchorbase,yscale=-1]
\draw[dusual] (0,0) to (0,0.5);
\draw[usual] (0.5,0) to (0.5,0.5);
\draw[usual] (1,0) to (1,0.5);
\draw[usual] (1.5,0) to (1.5,0.5);
\end{tikzpicture}
\end{tabular}$}
.
\end{gather*}
Here we have indicated examples of $\mathcal{H}$-cells that are 
isomorphic to $\KK\setstuff{F}_{g}$, with $\setstuff{F}_{g}$ corresponding 
to the leftmost through strand, regardless 
of the parameters $\cpar$. For the top cell 
we can only pick up a copy of $\KK$ if at least 
one of the corresponding parameters is invertible in $\KK$. 
\end{example}

For the following recall that $\setstuff{F}_{g}$ is trivial 
by definition for $\lambda=0$, see \autoref{remark:fgtrivial}.

\begin{theorem}
Let $\KK$ be a field.
\begin{enumerate}

\item If $\cpar\neq 0$, or $\cpar=0$ and $\lambda\neq 0$ is odd, 
then all $\lambda\in\Lambda$ are apexes. In the remaining case, 
$\cpar=0$ and $\lambda=0$ (this only happens if $n$ is even), 
all $\lambda\in\Lambda-\{0\}$ are apexes, but $\lambda=0$ is not an apex.

\item The simple $\setstuff{TL}_{n,g}(\cpar)$-modules of 
apex $\lambda\in\Lambda$ 
are parameterized by simple modules of $\KK\setstuff{F}_{g}$.

\item The simple $\setstuff{TL}_{n,g}(\cpar)$-modules of 
apex $\lambda\in\Lambda$ can be constructed as 
the simple heads of
$\mathrm{Ind}_{\KK\setstuff{F}_{g}}^{\setstuff{TL}_{n,g}(\cpar)}(K)$, 
where $K$ runs over (equivalence classes of) 
simple $\KK\setstuff{F}_{g}$-modules.

\end{enumerate}
\end{theorem}

\begin{proof}
Word-by-word as for the Brauer algebra, 
see the proof of \fullref{theorem:brauer}. In particular, 
this is a direct application of \fullref{theorem:classification}.
\end{proof}

\subsection{Handlebody blob algebras}\label{subsection:handlebody-blob}

An often applied strategy to turn an infinite-dimensional algebra 
into a finite-dimensional algebra is to impose 
a cyclotomic condition on generators of infinite order.
In our case we will impose relations on the coil 
generators $\tau_{u}$.

Following history, we will use a slightly 
different diagrammatic presentation for these algebras, 
namely using \emph{blob diagrams of $2n$ points of genus $g$}. First, 
for everything not involving any core strands 
the diagrams stay the same. Moreover, we will use
\begin{gather}\label{eq:blobs}
\tau_{u}=
\begin{tikzpicture}[anchorbase,scale=0.7,tinynodes]
\draw[pole,crosspole] (-0.5,0) to[out=90,in=270] (-0.5,1.5);
\draw[usual,crossline] (1,0)node[below,black]{$1$} 
to[out=90,in=270] (-0.25,0.75);
\draw[pole,crosspole] (0,0)node[below,black]{$u$} 
to[out=90,in=270] (0,1.5)node[above,black,yshift=-3pt]{$u$};
\draw[pole,crosspole] (0.5,0)node[below,black]{$v$} 
to[out=90,in=270] (0.5,1.5)node[above,black,yshift=-3pt]{$v$};
\draw[usual,crossline] (-0.25,0.75) to[out=90,in=270] 
(1,1.5)node[above,black,yshift=-3pt]{$1$};
\end{tikzpicture}
\rightsquigarrow
b_{u}=
\begin{tikzpicture}[anchorbase,scale=0.7,tinynodes]
\draw[usual,blobbed={0.5}{u}{spinach}] 
(1,0)node[below,black]{$1$} to[out=90,in=270]
(1,1.5)node[above,black,yshift=-3pt]{$1$};
\end{tikzpicture}
,\quad
\tau_{v}=
\begin{tikzpicture}[anchorbase,scale=0.7,tinynodes]
\draw[pole,crosspole] (-0.5,0) to[out=90,in=270] (-0.5,1.5);
\draw[usual,crossline] (1,0)node[below,black]{$1$} 
to[out=90,in=270] (0.25,0.75);
\draw[pole,crosspole] (0,0)node[below,black]{$u$} 
to[out=90,in=270] (0,1.5)node[above,black,yshift=-3pt]{$u$};
\draw[pole,crosspole] (0.5,0)node[below,black]{$v$} 
to[out=90,in=270] (0.5,1.5)node[above,black,yshift=-3pt]{$v$};
\draw[usual,crossline] (0.25,0.75) to[out=90,in=270] 
(1,1.5)node[above,black,yshift=-3pt]{$1$};
\end{tikzpicture}
\rightsquigarrow
b_{v}=
\begin{tikzpicture}[anchorbase,scale=0.7,tinynodes]
\draw[usual,blobbed={0.5}{v}{tomato}] (1,0)node[below,black]{$1$} to[out=90,in=270]
(1,1.5)node[above,black,yshift=-3pt]{$1$};
\end{tikzpicture}
.
\end{gather}
That is, we use colored \emph{blobs} 
instead of coils, which clarifies the nomenclature. 
We denote the elements corresponding to 
coils by $b_{u}$. We also say that a blob labeled $u$ has 
\emph{type $u$}.
Note that blobs are always reachable from the left by a straight line 
coming from $-\infty$, and move 
freely along strands in a vertical direction, but can not pass one another:
\begin{gather}\label{eq:not-defined}
\text{Ok:}
\begin{tikzpicture}[anchorbase,scale=0.7,tinynodes]
\draw[usual,blobbed={0.5}{u}{spinach}] (0.5,0)to[out=90,in=270](0.5,1.5);
\draw[usual] (1,0)to[out=90,in=270](1,1.5);
\end{tikzpicture}
,
\begin{tikzpicture}[anchorbase,scale=0.7,tinynodes]
\draw[usual,crossline,blobbed={0.2}{u}{spinach}] 
(1,0) to[out=270,in=180] (1.25,-0.5) to[out=0,in=270] (1.5,0);
\end{tikzpicture}
,\quad
\text{not defined}:
\begin{tikzpicture}[anchorbase,scale=0.7,tinynodes]
\draw[usual] (0.5,0)to[out=90,in=270](0.5,1.5);
\draw[usual,rblobbed={0.5}{u}{spinach}] (1,0)to[out=90,in=270](1,1.5);
\end{tikzpicture}
,
\begin{tikzpicture}[anchorbase,scale=0.7,tinynodes]
\draw[usual,crossline,rblobbed={0.8}{u}{spinach}] 
(1,0) to[out=270,in=180] (1.25,-0.5) to[out=0,in=270] (1.5,0);
\end{tikzpicture}
,\quad
\text{not equal}:
\begin{tikzpicture}[anchorbase,scale=0.7,tinynodes]
\draw[usual,blobbed={0.33}{u}{spinach},blobbed={0.66}{v}{tomato}] 
(1,0)to[out=90,in=270](1,1.5);
\end{tikzpicture}
\neq
\begin{tikzpicture}[anchorbase,scale=0.7,tinynodes]
\draw[usual,blobbed={0.33}{v}{tomato},blobbed={0.66}{u}{spinach}] 
(1,0)to[out=90,in=270](1,1.5);
\end{tikzpicture}.
\end{gather} 
An example of such a diagram is:
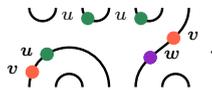
\begin{gather*}
\begin{tikzpicture}[anchorbase,scale=0.7,tinynodes]
\draw[usual] (1,1.5) to[out=270,in=180] (1.25,1.25) to[out=0,in=270] (1.5,1.5);
\draw[usual,blobbed={0.25}{u}{spinach}] (2,1.5) 
to[out=270,in=180] (2.25,1.25) to[out=0,in=270] (2.5,1.5);
\draw[usual,blobbed={0.25}{u}{spinach}] (3,1.5) 
to[out=270,in=180] (3.25,1.25) to[out=0,in=270] (3.5,1.5);
\draw[usual,rblobbed={0.33}{w}{orchid},rblobbed={0.66}{v}{tomato}] 
(3,0) to[out=90,in=270] (4,1.5);
\draw[usual,blobbed={0.1}{v}{tomato},blobbed={0.3}{u}{spinach}] (1,0) 
to[out=90,in=180] (1.75,0.75) to[out=0,in=90] (2.5,0);
\draw[usual] (1.5,0) to[out=90,in=180] (1.75,0.25) to[out=0,in=90] (2,0);
\draw[usual] (3.5,0) to[out=90,in=180] (3.75,0.25) to[out=0,in=90] (4,0);
\end{tikzpicture}
.
\end{gather*}
(Some of the blobs in this illustration
are strictly speaking not reachable 
from the left as they are behind cups and caps when drawing 
a straight line.
But here and throughout, to simplify illustrations, 
we will suppress the relevant height moves since they do not play any role.)

\begin{remark}
The reader might wonder whether 
(for appropriate parameters $\qvar^{\pm 1/2}$) 
one can not use \eqref{eq:kauffman} to define
\begin{gather*}
\begin{tikzpicture}[anchorbase,scale=0.7,tinynodes]
\draw[pole,crosspole,white] (2,0)node[below,white,yshift=-1pt]{$u$} to[out=90,in=270] (2,1.5)node[above,white,yshift=-3pt]{$u$};
\draw[usual] (2.5,0)node[below,black,yshift=-1pt]{$i$} to[out=90,in=270] (1.25,0.75);
\draw[usual,crossline] (1.5,0) to[out=90,in=270] (1.5,1.5);
\draw[usual,crossline] (2,0) to[out=90,in=270] (2,1.5);
\draw[usual,crossline] (1.25,0.75) to[out=90,in=270] (2.5,1.5)node[above,black,yshift=-3pt]{$i$};
\draw[usual,blobbed={1}{u}{spinach}] (1.25,0.65) to (1.25,0.85);
\end{tikzpicture}
=
\begin{tikzpicture}[anchorbase,scale=0.7,tinynodes]
\draw[pole,crosspole,white] (2,0)node[below,white,yshift=-1pt]{$u$} to[out=90,in=270] (2,1.5)node[above,white,yshift=-3pt]{$u$};
\draw[usual,rblobbed={0.5}{u}{spinach}] (2.5,0)node[below,black,yshift=-1pt]{$i$} to[out=90,in=270] (2.5,1.5)node[above,black,yshift=-3pt]{$i$};
\draw[usual,crossline] (1.5,0) to[out=90,in=270] (1.5,1.5);
\draw[usual,crossline] (2,0) to[out=90,in=270] (2,1.5);
\end{tikzpicture}
?
\end{gather*}
Indeed, in a topological model that is possible. However, as we will 
explore more carefully in \fullref{subsection:handlebody-blob-topology},
these diagrams will not be topological in nature, but 
have some error terms. So we decided to 
keep blobs to the left from the start.
\end{remark}

Using the same reading conventions as for $\setstuff{TL}_{g,n}(\cpar)$ but 
counting types of blobs instead of coils, we can associated a word in $\setstuff{F}_{g}$ 
to usual strands and circles. As before, all parameters 
below will be assumed to be constant on conjugacy classes.

Fix \emph{cyclotomic parameters} 
$\bpar=(\bvar_{u,i})\in\KK^{d_{1}+\dots+d_{g}}$.
We also fix $\dpar=(\dvar_{u})\in\N^{g}$ called the \emph{degree vector}.
Note that concatenation can create circles with blobs.
For each such circle $\gamma$ with 
at most $\dvar_{u}-1$ blobs of the corresponding kind
we choose a parameter $\cvar_{\gamma}\in\KK$, whose collection 
is denoted by $\cpar=(\cvar_{\gamma})_{\gamma}$.

\begin{definition}\label{definition:circle}
The evaluation of 
a closed circle $\gamma$ is defined to be the removal 
of a closed component, contributing a factor 
$\cvar_{\gamma}$. We call this \emph{blob circle evaluation}.
\end{definition}

\begin{example}
The circle evaluation of the handlebody Temperley--Lieb 
algebra becomes 
blob circle evaluation, {\eg}
\begin{gather*}
\begin{tikzpicture}[anchorbase,scale=0.7,tinynodes]
\draw[pole,crosspole] (0.5,0) to[out=90,in=270] (0.5,1.5);
\draw[pole,crosspole] (-0.5,0) to[out=90,in=270] (-0.5,1.5);
\draw[usual,crossline] (1.25,0.75) to[out=270,in=270] (-0.25,0.55);
\draw[usual,crossline] (0.25,0.75) to[out=90,in=270] (-0.25,0.95);
\draw[pole,crosspole] (0,0)node[below,black]{$u$} 
to[out=90,in=270] (0,1.5)node[above,black,yshift=-3pt]{$u$};
\draw[pole,crosspole] (1,0)node[below,black]{$v$} 
to[out=90,in=270] (1,1.5)node[above,black,yshift=-3pt]{$v$};
\draw[usual,crossline] (-0.25,0.9) to[out=90,in=90] (1.25,0.75);
\draw[usual,crossline] (-0.25,0.55) to[out=90,in=270] (0.25,0.75);
\end{tikzpicture}
=\cpar_{\gamma^{\prime}=u^{2}v}
\leftrightsquigarrow
\begin{tikzpicture}[anchorbase,scale=0.7,tinynodes]
\draw[usual] (1,1.2) to[out=90,in=180] (1.25,1.35) to[out=0,in=90] (1.5,1.2) 
to (1.5,0.3) to[out=270,in=0] (1.25,0.15) to[out=180,in=270] (1,0.3);
\draw[usual,blobbed={0}{u}{spinach},blobbed={0.5}{u}{spinach},blobbed={1}{v}{tomato}] (1,0.3) 
to[out=90,in=270] (1,1.2);
\end{tikzpicture}
=
\begin{tikzpicture}[anchorbase,scale=0.7,tinynodes]
\draw[usual] (1,1.2) to[out=90,in=180] (1.25,1.35) to[out=0,in=90] (1.5,1.2) 
to (1.5,0.3) to[out=270,in=0] (1.25,0.15) to[out=180,in=270] (1,0.3);
\draw[usual,blobbed={0}{v}{tomato},blobbed={0.5}{u}{spinach},blobbed={1}{u}{spinach}] (1,0.3) 
to[out=90,in=270] (1,1.2);
\end{tikzpicture}
=
\begin{tikzpicture}[anchorbase,scale=0.7,tinynodes]
\draw[usual] (1,1.2) to[out=90,in=180] (1.25,1.35) to[out=0,in=90] (1.5,1.2) 
to (1.5,0.3) to[out=270,in=0] (1.25,0.15) to[out=180,in=270] (1,0.3);
\draw[usual,blobbed={0}{u}{spinach},blobbed={0.5}{v}{tomato},blobbed={1}{u}{spinach}] (1,0.3) 
to[out=90,in=270] (1,1.2);
\end{tikzpicture}
=\cpar_{\gamma^{\prime}=u^{2}v}
.
\end{gather*}
\end{example}

Note that circles with more than $\dvar_{u}-1$ 
blobs of the corresponding kind 
can be evaluated by using \eqref{eq:cyclotomic-blob} and the 
above circle evaluation, so we do not need to define their evaluation. 

\begin{definition}\label{definition:blob}
We let the (cyclotomic) \emph{handlebody blob algebra} 
(in $n$ strands and of genus $g$)
$\setstuff{Bl}_{g,n}^{\dpar,\bpar}(\cpar)$ be 
the algebra whose underlying $\KK$-vector space is the
$\KK$-linear span of all 
blob diagrams on $2n$ points of genus $g$,
with multiplication given by concatenation of diagrams modulo 
blob circle evaluation, and the two-sided ideal 
generated by the \emph{cyclotomic relations}
\begin{gather}\label{eq:cyclotomic-blob}
(b_{u}-\bvar_{u,1})
\varpi_{1}
(b_{u}-\bvar_{u,2})
\varpi_{2}
\dots
(b_{u}-\bvar_{u,\dvar_{u}-1})
\varpi_{d_{u}-1}
(b_{u}-\bvar_{u,\dvar_{u}})
=0,
\end{gather}
where $\varpi_{j}$ is any finite (potentially empty) 
word in $b_{v}$ for $v\neq u$. 
\end{definition}

In words, any occurrence of $\dvar_{u}$ blobs colored $g$ 
on a strand can be replaced by the corresponding expression 
in the expansion of \eqref{eq:cyclotomic-blob}.

\begin{example}\label{example:blob}
For $g=2$, let us choose $\dvar_{1}=1$ and $\dvar_{2}=2$. We get
\begin{gather*}
\begin{gathered}
\bvar_{1,1}=0
\\[-5pt]
\bvar_{2,1}=0
\\[-5pt]
\bvar_{2,2}=0
\end{gathered}
\quad\text{gives}\quad
\begin{tikzpicture}[anchorbase,scale=0.7,tinynodes]
\draw[usual,blobbed={0.5}{1}{spinach}] (1,0) to[out=90,in=270](1,1.5);
\end{tikzpicture}
=
\begin{tikzpicture}[anchorbase,scale=0.7,tinynodes]
\draw[usual,blobbed={0.33}{2}{tomato},blobbed={0.66}{2}{tomato}] 
(1,0) to[out=90,in=270](1,1.5);
\end{tikzpicture}
=
\begin{tikzpicture}[anchorbase,scale=0.7,tinynodes]
\draw[usual,blobbed={0.25}{2}{tomato},blobbed={0.5}{1}{spinach},blobbed={0.75}{2}{tomato}] 
(1,0) to[out=90,in=270](1,1.5);
\end{tikzpicture}
=0
,\quad
\begin{gathered}
\bvar_{1,1}=1
\\[-5pt]
\bvar_{2,1}=0
\\[-5pt]
\bvar_{2,2}=0
\end{gathered}
\quad\text{gives}\quad
\begin{tikzpicture}[anchorbase,scale=0.7,tinynodes]
\draw[usual,blobbed={0.5}{1}{spinach}] (1,0) to[out=90,in=270](1,1.5);
\end{tikzpicture}
=
\begin{tikzpicture}[anchorbase,scale=0.7,tinynodes]
\draw[usual] (1,0) to[out=90,in=270](1,1.5);
\end{tikzpicture}
,
\begin{tikzpicture}[anchorbase,scale=0.7,tinynodes]
\draw[usual,blobbed={0.33}{2}{tomato},blobbed={0.66}{2}{tomato}] 
(1,0) to[out=90,in=270](1,1.5);
\end{tikzpicture}
=
\begin{tikzpicture}[anchorbase,scale=0.7,tinynodes]
\draw[usual,blobbed={0.25}{2}{tomato},blobbed={0.5}{1}{spinach},blobbed={0.75}{2}{tomato}] 
(1,0) to[out=90,in=270](1,1.5);
\end{tikzpicture}
=0
,
\\
\begin{gathered}
\bvar_{1,1}=1
\\[-5pt]
\bvar_{2,1}=1
\\[-5pt]
\bvar_{2,2}=0
\end{gathered}
\quad\text{gives}\quad
\begin{tikzpicture}[anchorbase,scale=0.7,tinynodes]
\draw[usual,blobbed={0.5}{1}{spinach}] (1,0) to[out=90,in=270](1,1.5);
\end{tikzpicture}
=
\begin{tikzpicture}[anchorbase,scale=0.7,tinynodes]
\draw[usual] (1,0) to[out=90,in=270](1,1.5);
\end{tikzpicture}
,
\begin{tikzpicture}[anchorbase,scale=0.7,tinynodes]
\draw[usual,blobbed={0.33}{2}{tomato},blobbed={0.66}{2}{tomato}] 
(1,0) to[out=90,in=270](1,1.5);
\end{tikzpicture}
=
\begin{tikzpicture}[anchorbase,scale=0.7,tinynodes]
\draw[usual,blobbed={0.5}{2}{tomato}] (1,0) to[out=90,in=270](1,1.5);
\end{tikzpicture}
\quad
\begin{tikzpicture}[anchorbase,scale=0.7,tinynodes]
\draw[usual,blobbed={0.25}{2}{tomato},blobbed={0.5}{1}{spinach},blobbed={0.75}{2}{tomato}] 
(1,0) to[out=90,in=270](1,1.5);
\end{tikzpicture}
=
\begin{tikzpicture}[anchorbase,scale=0.7,tinynodes]
\draw[usual,blobbed={0.33}{1}{spinach},blobbed={0.66}{2}{tomato}] 
(1,0) to[out=90,in=270](1,1.5);
\end{tikzpicture}
=
\begin{tikzpicture}[anchorbase,scale=0.7,tinynodes]
\draw[usual,blobbed={0.5}{2}{tomato}] (1,0) to[out=90,in=270](1,1.5);
\end{tikzpicture}
.
\end{gather*}
Note that the final expression can be resolved in two ways: First 
by removing the blob colored 1, and then by replacing the two blobs colored 2 
by one blob colored 2. Second, by applying \eqref{eq:cyclotomic-blob} 
directly, as we did above. 
Both give the same result.
\end{example}

\begin{remark}\label{remark:no-inverses}
We have not defined 
$\setstuff{Bl}_{g,n}^{\dpar,\bpar}(\cpar)$
as a quotient of $\setstuff{TL}_{g,n}(\cpar)$ 
because blobs would be invertible 
if we define them as the image of $\tau_{u}$ in the quotient
and we want to include the possibility of blobs being nilpotent, 
{\cf} \fullref{example:blob}.

Note that however that for certain choices of parameters
the blobs are invertible, and $\setstuff{Bl}_{g,n}^{\dpar,\bpar}(\cpar)$
is often a quotient of $\setstuff{TL}_{g,n}(\cpar)$ for these parameters.
\end{remark}

\begin{remark}\label{remark:blob}
We again discuss a few instances of \fullref{definition:blob}:
\begin{enumerate}

\setlength\itemsep{0.15cm}

\item For $g=0$ the blob and the Temperley--Lieb algebra 
coincide.

\item The algebra $\setstuff{Bl}_{1,n}^{\dpar,\bpar}(\cpar)$ 
is sometimes called the (cyclotomic) blob algebra, see {\eg} \cite{MaSa-blob}.

\item In genus $g=2$ there is a two-boundary 
blob algebra, see {\eg} 
\cite{deGiNi-two-boundary-tl} (beware that the version of $\setstuff{Bl}_{2,n}^{\dpar,\bpar}(\cpar)$ 
in that paper is called a two-boundary 
Temperley--Lieb algebra).

\end{enumerate}
For $g>0$
the terminology in the literature is not consistent, and 
Temperley--Lieb algebras and blob algebras might 
be the same or not. For us \fullref{definition:blob} 
generalizes \cite{MaSa-blob}.
\end{remark}

We need the 
analog of \fullref{lemma:tl-basis} which reads as follows.

\begin{lemma}\label{lemma:blob-basis}
The algebra 
$\setstuff{Bl}_{g,n}^{\dpar,\bpar}(\cpar)$ is an associative, unital 
algebra with a $\KK$-basis given by all handlebody blob diagrams 
on $2n$ points of genus 
$g$ where all strands have at most $\dvar_{u}$ 
blobs of the corresponding type.
\end{lemma}

\begin{proof}
The only fact to observe is that 
the cyclotomic condition \eqref{eq:cyclotomic-blob} 
ensures that it suffices to 
fix an evaluation for any circle whose number of blobs are bounded 
by the degree vector.
\end{proof}

For our fixed genus $g$ and degree vector $\dpar$ we define 
the (corresponding) \emph{blob numbers} using multinomial
coefficients:
\begin{gather}\label{eq:blob-numbers}
\bbvar_{g,\dpar}=
\sum_{k\in\N}\,
\sum_{\substack{0\leq k_{u}\leq\min(k,\dvar_{u}-1)\\k_{1}+\dots+k_{g}=k}}
\binom{k}{k_{1},\dots,k_{g}},
\end{gather}
where the sums run over all $k\in\N$ and all 
$0\leq k_{u}\leq\min(k,\dvar_{u}-1)$,
for $u\in\{1,\dots,g\}$, that sum up to $k$.
Note that this sum is finite.

\begin{example}\label{example:blob-numbers}
We have $\bbvar_{0,\dpar}=1$ (by definition), 
$\bbvar_{1,\dpar}=\dvar_{1}$ and, for $\dvar_{1}=2$ and $\dvar_{2}=3$, 
we get $\bbvar_{2,\dpar}=
\binom{0}{0,0}+\binom{1}{1,0}+\binom{1}{0,1}+\binom{2}{1,1}
+\binom{2}{0,2}+\binom{3}{1,2}=9$. 
Note that there are nine blob diagrams 
with one strand and at most one $1$ blob and two $2$ blobs:
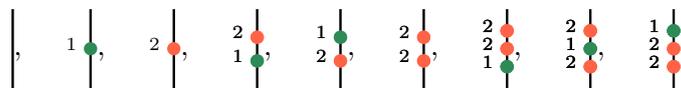
\begin{gather*}
\begin{tikzpicture}[anchorbase,scale=0.7,tinynodes]
\draw[usual] (1,0) to[out=90,in=270](1,1.5);
\end{tikzpicture}
,\quad
\begin{tikzpicture}[anchorbase,scale=0.7,tinynodes]
\draw[usual,blobbed={0.5}{1}{spinach}] (1,0) to[out=90,in=270](1,1.5);
\end{tikzpicture}
,\quad
\begin{tikzpicture}[anchorbase,scale=0.7,tinynodes]
\draw[usual,blobbed={0.5}{2}{tomato}] (1,0) to[out=90,in=270](1,1.5);
\end{tikzpicture}
,\quad
\begin{tikzpicture}[anchorbase,scale=0.7,tinynodes]
\draw[usual,blobbed={0.33}{1}{spinach},blobbed={0.66}{2}{tomato}] (1,0) to[out=90,in=270](1,1.5);
\end{tikzpicture}
,\quad
\begin{tikzpicture}[anchorbase,scale=0.7,tinynodes]
\draw[usual,blobbed={0.33}{2}{tomato},blobbed={0.66}{1}{spinach}] (1,0) to[out=90,in=270](1,1.5);
\end{tikzpicture}
,\quad
\begin{tikzpicture}[anchorbase,scale=0.7,tinynodes]
\draw[usual,blobbed={0.33}{2}{tomato},blobbed={0.66}{2}{tomato}] (1,0) to[out=90,in=270](1,1.5);
\end{tikzpicture}
,\quad
\begin{tikzpicture}[anchorbase,scale=0.7,tinynodes]
\draw[usual,blobbed={0.25}{1}{spinach},blobbed={0.5}{2}{tomato},blobbed={0.75}{2}{tomato}] 
(1,0) to[out=90,in=270](1,1.5);
\end{tikzpicture}
,\quad
\begin{tikzpicture}[anchorbase,scale=0.7,tinynodes]
\draw[usual,blobbed={0.25}{2}{tomato},blobbed={0.5}{1}{spinach},blobbed={0.75}{2}{tomato}] 
(1,0) to[out=90,in=270](1,1.5);
\end{tikzpicture}
,\quad
\begin{tikzpicture}[anchorbase,scale=0.7,tinynodes]
\draw[usual,blobbed={0.25}{2}{tomato},blobbed={0.5}{2}{tomato},blobbed={0.75}{1}{spinach}] 
(1,0) to[out=90,in=270](1,1.5);
\end{tikzpicture}
.
\end{gather*}
In fact, we have
$\bbvar_{2,\dpar=(\dvar_{1},\dvar_{2})}=\binom{\dvar_{1}+\dvar_{2}}{\dvar_{1}}-1$.
\end{example}

\begin{lemma}\label{lemma:blob-onestrand}
For any $\cpar$, we have an isomorphism of algebras
\begin{gather*}
\setstuff{Bl}_{g,1}^{\dpar,\bpar}(\cpar)
\cong
\KK\langle b_{u}\mid u\in\{1,\dots,g\}\rangle
/\eqref{eq:cyclotomic-blob}
.
\end{gather*}
In particular, 
$\dim_{\KK}\setstuff{Bl}_{g,1}^{\dpar,\bpar}(\cpar)=\bbvar_{g,\dpar}$.
\end{lemma}

\begin{proof}
Since blobs do not satisfy any other relation than \eqref{eq:cyclotomic-blob},
the only claim that is not immediate is the dimension count. 
To see that that works, we recall that the multinomial coefficient 
$\binom{k}{a_{1},\dots,a_{g}}$
counts the appearance of $x^{a_{1}}_{1}\dots x^{a_{g}}_{g}$ in the expansion 
of $(x_{1}+\dots+x_{g})^{k}$, which is the counting problem we need to solve.
Finally, note that $b_{u}^{d_{u}}$ can be expressed in terms of
lower order expressions, which explains our condition on the summation.
\end{proof}

We also calculate the dimension of 
$\setstuff{Bl}_{g,n}^{\dpar,\bpar}(\cpar)$:

\begin{proposition}\label{proposition:blob-dim}
We have $\dim_{\KK}\setstuff{Bl}_{g,0}^{\dpar,\bpar}(\cpar)=1$, 
$\dim_{\KK}\setstuff{Bl}_{g,1}^{\dpar,\bpar}(\cpar)=\bbvar_{g,\dpar}$ and
\begin{gather}\label{eq:blob-dim}
\dim_{\KK}\setstuff{Bl}_{g,n}^{\dpar,\bpar}(\cpar)
=
\bbvar_{g,\dpar}
{\textstyle\sum_{k\in\{2,4,6,\dots,2n\}}}\,C_{k-1}\dim_{\KK}
\setstuff{Bl}_{g,n-k}^{\dpar,\bpar}(\cpar),
\end{gather}
where $C_{k-1}$ is the $(k-1)$th Catalan number.
\end{proposition}

The proof of \fullref{proposition:blob-dim} 
is an inductive argument which works in quite some 
generality and that we learned from \cite{tDi-sym-bridges}.

\begin{proof}
By \fullref{lemma:blob-basis}, it suffices to count 
handlebody blob diagrams of genus $g$.
It actually suffices to study the clapped situation:
\begin{gather*}
\begin{tikzpicture}[anchorbase,scale=0.7,tinynodes]
\draw[usual] (1,1.5) to[out=270,in=180] (1.25,1.25) to[out=0,in=270] (1.5,1.5);
\draw[usual,blobbed={0.15}{u}{spinach}] (2,1.5) 
to[out=270,in=180] (2.25,1.25) to[out=0,in=270] (2.5,1.5);
\draw[usual,blobbed={0.15}{u}{spinach}] (3,1.5) 
to[out=270,in=180] (3.25,1.25) to[out=0,in=270] (3.5,1.5);
\draw[usual,blobbed={0.33}{w}{orchid},blobbed={0.66}{v}{tomato}] (3,0) to[out=90,in=270] (4,1.5);
\draw[usual,blobbed={0.1}{v}{tomato},blobbed={0.3}{u}{spinach}] (1,0) 
to[out=90,in=180] (1.75,0.75) to[out=0,in=90] (2.5,0);
\draw[usual] (1.5,0) to[out=90,in=180] (1.75,0.25) to[out=0,in=90] (2,0);
\draw[usual] (3.5,0) to[out=90,in=180] (3.75,0.25) to[out=0,in=90] (4,0);
\draw[very thick,mygray,densely dashed] (0.75,0) to (4.25,0);
\draw[very thick,mygray,->] (2.5,0) to[out=270,in=180] (3.75,-0.5) to[out=0,in=270] (5,0) to (5,1.5);
\end{tikzpicture}
\leftrightsquigarrow
\begin{tikzpicture}[anchorbase,scale=0.7,tinynodes]
\draw[usual] (1,1.5) to[out=270,in=180] (1.25,1.25) to[out=0,in=270] (1.5,1.5);
\draw[usual,blobbed={0.25}{u}{spinach}] (2,1.5) 
to[out=270,in=180] (2.25,1.25) to[out=0,in=270] (2.5,1.5);
\draw[usual,blobbed={0.25}{u}{spinach}] (3,1.5) 
to[out=270,in=180] (3.25,1.25) to[out=0,in=270] (3.5,1.5);
\draw[usual,blobbed={0.1}{v}{tomato},blobbed={0.3}{w}{orchid}] (4,1.5) 
to[out=270,in=180] (4.75,0.75) to[out=0,in=270] (5.5,1.5);
\draw[usual] (4.5,1.5) to[out=270,in=180] (4.75,1.25) to[out=0,in=270] (5,1.5);
\draw[usual,blobbed={0.1}{u}{spinach},blobbed={0.3}{v}{tomato}] (6,1.5) 
to[out=270,in=180] (6.75,0.75) to[out=0,in=270] (7.5,1.5);
\draw[usual] (6.5,1.5) to[out=270,in=180] (6.75,1.25) to[out=0,in=270] (7,1.5);
\draw[very thick,mygray,densely dashed] (3.75,1.5) to (7.75,1.5);
\end{tikzpicture}
.
\end{gather*}
We then argue by induction on $n$. 
The induction start is clear, so let $n>1$.
First, the 
strand with the leftmost point can end at position $2k$ for $k\in\N$, 
and one can divide the diagram into two parts: a part underneath it 
$TL_{k-1}$ and a part to the right of it $BL_{n-k}$. For example,
\begin{gather*}
n=9,k=4\colon
\begin{tikzpicture}[anchorbase,scale=0.7,tinynodes]
\draw[usual,orchid,blobbed={0.1}{w}{orchid},blobbed={0.3}{v}{tomato}] (0.5,1.5) 
to[out=270,in=180] (2.25,0.25) to[out=0,in=270] (4,1.5)node[above]{$2k$};
\draw[usual] (1,1.5) to[out=270,in=180] (1.25,1.25) to[out=0,in=270] (1.5,1.5);
\draw[usual] (2,1.5) to[out=270,in=180] (2.75,0.75) to[out=0,in=270] (3.5,1.5);
\draw[usual] (2.5,1.5) to[out=270,in=180] (2.75,1.25) to[out=0,in=270] (3,1.5);
\draw[usual,rblobbed={0.1}{u}{spinach},blobbed={0.3}{v}{tomato}] (4.5,1.5) 
to[out=270,in=180] (5.25,0.75) to[out=0,in=270] (6,1.5);
\draw[usual] (5,1.5) to[out=270,in=180] (5.25,1.25) to[out=0,in=270] (5.5,1.5);
\draw[usual,blobbed={0.25}{w}{orchid}] (6.5,1.5) 
to[out=270,in=180] (6.75,1.25) to[out=0,in=270] (7,1.5);
\draw[usual,blobbed={0.25}{w}{orchid}] (7.5,1.5) 
to[out=270,in=180] (7.75,1.25) to[out=0,in=270] (8,1.5);
\draw[usual] (8.5,1.5) to[out=270,in=180] (8.75,1.25) 
to[out=0,in=270] (9,1.5)node[above]{$2n$};
\draw[very thick,mygray,densely dashed] (4.25,2) to (4.25,0);
\node at (2.25,1.75) {$TL_{k-1}$};
\node at (6.25,1.75) {$BL_{n-k}$};
\end{tikzpicture}
.
\end{gather*}
The part underneath it, denoted $TL_{k-1}$ above, 
can not carry any blobs, so the number of possible diagrams is the same 
as for the corresponding Temperley--Lieb algebra, which is the Catalan 
number appearing in \eqref{eq:blob-dim}. The number of possible diagrams 
on the right, denoted $BL_{n-k}$ above, is the dimension of 
$\setstuff{Bl}_{g,n-k}^{\dpar,\bpar}(\cpar)$, 
and we get the corresponding 
number in \eqref{eq:blob-dim}. The remaining number is 
the number of possible blob placements on the 
strand with the leftmost point, see \fullref{lemma:blob-onestrand}.
\end{proof}

\begin{example}
For $g=0$ one obtains that 
$\dim_{\KK}\setstuff{Bl}_{0,n}^{\dpar,\bpar}(\cpar)=C_{n}$, which is 
of course expec\-ted. For $g=1$ one can solve 
the recursion in \eqref{eq:blob-dim} and obtains 
$\dim_{\KK}\setstuff{Bl}_{1,n}^{\dpar,\bpar}(\cpar)=\binom{2n}{n}$, 
the dimension of the blob algebra, {\cf} \cite[Lemma 5.7]{Gr-gen-tl-algebra}.
\end{example}

Regarding cellular structures, the same strategy as for 
the handlebody Temperley--Lieb algebra from 
\fullref{subsection:handlebody-tl}
works. Precisely, the $D$ part is allowed to have caps 
with blobs and through strands without blobs, 
the $U$ part is allowed to have cups 
with blobs and through strands without blobs, and the middle 
has blobs on through strands.
The picture to keep in mind is
\begin{gather*}
\begin{tikzpicture}[anchorbase,scale=0.7,tinynodes]
\draw[usual] (1,1.5) to[out=270,in=180] (1.25,1.25) to[out=0,in=270] (1.5,1.5);
\draw[usual,blobbed={0.25}{u}{spinach}] (2,1.5) to[out=270,in=180] 
(2.25,1.25) to[out=0,in=270] (2.5,1.5);
\draw[usual,blobbed={0.25}{u}{spinach}] (3,1.5) to[out=270,in=180] 
(3.25,1.25) to[out=0,in=270] (3.5,1.5);
\draw[usual,rblobbed={0.33}{w}{orchid},rblobbed={0.66}{v}{tomato}] 
(3,0) to[out=90,in=270] (4,1.5);
\draw[usual,blobbed={0.1}{v}{tomato},blobbed={0.3}{u}{spinach}] (1,0) 
to[out=90,in=180] (1.75,0.75) to[out=0,in=90] (2.5,0);
\draw[usual] (1.5,0) to[out=90,in=180] (1.75,0.25) to[out=0,in=90] (2,0);
\draw[usual] (3.5,0) to[out=90,in=180] (3.75,0.25) to[out=0,in=90] (4,0);
\end{tikzpicture}
,\quad
\begin{aligned}
\begin{tikzpicture}[anchorbase,scale=1]
\draw[mor] (0,1) to (0.25,0.5) to (0.75,0.5) to (1,1) to (0,1);
\node at (0.5,0.75){$U$};
\end{tikzpicture}
&=
\begin{tikzpicture}[anchorbase,scale=0.7,tinynodes]
\draw[usual] (1,0) to[out=270,in=180] (1.25,-0.25) to[out=0,in=270] (1.5,0);
\draw[usual,blobbed={0.25}{u}{spinach}] (2,0) 
to[out=270,in=180] (2.25,-0.25) to[out=0,in=270] (2.5,0);
\draw[usual,blobbed={0.25}{u}{spinach}] (3,0) 
to[out=270,in=180] (3.25,-0.25) to[out=0,in=270] (3.5,0);
\draw[usual] (4,0) to[out=270,in=90] (4,-1);
\end{tikzpicture}
,\\
\begin{tikzpicture}[anchorbase,scale=1]
\draw[mor] (0.25,0) to (0.25,0.5) to (0.75,0.5) to (0.75,0) to (0.25,0);
\node at (0.5,0.25){$b$};
\node at (0.8,0.25){\phantom{$b$}};
\end{tikzpicture}
&=
\begin{tikzpicture}[anchorbase,scale=0.7,tinynodes]
\draw[usual,rblobbed={0.3}{w}{orchid},rblobbed={0.7}{v}{tomato}] (1,0) to[out=90,in=270] (1,1);
\end{tikzpicture}
,
\\
\begin{tikzpicture}[anchorbase,scale=1]
\draw[mor] (0,-0.5) to (0.25,0) to (0.75,0) to (1,-0.5) to (0,-0.5);
\node at (0.5,-0.25){$D$};
\end{tikzpicture}
&=
\begin{tikzpicture}[anchorbase,scale=0.7,tinynodes]
\draw[usual,blobbed={0.1}{v}{tomato},blobbed={0.3}{u}{spinach}] (1,0) 
to[out=90,in=180] (1.75,0.75) to[out=0,in=90] (2.5,0);
\draw[usual] (1.5,0) to[out=90,in=180] (1.75,0.25) to[out=0,in=90] (2,0);
\draw[usual] (3,0) to[out=90,in=270] (3,1);
\draw[usual] (3.5,0) to[out=90,in=180] (3.75,0.25) to[out=0,in=90] (4,0);
\end{tikzpicture}
.
\end{aligned}
\end{gather*}
Note that the middle part is $\sand[0]=\KK$ and
\begin{gather*}
\sand=\KK\setstuff{B}_{g,1}^{\dpar,\bpar}=
\KK\langle b_{u}\mid u\in\{1,\dots,g\}\rangle
/\eqref{eq:cyclotomic-blob}.
\end{gather*}
We choose the monomial basis in the $b_{u}$ as our 
sandwich basis.
As the antiinvolution we choose the map
$(\placeholder)^{\star}\colon\setstuff{Bl}_{g,n}^{\dpar,\bpar}(\cpar)
\to\setstuff{Bl}_{g,n}^{\dpar,\bpar}(\cpar)$ that mirrors 
diagrams and fixes all blobs.

\begin{proposition}
The above defines an 
involutive sandwich cell datum for $\setstuff{Bl}_{g,n}^{\dpar,\bpar}(\cpar)$.
\end{proposition}

\begin{proof}
The proof is, {\muta}, as in \fullref{proposition:tl} and omitted.
\end{proof}

The cells look similar as in \fullref{example:cells-tl}.

\begin{theorem}\label{theorem:blob}
Let $\KK$ be a field.
\begin{enumerate}

\item If $\cpar\neq 0$, or $\cpar=0$ and $\lambda\neq 0$ is odd, 
then all $\lambda\in\Lambda$ are apexes. In the remaining case, 
$\cpar=0$ and $\lambda=0$ (this only happens if $n$ is even), all $\lambda\in\Lambda-\{0\}$ are apexes, but $\lambda=0$ is not an apex.

\item The simple $\setstuff{Bl}_{g,n}^{\dpar,\bpar}(\cpar)$-modules of 
apex $\lambda\in\Lambda$ 
are parameterized by simple modules of $\KK\setstuff{B}_{g,1}^{\dpar,\bpar}$.

\item The simple $\setstuff{Bl}_{g,n}^{\dpar,\bpar}(\cpar)$-modules of 
apex $\lambda\in\Lambda$ can be constructed as 
the simple heads of
$\mathrm{Ind}_{\KK\setstuff{B}_{g}^{\dpar,\bpar}}^{\setstuff{Bl}_{g,n}^{\dpar,\bpar}(\cpar)}(K)$, 
where $K$ runs over (the equivalence classes of) 
simple $\KK\setstuff{B}_{g,1}^{\dpar,\bpar}$-modules.

\end{enumerate}
\end{theorem}

\begin{proof}
This can be proven {\ver} as in the previous cases, see {\eg} 
\fullref{theorem:brauer}.
\end{proof}

The sandwiched algebras are identified in 
\autoref{lemma:blob-onestrand}. Thus, \autoref{theorem:blob} 
explicitly gives the desired 
classification of simple $\setstuff{Bl}_{g,n}^{\dpar,\bpar}(\cpar)$-modules.

\begin{example}
The low genus cases of \fullref{theorem:blob} are known 
(for simplicity we ignore the potential exception 
for $\lambda=0$ in the text below):	

\begin{enumerate}

\item For $g=0$ we have $\KK\setstuff{B}_{0,1}^{\dpar,\bpar}=\KK$, so we obtains 
the classical parametrization of the simple modules of 
the Temperley--Lieb algebra by through strands.

\item For $g=1$ we have that $\KK\setstuff{B}_{1,1}^{\dpar,\bpar}$ is 
a finite-dimensional quotient of a polynomial ring in one variable $b$
by an ideal of the form
\begin{gather*}
(b-\bvar_{1})
\dots
(b-\bvar_{\dvar_{1}})=0.
\end{gather*} 
In particular, the number of simple modules associated to 
$\KK\setstuff{B}_{1,1}^{\dpar,\bpar}$ equals the number of distinct 
parameters in $\bpar$. For example, if $\dvar_{1}=2$, then
\begin{gather*}
\bvar_{1}=\bvar_{2}=0\Rightarrow
\KK\setstuff{B}_{1,1}^{\dpar,\bpar}\cong\KK[b]/(b^{2})
\text{ has one simple module},
\\
\bvar_{1}=1,\bvar_{2}=0\Rightarrow
\KK\setstuff{B}_{1,1}^{\dpar,\bpar}\cong\KK[b]/(b-1)b
\text{ has two simple modules}.
\end{gather*}
In the latter case $b^{2}=b$, 
which is the situation studied in \cite{MaSa-blob}.
Thus, we recover the classification of simples of the blob 
algebra from \cite{MaSa-blob}.
\end{enumerate}

\end{example}

\subsection{Topology of handlebody Temperley--Lieb algebras}\label{subsection:handlebody-tl-topology}

Another way of defining the classical Temperley--Lieb algebra
would be as the quotient of the algebra of tangles by circle 
evaluation and the Kauffman skein relation. To discuss 
a handlebody analog let us define \emph{handlebody 
(framed) tangle diagrams of $2n$ 
points of genus $g$}, using handlebody braid 
diagrams as in \fullref{section:braids} and inspired 
by \fullref{proposition:embedding}, as the (isotopy classes of) tangle 
diagrams in $g+n$ strings, which are pure on the first $g$ strings, modulo 
the usual tangle relations. 
The framing is given by a vector field pointing to the $g$th core strand 
from the left. 
The examples to keep in mind are
\begin{gather*}
\begin{tikzpicture}[anchorbase,scale=0.7,tinynodes]
\draw[pole,crosspole] (-0.5,0) to[out=270,in=90] (-0.5,-1.5);
\draw[usual,crossline] (-0.25,-0.5) to[out=270,in=180] (0.25,-0.8) 
to[out=0,in=180] (0.75,-0.8) to[out=0,in=270] (1.5,0);
\draw[pole,crosspole] (0,0) to[out=270,in=90] (0,-1.5);
\draw[pole,crosspole] (0.5,0) to[out=270,in=90] (0.5,-1.5);
\draw[usual,crossline] (1,0) to[out=270,in=90] (-0.25,-0.5);
\draw[usual,crossline] (1,-1.5) to[out=90,in=180] (1.25,-1.25) 
to[out=0,in=90] (1.5,-1.5);
\end{tikzpicture}
,\quad
\jm_{u,i}
=
\begin{tikzpicture}[anchorbase,scale=0.7,tinynodes]
\draw[pole,crosspole] (-0.5,0) to[out=90,in=270] (-0.5,1.5);
\draw[usual,crossline] (2.5,0)node[below,black,yshift=-1pt]{$i$} 
to[out=90,in=270] (-0.25,0.75);
\draw[usual,crossline] (1.5,0) to[out=90,in=270] (1.5,1.5);
\draw[usual,crossline] (2,0) to[out=90,in=270] (2,1.5);
\draw[pole,crosspole] (0.5,0) to[out=90,in=270] (0.5,1.5);
\draw[pole,crosspole] (1,0) to[out=90,in=270] (1,1.5);
\draw[pole,crosspole] (0,0)node[below,black,yshift=-1pt]{$u$} 
to[out=90,in=270] (0,1.5)node[above,black,yshift=-3pt]{$u$};
\draw[usual,crossline] (-0.25,0.75) to[out=90,in=270] 
(2.5,1.5)node[above,black,yshift=-3pt]{$i$};
\end{tikzpicture}
,\quad
\tau_{u}=
\begin{tikzpicture}[anchorbase,scale=0.7,tinynodes]
\draw[pole,crosspole] (-0.5,0) to[out=90,in=270] (-0.5,1.5);
\draw[usual,crossline] (1,0) to[out=90,in=270] (-0.25,0.75);
\draw[pole,crosspole] (0,0)node[below,black,yshift=-1pt]{$u$} 
to[out=90,in=270] (0,1.5)node[above,black,yshift=-3pt]{$u$};
\draw[pole,crosspole] (0.5,0) to[out=90,in=270] (0.5,1.5);
\draw[usual,crossline] (-0.25,0.75) to[out=90,in=270] (1,1.5);
\end{tikzpicture}
.
\end{gather*}
Recall that we call the second diagram a Jucys--Murphy element.

\begin{remark}
A subtle behavior due to our choice of framing occurs under unknotting 
around a core, as one can see 
from the example below, see also
\cite[Figure 3]{HaOl-actions-tensor-categories}.
\begin{gather}\label{eq:framingblobs}
\scalebox{0.85}{$\begin{tikzpicture}[anchorbase,scale=0.7,tinynodes]
\draw[usual,crossline] (1.5,0) to[out=90,in=270] (1,1.5);
\draw[usual,crossline] (1,0) to[out=90,in=270] (-0.25,0.75);
\draw[pole,crosspole] (0,0) to[out=90,in=270] (0,1.5);
\draw[pole,crosspole] (0.5,0) to[out=90,in=270] (0.5,1.5);
\draw[usual,crossline] (-0.25,0.75) to[out=90,in=270] (1.5,1.5);
\draw[usual,crossline] (1,1.5) to[out=90,in=270] (-0.25,2.25);
\draw[pole,crosspole] (0,1.5) to[out=90,in=270] (0,3);
\draw[pole,crosspole] (0.5,1.5) to[out=90,in=270] (0.5,3);
\draw[usual,crossline] (-0.25,2.25) to[out=90,in=90] (1.5,1.5);
\end{tikzpicture}
=
\begin{tikzpicture}[anchorbase,scale=0.7,tinynodes]
\draw[pole,crosspole] (0.5,0) to[out=90,in=270] (0.5,3);
\draw[pole,crosspole] (0,0) to[out=90,in=270] (0,3);
\draw[pole,crosspole] (0.5,1.5) to[out=90,in=270] (0.5,3);
\draw[usual,crossline] (1,0) to[out=90,in=270] (1.5,0.75);
\draw[usual,crossline] (1.5,0) to[out=90,in=270] (1,0.75);
\draw[usual,crossline] (1,0.75) to[out=90,in=180] (1.25,1) to[out=0,in=90] (1.5,0.75);
\end{tikzpicture}$}
.
\end{gather}
Let us also mention that these diagrams satisfy the classical 
Reidemeister relations and other types of 
isotopy relations, {\eg}
\begin{gather*}
\scalebox{0.85}{$\begin{tikzpicture}[anchorbase,scale=0.7,tinynodes]
\draw[usual,crossline] (0,0) to[out=270,in=180] (0.25,-0.25) 
to[out=0,in=270] (0.5,0) to[out=90,in=270] (1,0.75) to (1,2);
\draw[usual,crossline] (1,-0.75) to[out=90,in=270] (1,0) 
to[out=90,in=270] (0.5,0.75) to (0.5,1.5) to[out=90,in=0] (0.25,1.75) to[out=180,in=90] (0,1.5);
\draw[usual,crossline] (0,0) to[out=90,in=270] (-1.75,0.75);
\draw[pole,crosspole] (-2,-0.75) to[out=90,in=270] (-2,2);
\draw[pole,crosspole] (-1.5,-0.75) to[out=90,in=270] (-1.5,2);
\draw[pole,crosspole] (-1,-0.75) to[out=90,in=270] (-1,2);
\draw[pole,crosspole] (-0.5,-0.75) to[out=90,in=270] (-0.5,2);
\draw[usual,crossline] (-1.75,0.75) to[out=90,in=270] (0,1.5);
\end{tikzpicture}
=
\begin{tikzpicture}[anchorbase,scale=0.7,tinynodes]
\draw[usual,crossline] (-1.75,0.75) to[out=90,in=270] (0,2);
\draw[usual,crossline] (0,-0.75) to[out=90,in=270] (-1.75,0.75);
\draw[pole,crosspole] (-2,-0.75) to[out=90,in=270] (-2,2);
\draw[pole,crosspole] (-1.5,-0.75) to[out=90,in=270] (-1.5,2);
\draw[pole,crosspole] (-1,-0.75) to[out=90,in=270] (-1,2);
\draw[pole,crosspole] (-0.5,-0.75) to[out=90,in=270] (-0.5,2);
\draw[usual,crossline] (0,-0.75) to[out=90,in=270] (-1.75,0.75);
\end{tikzpicture}$}
.
\end{gather*}
\end{remark}

As before we fix circle 
evaluations $\cpar=(\cvar_{\gamma})_{\gamma}$, one for any isotopy 
class of circles $\gamma$. 
We also fix any invertible 
element $\qvar\in\KK$ that has a square root in $\KK$.

\begin{definition}
The evaluation of a circle $\gamma$ is defined to be the removal 
of a closed component, contributing a factor 
$\cvar_{\gamma}$. We call this \emph{circle evaluation}.
\end{definition}

\begin{definition}\label{definition:handlebody-tl-second}
We let the \emph{topological handlebody Temperley--Lieb algebra} 
(in $n$ strands and of genus $g$) 
$\setstuff{TL}_{g,n}^{\mathrm{top}}(\cpar)$ be the
algebra whose underlying free $\KK$-module is the
$\KK$-linear span of all 
handlebody tangle diagrams of $2n$ points of genus $g$, and
with multiplication given by concatenation of diagrams modulo 
circle evaluation and \eqref{eq:kauffman}.
\end{definition}

Note that $\setstuff{TL}_{g,n}(\cpar)$ and 
$\setstuff{TL}_{g,n}^{\mathrm{top}}(\cpar)$ are very different algebras. 
For example, as we have seen in \eqref{eq:order-two}, 
the coils are of finite order 
in $\setstuff{TL}_{g,n}^{\mathrm{top}}(\cpar)$ but of infinite 
order in $\setstuff{TL}_{g,n}(\cpar)$.

\begin{remark}
\leavevmode

\begin{enumerate}

\item For $g=0$ \fullref{definition:handlebody-tl-second} 
is a classical definition, while the 
case $g=1$ is related by Schur--Weyl duality 
to Verma modules, see \cite{IoLeZh-verma-schur-weyl}.

\item Without evaluation of circles the case 
$g>2$ can also be found in 
\cite{IoLeZh-verma-schur-weyl} under the name multi-polar, 
but not much appears to be known.

\end{enumerate}
\end{remark}

\begin{lemma}\label{lemma:tl-span-topological}
The algebra 
$\setstuff{TL}_{g,n}^{\mathrm{top}}(\cpar)$ is an associative, unital 
algebra with a $\KK$-spanning set given by all crossingless 
matchings of $2n$ points of genus $g$ where each coil is positive and appears at most of order $1$.
\end{lemma}

\begin{proof}
That the claimed set is spanning is clear by circle evaluation and 
\eqref{eq:order-two}.
\end{proof}

Recall that $C_{n}$ denotes the $n$th Catalan 
number, which is also the dimension of the 
genus $g=0$ Temperley--Lieb algebra.

\begin{definition}
We call parameters $\cpar=(\cvar_{\gamma})_{\gamma}$ such that 
$\dim_{\KK}\setstuff{TL}_{g,n}^{\mathrm{top}}(\cpar)\geq C_{n}$
\emph{weakly admissible}.
\end{definition}

With \fullref{remark:non-topological} in mind, the 
following is in contrast to \fullref{lemma:tl-basis} 
a non-trivial result. To state it let 
$[k]_{\qvar}=\qvar^{k-1}+\dots+\qvar^{-k+1}$ denote 
the usual quantum number for $k\in\N$, and
note that
for the choice $\cpar=(-[2]_{\qvar})$
the algebra $\setstuff{TL}_{G}(-[2]_{\qvar})
=\setstuff{TL}_{0,G}^{\mathrm{top}}(\cpar)$ 
is the classical Temperley--Lieb algebra on $G$ strands 
whose circle evaluates to $-[2]_{\qvar}$.

\begin{lemma}\label{lemma:handlebody-tl-second}
Let $\KK$ be a field.
For any $\qvar\in\KK$ and $G\geq g$ such that 
$[k]_{\qvar}\neq 0$ 
for all $k<G+1$ there exists 
a set $\cpar^{G}$ and an algebra homomorphisms 
(explicitly given in the proof below)
\begin{gather*}
\iota^{G}\colon\setstuff{TL}_{g,n}^{\mathrm{top}}(\cpar^{G})
\to\setstuff{TL}_{G+n}\big(-[2]_{\qvar}\big).
\end{gather*}
Moreover, the parameters $\cpar^{G}$ are weakly admissible. 
\end{lemma}

Note that the assumptions in \fullref{lemma:handlebody-tl-second} 
are always satisfied for {\eg} $\KK=\mathbb{C}(\qvar^{1/2})$ 
and $\qvar$ being the generic variable.

\begin{proof}
For the algebra $\setstuff{TL}_{G+n}\big(-[2]_{\qvar}\big)$ 
there exists a Jones--Wenzl idempotent $\obstuff{e}_{G}$ 
on $G$ strands, see {\eg} 
\cite[Chapter 3]{KaLi-TL-recoupling} 
(strictly speaking \cite[Chapter 3]{KaLi-TL-recoupling} uses 
$\KK=\mathbb{C}(\qvar)$ but under the assumption that $[k]\neq 0$ 
for all $k<G+1$ the whole discussion therein goes though as long as 
the number of strands is less than $G+1$), 
and we can consider the idempotent truncation 
$\obstuff{e}_{G}\setstuff{TL}_{G}\big(-[2]_{\qvar}\big)\obstuff{e}_{G}$. 
By the properties of $\obstuff{e}_{G}$ we have 
$\obstuff{e}_{G}\setstuff{TL}_{G}\big(-[2]_{\qvar}\big)\obstuff{e}_{G}\cong
\KK\{\obstuff{e}_{G}\}$, 
so for any circle $\gamma$ in $\obstuff{e}_{G}\setstuff{TL}_{G}(-[2]_{\qvar})\obstuff{e}_{G}$
we get a scalar $\cvar_{\gamma}^{G}$. 
Fix some partitioning of $G=M_{1}+\dots+M_{g}$ for $M_{i}\in\N_{>0}$.
With the choices $\cvar_{\gamma}^{G}$ it is easy to see that 
blowing the $i$th core strand into $M_{i}$ parallel strands 
and flanking them then with $\obstuff{e}_{G}$
defines an algebra homomorphism
\begin{gather*}
\iota\colon\setstuff{TL}_{g,n}^{\mathrm{top}}(\cpar^{G})
\to(\obstuff{e}_{G}\hcirc\idmor_{n})
\setstuff{TL}_{G+n}\big(-[2]_{\qvar}\big)(\obstuff{e}_{G}\hcirc\idmor_{n})
\hookrightarrow\setstuff{TL}_{G+n}\big(-[2]_{\qvar}\big).
\end{gather*}
The pictures illustrating the above constructions are
\begin{gather}\label{eq:flanking}
\begin{tikzpicture}[anchorbase,scale=0.7,tinynodes]
\draw[pole,crosspole] (-0.5,0) to[out=90,in=270] (-0.5,1.5);
\draw[usual,crossline] (1,0) to[out=90,in=270] (-0.25,0.75);
\draw[pole,crosspole] (0,0) to[out=90,in=270] (0,1.5);
\draw[pole,crosspole] (0.5,0) to[out=90,in=270] (0.5,1.5);
\draw[usual,crossline] (-0.25,0.75) to[out=90,in=270] 
(1,1.5);
\end{tikzpicture}
\mapsto
\begin{tikzpicture}[anchorbase,scale=0.7,tinynodes]
\draw[usual,crossline] (-0.43,0) to[out=90,in=270] (-0.43,1.5);
\draw[usual,crossline] (-0.57,0) to[out=90,in=270] (-0.57,1.5);
\draw[usual,crossline] (1,0) to[out=90,in=270] (-0.25,0.75);
\draw[usual,crossline] (-0.07,0) to[out=90,in=270] (-0.07,1.5);
\draw[usual,crossline] (0.07,0) to[out=90,in=270] (0.07,1.5);
\draw[usual,crossline] (0.43,0) to[out=90,in=270] (0.43,1.5);
\draw[usual,crossline] (0.57,0) to[out=90,in=270] (0.57,1.5);
\draw[usual,crossline] (-0.25,0.75) to[out=90,in=270] 
(1,1.5);
\end{tikzpicture}
\mapsto
\begin{tikzpicture}[anchorbase,scale=0.7,tinynodes]
\draw[usual,crossline] (-0.43,0) to[out=90,in=270] (-0.43,1.5);
\draw[usual,crossline] (-0.57,0) to[out=90,in=270] (-0.57,1.5);
\draw[usual,crossline] (1,0) to[out=90,in=270] (-0.25,0.75);
\draw[usual,crossline] (-0.07,0) to[out=90,in=270] (-0.07,1.5);
\draw[usual,crossline] (0.07,0) to[out=90,in=270] (0.07,1.5);
\draw[usual,crossline] (0.43,0) to[out=90,in=270] (0.43,1.5);
\draw[usual,crossline] (0.57,0) to[out=90,in=270] (0.57,1.5);
\draw[usual,crossline] (-0.25,0.75) to[out=90,in=270] 
(1,1.5);
\draw[mor] (-0.65,0) rectangle (0.65,-0.25);
\draw[mor] (-0.65,1.5) rectangle (0.65,1.75);
\end{tikzpicture}
,\quad
\begin{tikzpicture}[anchorbase,scale=0.7,tinynodes]
\draw[pole,crosspole] (0.5,0) to[out=90,in=270] (0.5,1.5);
\draw[usual,crossline] (1.25,0.75) to[out=270,in=270] (-0.25,0.75);
\draw[pole,crosspole] (0,0) to[out=90,in=270] (0,1.5);
\draw[pole,crosspole] (1,0) to[out=90,in=270] (1,1.5);
\draw[usual,crossline] (-0.25,0.75) to[out=90,in=90] (1.25,0.75);
\end{tikzpicture}
\mapsto
\begin{tikzpicture}[anchorbase,scale=0.7,tinynodes]
\draw[usual,crossline] (0.43,0) to[out=90,in=270] (0.43,1.5);
\draw[usual,crossline] (0.57,0) to[out=90,in=270] (0.57,1.5);
\draw[usual,crossline] (1.25,0.75) to[out=270,in=270] (-0.25,0.75);
\draw[usual,crossline] (-0.07,0) to[out=90,in=270] (-0.07,1.5);
\draw[usual,crossline] (0.07,0) to[out=90,in=270] (0.07,1.5);
\draw[usual,crossline] (0.93,0) to[out=90,in=270] (0.93,1.5);
\draw[usual,crossline] (1.07,0) to[out=90,in=270] (1.07,1.5);
\draw[usual,crossline] (-0.25,0.75) to[out=90,in=90] (1.25,0.75);
\end{tikzpicture}
\mapsto
\begin{tikzpicture}[anchorbase,scale=0.7,tinynodes]
\draw[usual,crossline] (0.43,0) to[out=90,in=270] (0.43,1.5);
\draw[usual,crossline] (0.57,0) to[out=90,in=270] (0.57,1.5);
\draw[usual,crossline] (1.25,0.75) to[out=270,in=270] (-0.25,0.75);
\draw[usual,crossline] (-0.07,0) to[out=90,in=270] (-0.07,1.5);
\draw[usual,crossline] (0.07,0) to[out=90,in=270] (0.07,1.5);
\draw[usual,crossline] (0.93,0) to[out=90,in=270] (0.93,1.5);
\draw[usual,crossline] (1.07,0) to[out=90,in=270] (1.07,1.5);
\draw[usual,crossline] (-0.25,0.75) to[out=90,in=90] (1.25,0.75);
\draw[mor] (-0.15,0) rectangle (1.15,-0.25);
\draw[mor] (-0.15,1.5) rectangle (1.15,1.75);
\end{tikzpicture}
=
\cvar_{\gamma}
\cdot
\begin{tikzpicture}[anchorbase,scale=0.7,tinynodes]
\draw[usual,crossline] (0.43,0) to[out=90,in=270] (0.43,1.5);
\draw[usual,crossline] (0.57,0) to[out=90,in=270] (0.57,1.5);
\draw[usual,crossline] (-0.07,0) to[out=90,in=270] (-0.07,1.5);
\draw[usual,crossline] (0.07,0) to[out=90,in=270] (0.07,1.5);
\draw[usual,crossline] (0.93,0) to[out=90,in=270] (0.93,1.5);
\draw[usual,crossline] (1.07,0) to[out=90,in=270] (1.07,1.5);
\draw[mor] (-0.15,0) rectangle (1.15,-0.25);
\draw[mor] (-0.15,1.5) rectangle (1.15,1.75);
\end{tikzpicture}
,
\end{gather}
where the boxes represent the Jones--Wenzl idempotent $\obstuff{e}_{G}$.
To see the remaining statement let $\mathrm{Sym}^{G}_{\qvar}(\KK^{2})$ 
denote the $G$th quantum symmetric power of $\KK^{2}$.
We use $\iota$ 
and quantum symmetric Howe duality 
\cite[Theorem 2.6 (1) and (2)]{RoTu-symmetric-howe} 
(which work over $\Z[\qvar,\qvar^{-1}]$) in combination with 
\cite[Proposition 2.14]{RoTu-symmetric-howe} (which works 
under the assumptions on $[k]_{\qvar}$ 
to the point needed) to define 
a representation of $\setstuff{TL}_{g,n}^{\mathrm{top}}(\cpar^{G})$ 
on $\mathrm{Sym}^{G}_{\qvar}(\KK^{2})\hcirc\KK^{2}\hcirc\dots\hcirc\KK^{2}$.
Using this representation one can check that 
$\setstuff{TL}_{g,n}^{\mathrm{top}}(\cpar^{G})$ can not be trivial since, 
for example, $\setstuff{TL}_{0,n}^{\mathrm{top}}(\cpar^{G})$ 
acts faithfully by classical Schur--Weyl duality.
\end{proof}

\begin{remark}
Our proof of \fullref{lemma:handlebody-tl-second} is directly 
inspired by \cite{RoTu-homflypt-typea}: As pointed out in that paper, 
the handlebody closing \eqref{eq:handlebody-closing} can
be interpreted, in the appropriate 
algebraic framework, by putting an idempotent on bottom and top.
Moreover and alternatively to the usage of 
(growing) symmetric powers, one might want to associate Verma modules 
to the core strands as in the $g=1$ case, see \cite{IoLeZh-verma-schur-weyl}.
This approach has however the disadvantage that we do not know analogs 
of Jones--Wenzl idempotents which one could use to get
coherent circle evaluations.
\end{remark}

\begin{remark}
For $\qvar\in\KK$ being a root of unity and $\KK$ being a field
there are still Jones--Wenzl 
type idempotents, see {\eg} \cite{MaSp-lp-jw} or \cite{SuTuWeZh-mixedsl2} 
for general constructions of such idempotents. However their 
endomorphism spaces are no longer trivial, so we can not use the 
argument in the proof of \fullref{lemma:handlebody-tl-second}.
\end{remark}

We do not know any explicit condition to check 
whether parameter choices are \emph{admissible}, meaning 
that we do not know a 
``generic'' basis of $\setstuff{TL}_{g,n}^{\mathrm{top}}(\cpar)$.
We leave it to future work to find such a basis.

\subsection{Topology of handlebody blob algebras}\label{subsection:handlebody-blob-topology}

The purpose of this section is to explain a different 
presentation for $\setstuff{TL}_{g,n}^{\mathrm{top}}(\cpar)$.

\begin{remark}
We will also use a similar construction as 
presented in this section for 
cyclotomic handlebody Brauer and BMW algebras later on, 
so we decided to spell it out here despite 
\fullref{lemma:same}.
However, while the topological handlebody 
blob algebra is the same as the topological handlebody 
Temperley--Lieb algebra, this phenomena is no longer true for 
handlebody Brauer and BMW algebras.
\end{remark}

\begin{definition}\label{definition:handlebody-blob-second}
Retain the notation and conventions from \fullref{definition:handlebody-tl-second}.	
We let the \emph{topological handlebody blob algebra} 
(in $n$ strands and of genus $g$) 
$\setstuff{Bl}_{g,n}^{\mathrm{top}}(\cpar)$
be the subalgebra of 
$\setstuff{TL}_{g,n}^{\mathrm{top}}(\cpar)$ 
spanned by all elements containing only positive coils.
\end{definition}

\begin{lemma}\label{lemma:same}
We have $\setstuff{Bl}_{g,n}^{\mathrm{top}}(\cpar)
=\setstuff{TL}_{g,n}^{\mathrm{top}}(\cpar)$.
\end{lemma}

\begin{proof}
Clear by \eqref{eq:order-two}.
\end{proof}

To give a different diagrammatic presentation  
we need the notion of \emph{blobbed presentations}, that will 
spell out now.
The construction of this works {\muta} 
as in \fullref{subsection:handlebody-blob} 
(so we will be brief)
with one main difference: 
we allow crossings between the usual strands.
This gives us the notion of 
\emph{framed tangled blob diagrams of $2n$ points of genus $g$}, 
where blobs move freely around its strand, 
but always keep being reachable from the left, and do not pass 
each other, {\cf} \eqref{eq:not-defined}. 

Examples of such tangled blob diagrams and how they relate 
to the diagrammatic used in 
\fullref{subsection:handlebody-tl-topology} are
\begin{gather*}
\begin{tikzpicture}[anchorbase,scale=0.7,tinynodes]
\draw[pole,crosspole] (-0.5,0) to[out=270,in=90] (-0.5,-1.5);
\draw[usual,crossline] (-0.25,-0.5) to[out=270,in=180] (0.25,-0.8) 
to[out=0,in=180] (0.75,-0.8) to[out=0,in=270] (1.5,0);
\draw[pole,crosspole] (0,0)node[above,black,yshift=-3pt]{$u$} 
to[out=270,in=90] (0,-1.5)node[below,black,yshift=-1pt]{$u$};
\draw[pole,crosspole] (0.5,0) to[out=270,in=90] (0.5,-1.5);
\draw[usual,crossline] (1,0) to[out=270,in=90] (-0.25,-0.5);
\draw[usual,crossline] (1,-1.5) to[out=90,in=180] (1.25,-1.25) 
to[out=0,in=90] (1.5,-1.5);
\end{tikzpicture}
\rightsquigarrow
\begin{tikzpicture}[anchorbase,scale=0.7,tinynodes]
\draw[usual,crossline,blobbed={0.25}{u}{spinach}] (1,0) to[out=270,in=180] (1.25,-0.25) to[out=0,in=270] (1.5,0);
\draw[usual,crossline] (1,-1.5) to[out=90,in=180] (1.25,-1.25) to[out=0,in=90] (1.5,-1.5);
\end{tikzpicture}
,\quad
\jm_{u,i}
=
\begin{tikzpicture}[anchorbase,scale=0.7,tinynodes]
\draw[pole,crosspole] (-0.5,0) to[out=90,in=270] (-0.5,1.5);
\draw[usual,crossline] (2.5,0)node[below,black,yshift=-1pt]{$i$} to[out=90,in=270] (-0.25,0.75);
\draw[usual,crossline] (1.5,0) to[out=90,in=270] (1.5,1.5);
\draw[usual,crossline] (2,0) to[out=90,in=270] (2,1.5);
\draw[pole,crosspole] (0.5,0) to[out=90,in=270] (0.5,1.5);
\draw[pole,crosspole] (1,0) to[out=90,in=270] (1,1.5);
\draw[pole,crosspole] (0,0)node[below,black,yshift=-1pt]{$u$} 
to[out=90,in=270] (0,1.5)node[above,black,yshift=-3pt]{$u$};
\draw[usual,crossline] (-0.25,0.75) to[out=90,in=270] (2.5,1.5)node[above,black,yshift=-3pt]{$i$};
\end{tikzpicture}
\rightsquigarrow
b_{u,i}
=
\begin{tikzpicture}[anchorbase,scale=0.7,tinynodes]
\draw[pole,crosspole,white] (2,0)node[below,white,yshift=-1pt]{$u$} 
to[out=90,in=270] (2,1.5)node[above,white,yshift=-3pt]{$u$};
\draw[usual] (2.5,0)node[below,black,yshift=-1pt]{$i$} to[out=90,in=270] (1.25,0.75);
\draw[usual,crossline] (1.5,0) to[out=90,in=270] (1.5,1.5);
\draw[usual,crossline] (2,0) to[out=90,in=270] (2,1.5);
\draw[usual,crossline] (1.25,0.75) to[out=90,in=270] (2.5,1.5)node[above,black,yshift=-3pt]{$i$};
\draw[usual,blobbed={1}{u}{spinach}] (1.25,0.65) to (1.25,0.85);
\end{tikzpicture}
,
\end{gather*}
and
\begin{gather*}
\tau_{u}=
\begin{tikzpicture}[anchorbase,scale=0.7,tinynodes]
\draw[pole,crosspole] (-0.5,0) to[out=90,in=270] (-0.5,1.5);
\draw[usual,crossline] (1,0) to[out=90,in=270] (-0.25,0.75);
\draw[pole,crosspole] (0,0)node[below,black,yshift=-1pt]{$u$} 
to[out=90,in=270] (0,1.5)node[above,black,yshift=-3pt]{$u$};
\draw[pole,crosspole] (0.5,0) to[out=90,in=270] (0.5,1.5);
\draw[usual,crossline] (-0.25,0.75) to[out=90,in=270] (1,1.5);
\end{tikzpicture}
\rightsquigarrow
b_{u}=
\begin{tikzpicture}[anchorbase,scale=0.7,tinynodes]
\draw[pole,crosspole,white] (0,0)node[below,white,yshift=-1pt]{$u$} 
to[out=90,in=270] (0,1.5)node[above,white,yshift=-3pt]{$u$};
\draw[usual,crossline,blobbed={0.5}{u}{spinach}] (0,0) to[out=90,in=270] (0,1.5);
\end{tikzpicture}
.
\end{gather*}
where we use the same notation for the 
blob versions of the coils as in 
\fullref{subsection:handlebody-blob}.

Note that we can introduce blobs on any possible strand by
\begin{gather}\label{eq:blob-right}
\jm_{u,i}
=
\begin{tikzpicture}[anchorbase,scale=0.7,tinynodes]
\draw[pole,crosspole] (-0.5,0) to[out=90,in=270] (-0.5,1.5);
\draw[usual,crossline] (2.5,0)node[below,black,yshift=-1pt]{$i$} to[out=90,in=270] (-0.25,0.75);
\draw[usual,crossline] (1.5,0) to[out=90,in=270] (1.5,1.5);
\draw[usual,crossline] (2,0) to[out=90,in=270] (2,1.5);
\draw[pole,crosspole] (0.5,0) to[out=90,in=270] (0.5,1.5);
\draw[pole,crosspole] (1,0) to[out=90,in=270] (1,1.5);
\draw[pole,crosspole] (0,0)node[below,black,yshift=-1pt]{$u$} 
to[out=90,in=270] (0,1.5)node[above,black,yshift=-3pt]{$u$};
\draw[usual,crossline] (-0.25,0.75) to[out=90,in=270] (2.5,1.5)node[above,black,yshift=-3pt]{$i$};
\end{tikzpicture}
\rightsquigarrow
b_{u,i}
=
\begin{tikzpicture}[anchorbase,scale=0.7,tinynodes]
\draw[pole,crosspole,white] (2,0)node[below,white,yshift=-1pt]{$u$} 
to[out=90,in=270] (2,1.5)node[above,white,yshift=-3pt]{$u$};
\draw[usual] (2.5,0)node[below,black,yshift=-1pt]{$i$} to[out=90,in=270] (1.25,0.75);
\draw[usual,crossline] (1.5,0) to[out=90,in=270] (1.5,1.5);
\draw[usual,crossline] (2,0) to[out=90,in=270] (2,1.5);
\draw[usual,crossline] (1.25,0.75) to[out=90,in=270] (2.5,1.5)node[above,black,yshift=-3pt]{$i$};
\draw[usual,blobbed={1}{u}{spinach}] (1.25,0.65) to (1.25,0.85);
\end{tikzpicture}
=
\begin{tikzpicture}[anchorbase,scale=0.7,tinynodes]
\draw[pole,crosspole,white] (2,0)node[below,white,yshift=-1pt]{$u$} 
to[out=90,in=270] (2,1.5)node[above,white,yshift=-3pt]{$u$};
\draw[usual,rblobbed={0.5}{u}{spinach}] (2.5,0)node[below,black,yshift=-1pt]{$i$} 
to[out=90,in=270] (2.5,1.5)node[above,black,yshift=-3pt]{$i$};
\draw[usual,crossline] (1.5,0) to[out=90,in=270] (1.5,1.5);
\draw[usual,crossline] (2,0) to[out=90,in=270] (2,1.5);
\end{tikzpicture}
.
\end{gather}	
The latter is however just a shorthand notation
which is not quite topological in nature.
The point to is that the relations among 
Jucys--Murphy elements in \eqref{eq:jm-relations}
give relations among blobs, but not all
of them are topological manipulations of blobs.
Precisely, we have:

\begin{lemma}\label{lemma:blob-rels}
Blobs satisfy the following relations.
\begin{align}\label{eq:blobslides}
\begin{tikzpicture}[anchorbase,scale=0.7,tinynodes]
\draw[usual,crossline] (0,0) to[out=90,in=270] (0.5,0.75);
\draw[usual,crossline,rblobbed={0.15}{u}{spinach}] (0.5,0) to[out=90,in=270] (0,0.75);
\end{tikzpicture}
\hspace{-0.2cm}&=\hspace{-0.2cm}
\begin{tikzpicture}[anchorbase,scale=0.7,tinynodes]
\draw[usual,crossline,blobbed={0.85}{u}{spinach}] (0.5,0) to[out=90,in=270] (0,0.75);
\draw[usual,crossline] (0,0) to[out=90,in=270] (0.5,0.75);
\end{tikzpicture}
,\quad
\begin{tikzpicture}[anchorbase,scale=0.7,tinynodes]
\draw[usual,crossline] (0.5,0) to[out=90,in=270] (0,0.75);
\draw[usual,crossline,blobbed={0.15}{u}{spinach}] (0,0) to[out=90,in=270] (0.5,0.75);
\end{tikzpicture}
=
\begin{tikzpicture}[anchorbase,scale=0.7,tinynodes]
\draw[usual,crossline,rblobbed={0.85}{u}{spinach}] (0,0) to[out=90,in=270] (0.5,0.75);
\draw[usual,crossline] (0.5,0) to[out=90,in=270] (0,0.75);
\end{tikzpicture}
.
\\
\label{eq:dots-capscups}
\begin{tikzpicture}[anchorbase,scale=0.7,tinynodes]
\draw[usual,crossline,blobbed={0.2}{u}{spinach},rblobbed={0.8}{u}{spinach}] 
(1,0) to[out=90,in=180] (1.25,0.5) to[out=0,in=90] (1.5,0);
\end{tikzpicture}
\hspace{-0.2cm}&=
\begin{tikzpicture}[anchorbase,scale=0.7,tinynodes]
\draw[usual,crossline] (1,0) to[out=90,in=180] (1.25,0.5) to[out=0,in=90] (1.5,0);
\end{tikzpicture}
,\quad
\begin{tikzpicture}[anchorbase,scale=0.7,tinynodes,yscale=-1]
\draw[usual,crossline,blobbed={0.2}{u}{spinach},rblobbed={0.8}{u}{spinach}] 
(1,0) to[out=90,in=180] (1.25,0.5) to[out=0,in=90] (1.5,0);
\end{tikzpicture}
\hspace{-0.2cm}=
\begin{tikzpicture}[anchorbase,scale=0.7,tinynodes,yscale=-1]
\draw[usual,crossline] (1,0) to[out=90,in=180] (1.25,0.5) to[out=0,in=90] (1.5,0);
\end{tikzpicture}
.
\end{align}
\begin{gather}\label{eq:height-blob-switch}
\begin{tikzpicture}[anchorbase,scale=0.7,tinynodes]
\draw[usual,rblobbed={0.66}{u}{spinach}] (2.5,0) to[out=90,in=270] (2.5,1.5);
\draw[usual,blobbed={0.33}{v}{tomato}] (2,0) to[out=90,in=270] (2,1.5);
\end{tikzpicture}
=
\begin{cases}
\begin{tikzpicture}[anchorbase,scale=0.7,tinynodes]
\draw[usual,rblobbed={0.33}{u}{spinach}] (2.5,0) to[out=90,in=270] (2.5,1.5);
\draw[usual,blobbed={0.66}{v}{tomato}] (2,0) to[out=90,in=270] (2,1.5);
\end{tikzpicture}
&\text{if }u\leq v,
\\[2ex]
\begin{tikzpicture}[anchorbase,scale=0.7,tinynodes]
\draw[usual,crossline] (2.5,1.0) to[out=90,in=270] (2.0,1.5);
\draw[usual,crossline] (2.0,1.0) to[out=90,in=270] (2.5,1.5);
\draw[usual,crossline] (2.0,0) to[out=90,in=270] (2.5,0.5);
\draw[usual,crossline] (2.5,0) to[out=90,in=270] (2.0,0.5);
\draw[usual,blobbed={0}{u}{spinach}] (2,0.5) to[out=90,in=270] (2,1.0);
\draw[usual,rblobbed={1}{v}{tomato}] (2.5,0.5) to[out=90,in=270] (2.5,1.0);
\end{tikzpicture}
&\text{if }u\geq v.
\end{cases}
\end{gather}
\begin{gather}\label{eq:more-slides}
\begin{aligned}
\begin{tikzpicture}[anchorbase,scale=0.7,tinynodes]
\draw[usual,crossline] (0.5,0) to[out=90,in=270] (0,0.75);
\draw[usual,crossline,rblobbed={0.85}{u}{spinach}] (0,0) to[out=90,in=270] (0.5,0.75);
\end{tikzpicture}
\hspace{-0.2cm}&=\hspace{-0.2cm}
\begin{tikzpicture}[anchorbase,scale=0.7,tinynodes]
\draw[usual,crossline] (0.5,0) to[out=90,in=270] (0,0.75);
\draw[usual,crossline,blobbed={0.15}{u}{spinach}] (0,0) to[out=90,in=270] (0.5,0.75);
\end{tikzpicture}
+(\qvar^{1/2}-\qvar^{-1/2})\cdot
\left(
\begin{tikzpicture}[anchorbase,scale=0.7,tinynodes]
\draw[usual,crossline] (0,0) to (0,0.75);
\draw[usual,crossline,rblobbed={0.5}{u}{spinach}] (0.5,0) to (0.5,0.75);
\end{tikzpicture}
-
\begin{tikzpicture}[anchorbase,scale=0.7,tinynodes]
\draw[usual,crossline] (0,0) to[out=90,in=180] (0.25,0.25) 
to[out=0,in=90] (0.5,0);
\draw[usual,crossline,rblobbed={0.75}{u}{spinach}] (0,0.75) 
to[out=270,in=180] (0.25,0.5) to[out=0,in=270] (0.5,0.75);
\end{tikzpicture}
\hspace{-0.2cm}\right)
,\\
\begin{tikzpicture}[anchorbase,scale=0.7,tinynodes]
\draw[usual,crossline,rblobbed={0.12}{u}{spinach}] (0.5,0) to[out=90,in=270] (0,0.75);
\draw[usual,crossline] (0,0) to[out=90,in=270] (0.5,0.75);
\end{tikzpicture}
\hspace{-0.2cm}&=\hspace{-0.2cm}
\begin{tikzpicture}[anchorbase,scale=0.7,tinynodes]
\draw[usual,crossline,blobbed={0.87}{u}{spinach}] (0.5,0) to[out=90,in=270] (0,0.75);
\draw[usual,crossline] (0,0) to[out=90,in=270] (0.5,0.75);
\end{tikzpicture}
+(\qvar^{1/2}-\qvar^{-1/2})\cdot
\left(
\begin{tikzpicture}[anchorbase,scale=0.7,tinynodes]
\draw[usual,crossline] (0,0) to (0,0.75);
\draw[usual,crossline,rblobbed={0.5}{u}{spinach}] (0.5,0) to (0.5,0.75);
\end{tikzpicture}
-
\begin{tikzpicture}[anchorbase,scale=0.7,tinynodes]
\draw[usual,crossline,rblobbed={0.75}{u}{spinach}] (0,0) 
to[out=90,in=180] (0.25,0.25) to[out=0,in=90] (0.5,0);
\draw[usual,crossline] (0,0.75) to[out=270,in=180] (0.25,0.5) 
to[out=0,in=270] (0.5,0.75);
\end{tikzpicture}
\hspace{-0.2cm}\right)
,
\\
\begin{tikzpicture}[anchorbase,scale=0.7,tinynodes]
\draw[usual,crossline,blobbed={0.15}{u}{spinach}] (0,0) to[out=90,in=270] (0.5,0.75);
\draw[usual,crossline] (0.5,0) to[out=90,in=270] (0,0.75);
\end{tikzpicture}
&=
\begin{tikzpicture}[anchorbase,scale=0.7,tinynodes]
\draw[usual,crossline,rblobbed={0.85}{u}{spinach}] (0,0) to[out=90,in=270] (0.5,0.75);
\draw[usual,crossline] (0.5,0) to[out=90,in=270] (0,0.75);
\end{tikzpicture}
\hspace{-0.2cm}+(\qvar^{1/2}-\qvar^{-1/2})\cdot
\left(\hspace{-0.2cm}
\begin{tikzpicture}[anchorbase,scale=0.7,tinynodes]
\draw[usual,crossline,blobbed={0.5}{u}{spinach}] (0,0) to (0,0.75);
\draw[usual,crossline] (0.5,0) to (0.5,0.75);
\end{tikzpicture}
-
\begin{tikzpicture}[anchorbase,scale=0.7,tinynodes]
\draw[usual,crossline,blobbed={0.25}{u}{spinach}] (0,0) 
to[out=90,in=180] (0.25,0.25) to[out=0,in=90] (0.5,0);
\draw[usual,crossline] (0,0.75) to[out=270,in=180] (0.25,0.5) 
to[out=0,in=270] (0.5,0.75);
\end{tikzpicture}
\right)
,\\
\begin{tikzpicture}[anchorbase,scale=0.7,tinynodes]
\draw[usual,crossline] (0,0) to[out=90,in=270] (0.5,0.75);
\draw[usual,crossline,blobbed={0.85}{u}{spinach}] (0.5,0) to[out=90,in=270] (0,0.75);
\end{tikzpicture}
&=
\begin{tikzpicture}[anchorbase,scale=0.7,tinynodes]
\draw[usual,crossline] (0,0) to[out=90,in=270] (0.5,0.75);
\draw[usual,crossline,rblobbed={0.15}{u}{spinach}] (0.5,0) to[out=90,in=270] (0,0.75);
\end{tikzpicture}
\hspace{-0.2cm}+(\qvar^{1/2}-\qvar^{-1/2})\cdot
\left(\hspace{-0.2cm}
\begin{tikzpicture}[anchorbase,scale=0.7,tinynodes]
\draw[usual,crossline,blobbed={0.5}{u}{spinach}] (0,0) to (0,0.75);
\draw[usual,crossline] (0.5,0) to (0.5,0.75);
\end{tikzpicture}
-
\begin{tikzpicture}[anchorbase,scale=0.7,tinynodes]
\draw[usual,crossline,] (0,0) to[out=90,in=180] (0.25,0.25) 
to[out=0,in=90] (0.5,0);
\draw[usual,crossline,blobbed={0.25}{u}{spinach}] (0,0.75) 
to[out=270,in=180] (0.25,0.5) to[out=0,in=270] (0.5,0.75);
\end{tikzpicture}
\right)
.                     
\end{aligned}
\end{gather}
\end{lemma}

The relations in \eqref{eq:blobslides} and \eqref{eq:more-slides} 
are called \emph{blob slides}. 
Also note the distorted topology in \fullref{lemma:blob-rels}, 
which is however gets resolved for $\qvar=1$.

\begin{proof}
Relations \eqref{eq:blobslides} are 
a blob version of the equalities
$\beta_{i-1}^{-1}\jm_{u,i}=\jm_{u,i-1}\beta_{i-1}$ and
$\beta_{i}\jm_{u,i}=\jm_{u,i+1}\beta_{i}^{-1}$ 
between Jucys--Murphy elements, {\cf} \fullref{lemma:jm-elements}.

To prove the first relation 
in \eqref{eq:dots-capscups} slide down the blob on the 
right and use \eqref{eq:blob-right} to 
write it as a blob on the strand to the left. 
The rest follows from \eqref{eq:framingblobs}. 
The second relation in \eqref{eq:dots-capscups} 
is proved analogously.

Relation \eqref{eq:height-blob-switch} 
in the case $v\geq u$ is the last 
relation in \fullref{lemma:jm-elements}. 
In the case $v\leq u$ we calculate
\begin{gather*}
\begin{tikzpicture}[anchorbase,scale=0.7,tinynodes]
\draw[usual,rblobbed={0.66}{u}{spinach}] (2.5,0) to[out=90,in=270] (2.5,1.5);
\draw[usual,blobbed={0.33}{v}{tomato}] (2,0) to[out=90,in=270] (2,1.5);
\end{tikzpicture}
=
\begin{tikzpicture}[anchorbase,scale=0.7,tinynodes]
\draw[usual] (2,1) to[out=90,in=270] (2,1.5);
\draw[usual,blobbed={0.58}{u}{spinach},crossline] (2.5,0) 
to (2.5,0.5) to[out=90,in=270] (1.5,1) to[out=90,in=270] (2.5,1.5);
\draw[usual,blobbed={0.4}{v}{tomato},crossline] (2,0) to[out=90,in=270] (2,1);
\end{tikzpicture}
=
\begin{tikzpicture}[anchorbase,scale=0.7,tinynodes]
\draw[usual,rblobbed={0.42}{v}{tomato},crossline] (2,0) 
to[out=90,in=270] (2,1.5);
\draw[usual,blobbed={0.58}{u}{spinach},crossline] (2.5,0) 
to[out=90,in=270] (1.5,0.5) to[out=90,in=270] (1.5,1) to[out=90,in=270] (2.5,1.5);
\end{tikzpicture}
=
\begin{tikzpicture}[anchorbase,scale=0.7,tinynodes]
\draw[usual,rblobbed={0.6}{v}{tomato},crossline] (2,0) 
to[out=90,in=270] (2,1.5);
\draw[usual,blobbed={0.425}{u}{spinach},crossline] (2.5,0) 
to[out=90,in=270] (1.5,0.5) to[out=90,in=270] (1.5,1) to[out=90,in=270] (2.5,1.5);
\end{tikzpicture}
,
\end{gather*}
where we use the case $v\geq u$ and \eqref{eq:blobslides}.

For the final set of relations in \eqref{eq:more-slides}
we combine \eqref{eq:blobslides} with the 
Kauffman skein relation \eqref{eq:kauffman}. 
Precisely, \eqref{eq:blobslides} and \eqref{eq:kauffman}
imply {\eg}
\begin{gather*}
\begin{tikzpicture}[anchorbase,scale=0.7,tinynodes]
\draw[usual,crossline] (0.5,0) to[out=90,in=270] (0,0.75);
\draw[usual,crossline,rblobbed={0.85}{u}{spinach}] (0,0) to[out=90,in=270] (0.5,0.75);
\end{tikzpicture}
\hspace{-0.2cm}=
\qvar^{1/2}
\cdot
\begin{tikzpicture}[anchorbase,scale=0.7,tinynodes]
\draw[usual,crossline] (0,0) to (0,0.75);
\draw[usual,crossline,rblobbed={0.5}{u}{spinach}] (0.5,0) to (0.5,0.75);
\end{tikzpicture}
+\qvar^{-1/2}\cdot
\begin{tikzpicture}[anchorbase,scale=0.7,tinynodes]
\draw[usual,crossline] (0,0) to[out=90,in=180] (0.25,0.25) 
to[out=0,in=90] (0.5,0);
\draw[usual,crossline,rblobbed={0.75}{u}{spinach}] (0,0.75) 
to[out=270,in=180] (0.25,0.5) to[out=0,in=270] (0.5,0.75);
\end{tikzpicture}
,\quad
\begin{tikzpicture}[anchorbase,scale=0.7,tinynodes]
\draw[usual,crossline] (0.5,0) to[out=90,in=270] (0,0.75);
\draw[usual,crossline,blobbed={0.15}{u}{spinach}] (0,0) to[out=90,in=270] (0.5,0.75);
\end{tikzpicture}
=
\begin{tikzpicture}[anchorbase,scale=0.7,tinynodes]
\draw[usual,crossline,rblobbed={0.85}{u}{spinach}] (0,0) to[out=90,in=270] (0.5,0.75);
\draw[usual,crossline] (0.5,0) to[out=90,in=270] (0,0.75);
\end{tikzpicture}
\hspace{-0.2cm}=
\qvar^{-1/2}
\cdot
\begin{tikzpicture}[anchorbase,scale=0.7,tinynodes]
\draw[usual,crossline] (0,0) to (0,0.75);
\draw[usual,crossline,rblobbed={0.5}{u}{spinach}] (0.5,0) to (0.5,0.75);
\end{tikzpicture}
+\qvar^{1/2}\cdot
\begin{tikzpicture}[anchorbase,scale=0.7,tinynodes]
\draw[usual,crossline] (0,0) to[out=90,in=180] (0.25,0.25) 
to[out=0,in=90] (0.5,0);
\draw[usual,crossline,rblobbed={0.75}{u}{spinach}] (0,0.75) 
to[out=270,in=180] (0.25,0.5) to[out=0,in=270] (0.5,0.75);
\end{tikzpicture}
,
\end{gather*}
which in turn prove the claimed formulas.
\end{proof}

\section{Handlebody Brauer and BMW algebras}\label{section:brauer}

We will be brief in this section as it is very similar 
to the previous discussions. Recall that we have fixed the genus $g$ and 
the number of strands $n$.

\subsection{Handlebody Brauer and BMW algebras}\label{subsection:BMW}

We use handlebody 
(framed) tangle diagrams of $2n$ 
points of genus $g$ as in \fullref{subsection:handlebody-tl-topology}.
We fix invertible scalars 
$\qvar,\avar\in\KK$ with $\qvar\neq\qvar^{-1}$.
Circles that may appear after concatenation of these diagrams 
are evaluated as in \fullref{subsection:handlebody-tl-topology}, and we 
fix $\cvar_{\gamma}\in\KK$ as before, except for 
$\cvar_{1}$ which we define as
\begin{gather*}
\cvar_{1}=1+\frac{\avar-\avar^{-1}}{\qvar-\qvar^{-1}}.
\end{gather*}

\begin{definition}\label{definition:handlebody-bmw}
We let the \emph{handlebody BMW algebra} (in $n$ strands and 
of genus $g$) $\setstuff{BMW}_{g,n}(\cpar,\qvar,\avar)$ be 
the algebra whose underlying free $\KK$-module is the
$\KK$-linear span of all 
handlebody tangle diagrams of $2n$ points 
of genus $g$, and
with multiplication given by concatenation of diagrams modulo 
circle evaluation, and the \emph{skein 
relation} \eqref{eq:skein} and \eqref{eq:framing-rel}:
\begin{gather}\label{eq:skein}
\begin{tikzpicture}[anchorbase,scale=0.7,tinynodes]
\draw[usual,crossline] (0.5,0) to[out=90,in=270] (0,0.75);
\draw[usual,crossline] (0,0) to[out=90,in=270] (0.5,0.75);
\end{tikzpicture}
-
\begin{tikzpicture}[anchorbase,scale=0.7,tinynodes,xscale=-1]
\draw[usual,crossline] (0.5,0) to[out=90,in=270] (0,0.75);
\draw[usual,crossline] (0,0) to[out=90,in=270] (0.5,0.75);
\end{tikzpicture}
= (\qvar-\qvar^{-1})\cdot
\left(
\begin{tikzpicture}[anchorbase,scale=0.7,tinynodes]
\draw[usual,crossline] (0.5,0) to(0.5,0.75);
\draw[usual,crossline] (0,0) to (0,0.75);
\end{tikzpicture}
-
\begin{tikzpicture}[anchorbase,scale=0.7,tinynodes]
\draw[usual,crossline] (0,0) to[out=90,in=180] (0.25,0.25) to[out=0,in=90] (0.5,0);
\draw[usual,crossline] (0,0.75) to[out=270,in=180] (0.25,0.5) to[out=0,in=270] (0.5,0.75);
\end{tikzpicture}
\right)
,\\
\label{eq:framing-rel}
\begin{tikzpicture}[anchorbase,scale=0.7,tinynodes]
\draw[usual,crossline] (0.5,0) to[out=90,in=270] (0,0.75);
\draw[usual,crossline] (0,0) to[out=90,in=270] (0.5,0.75);
\draw[usual,crossline] (0,0.75) to[out=90,in=180] (0.25,1) to[out=0,in=90] (0.5,0.75);
\end{tikzpicture}
= 
\avar\cdot
\begin{tikzpicture}[anchorbase,scale=0.7,tinynodes]
\draw[usual,crossline] (0.5,0) to(0.5,0.75);
\draw[usual,crossline] (0,0) to (0,0.75);
\draw[usual,crossline] (0,0.75) to[out=90,in=180] (0.25,1) to[out=0,in=90] (0.5,0.75);
\end{tikzpicture}
,\quad 
\begin{tikzpicture}[anchorbase,scale=0.7,tinynodes,xscale=-1]
\draw[usual,crossline] (0.5,0) to[out=90,in=270] (0,0.75);
\draw[usual,crossline] (0,0) to[out=90,in=270] (0.5,0.75);
\draw[usual,crossline] (0,0.75) to[out=90,in=180] (0.25,1) to[out=0,in=90] (0.5,0.75);
\end{tikzpicture}
=
\avar^{-1}\cdot
\begin{tikzpicture}[anchorbase,scale=0.7,tinynodes]
\draw[usual,crossline] (0.5,0) to (0.5,0.75);
\draw[usual,crossline] (0,0) to (0,0.75);
\draw[usual,crossline] (0,0.75) to[out=90,in=180] (0.25,1) to[out=0,in=90] (0.5,0.75);
\end{tikzpicture}
.
\end{gather}
\end{definition}

For $\qvar=\qvar^{-1}$ we also let $\avar=\avar^{-1}$
and let $\cvar_{1}\in\KK$ be any parameter.

\begin{definition}
We call the specialization $\qvar=\avar=1$ the \emph{Brauer 
specialization} and the resulting algebra
the \emph{handlebody Brauer algebra} 
(in $n$ strands and of genus $g$) 
$\setstuff{Br}_{g,n}(\cpar)$. 
\end{definition}

\begin{remark}
Note that $\setstuff{Br}_{g,n}(\cpar)$ is 
not a specialization of $\setstuff{BMW}_{g,n}(\cpar,\qvar,\avar)$ 
as the $\qvar=\avar=1$ limit of $\setstuff{BMW}_{g,n}(\cpar,\qvar,\avar)$ 
has $\cvar_{1}=2$. This technicality will not play any role for us, so 
we strategically ignore it.
\end{remark}

In the specialization $\setstuff{Br}_{g,n}(\cpar)$ we can 
use a simplified notation for the crossings
\begin{gather*}
\begin{tikzpicture}[anchorbase,scale=0.7,tinynodes]
\draw[usual] (0,0) to[out=90,in=270] (0.5,0.75);
\draw[usual] (0.5,0) to[out=90,in=270] (0,0.75);
\end{tikzpicture}  
=
\begin{tikzpicture}[anchorbase,scale=0.7,tinynodes]
\draw[usual,crossline] (0,0) to[out=90,in=270] (0.5,0.75);
\draw[usual,crossline] (0.5,0) to[out=90,in=270] (0,0.75);
\end{tikzpicture}
=
\begin{tikzpicture}[anchorbase,scale=0.7,tinynodes]
\draw[usual,crossline] (0.5,0) to[out=90,in=270] (0,0.75);
\draw[usual,crossline] (0,0) to[out=90,in=270] (0.5,0.75);
\end{tikzpicture}
.
\end{gather*}

\begin{remark}
Special cases of \fullref{definition:handlebody-bmw} have 
appeared in the literature:
\begin{enumerate}

\item The case $g=0$ is the case of the 
Birman--Murakami--Wenzl algebra, respectively Brauer algebra, 
and is classical.

\item For $g=1$ these algebras appear under the name 
of affine BMW or Brauer algebras in many works, {\eg} in
\cite{HaOl-actions-tensor-categories} or \cite{OrRa-affine-braids}.

\item For $g=2$ we were not able to find a reference, but the 
definition is easily deduced from the affine type C 
braid group combinatorics. However, this would give 
a two-boundary version of the above, with one 
core strand to the left and one to the right, 
see {\eg} \cite{DaRa-two-boundary-hecke} for the 
corresponding pictures.

\end{enumerate}
\end{remark}

In order to define a spanning set for 
$\setstuff{BMW}_{g,n}(\cpar,\qvar,\avar)$ we 
need the notion of 
\emph{perfect matchings of $2n$ points of genus $g$}: these are 
perfect matchings of $2n$ points where strands can wind around the cores.
We keep the conventions of the 
previous sections and consider that all cores are at the left. 
For example, if $g=3$ and $n=2$, then
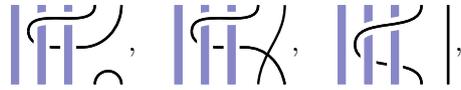
\begin{gather*}
\begin{tikzpicture}[anchorbase,scale=0.7,tinynodes]
\draw[pole,crosspole] (-0.5,0) to[out=270,in=90] (-0.5,-1.5);
\draw[usual,crossline] (-0.25,-0.5) to[out=270,in=180] (0.25,-0.8) 
to[out=0,in=180] (0.75,-0.8) to[out=0,in=270] (1.5,0);
\draw[pole,crosspole] (0,0) to[out=270,in=90] (0,-1.5);
\draw[pole,crosspole] (0.5,0) to[out=270,in=90] (0.5,-1.5);
\draw[usual,crossline] (1,0) to[out=270,in=90] (-0.25,-0.5);
\draw[usual,crossline] (1,-1.5) to[out=90,in=180] (1.25,-1.25) 
to[out=0,in=90] (1.5,-1.5);
\end{tikzpicture}
,\quad
\begin{tikzpicture}[anchorbase,scale=0.7,tinynodes]
\draw[pole,crosspole] (-0.5,0) to[out=270,in=90] (-0.5,-1.5);
\draw[usual,crossline] (-0.25,-0.5) to[out=270,in=180] (0.25,-0.8) 
to[out=0,in=180] (0.75,-0.8) to[out=0,in=100] (1.5,-1.5);
\draw[pole,crosspole] (0,0) to[out=270,in=90] (0,-1.5);
\draw[pole,crosspole] (0.5,0) to[out=270,in=90] (0.5,-1.5);
\draw[usual,crossline] (1,0) to[out=270,in=90] (-0.25,-0.5);
\draw[usual] (1,-1.5) to[out=85,in=-90] (1.5,0);
\end{tikzpicture}
,\quad
\begin{tikzpicture}[anchorbase,scale=0.7,tinynodes]
\draw[pole,crosspole] (-0.5,0) to[out=90,in=270] (-0.5,1.5);
\draw[usual,crossline] (1,0) to[out=90,in=270] (-0.25,0.75);
\draw[usual,crossline] (1.5,0) to[out=90,in=270] (1.5,1.5);
\draw[pole,crosspole] (0,0) to[out=90,in=270] (0,1.5);
\draw[pole,crosspole] (0.5,0) to[out=90,in=270] (0.5,1.5);
\draw[usual,crossline] (-0.25,0.75) to[out=90,in=270] (1,1.5);
\end{tikzpicture}
,
\end{gather*}
are examples of such perfect matchings.
These perfect matchings are equal if they describe the same perfect matching 
with the same winding around cores. 

\begin{remark}
With blobs the situation is trickier 
because we need to be careful 
with relation \eqref{eq:dots-capscups}.
Without blobs we do not have this problem and we can treat 
these perfect matchings as topological objects.
\end{remark}

Forgetting isotopy, each perfect matching as above defines an element of 
$\setstuff{BMW}_{g,n}(\cpar,\qvar,\avar)$ 
by lifting crossings.
Note however that we have an ambiguity coming from
\begin{gather*}
\left(
\begin{tikzpicture}[anchorbase,scale=0.7,tinynodes]
\draw[usual] (1,0) to[out=90,in=180] (1.75,0.75) 
to[out=0,in=90] (2.5,0);
\draw[usual] (1.75,0) to[out=90,in=270] (2.25,0.75) to[out=90,in=270] (1.75,1.5);
\end{tikzpicture}
=
\begin{tikzpicture}[anchorbase,scale=0.7,tinynodes]
\draw[usual] (1,0) to[out=90,in=180] (1.75,0.75) 
to[out=0,in=90] (2.5,0);
\draw[usual] (1.75,0) to[out=90,in=270] (1.25,0.75) to[out=90,in=270] (1.75,1.5);
\end{tikzpicture}
\right)
\mapsto
\left(
\begin{tikzpicture}[anchorbase,scale=0.7,tinynodes]
\draw[usual,crossline] (1,0) to[out=90,in=180] (1.75,0.75) 
to[out=0,in=90] (2.5,0);
\draw[usual,crossline] (1.75,0) to[out=90,in=270] (2.25,0.75) to[out=90,in=270] (1.75,1.5);
\end{tikzpicture}
=
\begin{tikzpicture}[anchorbase,scale=0.7,tinynodes]
\draw[usual,crossline] (1,0) to[out=90,in=180] (1.75,0.75) 
to[out=0,in=90] (2.5,0);
\draw[usual,crossline] (1.75,0) to[out=90,in=270] (1.25,0.75) to[out=90,in=270] (1.75,1.5);
\end{tikzpicture}
\right)
.
\end{gather*}
In order to avoid this we call a \emph{positive lift} 
a lift such that all through strands form a positive braid monoid, 
using 
\begin{gather}\label{eq:positive-braid}
\begin{tikzpicture}[anchorbase,scale=0.7,tinynodes]
\draw[usual] (0.5,0) to[out=90,in=270] (0,0.75);
\draw[usual] (0,0) to[out=90,in=270] (0.5,0.75);
\end{tikzpicture}
\mapsto
\begin{tikzpicture}[anchorbase,scale=0.7,tinynodes]
\draw[usual,crossline] (0.5,0) to[out=90,in=270] (0,0.75);
\draw[usual,crossline] (0,0) to[out=90,in=270] (0.5,0.75);
\end{tikzpicture}
,
\end{gather}
where the image is a positive crossing.
All caps and cups are assumed to be underneath any through strand, 
and any 
caps and cups are also underneath one another, going from left 
(lowest) to right (highest) along their left boundary points. 
For example,
\begin{gather*}
\text{Ok}\colon
\begin{tikzpicture}[anchorbase,scale=0.7,tinynodes]
\draw[usual,crossline] (0.5,0) to[out=90,in=270] (0,0.75);
\draw[usual,crossline] (0,0) to[out=90,in=270] (0.5,0.75);
\end{tikzpicture}
,\quad
\begin{tikzpicture}[anchorbase,scale=0.7,tinynodes]
\draw[usual,crossline] (1,0) to[out=90,in=180] (1.75,0.5) 
to[out=0,in=90] (2.5,0);
\draw[usual,crossline] (1.75,0) to[out=90,in=270] (1.75,0.75);
\end{tikzpicture}
,\quad
\begin{tikzpicture}[anchorbase,scale=0.7,tinynodes]
\draw[usual,white] (1.75,0) to[out=90,in=270] (1.75,0.75);
\draw[usual,crossline] (1,0) to[out=90,in=180] (1.75,0.5) 
to[out=0,in=90] (2.5,0);
\draw[usual,crossline] (1.5,0) to[out=90,in=180] (2.25,0.5) 
to[out=0,in=90] (3,0);
\end{tikzpicture}
,\quad
\text{not allowed}\colon
\begin{tikzpicture}[anchorbase,scale=0.7,tinynodes]
\draw[usual,crossline] (0,0) to[out=90,in=270] (0.5,0.75);
\draw[usual,crossline] (0.5,0) to[out=90,in=270] (0,0.75);
\end{tikzpicture}
,\quad
\begin{tikzpicture}[anchorbase,scale=0.7,tinynodes]
\draw[usual,crossline] (1.75,0) to[out=90,in=270] (1.75,0.75);
\draw[usual,crossline] (1,0) to[out=90,in=180] (1.75,0.5) 
to[out=0,in=90] (2.5,0);
\end{tikzpicture}
,\quad
\begin{tikzpicture}[anchorbase,scale=0.7,tinynodes]
\draw[usual,white] (1.75,0) to[out=90,in=270] (1.75,0.75);
\draw[usual,crossline] (1.5,0) to[out=90,in=180] (2.25,0.5) 
to[out=0,in=90] (3,0);
\draw[usual,crossline] (1,0) to[out=90,in=180] (1.75,0.5) 
to[out=0,in=90] (2.5,0);
\end{tikzpicture}
.
\end{gather*}

Now we can proceed as before:

\begin{lemma}\label{lemma:bmw-span}
The algebra 
$\setstuff{BMW}_{g,n}(\cvar,\qvar,\avar)$ is an associative, unital 
algebra with a $\KK$-spanning set given by positive lifts 
of perfect matchings of $2n$ points of genus $g$.
\end{lemma}

\begin{proof}
Clear by \eqref{eq:skein} and \eqref{eq:framing-rel}, and isotopies.
\end{proof}

The dimension bound in the next definition 
comes from the classical Brauer/BMW algebra. 
That is, the number $(2n-1)!!$ 
in the definition of weakly admissible parameters up next
is the dimension of the classical BMW and Brauer algebras.

\begin{definition}\label{definition:handlebody-bmw-second}
We call parameters $(\cpar,\qvar,\avar)$ such that the $\KK$-spanning set 
in \fullref{lemma:bmw-span} is a $\KK$-basis
\emph{admissible}. We call parameters \emph{weakly admissible} 
if $\dim_{\KK}\setstuff{BMW}_{g,n}(\cvar,\qvar,\avar)\geq (2n-1)!!$.
\end{definition}

To state the analog of \fullref{lemma:handlebody-tl-second} denote by $\setstuff{BMW}_{G+n}(\cvar,\qvar,\avar)
=\setstuff{BMW}_{0,G+n}(\cvar,\qvar,\avar)$ 
the classical BMW algebra in $G+n$ strands. 

\begin{lemma}\label{lemma:handlebody-bmw-second}
Let $\KK$ be a field.
For any choice of $\qvar\in\KK$ and $G\geq g$ such that 
$[k]_{\qvar}\neq 0$ 
for all $1\ll k$ there exists 
a triple $(\cpar^{G},\qvar^{G},\avar^{G})$ and 
an algebra homomorphisms 
(explicitly given in the proof below)
\begin{gather*}
\iota^{G}\colon\setstuff{BMW}_{g,n}(\cpar^{G},\qvar^{G},\avar^{G})
\to\setstuff{BMW}_{G+n}(\tilde{\cvar},\tilde{\qvar},\tilde{\avar}).
\end{gather*}
Moreover, the parameters $(\cpar^{G},\qvar^{G},\avar^{G})$ are 
weakly admissible.
\end{lemma}

If $\qvar$ is not a root of unity, then there is 
always a corresponding triple $(\cpar^{G},\qvar^{G},\avar^{G})$.

\begin{proof}
Denote by $(\tilde{\cpar},\tilde{\qvar},\tilde{\avar})$ 
a choice of parameters such that Schur--Weyl--Brauer duality 
(for this duality 
we refer the reader to, for example, \cite[Theorem 3.17]{AnStTu-semisimple-tilting} 
and \cite[Theorem 1.3]{Hu-bmw-typec})
gives a well-defined algebra homomorphism from 
$\setstuff{BMW}_{G+n}(\tilde{\cvar},\qvar,\tilde{\avar})$
into the endomorphism space of the ($G+n$)-fold 
tensor product of the vector representation of an associated quantum 
group of types $BCD$, called tensor space. 
The proof is then essentially the same as the one in 
\fullref{lemma:handlebody-tl-second}, using 
Schur--Weyl--Brauer duality instead of 
(a special case of) Schur--Weyl duality. Precisely, 
we define $\iota^{G}$ using the very same pictures as in \eqref{eq:flanking}, 
but the boxes represent 
the pullbacks of the highest weight projectors 
(splitting off the highest weight summands) in tensor space. 
Under the appropriate 
assumptions on the involved quantum numbers these 
projectors exist. Explicitly:
\begin{enumerate}[label=$\bullet$]

\item In types $B$ and $D$ one can take {\eg} $\tilde{\cvar}=[2m+1]_{\qvar}$
and $\tilde{\avar}=\qvar^{2m+1}$ and the quantum group 
for $\mathfrak{so}_{2m+1}$ to get a well-defined 
$\iota^{G}$. To ensure the existence of the idempotent 
one takes $m>\tfrac{G-1}{2}$.

\item In type $C$ to define $\iota^{G}$ one 
can take {\eg} $\tilde{\cvar}=-[2m]_{\qvar}$
and $\tilde{\avar}=-\qvar^{-2m}$ and the quantum group 
for $\mathfrak{sp}_{2m}$, and also $m>G$ to ensure the existence of the idempotent.

\item In type $D$ choices that work are {\eg} $\tilde{\cvar}=[2m]_{\qvar}$,
$\tilde{\avar}=\qvar^{2m}$ and the quantum group 
for $\mathfrak{so}_{2m}$, as well as $m>G+1$.

\end{enumerate}
To see this we can use the explicit bounds given in 
{\eg} \cite[Theorem 3.17]{AnStTu-semisimple-tilting}.
\end{proof}

\begin{remark}
The idempotents used in the proof of 
\fullref{lemma:handlebody-bmw-second} do not satisfy 
an easy recursion as the Jones--Wenzl projectors used in 
the proof of \fullref{lemma:handlebody-tl-second}.
See however \cite{IsMoOg-bmw-idempotents} or \cite{LeZh-brauer-invariant-theory} 
for some work on projectors in Brauer respectively BMW algebras.
\end{remark}

As before for the topological handlebody Temperley--Lieb algebra, we 
do not know any explicit way to construct admissible parameters.
However, under the assumption that these exists we 
conclude this section as follows.

Note that our definition realizes $\setstuff{Br}_{g,n}(\cpar)$ 
as a specialization of $\setstuff{BMW}_{g,n}(\cpar,\qvar,\avar)$, which 
is not equal to a definition using perfect matchings. So the following
needs admissible parameters but is then immediate from 
\autoref{definition:handlebody-bmw-second}:

\begin{proposition}\label{proposition:brauer-good}
For admissible parameters the handlebody 
Brauer algebra $\setstuff{Br}_{g,n}(\cpar)$ can be alternatively 
described as the free $\KK$-module spanned by 
perfect matchings of $2n$ points of genus $g$ with multiplication 
given by concatenation and circle evaluation.\qed
\end{proposition}

For the cellular structure we fix a sandwich cell datum as in 
\fullref{subsection:brauer}, and also with a very similar 
cellular basis. Let us for brevity just stress the differences.
First, we only consider positive lifts and we let 
$\setstuff{B}_{g,\lambda}^{+}$ denote the handlebody braid monoid, {\ie}
\fullref{definition:handlebody-braidgroup} 
using only positive coils and positive crossings \eqref{eq:positive-braid} 
for the definition.

Note that there is a map 
$\setstuff{B}_{g,\lambda}^{+}\to\setstuff{BMW}_{g,n}(\cpar,\qvar,\avar)$ 
that sends positive coils to 
positive coils and positive crossings to positive crossings.
We let $\sand=
\KK\setstuff{B}_{g,\lambda}^{+}$ with the monoid basis as the sandwiched basis.
We then get
\begin{gather}\label{eq:bmw-basis}
\{c_{D,b,U}^{\lambda}\mid\lambda\in\Lambda,D,U\in M_{\lambda},
b\in\sand\}
\end{gather}
as before, with the $D$ and $U$ part being as for the 
handlebody Temperley--Lieb algebra, {\cf} \eqref{eq:tl-basis-illustration}, 
namely only caps respectively caps are allowed to wind around the cores.
The picture is
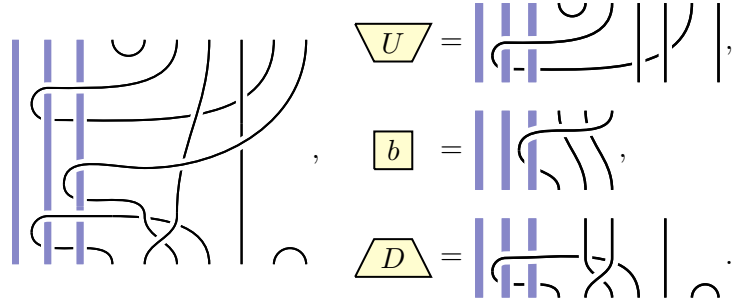
\begin{gather*}
\scalebox{0.85}{$\begin{tikzpicture}[anchorbase,scale=1]
\draw[usual,crossline] (3,-4) to[out=90,in=270] (3,-1.8);
\draw[pole,crosspole] (-0.5,-0.5) to[out=270,in=90] (-0.5,-2);
\draw[usual,crossline] (-0.25,-1.5) to[out=270,in=180] (0,-1.75) 
to[out=0,in=270] (3.5,-0.5);
\draw[pole,crosspole] (0,-0.5) to[out=270,in=90] (0,-2);
\draw[pole,crosspole] (0.5,-0.5) to[out=270,in=90] (0.5,-2);
\draw[usual,crossline] (2,-0.5) to[out=270,in=0] (0,-1.25) 
to[out=180,in=90] (-0.25,-1.5);
\draw[usual,crossline] (1,-0.5) to[out=270,in=180] (1.25,-0.75) 
to[out=0,in=270] (1.5,-0.5);
\draw[usual,crossline] (1,-3.25) to[out=0,in=90] (2.5,-4);
\draw[usual,crossline] (2.5,-0.5) to[out=270,in=90] (2,-3.2);
\draw[usual,crossline] (2,-4) to[out=90,in=270] (1.5,-3.25) to[out=90,in=270] (0.25,-2.75);
\draw[pole,crosspole] (0.5,-3.1) to[out=90,in=270] (0.5,-2);
\draw[usual,crossline] (1,-4) to[out=90,in=270] (-0.25,-3.5);
\draw[usual,crossline] (0.25,-2.75) to[out=90,in=180] (1.5,-2.5) 
to[out=0,in=270] (4,-0.5);
\draw[pole,crosspole] (-0.5,-4) to[out=90,in=270] (-0.5,-2);
\draw[pole,crosspole] (0,-4) to[out=90,in=270] (0,-2);
\draw[pole,crosspole] (0.5,-4) to[out=90,in=270] (0.5,-3.1);
\draw[usual,crossline] (-0.25,-3.5) to[out=90,in=180] (0,-3.25) to (1,-3.25);
\draw[usual,crossline] (1.5,-4) to[out=90,in=270] (2,-3.25) to (2,-3.2);
\draw[usual,crossline] (3.5,-4) to[out=90,in=180] (3.75,-3.75) 
to[out=0,in=90] (4,-4);
\draw[usual,crossline] (3,-1.8) to[out=90,in=270] (3,-0.5);
\end{tikzpicture}$}
,\quad
\begin{aligned}
\begin{tikzpicture}[anchorbase,scale=1]
\draw[mor] (0,1) to (0.25,0.5) to (0.75,0.5) to (1,1) to (0,1);
\node at (0.5,0.75){$U$};
\end{tikzpicture}
&=
\begin{tikzpicture}[anchorbase,scale=0.7,tinynodes]
\draw[pole,crosspole] (-0.5,0) to[out=270,in=90] (-0.5,-1.5);
\draw[usual,crossline] (-0.25,-1) to[out=270,in=180] (0,-1.25) 
to[out=0,in=270] (3.5,0);
\draw[pole,crosspole] (0,0) to[out=270,in=90] (0,-1.5);
\draw[pole,crosspole] (0.5,0) to[out=270,in=90] (0.5,-1.5);
\draw[usual,crossline] (2,0) to[out=270,in=0] (0,-0.75) 
to[out=180,in=90] (-0.25,-1);
\draw[usual,crossline] (1,0) to[out=270,in=180] (1.25,-0.25) 
to[out=0,in=270] (1.5,0);
\draw[usual,crossline] (2.5,0) to[out=270,in=90] (2.5,-1.5);
\draw[usual,crossline] (3,0) to[out=270,in=90] (3,-1.5);
\draw[usual,crossline] (4,0) to[out=270,in=90] (4,-1.5);
\end{tikzpicture}
,\\
\begin{tikzpicture}[anchorbase,scale=1]
\draw[mor] (0.25,0) to (0.25,0.5) to (0.75,0.5) to (0.75,0) to (0.25,0);
\node at (0.5,0.25){$b$};
\node at (0.8,0.25){\phantom{$b$}};
\end{tikzpicture}
&=
\begin{tikzpicture}[anchorbase,scale=0.7,tinynodes]
\draw[usual,crossline] (1.5,0) to[out=90,in=270] (1,1.5);
\draw[usual,crossline] (2,0) to[out=90,in=270] (1.5,1.5);
\draw[pole,crosspole] (-0.5,0) to[out=90,in=270] (-0.5,1.5);
\draw[pole,crosspole] (0,0) to[out=90,in=270] (0,1.5);
\draw[usual,crossline] (1,0) to[out=90,in=270] (0.25,0.75);
\draw[pole,crosspole] (0.5,0) to[out=90,in=270] (0.5,1.5);
\draw[usual,crossline] (0.25,0.75) to[out=90,in=270] (2,1.5);
\end{tikzpicture}
,
\\
\begin{tikzpicture}[anchorbase,scale=1]
\draw[mor] (0,-0.5) to (0.25,0) to (0.75,0) to (1,-0.5) to (0,-0.5);
\node at (0.5,-0.25){$D$};
\end{tikzpicture}
&=
\begin{tikzpicture}[anchorbase,scale=0.7,tinynodes]
\draw[pole,crosspole] (-0.5,0) to[out=90,in=270] (-0.5,1.5);
\draw[usual,crossline] (1,0) to[out=90,in=270] (-0.25,0.5);
\draw[pole,crosspole] (0,0) to[out=90,in=270] (0,1.5);
\draw[pole,crosspole] (0.5,0) to[out=90,in=270] (0.5,1.5);
\draw[usual,crossline] (-0.25,0.5) to[out=90,in=180] (0,0.75) 
to[out=0,in=90] (2.5,0);
\draw[usual,crossline] (2,0) to[out=90,in=270] (1.5,0.75) to (1.5,1.5);
\draw[usual,crossline] (1.5,0) to[out=90,in=270] (2,0.75) to (2,1.5);
\draw[usual,crossline] (3,0) to[out=90,in=270] (3,1.5);
\draw[usual,crossline] (3.5,0) to[out=90,in=180] (3.75,0.25) 
to[out=0,in=90] (4,0);
\end{tikzpicture}
.
\end{aligned}
\end{gather*}
The antiinvolution $(\placeholder)^{\star}\colon\setstuff{BMW}_{g,n}(\cpar,\qvar,\avar)
\to\setstuff{BMW}_{g,n}(\cpar,\qvar,\avar)$ 
is defined as for $\setstuff{TL}_{g,n}(\cpar)$ 
with the addition that it maps positive crossings to 
positive crossings.

\begin{proposition}\label{proposition:bmw-basis}
For any admissible parameters, the above defines an involutive sandwich
cell datum for the algebra $\setstuff{BMW}_{g,n}(\cpar,\qvar,\avar)$.
\end{proposition}

\begin{proof}
The only claim which is not immediate by construction is 
that \eqref{eq:bmw-basis} is a $\KK$-basis, which
follows from \fullref{definition:handlebody-bmw-second}.
\end{proof}

Not surprisingly, the cell 
structure is, {\muta}, as in \fullref{example:brauer}.
The next theorem is 
clear by the previous discussions.

\begin{theorem}\label{theorem:bmw}
Let $\KK$ be a field, and choose admissible parameters.
\begin{enumerate}

\item If $\cvar\neq 0$, or $\cvar=0$ and $\lambda\neq 0$ is odd, 
then all $\lambda\in\Lambda$ are apexes. In the remaining case, 
$\cvar=0$ and $\lambda=0$ (this only happens if $n$ is even), all $\lambda\in\Lambda-\{0\}$ are apexes, but $\lambda=0$ is not an apex.

\item The simple $\setstuff{BMW}_{g,n}(\cpar,\qvar,\avar)$-modules of 
apex $\lambda\in\Lambda$ 
are parameterized by simple modules of $\KK\setstuff{B}_{g,\lambda}^{+}$.

\item The simple $\setstuff{BMW}_{g,n}(\cpar,\qvar,\avar)$-modules of 
apex $\lambda\in\Lambda$ can be constructed as 
the simple heads of
$\mathrm{Ind}_{\KK\setstuff{B}_{g,\lambda}^{+}}^{\setstuff{BMW}_{g,n}(\cpar,\qvar,\avar)}(K)$, 
where $K$ runs over (equivalence classes of) 
simple $\KK\setstuff{B}_{g,\lambda}^{+}$-modules.

\end{enumerate}
\end{theorem}

\begin{proof}
No difference to the Brauer case in \fullref{theorem:brauer}.
\end{proof}

\begin{example}
Let us comment on the parametrization given by 
\fullref{theorem:bmw}:
\begin{enumerate}

\item For $g=0$ the algebras $\KK\setstuff{B}_{0,\lambda}^{+}$ 
are Hecke algebras associated to Coxeter type 
$A_{\lambda-1}$, so we get the same classification 
as in \fullref{theorem:brauer}, but for the BMW algebra. 
This was of course known, see {\eg} \cite[Corollary 3.14]{Xi-bmw}.

\item For $g=1$ the algebras $\KK\setstuff{B}_{1,\lambda}^{+}$
are extended affine Hecke algebras associated to Coxeter type 
$A_{\lambda-1}$.

\end{enumerate}
\end{example}

\subsection{Cyclotomic handlebody Brauer and BMW algebras}\label{subsection:cbrauer}

Recall the notion of blobbed presentations from 
\fullref{subsection:handlebody-blob-topology}, and retain 
the notation from that section.
One difference to that section is that 
here we identify blobbed diagrams with a subalgebra 
of $\setstuff{BMW}_{g,n}(\cpar,\qvar,\avar)$ 
spanned by all elements containing only positive coils.

We have the following analog of \fullref{lemma:blob-rels}. The proof 
is the same as that of \fullref{lemma:blob-rels} and omitted.

\begin{lemmaqed}\label{lemma:blob-bmw-rels}
Blobs satisfy the following relations. First, 
\eqref{eq:blobslides}, \eqref{eq:dots-capscups} 
and \eqref{eq:height-blob-switch},
and also
\begin{gather}\label{eq:bmw-more-slides}
\begin{aligned}
\begin{tikzpicture}[anchorbase,scale=0.7,tinynodes]
\draw[usual,crossline] (0.5,0) to[out=90,in=270] (0,0.75);
\draw[usual,crossline,rblobbed={0.85}{u}{spinach}] (0,0) to[out=90,in=270] (0.5,0.75);
\end{tikzpicture}
\hspace{-0.2cm}&=\hspace{-0.2cm}
\begin{tikzpicture}[anchorbase,scale=0.7,tinynodes]
\draw[usual,crossline] (0.5,0) to[out=90,in=270] (0,0.75);
\draw[usual,crossline,blobbed={0.15}{u}{spinach}] (0,0) to[out=90,in=270] (0.5,0.75);
\end{tikzpicture}
+(\qvar-\qvar^{-1})\cdot
\left(
\begin{tikzpicture}[anchorbase,scale=0.7,tinynodes]
\draw[usual,crossline] (0,0) to (0,0.75);
\draw[usual,crossline,rblobbed={0.5}{u}{spinach}] (0.5,0) to (0.5,0.75);
\end{tikzpicture}
-
\begin{tikzpicture}[anchorbase,scale=0.7,tinynodes]
\draw[usual,crossline] (0,0) to[out=90,in=180] (0.25,0.25) 
to[out=0,in=90] (0.5,0);
\draw[usual,crossline,rblobbed={0.75}{u}{spinach}] (0,0.75) 
to[out=270,in=180] (0.25,0.5) to[out=0,in=270] (0.5,0.75);
\end{tikzpicture}
\hspace{-0.2cm}\right)
,\\
\begin{tikzpicture}[anchorbase,scale=0.7,tinynodes]
\draw[usual,crossline,rblobbed={0.12}{u}{spinach}] (0.5,0) to[out=90,in=270] (0,0.75);
\draw[usual,crossline] (0,0) to[out=90,in=270] (0.5,0.75);
\end{tikzpicture}
\hspace{-0.2cm}&=\hspace{-0.2cm}
\begin{tikzpicture}[anchorbase,scale=0.7,tinynodes]
\draw[usual,crossline,blobbed={0.87}{u}{spinach}] (0.5,0) to[out=90,in=270] (0,0.75);
\draw[usual,crossline] (0,0) to[out=90,in=270] (0.5,0.75);
\end{tikzpicture}
+(\qvar-\qvar^{-1})\cdot
\left(
\begin{tikzpicture}[anchorbase,scale=0.7,tinynodes]
\draw[usual,crossline] (0,0) to (0,0.75);
\draw[usual,crossline,rblobbed={0.5}{u}{spinach}] (0.5,0) to (0.5,0.75);
\end{tikzpicture}
-
\begin{tikzpicture}[anchorbase,scale=0.7,tinynodes]
\draw[usual,crossline,rblobbed={0.75}{u}{spinach}] (0,0) 
to[out=90,in=180] (0.25,0.25) to[out=0,in=90] (0.5,0);
\draw[usual,crossline] (0,0.75) to[out=270,in=180] (0.25,0.5) 
to[out=0,in=270] (0.5,0.75);
\end{tikzpicture}
\hspace{-0.2cm}\right)
,
\\
\begin{tikzpicture}[anchorbase,scale=0.7,tinynodes]
\draw[usual,crossline,blobbed={0.15}{u}{spinach}] (0,0) to[out=90,in=270] (0.5,0.75);
\draw[usual,crossline] (0.5,0) to[out=90,in=270] (0,0.75);
\end{tikzpicture}
&=
\begin{tikzpicture}[anchorbase,scale=0.7,tinynodes]
\draw[usual,crossline,rblobbed={0.85}{u}{spinach}] (0,0) to[out=90,in=270] (0.5,0.75);
\draw[usual,crossline] (0.5,0) to[out=90,in=270] (0,0.75);
\end{tikzpicture}
\hspace{-0.2cm}+(\qvar-\qvar^{-1})\cdot
\left(\hspace{-0.2cm}
\begin{tikzpicture}[anchorbase,scale=0.7,tinynodes]
\draw[usual,crossline,blobbed={0.5}{u}{spinach}] (0,0) to (0,0.75);
\draw[usual,crossline] (0.5,0) to (0.5,0.75);
\end{tikzpicture}
-
\begin{tikzpicture}[anchorbase,scale=0.7,tinynodes]
\draw[usual,crossline,blobbed={0.25}{u}{spinach}] (0,0) 
to[out=90,in=180] (0.25,0.25) to[out=0,in=90] (0.5,0);
\draw[usual,crossline] (0,0.75) to[out=270,in=180] (0.25,0.5) 
to[out=0,in=270] (0.5,0.75);
\end{tikzpicture}
\right)
,\\
\begin{tikzpicture}[anchorbase,scale=0.7,tinynodes]
\draw[usual,crossline] (0,0) to[out=90,in=270] (0.5,0.75);
\draw[usual,crossline,blobbed={0.85}{u}{spinach}] (0.5,0) to[out=90,in=270] (0,0.75);
\end{tikzpicture}
&=
\begin{tikzpicture}[anchorbase,scale=0.7,tinynodes]
\draw[usual,crossline] (0,0) to[out=90,in=270] (0.5,0.75);
\draw[usual,crossline,rblobbed={0.15}{u}{spinach}] (0.5,0) to[out=90,in=270] (0,0.75);
\end{tikzpicture}
\hspace{-0.2cm}+(\qvar-\qvar^{-1})\cdot
\left(\hspace{-0.2cm}
\begin{tikzpicture}[anchorbase,scale=0.7,tinynodes]
\draw[usual,crossline,blobbed={0.5}{u}{spinach}] (0,0) to (0,0.75);
\draw[usual,crossline] (0.5,0) to (0.5,0.75);
\end{tikzpicture}
-
\begin{tikzpicture}[anchorbase,scale=0.7,tinynodes]
\draw[usual,crossline,] (0,0) to[out=90,in=180] (0.25,0.25) 
to[out=0,in=90] (0.5,0);
\draw[usual,crossline,blobbed={0.25}{u}{spinach}] (0,0.75) 
to[out=270,in=180] (0.25,0.5) to[out=0,in=270] (0.5,0.75);
\end{tikzpicture}
\right)
.                     
\end{aligned}
\end{gather}
(Note the difference in the powers of 
$\qvar$ between \eqref{eq:more-slides} and 
\eqref{eq:bmw-more-slides}.)
\end{lemmaqed}

We call the relations collected in 
\fullref{lemma:blob-bmw-rels} \emph{blob sliding relations}. Because of 
their slightly distorted topology, we introduce 
cyclotomic handlebody Brauer algebra 
before their BMW counterparts.
That is, we first treat the Brauer specialization 
$\qvar=\avar=1$.

In the following lemma we 
collect some relations that are handy in the Brauer case. 
In particular, up to \eqref{eq:dots-capscups}, these \emph{Brauer blobs}
move freely along strands, {\cf} \eqref{eq:brauer-move1}
and \eqref{eq:brauer-move2}.

\begin{lemma}\label{lemma:blob-brauer-rels}
Brauer blobs satisfy the following relations
additionally to \eqref{eq:dots-capscups}.
\begin{gather}\label{eq:brauer-move1}
\begin{tikzpicture}[anchorbase,scale=0.7,tinynodes]
\draw[usual] (0,0) to[out=90,in=270] (0.5,0.75);
\draw[usual,rblobbed={0.15}{u}{spinach}] (0.5,0) to[out=90,in=270] (0,0.75);
\end{tikzpicture}
\hspace{-0.2cm}=\hspace{-0.2cm}
\begin{tikzpicture}[anchorbase,scale=0.7,tinynodes]
\draw[usual,blobbed={0.85}{u}{spinach}] (0.5,0) to[out=90,in=270] (0,0.75);
\draw[usual] (0,0) to[out=90,in=270] (0.5,0.75);
\end{tikzpicture}
,\quad
\begin{tikzpicture}[anchorbase,scale=0.7,tinynodes]
\draw[usual] (0.5,0) to[out=90,in=270] (0,0.75);
\draw[usual,blobbed={0.15}{u}{spinach}] (0,0) to[out=90,in=270] (0.5,0.75);
\end{tikzpicture}
=
\begin{tikzpicture}[anchorbase,scale=0.7,tinynodes]
\draw[usual,rblobbed={0.85}{u}{spinach}] (0,0) to[out=90,in=270] (0.5,0.75);
\draw[usual] (0.5,0) to[out=90,in=270] (0,0.75);
\end{tikzpicture}
.
\end{gather}
\begin{gather}\label{eq:brauer-move2}
\begin{tikzpicture}[anchorbase,scale=0.7,tinynodes]
\draw[usual,blobbed={0.66}{u}{spinach}] (0,0) to[out=90,in=270] (0,1.5);
\draw[usual,rblobbed={0.33}{v}{tomato}] (0.5,0) to[out=90,in=270] (0.5,1.5);
\end{tikzpicture}
\hspace{-0.2cm}=\hspace{-0.2cm}
\begin{tikzpicture}[anchorbase,scale=0.7,tinynodes]
\draw[usual,blobbed={0.33}{u}{spinach}] (0,0) to[out=90,in=270] (0,1.5);
\draw[usual,rblobbed={0.66}{v}{tomato}] (0.5,0) to[out=90,in=270] (0.5,1.5);
\end{tikzpicture}
.
\end{gather}
\end{lemma}

\begin{proof}
The relation \eqref{eq:bmw-more-slides} 
specializes to \eqref{eq:brauer-move1}, which we then use to 
derive \eqref{eq:brauer-move2} from \eqref{eq:height-blob-switch}.
\end{proof}

Fix cyclotomic parameters
$\bpar=(\bvar_{u,i})\in\KK^{d_{1}+\dots+d_{g}}$, a
degree vector $\dpar=(\dvar_{u})\in\N^{g}$. Moreover, 
let $\setstuff{Br}^{+}_{g,n}(\cpar)$ denote 
the subalgebra of $\setstuff{Br}_{g,n}(\cpar)$ with only 
positive coils.

\begin{definition}\label{definition:blob-brauer}
We let the \emph{cyclotomic handlebody Brauer algebra} 
(in $n$ strands and of genus $g$) 
$\setstuff{Br}_{g,n}^{\dpar,\bpar}(\cpar)$ be 
the quotient of $\setstuff{Br}^{+}_{g,n}(\cpar)$
by the two-sided ideal generated by the \emph{cyclotomic relations}
\begin{gather}\label{eq:cyclotomic-brauer-blob}
(b_{u}-\bvar_{u,1})
\varpi_{1}
(b_{u}-\bvar_{u,2})
\varpi_{2}
\dots
(b_{u}-\bvar_{u,\dvar_{u}-1})
\varpi_{d_{u}-1}
(b_{u}-\bvar_{u,\dvar_{u}})
=0,
\end{gather}
where $\varpi_{j}$ is any finite (potentially empty)
expression not involving $b_{u}$ such that all 
$b_{u}$ in \eqref{eq:cyclotomic-brauer-blob} are on the same strand.
\end{definition}

\begin{remark}
Note that the skein relation \eqref{eq:skein} and also 
\eqref{eq:dots-capscups} imply compatibility 
conditions between parameters. See \fullref{def:bmw-admissible}, \fullref{lemma:handlebody-blobbmw-second} and \fullref{proposition:brauer-dim} below.
\end{remark}

We consider \emph{clapped, blobbed perfect matchings of $2n$ 
points of genus $g$} which we need 
for counting purposes.
We assume that these perfect 
matchings points are clapped, similar to the crossingless 
matchings in the proof of 
\fullref{proposition:blob-dim}.
Strands can carry blobs that move 
freely on the strand they belong but are not allowed 
to pass each other. In particular, there is no condition 
on being reachable from the left.
This implies that 
each strand with blobs defines a word in 
the free monoid $\setstuff{F}_{g}^{+}\cong
\setstuff{Br}^{+}_{g,1}(\cpar)\hookrightarrow\setstuff{Br}^{+}_{g,n}(\cpar)$. 

Note that we consider them as perfect matchings, so 
there is an ambiguity in how to illustrate 
these without further conditions. This problem is 
however resolved by demanding that each matched pair 
is connected by a cup with exactly one Morse point, and there is 
a minimal number of intersection between the cups.
Here is an example:
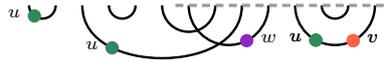
\begin{gather*}
\begin{tikzpicture}[anchorbase,scale=0.7,tinynodes]
\draw[usual,blobbed={0.25}{u}{spinach}] (1,1.5) to[out=270,in=180] (1.25,1.25) to[out=0,in=270] (1.5,1.5);
\draw[usual,blobbed={0.25}{u}{spinach}] (2,1.5) to[out=270,in=180] (3.5,0.5) to[out=0,in=270] (5,1.5);
\draw[usual] (2.5,1.5) to[out=270,in=180] (2.75,1.25) to[out=0,in=270] (3,1.5);
\draw[usual,rblobbed={0.66}{w}{orchid}] (4,1.5) 
to[out=270,in=180] (4.75,0.75) to[out=0,in=270] (5.5,1.5);
\draw[usual] (3.5,1.5) to[out=270,in=180] (4,1.0) to[out=0,in=270] (4.5,1.5);
\draw[usual,blobbed={0.33}{u}{spinach},rblobbed={0.66}{v}{tomato}] (6,1.5) 
to[out=270,in=180] (6.75,0.75) to[out=0,in=270] (7.5,1.5);
\draw[usual] (6.5,1.5) to[out=270,in=180] (6.75,1.25) to[out=0,in=270] (7,1.5);
\draw[very thick,mygray,densely dashed] (3.75,1.5) to (7.75,1.5);
\end{tikzpicture}
.
\end{gather*}
Each such perfect matching defines an element of $\setstuff{Br}^{+}_{g,n}(\cpar)$ by 
unclapping and simultaneously sliding all blobs
to the left that may fall on caps or cups.
This procedure is illustrated in the diagram below. 
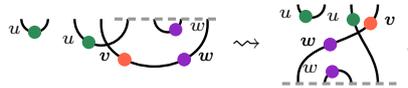
\begin{gather*}
\begin{tikzpicture}[anchorbase,scale=0.7,tinynodes]
\draw[usual,blobbed={0.5}{u}{spinach}] (2,1.5) to[out=270,in=180] (2.25,1.25) to[out=0,in=270] (2.5,1.5);
\draw[usual,blobbed={0.33}{u}{spinach}] (3,1.5) to[out=270,in=180] (3.5,1.0) to[out=0,in=270] (4.0,1.5);
\draw[usual,blobbed={0.3}{v}{tomato},rblobbed={0.7}{w}{orchid}] (3.5,1.5) to (3.5,1.4) to[out=270,in=180] (4.5,0.6) to[out=0,in=270] (5.5,1.4) to (5.5,1.5);
\draw[usual,rblobbed={0.75}{w}{orchid}] (4.5,1.5) to[out=270,in=180] (4.75,1.25) to[out=0,in=270] (5,1.5);
\draw[very thick,mygray,densely dashed] (3.75,1.5) to (5.75,1.5);
\end{tikzpicture}
\rightsquigarrow
\begin{tikzpicture}[anchorbase,scale=0.7,tinynodes]
\draw[usual,blobbed={0.33}{u}{spinach}] (2,1.5) to[out=270,in=180] (2.25,1.25) to[out=0,in=270] (2.5,1.5);
\draw[usual,blobbed={0.15}{u}{spinach}] (3,1.5) to[out=270,in=90]  (3.5,0);
\draw[usual,blobbed={0.45}{w}{orchid},rblobbed={0.85}{v}{tomato}] (2,0) to (2,0.1) to[out=90,in=270] (3.5,1.5);
\draw[usual,blobbed={0.33}{w}{orchid}] (2.5,0) to[out=90,in=180] (2.75,0.25) to[out=0,in=90] (3.0,0);
\draw[very thick,mygray,densely dashed] (1.7,0) to (3.85,0);
\end{tikzpicture}
.
\end{gather*}
This operation can described rigorously: 
If we label the boundary 
points $1,\dots,2n$, then  
the set of perfect matchings is in 
bijection with the set of all 
unordered $n$-tuples of pairs 
$\{(i_{1},j_{1}),\dots,(i_{n},j_{n})\}$ that can be formed from the $2n$ 
distinct elements of $\{1,\dots,2n\}$ with $i_{k}<j_{k}$. Each such pair corresponds to a strand in the diagrammatics. 
A blob on the strand $(i,j)$ is to be slid to the 
left if $j\leq n$, to the right if $n\leq i$ 
and left untouched otherwise in the unclapping process.

The next lemma follows from this discussion.

\begin{lemma}\label{lemma:brauer-basis2}
The algebra 
$\setstuff{Br}_{g,n}^{\dpar,\bpar}(\cpar)$ 
has a $\KK$-spanning set in bijection with clapped, blobbed
perfect matchings of $2n$ points of genus  
$g$ where each strand has at 
most $\dvar_{u}-1$ blobs of the corresponding type.
\end{lemma}

\begin{proof}
The above defines a $\KK$-module map from 
the free $\KK$-module spanned by all 
clap\-ped, blobbed perfect matchings to 
$\setstuff{Br}^{+}_{g,n}(\cpar)$.	
On the other hand, each diagram of 
$\setstuff{Br}^{+}_{g,n}(\cpar)$ defines a 
clapped, blobbed perfect 
matching by clapping and then arranging the result until it is of the required form.
These are inverse 
procedures and so the $\KK$-module 
$\setstuff{Br}^{+}_{g,n}(\cpar)$ 
is contained in the free 
$\KK$-module spanned by all 
clapped, blobbed perfect matchings.
The cyclotomic condition then gets rid of 
having too many blobs on strands.
\end{proof}

\begin{definition}\label{def:bmw-admissible}
We call parameters $(\dpar,\bpar,\cpar)$ such that the $\KK$-spanning set 
in \fullref{lemma:brauer-basis2} is a $\KK$-basis \emph{admissible}.
\end{definition}

We do not know any representation theoretical space 
where $\setstuff{Br}_{g,n}^{\dpar,\bpar}(\cpar)$ acts on, 
so we can not copy the arguments from {\eg} the proof of \fullref{lemma:handlebody-tl-second}. We rather do the following 
(which includes the cases where the parameters are generic):

\begin{lemma}\label{lemma:handlebody-blobbmw-second}
Let $\KK$ be a field and let
$\KKL$ be a field extension of $\KK$ 
of degree $[\KKL:\KK]\gg 1$.
For each choice of $\dpar$ and $\bpar$ 
there exist $\cpar$ and $\bpar$ with entries in
in $\KKL$ such that $(\dpar,\bpar,\cpar)$ is admissible.
\end{lemma}

\begin{proof}
The same arguments as in \cite{GoHaMo-cyclotomic-bmw2}
prove that admissibility is equivalent to 
the left 
and right $\setstuff{Br}_{g,n}^{\dpar,\bpar}(\cpar)$-modules
\begin{gather*}
\setstuff{Caps}=
\left\{
\begin{tikzpicture}[anchorbase,scale=0.7,tinynodes]
\draw[usual,crossline] (0,-0.4) to (0,0) to[out=90,in=180] (0.25,0.25) 
to[out=0,in=90] (0.5,0) to (0.5,-0.4);
\end{tikzpicture}
,
\begin{tikzpicture}[anchorbase,scale=0.7,tinynodes]
\draw[usual,crossline,blobbed={0.1}{u}{spinach}] (0,-0.4) to (0,0) to[out=90,in=180] (0.25,0.25) 
to[out=0,in=90] (0.5,0) to (0.5,-0.4);
\end{tikzpicture}
,
\begin{tikzpicture}[anchorbase,scale=0.7,tinynodes]
\draw[usual,crossline,blobbed={0.1}{v}{tomato}] (0,-0.4) to (0,0) to[out=90,in=180] (0.25,0.25) 
to[out=0,in=90] (0.5,0) to (0.5,-0.4);
\end{tikzpicture}
,
\begin{tikzpicture}[anchorbase,scale=0.7,tinynodes]
\draw[usual,crossline,blobbed={0.075}{u}{spinach},blobbed={0.3}{v}{tomato}] (0,-0.4) to (0,0) to[out=90,in=180] (0.25,0.25) 
to[out=0,in=90] (0.5,0) to (0.5,-0.4);
\end{tikzpicture}
,\dots
\right\}
,\quad
\setstuff{Cups}=
\left\{
\begin{tikzpicture}[anchorbase,scale=0.7,tinynodes]
\draw[usual,crossline] (0,0.4) to (0,0) to[out=270,in=180] (0.25,-0.25) 
to[out=0,in=270] (0.5,0) to (0.5,0.4);
\end{tikzpicture}
,
\begin{tikzpicture}[anchorbase,scale=0.7,tinynodes]
\draw[usual,crossline,blobbed={0.25}{u}{spinach}] (0,0.4) to (0,0) to[out=270,in=180] (0.25,-0.25) 
to[out=0,in=270] (0.5,0) to (0.5,0.4);
\end{tikzpicture}
,
\begin{tikzpicture}[anchorbase,scale=0.7,tinynodes]
\draw[usual,crossline,blobbed={0.25}{v}{tomato}] (0,0.4) to (0,0) to[out=270,in=180] (0.25,-0.25) 
to[out=0,in=270] (0.5,0) to (0.5,0.4);
\end{tikzpicture}
,
\begin{tikzpicture}[anchorbase,scale=0.7,tinynodes]
\draw[usual,crossline,blobbed={0.075}{u}{spinach},blobbed={0.3}{v}{tomato}] (0,0.4) to (0,0) to[out=270,in=180] (0.25,-0.25) 
to[out=0,in=270] (0.5,0) to (0.5,0.4);
\end{tikzpicture}
,\dots
\right\},
\end{gather*}
of all possible blob placements on a cap respectively 
cup not affected by 
the cyclotomic relation \eqref{eq:cyclotomic-brauer-blob}
being free of rank $\bbvar_{1,\dpar}$. 
Closing diagrams implies that this 
happens if and only if the $\bbvar_{1,\dpar}$-$\bbvar_{1,\dpar}$ pairing matrix
\begin{gather*}
\scalebox{0.85}{$P=
\begin{pmatrix}
\hspace*{0.5cm}
\begin{tikzpicture}[anchorbase,scale=0.7,tinynodes]
\draw[usual] (0,-0.75) to (0,0) to[out=90,in=180] (0.25,0.25) 
to[out=0,in=90] (0.5,0) to (0.5,-0.75) to[out=270,in=0] (0.25,-1) 
to[out=180,in=270] (0,-0.75);
\end{tikzpicture}
&
\begin{tikzpicture}[anchorbase,scale=0.7,tinynodes]
\draw[usual,blobbed={0.1}{u}{spinach}] (0,-0.75) to (0,0) to[out=90,in=180] (0.25,0.25) 
to[out=0,in=90] (0.5,0) to (0.5,-0.75) to[out=270,in=0] (0.25,-1) 
to[out=180,in=270] (0,-0.75);
\end{tikzpicture}
&
\begin{tikzpicture}[anchorbase,scale=0.7,tinynodes]
\draw[usual,blobbed={0.2}{v}{tomato}] (0,-0.75) to (0,0) to[out=90,in=180] (0.25,0.25) 
to[out=0,in=90] (0.5,0) to (0.5,-0.75) to[out=270,in=0] (0.25,-1) 
to[out=180,in=270] (0,-0.75);
\end{tikzpicture}
&
\begin{tikzpicture}[anchorbase,scale=0.7,tinynodes]
\draw[usual,blobbed={0.1}{u}{spinach},blobbed={0.2}{v}{tomato}] (0,-0.75) to (0,0) to[out=90,in=180] (0.25,0.25) 
to[out=0,in=90] (0.5,0) to (0.5,-0.75) to[out=270,in=0] (0.25,-1) 
to[out=180,in=270] (0,-0.75);
\end{tikzpicture}
&
\dots
\\[10pt]
\begin{tikzpicture}[anchorbase,scale=0.7,tinynodes]
\draw[usual,blobbed={0}{u}{spinach}] (0,-0.75) to (0,0) to[out=90,in=180] (0.25,0.25) 
to[out=0,in=90] (0.5,0) to (0.5,-0.75) to[out=270,in=0] (0.25,-1) 
to[out=180,in=270] (0,-0.75);
\end{tikzpicture}
&
\begin{tikzpicture}[anchorbase,scale=0.7,tinynodes]
\draw[usual,blobbed={0}{u}{spinach},blobbed={0.1}{u}{spinach}] (0,-0.75) to (0,0) to[out=90,in=180] (0.25,0.25) 
to[out=0,in=90] (0.5,0) to (0.5,-0.75) to[out=270,in=0] (0.25,-1) 
to[out=180,in=270] (0,-0.75);
\end{tikzpicture}
&
\begin{tikzpicture}[anchorbase,scale=0.7,tinynodes]
\draw[usual,blobbed={0}{u}{spinach},blobbed={0.2}{v}{tomato}] (0,-0.75) to (0,0) to[out=90,in=180] (0.25,0.25) 
to[out=0,in=90] (0.5,0) to (0.5,-0.75) to[out=270,in=0] (0.25,-1) 
to[out=180,in=270] (0,-0.75);
\end{tikzpicture}
&
\begin{tikzpicture}[anchorbase,scale=0.7,tinynodes]
\draw[usual,blobbed={0}{u}{spinach},blobbed={0.1}{u}{spinach},blobbed={0.2}{v}{tomato}] (0,-0.75) to (0,0) to[out=90,in=180] (0.25,0.25) 
to[out=0,in=90] (0.5,0) to (0.5,-0.75) to[out=270,in=0] (0.25,-1) 
to[out=180,in=270] (0,-0.75);
\end{tikzpicture}
&
\dots
\\
\vdots
&
\ddots
&
\ddots
&
\ddots
&
\ddots
\end{pmatrix}$}
\end{gather*}
is non-degenerate, {\ie} of rank $\bbvar_{1,\dpar}$. 
Note that each row contains a unique 
combination of $\bvar_{i}$ and $\cvar_{j}$. Thus, 
choosing them all such that they do not satisfy 
any polynomial equation of low order ensures 
that the determinant of $P$ is non-zero.
\end{proof}

\begin{remark}
For $g=1$ there is such an effective 
criterion to get admissible parameters, see 
\cite[Theorem 3.2]{GoHaMo-cyclotomic-bmw2},
\cite{WiYu-cyclotomic-bmw}. It would be interesting 
to have such a criterion for $g>1$.
\end{remark}

The key players to calculate 
$\dim_{\KK}\setstuff{Br}_{g,n}^{\dpar,\bpar}(\cpar)$
are again the blob numbers $\bbvar_{g,\dpar}$
defined in \eqref{eq:blob-numbers}:

\begin{proposition}\label{proposition:brauer-dim}
For admissible parameters we have 
\begin{gather}\label{eq:brauer-dim}
\dim_{\KK}\setstuff{Br}_{g,n}^{\dpar,\bpar}(\cpar)
=
(\bbvar_{g,\dpar})^{n}(2n-1)!!.
\end{gather}
Moreover, for admissible parameters, the cyclotomic handlebody 
Brauer algebra $\setstuff{Br}_{g,n}^{\dpar,\bpar}(\cpar)$ can be alternatively 
described as the free $\KK$-module spanned by 
perfect matchings of $2n$ points of genus $g$ with multiplication 
given by concatenation and circle evaluation modulo the cyclotomic 
condition. 
\end{proposition}

\begin{proof}
By \fullref{lemma:blob-onestrand}, the number of ways to put blobs 
on a single strand is $\bbvar_{g,\dpar}$. Hence, 
the number of ways to put these on $n$ strands is 
$(\bbvar_{g,\dpar})^{n}$. Since there are $(2n-1)!!$ corresponding 
perfect matchings all having $n$ strands, 
the dimension formula follows from \fullref{lemma:brauer-basis2}.
\end{proof}

\begin{example}
For $g=0$ \eqref{eq:brauer-dim} is the dimension 
of the Brauer algebra. For $g=1$ \eqref{eq:brauer-dim} 
is the formula in {\eg} \cite[Section 11]{HaOl-cyclotomic-bmw} or
\cite[Theorem 4.11]{Yu-cyclotomic-bmw}.
\end{example}

\begin{remark}
Note also the difference of our construction to 
\cite{HaOl-cyclotomic-bmw} or
\cite{Yu-cyclotomic-bmw}: we define 
$\setstuff{Br}_{g,n}^{\dpar,\bpar}(\cpar)$ as a tangle algebra 
to begin with, while those papers 
start with an algebraic formulation and
then prove that they have the suggestive diagrammatic presentation.
\end{remark}

We keep the cyclotomic parameters $\bpar$ 
and degree vector $\dpar$ as before.
We are now going to define cyclotomic handlebody BMW algebras, 
where we as before use the corresponding positive monoid 
$\setstuff{BMW}^{+}_{g,n}(\cpar,\qvar,\avar)$.

\begin{definition}\label{definition:blob-bmw}
We let the \emph{cyclotomic handlebody BMW algebra} 
(in $n$ strands and of genus $g$)
$\setstuff{BMW}_{g,n}^{\dpar,\bpar}(\cpar,\qvar,\avar)$ be 
the quotient of $\setstuff{BMW}^{+}_{g,n}(\cpar,\qvar,\avar)$
by the two-sided ideal generated by the \emph{cyclotomic relations}
\begin{gather}\label{eq:cyclotomic-bmw-blob}
(b_{u}-\bvar_{u,1})
\varpi_{1}
(b_{u}-\bvar_{u,2})
\varpi_{2}
\dots
(b_{u}-\bvar_{u,\dvar_{u}-1})
\varpi_{d_{u}-1}
(b_{u}-\bvar_{u,\dvar_{u}})
=0,
\end{gather}
where $\varpi_{j}$ is any finite (potentially empty)
expression not involving $b_{u}$ such that all 
$b_{u}$ in \eqref{eq:cyclotomic-bmw-blob} are on the same strand.
\end{definition}

\begin{remark}\label{remark:blob-bmw}
Cyclotomic versions of BMW and Brauer algebras have appeared in the literature:
\begin{enumerate}

\setlength\itemsep{0.15cm}

\item For $g=0$ they are of course just the BMW respectively 
Brauer algebra.

\item For $g=1$ the definition goes back to
\cite{HaOl-cyclotomic-bmw}.

\end{enumerate}
\end{remark}

\begin{lemma}\label{lemma:bmw-basis-2}
For an admissible choice of parameters, the algebra 
$\setstuff{BMW}_{g,n}^{\dpar,\bpar}(\cpar,\qvar,\avar)$ 
has a $\KK$-basis in bijection with clapped, blobbed
perfect matchings of 
$2n$ points of genus  
$g$ where each strand has at 
most $\dvar_{u}$ blobs of the corresponding type.
\end{lemma}

In particular, the dimension of  $\setstuff{BMW}_{g,n}^{\dpar,\bpar}(\cpar,\qvar,\avar)$
is also given by the formula in \eqref{eq:brauer-dim}, namely 
it is $\dim_{\KK}\setstuff{BMW}_{g,n}^{\dpar,\bpar}(\cpar,\qvar,\avar)
=(\bbvar_{g,\dpar})^{n}(2n-1)!!$.

\begin{proof}
By the procedure described in 
\fullref{lemma:brauer-basis2} we see that the 
corresponding images in 
$\setstuff{BMW}_{g,n}^{\dpar,\bpar}(\cpar,\qvar,\avar)$ are
linearly independent since they are so in the Brauer 
specialization. To see that these span we use 
\fullref{lemma:blob-bmw-rels}.
\end{proof}

For the cell structure we now combine 
the one from \fullref{subsection:handlebody-blob} with the 
one from \fullref{subsection:BMW} 
(in particular, using positive lifts). The two
differences worthwhile spelling out are first 
that \eqref{eq:dots-capscups} 
tell us to demand that blobs on caps and cups are 
to the right of any of their Morse points. 
Let $\sand=
\KK\setstuff{B}_{g,\lambda}^{+,\dpar,\bpar}$ denote the 
blobbed positive braid monoid in $\lambda$ 
strands, {\ie} the positive braid monoid with blobs satisfying the relations in 
\fullref{lemma:blob-bmw-rels}
and \eqref{eq:cyclotomic-bmw-blob}. As usual we fix the monoid element basis as our sandwiched basis.
Because of these relations
we need to choose an order of blobs and crossings. We choose 
to put blobs below all crossings, with blobs on the first strand 
below blobs on the second strand {\etc}
We leave the details to the reader, and rather give an example:
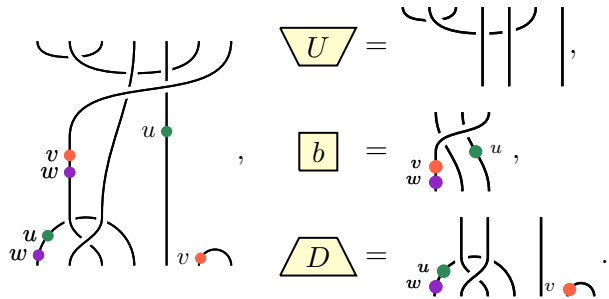
\begin{gather*}
\scalebox{0.85}{$\begin{tikzpicture}[anchorbase,scale=1]
\draw[usual,crossline] (1,-0.5) to[out=270,in=180] (1.5,-0.75) 
to[out=0,in=270] (2,-0.5);
\draw[usual,crossline] (1.5,-0.5) to[out=270,in=180] (2.5,-1) to[out=0,in=270] (3.5,-0.5);
\draw[usual,crossline,blobbed={0.6}{u}{spinach}] (3,-4) to[out=90,in=270] (3,-0.5);
\draw[usual,crossline] (2.5,-0.5) to[out=270,in=90] (2,-3.2);
\draw[usual,crossline,blobbed={0.05}{w}{orchid},blobbed={0.2}{u}{spinach}] (1,-4) to[out=90,in=180] (1.75,-3.25) 
to[out=0,in=90] (2.5,-4);
\draw[usual,crossline,blobbed={0.3}{w}{orchid},blobbed={0.35}{v}{tomato}] (2,-4) to[out=90,in=270] (1.5,-3.25) to (1.5,-2.0) to[out=90,in=270] (4,-0.5);
\draw[usual,crossline] (1.5,-4) to[out=90,in=270] (2,-3.25) to (2,-3.2);
\draw[usual,crossline,blobbed={0.1}{v}{tomato}] (3.5,-4) to[out=90,in=180] (3.75,-3.75) 
to[out=0,in=90] (4,-4);
\end{tikzpicture}$}
,\quad
\begin{aligned}
\begin{tikzpicture}[anchorbase,scale=1]
\draw[mor] (0,1) to (0.25,0.5) to (0.75,0.5) to (1,1) to (0,1);
\node at (0.5,0.75){$U$};
\end{tikzpicture}
&=
\begin{tikzpicture}[anchorbase,scale=0.7,tinynodes]
\draw[usual,crossline] (1,0) to[out=270,in=180] (1.5,-0.25) 
to[out=0,in=270] (2,0);
\draw[usual,crossline] (1.5,0) to[out=270,in=180] (2.5,-0.5) to[out=0,in=270] (3.5,0);
\draw[usual,crossline] (2.5,0) to[out=270,in=90] (2.5,-1.5);
\draw[usual,crossline] (3,0) to[out=270,in=90] (3,-1.5);
\draw[usual,crossline] (4,0) to[out=270,in=90] (4,-1.5);
\end{tikzpicture}
,\\
\begin{tikzpicture}[anchorbase,scale=1]
\draw[mor] (0.25,0) to (0.25,0.5) to (0.75,0.5) to (0.75,0) to (0.25,0);
\node at (0.5,0.25){$b$};
\node at (0.8,0.25){\phantom{$b$}};
\end{tikzpicture}
&=
\begin{tikzpicture}[anchorbase,scale=0.7,tinynodes]
\draw[usual,crossline] (1.5,0) to[out=90,in=270] (1,1.5);
\draw[usual,crossline,rblobbed={0.5}{u}{spinach}] (2,0) to[out=90,in=270] (1.5,1.5);
\draw[usual,crossline,blobbed={0.05}{w}{orchid},blobbed={0.2}{v}{tomato}] (1,0) to (1,0.75) to[out=90,in=270] (2,1.5);
\end{tikzpicture}
,
\\
\begin{tikzpicture}[anchorbase,scale=1]
\draw[mor] (0,-0.5) to (0.25,0) to (0.75,0) to (1,-0.5) to (0,-0.5);
\node at (0.5,-0.25){$D$};
\end{tikzpicture}
&=
\begin{tikzpicture}[anchorbase,scale=0.7,tinynodes]
\draw[usual,crossline,blobbed={0.05}{w}{orchid},blobbed={0.2}{u}{spinach}] (1,0) to[out=90,in=180] (1.75,0.75) 
to[out=0,in=90] (2.5,0);
\draw[usual,crossline] (2,0) to[out=90,in=270] (1.5,0.75) to (1.5,1.5);
\draw[usual,crossline] (1.5,0) to[out=90,in=270] (2,0.75) to (2,1.5);
\draw[usual,crossline] (3,0) to[out=90,in=270] (3,1.5);
\draw[usual,crossline,blobbed={0.1}{v}{tomato}] (3.5,0) to[out=90,in=180] (3.75,0.25) 
to[out=0,in=90] (4,0);
\end{tikzpicture}
.
\end{aligned}
\end{gather*}

The next two statements follow as before. The proofs are omitted.

\begin{proposition}
For a choice of admissible parameters 
the above defines an involutive sandwich
cell datum for $\setstuff{BMW}_{g,n}(\cpar,\qvar,\avar)$.\qed
\end{proposition}

\begin{theorem}\label{theorem:brauerblob}
Let $\KK$ be a field, and choose admissible parameters.
\begin{enumerate}

\item If $\cpar\neq 0$, or $\cpar=0$ and $\lambda\neq 0$ is odd, 
then all $\lambda\in\Lambda$ are apexes. In the remaining case, 
$\cpar=0$ and $\lambda=0$ (this only happens if $n$ is even), all $\lambda\in\Lambda-\{0\}$ are apexes, but $\lambda=0$ is not an apex.

\item The simple $\setstuff{BMW}_{g,n}(\cpar,\qvar,\avar)$-modules of 
apex $\lambda\in\Lambda$ 
are parameterized by simple modules of $\KK\setstuff{B}_{g,\lambda}^{+,\dpar,\bpar}$.

\item The simple $\setstuff{BMW}_{g,n}(\cpar,\qvar,\avar)$-modules of 
apex $\lambda\in\Lambda$ can be constructed as 
the simple heads of
$\mathrm{Ind}_{\KK\setstuff{B}_{g,\lambda}^{+,\dpar,\bpar}}^{\setstuff{BMW}_{g,n}(\cpar,\qvar,\avar)}(K)$, 
where $K$ runs over (equivalence classes of) 
simple $\KK\setstuff{B}_{g,\lambda}^{+,\dpar,\bpar}$-modules.\qed

\end{enumerate}
\end{theorem}

\begin{example}
For $g=0$ \fullref{theorem:brauerblob} is, 
of course, a BMW version of
\fullref{theorem:brauer}. For $g=1$ and the Brauer 
specialization this recovers 
\cite[Appendix 6]{BoCoDeVi-decomposition-cyclotomic-brauer}, where 
$\KK\setstuff{B}_{1,\lambda}^{+,\dpar,\bpar}$ is a complex reflection 
group of type $G(\dvar,1,\lambda)$.
In the semisimple case the parametrization of simples of
$\KK\setstuff{B}_{1,\lambda}^{+,\dpar,\bpar}$ (and thus, of $\setstuff{Br}_{1,n}(\cvar)$ per apex $\lambda$) is given 
by $\dvar$-multipartitions of $\lambda$. 
\end{example}

\section{Handlebody Hecke and Ariki--Koike algebras}\label{section:hecke}

We define handlebody Hecke algebras as quotients of handlebody braid groups.  
All algebras can alternatively be defined as quotients of an appropriate
BMW algebra from \fullref{section:brauer}. 

\subsection{Handlebody Hecke algebras}\label{subsection:heckealgebras}

Fix an invertible scalar $\qvar\in\KK$ and recall the definitions regarding
the handlebody braid groups $\setstuff{B}_{g,n}$ from \fullref{subsection:braids}. 

\begin{definition}\label{definition:hecke}
The \emph{handlebody Hecke algebra} 
(in $n$ strands and of genus $g$) $\setstuff{H}_{g,n}$
is the algebra whose underlying free $\KK$-module is
spanned by isotopy classes of all handlebody braid diagrams, with multiplication 
given by concatenation of diagrams, modulo the \emph{skein relation} 
\begin{gather}\label{eq:hecke-skein}
\begin{tikzpicture}[anchorbase,scale=0.7,tinynodes]
\draw[usual,crossline] (0.5,0) to[out=90,in=270] (0,0.75);
\draw[usual,crossline] (0,0) to[out=90,in=270] (0.5,0.75);
\end{tikzpicture}
-
\begin{tikzpicture}[anchorbase,scale=0.7,tinynodes,xscale=-1]
\draw[usual,crossline] (0.5,0) to[out=90,in=270] (0,0.75);
\draw[usual,crossline] (0,0) to[out=90,in=270] (0.5,0.75);
\end{tikzpicture}
=(\qvar-\qvar^{-1})\cdot
\begin{tikzpicture}[anchorbase,scale=0.7,tinynodes]
\draw[usual,crossline] (0.5,0) to (0.5,0.75);
\draw[usual,crossline] (0,0) to  (0,0.75);
\end{tikzpicture}
. 
\end{gather}
\end{definition}

Algebraically, 
$\setstuff{H}_{g,n}$ is the quotient 
of $\KK\setstuff{B}_{g,n}$ by the two-sided ideal generated by the elements
\begin{gather*}
(\beta_{i}+\qvar)(\beta_{i}-\qvar^{-1}), 
\end{gather*}
for $i=1,\dots,n-1$.
In this interpretation, we write $H_{i}$ for the image of 
$\beta_{i}$ in the quotient, but keep the notation $\tau_{u}$ for the others.
In defining $\setstuff{H}_{g,n}$ we only demand that
the $H_{i}$ satisfy a quadratic relation which 
is equivalent to either of
\begin{gather*}
H_{i}^{2}=(\qvar-\qvar^{-1})H_{i}+1,
\quad
H_{i}^{-1}=H_{i}-(\qvar-\qvar^{-1}).
\end{gather*}
Note that the coils $\tau_{u}$ do not satisfy 
a polynomial relation. Hence, as for the handlebody Coxeter group, 
$\setstuff{H}_{g,n}$ does not embed into 
a Hecke algebra of type A.

\begin{remark}\label{remark:hecke}
With respect to \fullref{definition:hecke} we note:	
\begin{enumerate}

\item In case $g=0$ the algebra $\setstuff{H}_{0,n}$ is the type A Hecke algebra.

\item For $g=1$ the algebra $\setstuff{H}_{1,n}$ is the extended affine Hecke algebra of type A. 

\item The algebra $\setstuff{H}_{g,n}$  
has been studied in \cite{Ba-braid-handlebodies}, 
but not much appears to be known.
\end{enumerate}
\end{remark}

For all $w\in\setstuff{S}_{n}$ we choose a 
reduced expression $\underline{w}=s_{i_{k}}\dots 
s_{i_{1}}\in\setstuff{S}_{n}$ 
once and for all. We define $H_{w}=H_{\underline{w}}=H_{i_{k}}\dots H_{i_{1}}$. 
We still have the Jucys--Murphy elements $\jm_{u,i}$ 
from \eqref{eq:jm-elements}, and the analog of \fullref{proposition:jm-elements-basis} is the following.

\begin{proposition}\label{proposition:heckejm-elements-basis}
The set
\begin{gather}\label{eq:jm-basis-hecke1}
\left\{ 
\jm_{u_{1},i_{1}}^{a_{1}}\dots 
\jm_{u_{m},i_{m}}^{a_{m}}H_{w} 
\,\middle\vert\,
\begin{gathered}
w\in\setstuff{S}_{n},
m\in\N,
\bsym{a}\in\Z^{m},
\\
(\bsym{u},\bsym{i})\in(\{1,\dots,g\}\times\{1,\dots,n\})^{m},
i_{1}\leq\dots\leq i_{m}
\end{gathered}
\right\}
\end{gather}
is a $\KK$-basis of $\setstuff{H}_{g,n}$.
\end{proposition}

\fullref{proposition:heckejm-elements-basis} can be seen 
as a higher genus version of
\cite[Equation (3.10)]{ArKo-hecke-algebra}.

\begin{proof}
That the set in \eqref{eq:jm-basis-hecke1} spans can be proven {\muta} as 
in \fullref{lemma:jm-elements-span}. 
That is, we use the relations in \eqref{eq:jm-relations} 
together with the following immediate consequences 
of \eqref{eq:jm-relations} and the skein relation \eqref{eq:hecke-skein}:
\begin{gather}\label{eq:hecke-pull}
\begin{gathered}
H_{i-1}\jm_{u,i}
=
\jm_{u,i-1}H_{i-1}
-(\qvar-\qvar^{-1})\jm_{u,i}
,\quad
H_{i}^{-1}\jm_{u,i}
=
\jm_{u,i+1}H_{i}^{-1}
+(\qvar-\qvar^{-1})\jm_{u,i}
,\\
H_{i-1}^{-1}\jm_{u,i}^{-1}
=
\jm_{u,i-1}^{-1}H_{i-1}^{-1}
+(\qvar-\qvar^{-1})\jm_{u,i}^{-1}
,\quad
H_{i}\jm_{u,i}^{-1}
=
\jm_{u,i+1}^{-1}H_{i}
-(\qvar-\qvar^{-1})\jm_{u,i}^{-1}.
\end{gathered}
\end{gather}
Moreover, 
\fullref{proposition:jm-elements-basis} shows the elements in \eqref{eq:jm-basis-hecke1} are 
linearly independent if $\qvar=1$ 
(note that the handlebody Coxeter group $\setstuff{W}_{g,n}$
can be obtained from $\setstuff{H}_{g,n}$
by specializing $\qvar=1$), which implies that they are linearly independent for generic $\qvar$.
\end{proof}

Combining \fullref{proposition:heckejm-elements-basis} 
with the respective 
statement for the handlebody Coxeter 
group \fullref{proposition:jm-elements-basis} we get 
the following.

\begin{corollary}
For $\qvar=1$
the map given by
$\tau_{u}\mapsto t_{u},H_{i}\mapsto s_{i}$
is an isomorphism of $\KK$-algebras $\setstuff{H}_{g,n}\xrightarrow{\cong}\KK\setstuff{W}_{g,n}$.
\end{corollary}

In order to define a sandwich cell datum we first note that
there is an antiinvolution $(\placeholder)^{\star}$ 
on $\setstuff{H}_{g,n}$ given by mirroring along a horizontal 
axis, but not changing crossings ({\eg} positive crossings
as in \eqref{eq:positive-braid} remain positive).
Moreover, we now recall the 
cellular basis of $\setstuff{H}_{0,n}$ from \cite{Mu-typea-hecke}, see also \cite[Chapter 3]{Ma-hecke-schur}. 

\begin{remark}
The cellular basis we are going to recall is not the Kazhdan--Lusztig 
basis, but the so-called \emph{Murphy basis} or
\emph{standard basis}. This basis 
has the advantage that it has a known generalization to the case of $g=1$, 
see \cite{DiJaMa-cyclotomic-q-schur}. 
On the other hand, we are not aware of a generalization 
of Kazhdan--Lusztig theory to higher genus.
\end{remark}

Recall that a \emph{partition} $\lambda$ of $n$ of length 
$\ell(\lambda)=k$ is a non-increasing sequence $\lambda_{1}\geq\lambda_{2}\geq\dots\geq\lambda_{k}>0$ of 
integers adding up to $|\lambda|=n$. 
Associated to $\lambda$ is its \emph{Young diagram} $[\lambda]$
which we will illustrate using the English convention.
Moreover, any filling $T$ of $[\lambda]$ with non-repeating 
numbers from $\{1,\dots,n\}$ is called a \emph{tableaux} 
of shape $\lambda$. Such a $T$ is further 
called \emph{standard} if its entries increase along rows and columns.
The \emph{canonical tableaux} $T_{c}$ is the filling where 
the numbers increase 
from left to right and top to bottom.

To a standard tableaux $T$ of shape $\lambda$ we can associate 
a down diagram $D$ 
using the classical \emph{permute entries strategy}.
We also need the Young subgroup $\setstuff{S}_{\lambda}^{row}=\prod_{i}\setstuff{S}_{r(i)}$ 
being the row stabilizer of $\lambda$, 
{\ie} $r(i)$ is the length of the $i$th row of $[\lambda]$.
Then let $H_{\lambda}=\prod_{i}(\sum_{w\in\setstuff{S}_{r(i)}}H_{w})$ 
(note that the bracketed terms commute, so 
the order is not important)
and let $H_{T}$ be the positive braid lifts 
of (a choice of)
minimal permutation from $T$ to $T_{c}$, 
identifying $s_{i}$ with the transposition $(i,i+1)$. We then let 
$D=D(T)=H_{T}$, $U=U(T^{\prime})=D(T^{\prime})^{\star}$ and 
$c_{D,U}^{\lambda}=DH_{\lambda}U$. 
Note that the middle $H_{\lambda}$ of $c_{D,U}^{\lambda}$ is a 
quasi-idempotent, {\ie} $H_{\lambda}^{2}=rH_{\lambda}$ for some 
potentially not invertible
$r\in\KK$.
Note also that $H_{0,n}$ is a subalgebra of $H_{g,n}$, 
so we can use the above construction for $H_{g,n}$ as well.

\begin{example}
The above is best illustrated in an example. For $\lambda=(4,2)$ 
we have $\setstuff{S}_{\lambda}^{row}=\setstuff{S}_{4}\times\setstuff{S}_{2}$ and
$H_{\lambda}=(1+H_{1}+\dots+H_{w_{0}})(1+H_{5})$, where 
$w_{0}$ is the longest element in $S_{4}$ written in terms of the 
generators $s_{1}$, $s_{2}$ and $s_{3}$. Moreover, for $T$ 
respectively $T^{\prime}$ as below 
we construct $D=D(T)$ respectively $U=D(T^{\prime})^{\star}$ by
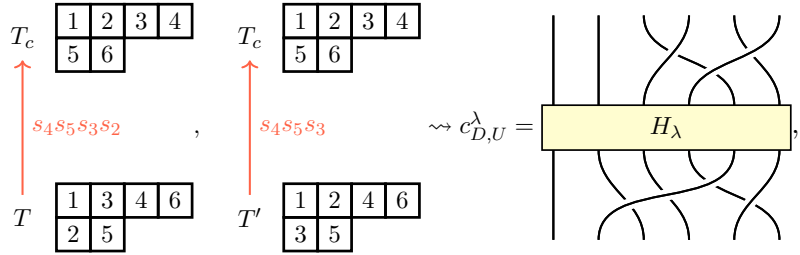
\begin{gather*}
\scalebox{0.85}{$\begin{tikzpicture}[anchorbase,scale=0.7]
\node at (-1.5,0) {$T_{c}$};
\draw[very thick] (0,0) rectangle node[pos=0.5]{$1$} (-0.75,0.75);
\draw[very thick] (0,0) rectangle node[pos=0.5]{$5$} (-0.75,-0.75);
\draw[very thick] (0,0) rectangle node[pos=0.5]{$2$} (0.75,0.75);
\draw[very thick] (0.75,0) rectangle node[pos=0.5]{$3$} (1.5,0.75);
\draw[very thick] (1.5,0) rectangle node[pos=0.5]{$4$} (2.25,0.75);
\draw[very thick] (0,-0.75) rectangle node[pos=0.5]{$6$} (0.75,0);
\node at (-1.5,-4) {$T$};
\draw[very thick] (0,-4) rectangle node[pos=0.5]{$1$} (-0.75,-3.25);
\draw[very thick] (0,-4) rectangle node[pos=0.5]{$2$} (-0.75,-4.75);
\draw[very thick] (0,-4) rectangle node[pos=0.5]{$3$} (0.75,-3.25);
\draw[very thick] (0.75,-4) rectangle node[pos=0.5]{$4$} (1.5,-3.25);
\draw[very thick] (1.5,-4) rectangle node[pos=0.5]{$6$} (2.25,-3.25);
\draw[very thick] (0,-4.75) rectangle node[pos=0.5]{$5$} (0.75,-4);
\draw[usual,tomato,->] (-1.5,-3.5) to (-1.5,-2)node[right]{$s_{4}s_{5}s_{3}s_{2}$} to (-1.5,-0.5);
\end{tikzpicture}
,\quad
\begin{tikzpicture}[anchorbase,scale=0.7]
\node at (-1.5,0) {$T_{c}$};
\draw[very thick] (0,0) rectangle node[pos=0.5]{$1$} (-0.75,0.75);
\draw[very thick] (0,0) rectangle node[pos=0.5]{$5$} (-0.75,-0.75);
\draw[very thick] (0,0) rectangle node[pos=0.5]{$2$} (0.75,0.75);
\draw[very thick] (0.75,0) rectangle node[pos=0.5]{$3$} (1.5,0.75);
\draw[very thick] (1.5,0) rectangle node[pos=0.5]{$4$} (2.25,0.75);
\draw[very thick] (0,-0.75) rectangle node[pos=0.5]{$6$} (0.75,0);
\node at (-1.5,-4) {$T^{\prime}$};
\draw[very thick] (0,-4) rectangle node[pos=0.5]{$1$} (-0.75,-3.25);
\draw[very thick] (0,-4) rectangle node[pos=0.5]{$3$} (-0.75,-4.75);
\draw[very thick] (0,-4) rectangle node[pos=0.5]{$2$} (0.75,-3.25);
\draw[very thick] (0.75,-4) rectangle node[pos=0.5]{$4$} (1.5,-3.25);
\draw[very thick] (1.5,-4) rectangle node[pos=0.5]{$6$} (2.25,-3.25);
\draw[very thick] (0,-4.75) rectangle node[pos=0.5]{$5$} (0.75,-4);
\draw[usual,tomato,->] (-1.5,-3.5) to (-1.5,-2)node[right]{$s_{4}s_{5}s_{3}$} to (-1.5,-0.5);
\end{tikzpicture}
\rightsquigarrow
c_{D,U}^{\lambda}=
\begin{tikzpicture}[anchorbase,scale=0.7]
\draw[usual,crossline] (5,0) to[out=90,in=270] (3,2);
\draw[usual,crossline] (4,0) to[out=90,in=270] (5,2);
\draw[usual,crossline] (3,0) to[out=90,in=270] (2,2);
\draw[usual,crossline] (2,0) to[out=90,in=270] (1,2);
\draw[usual,crossline] (1,0) to[out=90,in=270] (4,2);
\draw[usual,crossline] (0,0) to[out=90,in=270] (0,2);
\draw[usual,crossline] (5,3) to[out=90,in=270] (4,5);
\draw[usual,crossline] (4,3) to[out=90,in=270] (2,5);
\draw[usual,crossline] (3,3) to[out=90,in=270] (5,5);
\draw[usual,crossline] (2,3) to[out=90,in=270] (3,5);
\draw[usual,crossline] (1,3) to[out=90,in=270] (1,5);
\draw[usual,crossline] (0,3) to[out=90,in=270] (0,5);
\draw[mor] (-0.25,2) rectangle node[pos=0.5]{$H_{\lambda}$} (5.25,3);
\end{tikzpicture}$}
,
\end{gather*}
where we illustrated $H_{\lambda}$ as a box.
\end{example}

We let $\Lambda=(\Lambda,\leq_{d})$ be the set of all partitions of $n$, 
which are ordered by the dominance order $\leq_{d}$ (we use the convention 
from \cite[Section 3.1]{Ma-hecke-schur}).
The set $M_{\lambda}$ is the set of all standard tableaux 
of shape $\lambda$.
The middle part is $\sand=
\setstuff{L}_{g,n}H_{\lambda}$, where $\setstuff{L}_{g,n}$ 
is the subalgebra of $\setstuff{H}_{g,n}$ spanned 
by the Jucys--Murphy elements. The sandwiched basis of 
$\setstuff{L}_{g,n}H_{\lambda}$ that we use is \eqref{eq:jm-basis-hecke1} restricted to $\setstuff{L}_{g,n}$ and then composed with $H_{\lambda}$.
We then let $c_{D,b,U}^{\lambda}=D(T)bU(T^{\prime})$ 
where $b=b(\bsym{a},\bsym{u},\bsym{i})$ runs over elements 
of the form $\jm_{u_{1},i_{1}}^{a_{1}}\dots 
\jm_{u_{m},i_{m}}^{a_{m}}H_{\lambda}$ with the 
same indices as in \eqref{eq:jm-basis-hecke1}. We then get
\begin{gather}\label{eq:hecke-basis}
\{c_{D,b,U}^{\lambda}\mid\lambda\in\Lambda,D,U\in M_{\lambda},
b\in\sand\}.
\end{gather} 

\begin{proposition}
The above defines a sandwich
cell datum for $\setstuff{H}_{g,n}$.
\end{proposition}

\begin{proof}
To argue that \eqref{eq:hecke-basis} spans
$\setstuff{H}_{g,n}$ we use \cite[Proposition 3.16]{Ma-hecke-schur} as follows.
Instead of the elements $H_{w}\in\setstuff{H}_{0,n}\hookrightarrow
\setstuff{H}_{g,n}$ in \eqref{eq:jm-basis-hecke1} we could 
also use the Murphy basis $DH_{\lambda}U$ (as explained above, {\ie} 
\eqref{eq:hecke-basis} for $g=0$). Then, by 
\eqref{eq:jm-relations} and \eqref{eq:hecke-pull}, we can push 
the $D$ to the left, that is, 
\begin{gather*}
\jm_{u_{1}^{},i_{1}}^{a_{1}}\dots 
\jm_{u_{m},i_{m}}^{a_{m}}DH_{\lambda}U
=D\jm_{u_{1},i_{1}^{\prime}}^{a_{1}}\dots 
\jm_{u_{m},i_{m}^{\prime}}^{a_{m}}H_{\lambda}U+\text{error terms}
\end{gather*} 
which shows that the $c_{D,b,U}^{\lambda}$ span.
To show that \eqref{eq:hecke-basis} is a basis we observe that,
left aside error terms, the process of moving $D$ to the left 
does not change the number $m$, 
the powers $\bsym{a}$, and associated cores $\bsym{u}$, while 
$\bsym{i}$ changes, but the property 
$i_{1}\leq\dots\leq i_{m}$ is preserved. Hence, 
a counting argument ensures that \eqref{eq:jm-basis-hecke1} and 
\eqref{eq:hecke-basis} have the same (finite) size per fixed 
$(m,\bsym{a},\bsym{u})$, 
so \eqref{eq:hecke-basis} is indeed a basis.

It remains to prove \eqref{eq:cell-mult} as all other claimed 
properties hold by construction. Let us write 
$\bsym{L}=\jm_{u_{1},i_{1}}^{a_{1}}\dots 
\jm_{u_{m},i_{m}}^{a_{m}}$ for short, omitting the precise indices.
We calculate: 
\begin{align*}
D\bsym{L}H_{\lambda}U
D^{\prime}\bsym{L}^{\prime}H_{\mu}U^{\prime}
&=
\bsym{L}^{\dagger}DH_{\lambda}U
D^{\prime}H_{\mu}U^{\prime}(\bsym{L}^{\prime})^{\dagger}
+\text{error terms}
\\
&\equiv_{\leq_{\Lambda}}
r(U,D^{\prime})
\bsym{L}^{\dagger}DH_{\max_{\leq_{d}}(\lambda,\mu)}
U^{\prime}(\bsym{L}^{\prime})^{\dagger}
+\text{error terms}
\\
&=
D\bsym{L}H_{\max_{\leq_{d}}(\lambda,\mu)}
\bsym{L}^{\prime}U^{\prime}
+\text{error terms}
\\
&=
D\bsym{L}
(\bsym{L}^{\prime})^{\dagger}H_{\max_{\leq_{d}}(\lambda,\mu)}U^{\prime}
+\text{error terms}
\\
&\equiv_{\leq_{\Lambda}}
D\bsym{L}
(\bsym{L}^{\prime})^{\dagger}H_{\max_{\leq_{d}}(\lambda,\mu)}U^{\prime}
=
D\bsym{L}
F\bsym{L}^{\prime}H_{\max_{\leq_{d}}(\lambda,\mu)}U^{\prime}
,
\end{align*}
where the $\dagger$ indicates
a shift of the indices, as before, and $F$ 
is some product of Jucys--Murphy elements. 
In this calculation we used 
\eqref{eq:jm-relations} 
and \eqref{eq:hecke-pull} several times and also:
The crucial first congruence 
follows from the classical case, see {\eg} \cite[Theorem 3.20]{Ma-hecke-schur}.
The final equality reorders the 
Jucys--Murphy elements to match the expression in 
\eqref{eq:cell-mult}, while the second congruence uses 
the observation that the error terms in \eqref{eq:hecke-pull} 
have a lower number of $H_{i}$, so correspond to Young diagrams with shorter rows.
\end{proof} 

Let $e\in\N$ be the smallest number such that 
$[e]_{\qvar}=0$, or let $e=\infty$ if no such 
$e$ exists. An element $\lambda\in\Lambda$ is called 
\emph{$e$-restricted} if $\lambda_{i}-\lambda_{i+1}<e$.

\begin{theorem}\label{theorem:cellhecke}
Let $\KK$ be a field.
\begin{enumerate}

\item An element $\lambda\in\Lambda$ is an apex 
if and only if $\lambda$ is $e$-restricted.

\item The simple $\setstuff{H}_{n,g}$-modules of 
apex $\lambda\in\Lambda$ 
are parameterized by simple modules of $\setstuff{L}_{g,n}H_{\lambda}$.

\item The simple $\setstuff{H}_{n,g}$-modules of 
apex $\lambda\in\Lambda$ can be constructed as 
the simple heads of
$\mathrm{Ind}_{\setstuff{L}_{g,n}H_{\lambda}}^{\setstuff{H}_{n,g}}(K)$, 
where $K$ runs over (equivalence classes of) 
simple $\setstuff{L}_{g,n}H_{\lambda}$-modules.

\end{enumerate}
\end{theorem}

\begin{proof}
Claims (b) and (c) follow immediately from the abstract theory, as 
in the cases discussed before. The statement (a) follows 
because whether the form $\phi^{\lambda}$ is constant 
zero or not can be detected on the component where the Jucys--Murphy 
elements are trivial, meaning those $c_{D,b,U}^{\lambda}$ 
with $b=H_{\lambda}$.
Hence, the classical theory applies, see {\eg} 
\cite[Section 3.4]{Ma-hecke-schur}.
\end{proof}

\begin{example}
The cases $g=0$ and $g=1$ of \fullref{theorem:cellhecke}
are well-known:
\begin{enumerate}

\item We have $\setstuff{L}_{0,n}H_{\lambda}=\KK$, so the above 
is the classical parametrization 
of simple $\setstuff{H}_{0,n}$-modules, see {\eg} 
\cite[Section 3.4]{Ma-hecke-schur}.

\item For $g=1$ \fullref{theorem:cellhecke}
can be matched with {\eg} \cite[Theorem 5.8]{KoXi-affine-cellular}.

\end{enumerate}
\end{example}

\subsection{Cyclotomic handlebody Hecke algebras}\label{subsection:ak-hecke}

We keep the terminology from 
the previous sections. In order 
to define and work with cyclotomic 
quotients, we use \emph{blob diagrams of braids} instead of coils.

Denote by $\setstuff{H}_{g,n}^{+}$ the 
subalgebra of $\setstuff{H}_{g,n}$ 
generated by $\setstuff{H}_{0,n}$ and positive coils.
Imitating \fullref{subsection:handlebody-blob} we 
introduce elements $b_{1},\dots,b_{g}$ in $\setstuff{H}_{g,n}^{+}$ 
using the same pictures as in \eqref{eq:blobs}:
\begin{gather*}
\tau_{u}=
\begin{tikzpicture}[anchorbase,scale=0.7,tinynodes]
\draw[pole,crosspole] (-0.5,0) to[out=90,in=270] (-0.5,1.5);
\draw[usual,crossline] (1,0)node[below,black]{$1$} to[out=90,in=270] (-0.25,0.75);
\draw[pole,crosspole] (0,0)node[below,black]{$u$} to[out=90,in=270] (0,1.5)node[above,black,yshift=-3pt]{$u$};
\draw[pole,crosspole] (0.5,0)node[below,black]{$v$} to[out=90,in=270] (0.5,1.5)node[above,black,yshift=-3pt]{$v$};
\draw[usual,crossline] (-0.25,0.75) to[out=90,in=270] (1,1.5)node[above,black,yshift=-3pt]{$1$};
\end{tikzpicture}
\rightsquigarrow
b_{u}=
\begin{tikzpicture}[anchorbase,scale=0.7,tinynodes]
\draw[usual,blobbed={0.5}{u}{spinach}] (1,0)node[below,black]{$1$} to[out=90,in=270]
(1,1.5)node[above,black,yshift=-3pt]{$1$};
\end{tikzpicture}
,\quad
\tau_{v}=
\begin{tikzpicture}[anchorbase,scale=0.7,tinynodes]
\draw[pole,crosspole] (-0.5,0) to[out=90,in=270] (-0.5,1.5);
\draw[usual,crossline] (1,0)node[below,black]{$1$} to[out=90,in=270] (0.25,0.75);
\draw[pole,crosspole] (0,0)node[below,black]{$u$} to[out=90,in=270] (0,1.5)node[above,black,yshift=-3pt]{$u$};
\draw[pole,crosspole] (0.5,0)node[below,black]{$v$} to[out=90,in=270] (0.5,1.5)node[above,black,yshift=-3pt]{$v$};
\draw[usual,crossline] (0.25,0.75) to[out=90,in=270] (1,1.5)node[above,black,yshift=-3pt]{$1$};
\end{tikzpicture}
\rightsquigarrow
b_{v}=
\begin{tikzpicture}[anchorbase,scale=0.7,tinynodes]
\draw[usual,blobbed={0.5}{v}{tomato}] (1,0)node[below,black]{$1$} to[out=90,in=270]
(1,1.5)node[above,black,yshift=-3pt]{$1$};
\end{tikzpicture}
.
\end{gather*}
Although blobs are defined on the first 
strand from the left, one can define 
blobs on other strands exactly as in \eqref{eq:blob-right}.
These are the Jucys--Murphy elements 
from \fullref{definition:jm-elements} which 
we denote by $b_{u,i}$ to emphasize that 
they are not necessarily invertible.

Using the skein relation \eqref{eq:hecke-skein},
the same calculations as in \fullref{lemma:blob-rels} give:

\begin{lemma}\label{lemma:blob-hecke-rels}
Blobs satisfy relations \eqref{eq:blobslides} 
and \eqref{eq:height-blob-switch} and
\begin{gather}\label{eq:hecke:more-slides}
\begin{aligned}
\begin{tikzpicture}[anchorbase,scale=0.7,tinynodes]
\draw[usual,crossline] (0.5,0) to[out=90,in=270] (0,0.75);
\draw[usual,crossline,rblobbed={0.85}{u}{spinach}] (0,0) to[out=90,in=270] (0.5,0.75);
\end{tikzpicture}
\hspace{-0.2cm}&=\hspace{-0.2cm}
\begin{tikzpicture}[anchorbase,scale=0.7,tinynodes]
\draw[usual,crossline] (0.5,0) to[out=90,in=270] (0,0.75);
\draw[usual,crossline,blobbed={0.15}{u}{spinach}] (0,0) to[out=90,in=270] (0.5,0.75);
\end{tikzpicture}
+(\qvar-\qvar^{-1})\cdot
\begin{tikzpicture}[anchorbase,scale=0.7,tinynodes]
\draw[usual,crossline] (0,0) to (0,0.75);
\draw[usual,crossline,rblobbed={0.5}{u}{spinach}] (0.5,0) to (0.5,0.75);
\end{tikzpicture}
,\\
\begin{tikzpicture}[anchorbase,scale=0.7,tinynodes]
\draw[usual,crossline,rblobbed={0.12}{u}{spinach}] (0.5,0) to[out=90,in=270] (0,0.75);
\draw[usual,crossline] (0,0) to[out=90,in=270] (0.5,0.75);
\end{tikzpicture}
\hspace{-0.2cm}&=\hspace{-0.2cm}
\begin{tikzpicture}[anchorbase,scale=0.7,tinynodes]
\draw[usual,crossline,blobbed={0.87}{u}{spinach}] (0.5,0) to[out=90,in=270] (0,0.75);
\draw[usual,crossline] (0,0) to[out=90,in=270] (0.5,0.75);
\end{tikzpicture}
+(\qvar-\qvar^{-1})\cdot
\begin{tikzpicture}[anchorbase,scale=0.7,tinynodes]
\draw[usual,crossline] (0,0) to (0,0.75);
\draw[usual,crossline,rblobbed={0.5}{u}{spinach}] (0.5,0) to (0.5,0.75);
\end{tikzpicture}
,
\\
\begin{tikzpicture}[anchorbase,scale=0.7,tinynodes]
\draw[usual,crossline,blobbed={0.15}{u}{spinach}] (0,0) to[out=90,in=270] (0.5,0.75);
\draw[usual,crossline] (0.5,0) to[out=90,in=270] (0,0.75);
\end{tikzpicture}
&=
\begin{tikzpicture}[anchorbase,scale=0.7,tinynodes]
\draw[usual,crossline,rblobbed={0.85}{u}{spinach}] (0,0) to[out=90,in=270] (0.5,0.75);
\draw[usual,crossline] (0.5,0) to[out=90,in=270] (0,0.75);
\end{tikzpicture}\hspace{-0.2cm}
+(\qvar-\qvar^{-1})\cdot
\hspace{-0.2cm}
\begin{tikzpicture}[anchorbase,scale=0.7,tinynodes]
\draw[usual,crossline,blobbed={0.5}{u}{spinach}] (0,0) to (0,0.75);
\draw[usual,crossline] (0.5,0) to (0.5,0.75);
\end{tikzpicture}
,\\
\begin{tikzpicture}[anchorbase,scale=0.7,tinynodes]
\draw[usual,crossline] (0,0) to[out=90,in=270] (0.5,0.75);
\draw[usual,crossline,blobbed={0.85}{u}{spinach}] (0.5,0) to[out=90,in=270] (0,0.75);
\end{tikzpicture}
&=
\begin{tikzpicture}[anchorbase,scale=0.7,tinynodes]
\draw[usual,crossline] (0,0) to[out=90,in=270] (0.5,0.75);
\draw[usual,crossline,rblobbed={0.15}{u}{spinach}] (0.5,0) to[out=90,in=270] (0,0.75);
\end{tikzpicture}\hspace{-0.2cm}
+(\qvar-\qvar^{-1})\cdot
\hspace{-0.2cm}
\begin{tikzpicture}[anchorbase,scale=0.7,tinynodes]
\draw[usual,crossline,blobbed={0.5}{u}{spinach}] (0,0) to (0,0.75);
\draw[usual,crossline] (0.5,0) to (0.5,0.75);
\end{tikzpicture}
.                     
\end{aligned}
\end{gather}
(In comparison with 
\fullref{lemma:blob-rels}, note the missing cup-cap term.)\qed
\end{lemma}

Of course, \fullref{proposition:heckejm-elements-basis} 
and \eqref{eq:hecke:more-slides} give:

\begin{lemma}\label{lemma:heckeblob-elements-basis}
The set
\begin{gather*}
\left\{ 
b_{k_{1},i_{1}}^{a_{1}}\dots 
b_{k_{m},i_{m}}^{a_{m}}H_{w} 
\,\middle\vert\,
\begin{gathered}
w\in\setstuff{S}_{n},
m\in\N,
\bsym{a}\in\N^{m},
\\
(\bsym{u},\bsym{i})\in(\{1,\dots,g\}\times\{1,\dots,n\})^{m},
i_{1}\leq\dots\leq i_{m}
\end{gathered}
\right\}
\end{gather*}
is a $\KK$-basis of $\setstuff{H}_{g,n}^{+}$.\qed
\end{lemma}

As in the previous sections, we fix cyclotomic parameters 
$\bpar=(\bvar_{u,i})\in\KK^{d_{1}+\dots+d_{g}}$, and a degree vector
$\dpar=(\dvar_{u})\in\N^{g}$.

\begin{definition}\label{definition:cyclotomic-hecke}
We define  the \emph{cyclotomic handlebody Hecke algebra}
(in $n$ strands and of genus $g$)  
$\setstuff{H}_{g,n}^{\dpar,\bpar}$ as the
quotient of $\setstuff{H}_{g,n}^{+}$ by the two-sided ideal 
generated by the \emph{cyclotomic relations}
\begin{gather}\label{eq:hecke-cyclotomic}
(b_{u}-\beta_{u,1})
\varpi_{1}
(b_{u}-\beta_{u,2})
\varpi_{2}
\dots
(b_{u}-\beta_{u,d_{u}-1})
\varpi_{d_{u-1}}
(b_{u}-\beta_{u,d_{u}})
=0,
\end{gather}
where $\varpi_{i}$ is any finite (potentially empty) 
word in $b_{v}$ for $v\neq u$. 
\end{definition}

The relations imply that no strand can 
carry more than $d_{u}$ blobs of type $u$.

\begin{remark}\label{remark:cyclotomic-hecke}
With respect to \fullref{definition:cyclotomic-hecke} we note:	
\begin{enumerate}

\item In case $g=0$ the algebra 
$\setstuff{H}_{0,n}^{\dpar,\bpar}$ 
is the type A Hecke algebra.

\item For $g=1$ the algebra $\setstuff{H}_{1,n}^{\dpar,\bpar}$ 
is the Ariki--Koike algebra as defined and studied in {\eg} \cite{ArKo-hecke-algebra}, \cite{BrMa-hecke} or \cite{Ch-gelfandtzetlin}. 

\item For $g=2$ and $d_{1}=d_{2}=2$ the algebra 
$\setstuff{H}_{2,n}^{\dpar,\bpar}$ can be compared to the two 
boundary Hecke algebra as in \cite{DaRa-two-boundary-hecke}.

\end{enumerate}
Note also that imposing a single relation involving only $\tau_{g}$ 
(this is the first coil counting from the right, {\cf} \eqref{eq:lefttoright-labels}) gives a handlebody 
version of Ariki--Koike algebras.
For $g=2$ this case can be interpreted as an extended 
affine version of Ariki--Koike algebras.
\end{remark}

Our next aim to find a basis 
and a dimension formula 
for $\setstuff{H}_{g,n}^{\dpar,\bpar}$.

\begin{lemma}\label{lemma:hecke-blobsorder}
If $b_{u,1}$ is of 
order $d_{u}$, then $b_{u,i}$ is also of 
order $d_{u}$ for all $1\leq i\leq n$.
\end{lemma}

\begin{proof}
Since the relations in \fullref{lemma:blob-hecke-rels} 
preserve the number and the type of the blobs involved,
the result follows at once.
\end{proof}

\begin{proposition}\label{proposition:heckeblobcyclo-elements-basis}
The set 
\begin{gather}\label{eq:hecke-blob-basis}
\left\{ 
b_{k_{1},i_{1}}^{a_{1}}\dots 
b_{k_{m},i_{m}}^{a_{m}}H_{w} 
\,\middle\vert\,
\begin{gathered}
w\in\setstuff{S}_{n},
m\in\N,
\bsym{a}\in\N^{m},
\\
(\bsym{u},\bsym{i})\in(\{1,\dots,g\}\times\{1,\dots,n\})^{m},
i_{1}\leq\dots\leq i_{m},a_{j}<d_{j}
\end{gathered}
\right\} 
\end{gather}
is a $\KK$-basis of $\setstuff{H}_{g,n}^{\dpar,\bpar}$.
\end{proposition}

\begin{proof}
That the set spans is a consequence 
of \fullref{lemma:blob-hecke-rels}, which we use to pull 
blobs to the bottom,
and \fullref{lemma:hecke-blobsorder}, which gives 
the restriction $a_{j}<d_{j}$.
Now let $A$ be the $\KK$-span of $\{b_{k_{1},i_{1}}^{a_{1}}\dots  b_{k_{m},i_{m}}^{a_{m}}j\}$ with the same indexing sets as in 
\eqref{eq:hecke-blob-basis}. Let further 
$J_{u}\subset\setstuff{H}_{g,n}^{+}$ 
the two-sided ideal 
generated by  all elements having $d_{u}$ blobs 
on the first strand. 
Linear independence follows from 
\fullref{lemma:heckeblob-elements-basis} 
together with the observation that $A\cap J_{u}=\emptyset$. 
\end{proof}

Recall the blob numbers $\bbvar_{g,\dpar}$ from \eqref{eq:blob-numbers}.

\begin{proposition}
The dimension of the 
free $\KK$-module $\setstuff{H}_{g,n}^{\dpar,\bpar}$ is
\begin{gather*}
\dim_{\KK}\setstuff{H}_{g,n}^{\dpar,\bpar}
=(\bbvar_{g,\dpar})^{n}n!. 
\end{gather*}
\end{proposition}

\begin{proof}
There are $n!$ elements $H_{w}$, each one having $n$ strands.
Since every strand can carry up 
to $\bbvar_{g,\dpar}$ blobs the 
claim follows from \fullref{proposition:heckeblobcyclo-elements-basis}. 
\end{proof}

\begin{remark}
Recall that $\bbvar_{0,\dpar}=1$ and $\bbvar_{1,\dpar}=\dvar_{1}$, see 
\fullref{example:blob-numbers}. Thus, we recover the dimensions
\begin{gather*}
\dim_{\KK}\setstuff{H}_{0,n}^{\dpar,\bpar}=n!,
\quad
\dim_{\KK}\setstuff{H}_{1,n}^{\dpar,\bpar}=\dvar_{1}^{n}n!,
\end{gather*}
the former being well-known, of course, the latter 
appears in \cite[(3.10)]{ArKo-hecke-algebra}.
\end{remark}

To construct a sandwich cell datum one can use the same strategy 
as in \fullref{subsection:heckealgebras}. 
Keeping the above discussion in mind, {\eg} 
\fullref{proposition:heckeblobcyclo-elements-basis}, the 
only difference is that $\sand$ 
is now the algebra obtained by taking the quotient of
$\setstuff{L}_{g,n}H_{\lambda}$ by \eqref{eq:hecke-cyclotomic}. 
The sandwiched basis we choose is the restriction of 
\eqref{eq:hecke-blob-basis}. 
Otherwise nothing changes and we obtain a set 
\begin{gather*}
\{c_{D,b,U}^{\lambda}\mid\lambda\in\Lambda,D,U\in M_{\lambda},
b\in\sand\},
\end{gather*} 
as well as the two results, where the denote the 
aforementioned quotient by $\setstuff{L}_{g,n}^{\dpar,\bpar}H_{\lambda}$:

\begin{proposition}
The above defines a sandwich
cell datum for $\setstuff{H}_{g,n}^{\dpar,\bpar}$.\qed
\end{proposition}

\begin{theorem}\label{theorem:checke}
Let $\KK$ be a field.
\begin{enumerate}

\item A $\lambda\in\Lambda$ is an apex 
if and only if $\lambda$ is $e$-restricted.

\item The simple $\setstuff{H}_{n,g}^{\dpar,\bpar}$-modules of 
apex $\lambda\in\Lambda$ 
are parameterized by simple modules of $\setstuff{L}_{g,n}^{\dpar,\bpar}H_{\lambda}$.

\item The simple $\setstuff{H}_{n,g}^{\dpar,\bpar}$-modules of 
apex $\lambda\in\Lambda$ can be constructed as 
the simple heads of
$\mathrm{Ind}_{\setstuff{L}_{g,n}^{\dpar,\bpar}
H_{\lambda}}^{\setstuff{H}_{n,g}^{\dpar,\bpar}}(K)$, 
where $K$ runs over (equivalence classes of) 
simple $\setstuff{L}_{g,n}^{\dpar,\bpar}H_{\lambda}$-modules.\qed

\end{enumerate}
\end{theorem}

\begin{example}
For $g=0$ nothing new happens of course 
compared to \fullref{theorem:cellhecke}, 
while the $g=1$ version of \fullref{theorem:checke} 
is a non-explicit version of \cite[Theorem 3.30]{DiJaMa-cyclotomic-q-schur}.
\end{example}

\end{document}